\documentclass[reqno,a4paper]{amsart}
\pdfoutput=1

\usepackage[T1]{fontenc}
\usepackage{lmodern}
\usepackage{microtype}

\usepackage{amssymb}
\usepackage{stmaryrd} % for llbracket/rrbracket
\usepackage{mathrsfs}
\usepackage{mathdots}

\usepackage{fullpage}

\makeatletter
\let\origsection\section
\renewcommand\section{\@ifstar{\starsection}{\nostarsection}}

\newcommand\nostarsection[1]
{\sectionprelude\origsection{#1}\sectionpostlude}

\newcommand\starsection[1]
{\sectionprelude\origsection*{#1}\sectionpostlude}

\newcommand\sectionprelude{%
  \vspace{0.5em}
}

\newcommand\sectionpostlude{%
}
\makeatother

\usepackage{mathtools}
\usepackage{subcaption}
\usepackage[]{enumitem}
\setlist{itemsep=1ex,topsep=1ex} % without parskip
\setlist[1]{leftmargin=*}
\setlist[itemize,1]{labelindent=1em}
\setlist[itemize,2]{leftmargin=2pc,labelsep=*}

\usepackage[backref]{hyperref} %\usepackage[backref]{hyperref} % this should come before amsrefs; backref for ref check

\usepackage[nobysame,alphabetic,initials,msc-links]{amsrefs}

\DefineSimpleKey{bib}{how}
\DefineSimpleKey{bib}{mrclass}
\DefineSimpleKey{bib}{mrnumber}
\DefineSimpleKey{bib}{fjournal}
\DefineSimpleKey{bib}{mrreviewer}

\renewcommand{\PrintDOI}[1]{%
  \href{http://dx.doi.org/#1}{{\tt DOI:#1}}%
}
\renewcommand{\eprint}[1]{#1}
\BibSpec{book}{%
    +{}  {\PrintPrimary}                {transition}
    +{.} { \PrintDate}                  {date}
    +{.} { \textit}                     {title}
    +{.} { }                            {part}
    +{:} { \textit}                     {subtitle}
    +{,} { \PrintEdition}               {edition}
    +{}  { \PrintEditorsB}              {editor}
    +{,} { \PrintTranslatorsC}          {translator}
    +{,} { \PrintContributions}         {contribution}
    +{,} { }                            {series}
    +{,} { \voltext}                    {volume}
    +{,} { }                            {publisher}
    +{,} { }                            {organization}
    +{,} { }                            {address}
    +{,} { }                            {status}
    +{,} { \PrintDOI}                   {doi}
    +{,} { \PrintISBNs}                 {isbn}
    +{}  { \parenthesize}               {language}
    +{}  { \PrintTranslation}           {translation}
    +{;} { \PrintReprint}               {reprint}
    +{.} { }                            {note}
    +{.} {}                             {transition}
    +{}  {\SentenceSpace \PrintReviews} {review}
}
\BibSpec{article}{%
    +{}  {\PrintAuthors}                {author}
    +{,} { \textit}                     {title}
    +{.} { }                            {part}
    +{:} { \textit}                     {subtitle}
    +{,} { \PrintContributions}         {contribution}
    +{.} { \PrintPartials}              {partial}
    +{,} { }                            {journal}
    +{}  { \textbf}                     {volume}
    +{}  { \PrintDatePV}                {date}
    +{,} { \issuetext}                  {number}
    +{,} { \eprintpages}                {pages}
    +{,} { }                            {status}
    +{,} { \PrintDOI}                   {doi}
    +{,} { \eprint}        {eprint}
    +{}  { \parenthesize}               {language}
    +{}  { \PrintTranslation}           {translation}
    +{;} { \PrintReprint}               {reprint}
    +{.} { }                            {note}
    +{.} {}                             {transition}
    +{}  {\SentenceSpace \PrintReviews} {review}
}
\BibSpec{collection.article}{%
    +{}  {\PrintAuthors}                {author}
    +{,} { \textit}                     {title}
    +{.} { }                            {part}
    +{:} { \textit}                     {subtitle}
    +{,} { \PrintContributions}         {contribution}
    +{,} { \PrintConference}            {conference}
    +{}  {\PrintBook}                   {book}
    +{,} { }                            {booktitle}
    +{,} { \PrintDateB}                 {date}
    +{,} { pp.~}                        {pages}
    +{,} { }                            {publisher}
    +{,} { }                            {organization}
    +{,} { }                            {address}
    +{,} { }                            {status}
    +{,} { \PrintDOI}                   {doi}
    +{,} { \eprint}        {eprint}
    +{}  { \parenthesize}               {language}
    +{}  { \PrintTranslation}           {translation}
    +{;} { \PrintReprint}               {reprint}
    +{.} { }                            {note}
    +{.} {}                             {transition}
    +{}  {\SentenceSpace \PrintReviews} {review}
}
\BibSpec{misc}{%
  +{}{\PrintAuthors}  {author}
  +{,}{ \textit}      {title}
  +{.}{ }             {how}
  +{}{ \parenthesize} {date}
  +{,} { available at \eprint}        {eprint}
  +{,}{ available at \url}{url}
  +{,}{ }             {note}
  +{.}{}              {transition}
}

\usepackage{tikz}
\usetikzlibrary{knots,cd,backgrounds}
\usetikzlibrary{math}

\numberwithin{equation}{section}

\theoremstyle{plain}
\newtheorem{Thm}{Theorem}[section]
\newtheorem{Lem}[Thm]{Lemma}
\newtheorem{Prop}[Thm]{Proposition}
\newtheorem{Cor}[Thm]{Corollary}

\newtheorem{Thmintro}{Theorem}

\theoremstyle{definition}
\newtheorem{Def}[Thm]{Definition}

\theoremstyle{remark} % no extra space above or below
\newtheorem{Rem}[Thm]{Remark}

\newcommand\bp{\begin{proof}}
\newcommand\ep{\end{proof}}

\mathchardef\mhyph="2D

\newcommand{\msE}{\mathscr{E}}

\newcommand{\msK}{\mathscr{K}}

\newcommand{\msR}{\mathscr{R}}

\newcommand{\mfa}{\mathfrak{a}}
\newcommand{\mfb}{\mathfrak{b}}

\newcommand{\mfg}{\mathfrak{g}}
\newcommand{\mfh}{\mathfrak{h}}
\newcommand{\mfk}{\mathfrak{k}}

\newcommand{\mfm}{\mathfrak{m}}
\newcommand{\mfn}{\mathfrak{n}}
\newcommand{\mfp}{\mathfrak{p}}
\newcommand{\mfq}{\mathfrak{q}}
\newcommand{\mfs}{\mathfrak{s}}
\newcommand{\mfsl}{\mathfrak{sl}}

\newcommand{\mfsu}{\mathfrak{su}}
\newcommand{\mft}{\mathfrak{t}}
\newcommand{\mfu}{\mathfrak{u}}

\newcommand{\mfz}{\mathfrak{z}}
\newcommand{\mfX}{\mathfrak{X}}

\newcommand{\mcC}{\mathcal{C}}

\newcommand{\mcE}{\mathcal{E}}
\newcommand{\mcF}{\mathcal{F}}
\newcommand{\mcG}{\mathcal{G}}
\newcommand{\mcH}{\mathcal{H}}

\newcommand{\mcJ}{\mathcal{J}}
\newcommand{\mcK}{\mathcal{K}}

\newcommand{\mcM}{\mathcal{M}}
\newcommand{\mcO}{\mathcal{O}}
\newcommand{\mcR}{\mathcal{R}}
\newcommand{\mcS}{\mathcal{S}}
\newcommand{\mcT}{\mathcal{T}}
\newcommand{\mcU}{\mathcal{U}}

\newcommand{\mbc}{\mathbf{c}}

\newcommand{\mbs}{\mathbf{s}}
\newcommand{\mbt}{\mathbf{t}}

\newcommand{\C}{\mathbb{C}}
\newcommand{\K}{\mathbb{K}}
\newcommand{\N}{\mathbb{N}}
\newcommand{\Q}{\mathbb{Q}}
\newcommand{\R}{\mathbb{R}}
\newcommand{\Z}{\mathbb{Z}}

\newcommand{\rC}{\mathrm{C}}
\newcommand{\rH}{\mathrm{H}}

\DeclareMathOperator{\ad}{ad}
\DeclareMathOperator{\Ad}{Ad}
\DeclareMathOperator{\Alt}{Alt}

\DeclareMathOperator{\End}{End}

\DeclareMathOperator{\Hom}{Hom}

\DeclareMathOperator{\Rep}{Rep}
\DeclareMathOperator{\Irr}{Irr}
\DeclareMathOperator{\Sym}{Sym}
\DeclareMathOperator{\Tr}{Tr}

\DeclareMathOperator{\Tor}{Tor}

\DeclareMathOperator{\ord}{ord}
\DeclareMathOperator{\diag}{diag}

\DeclareMathOperator{\sgn}{sgn}
\DeclareMathOperator{\sech}{sech}

\newcommand{\id}{\mathrm{id}} % we are primarilly treating id as an object rather than operator; (we need \mathord for better spacing between \otimes)
\newcommand{\hotimes}{\mathbin{\hat\otimes}}

\newcommand{\ns}{\mathrm{ns}}
\newcommand{\nc}{\mathrm{nc}}
\newcommand{\htt}{\mathrm{ht}}
\newcommand{\cH}{\mathrm{cH}}
\newcommand{\rBC}{\mathrm{BC}}

\newcommand{\KZ}{\mathrm{KZ}}
\newcommand{\inte}{\mathrm{int}}
\newcommand{\quasiK}{\mathfrak{X}}

\newcommand{\RE}{\mathrm{RE}}
\newcommand{\Sk}{\mathrm{Sk}}
\newcommand{\tB}{\tilde{B}}

\newcommand{\HH}{\mathrm{HH}}
\newcommand{\fps}{\llbracket h \rrbracket}
\newcommand{\fpser}{\fps}%{\llbracket \hbar \rrbracket}
\newcommand{\fLauser}{[ h^{-1}, h \rrbracket}
\newcommand{\OHbimod}{{}_{\mathrm{res}}\mcO(H)_\counit}
\newcommand{\counit}{\epsilon}

\newcommand{\hlf}[1]{\frac{#1}{2}}
\newcommand{\absv}[1]{\left\lvert#1\right\rvert}
\newcommand{\medwedge}{{\textstyle\bigwedge}}
\newcommand{\qbin}[3]{\genfrac{[}{]}{0pt}{}{#1}{#2}_{#3}}

\newcommand{\eps}{\varepsilon}

\tikzset{
  x=0.7cm, % unit lengths
  y=0.7cm,
  knot diagram/every knot diagram/.append style={
    consider self intersections,
    clip width=5,
    end tolerance=0.05pt,
    clip radius=7pt,
  },
  knot diagram/every strand/.append style={
    thick,
  }
}
\newcommand{\stwbr}[2][]{% #1: style (optional), #2: y-orig; "simple" twist braid starting from (1,y-orig)
\strand[#1] (1, #2+2) .. controls +(0,-0.3) and +(0,0.3) .. (0,#2+1)
.. controls +(0,-0.3) and +(0,0.3) .. (1,#2); }

\newcommand{\braidgen}[3][]{% #1: style (optional), #2: x-orig, #3: y-orig; crossing between (x-orig, y-orig) -> (x-orig+1,y-orig+1) and (x-orig+1,y-orig) -> (x-orig,y-orig+1)
\strand[#1] (#2,#3) .. controls +(0,0.3) and +(0,-0.3) ..
(#2+1,#3+1); \strand[#1] (#2+1,#3) .. controls +(0,0.3) and
+(0,-0.3) .. (#2,#3+1); }

\begin{document}

\title{Comparison of quantizations of symmetric spaces: cyclotomic Knizhnik--Zamolodchikov equations and Letzter--Kolb coideals}

\date{September 13, 2020; minor changes February 28, 2023}

\author{Kenny De Commer}
\address{Vrije Universiteit Brussel}
\email{kenny.de.commer@vub.be}

\author{Sergey Neshveyev}
\address{Universitetet i Oslo}
\email{sergeyn@math.uio.no}

\author{Lars Tuset}
\address{OsloMet - storbyuniversitetet}
\email{larst@oslomet.no}

\author{Makoto Yamashita}
\address{Universitetet i Oslo}
\email{makotoy@math.uio.no}

\thanks{The work of K.DC.~was supported by the FWO grants G025115N and G032919N. The work of S.N.~and M.Y.~was partially supported by the NFR funded project 300837 ``Quantum Symmetry''. M.Y.~also acknowledges support by Grant for Basic Science Research Projects from The Sumitomo Foundation and JSPS Kakenhi 18K13421 at an early stage of collaboration.}

\begin{abstract}
We establish an equivalence between two approaches to quantization of irreducible symmetric spaces of compact type within the framework of quasi-coactions, one based on the Enriquez--Etingof cyclotomic Knizhnik--Zamolodchikov (KZ) equations and the other on the Letzter--Kolb coideals.
This equivalence can be upgraded to that of ribbon braided quasi-coactions, and then the associated reflection operators ($K$-matrices) become a tangible invariant of the quantization.
As an application we obtain a Kohno--Drinfeld type theorem on type $\mathrm{B}$ braid group representations defined by the monodromy of KZ-equations and by the Balagovi\'{c}--Kolb universal $K$-matrices.
The cases of Hermitian and non-Hermitian symmetric spaces are significantly different. In particular, in the latter case a quasi-coaction is essentially unique, while in the former we show that there is a one-parameter family of mutually nonequivalent quasi-coactions.
\end{abstract}

\maketitle

\tableofcontents

\section*{Introduction}

This paper is about quantization of symmetric spaces of compact type.
It will be sufficient to concentrate on the irreducible simply connected symmetric spaces of \emph{type I}, that is, the spaces of the form $U/U^\sigma$ for a compact simply connected simple Lie group $U$ with an involutive automorphism $\sigma$.
Our approach is motivated by the groundbreaking work of Drinfeld \cite{MR1047964}, in which he gave a new algebraic proof of Kohno's theorem \cite{MR927394} on equivalence of the braid group representations that appear as deformations of representations of the symmetric group on tensor powers of some representation of $\mfg=\mfu^\C$. The representations in question are defined by the monodromy of the Knizhnik--Zamolodchikov (KZ) equations, on the one hand, and by the universal $R$-matrix of the Hopf algebraic deformation $U_h(\mfg)$ of the universal enveloping algebra $U(\mfg)$, on the other.

Drinfeld developed a framework of \emph{quasi-triangular quasi-bialgebras}, which captures both types of representations.
He showed that a deformation of $U(\mfg)$ among such quasi-bialgebras is controlled by the \emph{co-Hochschild cohomology} of the coalgebra $U(\mfg)$, up to a natural notion of equivalence derived from tensor categorical considerations.
This cohomology is the exterior algebra $\bigwedge \mfg$, and the part giving the deformation parameter is the one-dimensional space $(\bigwedge^3 \mfg)^\mfg$.
Moreover, this parameter is detected by the eigenvalues of the square of the braiding.

In the course of developing the theory Drinfeld also clarified the geometric structures behind such deformations.
Namely, the first order terms of the deformations correspond to Poisson--Lie group structures on~$U$, or structures of a Lie bialgebra on $\mfu$.
The two types of representations of the braid groups arise from different models of quantizations of Poisson--Lie groups, and Drinfeld's result says that such quantizations are essentially unique. In hindsight, his result can be interpreted as an instance of the \emph{formality principle}, which roughly says that deformations of algebraic structures are controlled by first order terms through a quasi-isomorphism of differential graded Lie algebras.

Having understood quantizations of Poisson--Lie groups, one natural next direction is to look at quantizations of the Poisson homogeneous spaces.
The first important step towards a classification of such spaces was again made by Drinfeld~\cite{MR1243249}: for the standard Poisson--Lie group structure on $U$, they correspond to the real Lagrangian subalgebras of~$\mfg$. A complete classification of these spaces (with connected stabilizers) was then given by Karolinsky~\cite{MR1485682}.

The first classification result for quantizations of Poisson homogeneous spaces was obtained by Podle\'{s} \cite{MR919322}.
He classified the actions of Woronowicz's compact quantum group $\mathrm{SU}_q(2)$ \cite{MR890482} with the same spectral pattern as that of $\mathrm{SU}(2)$ acting on (the functions on) the $2$-sphere $S^2$.
In other words, he considered coactions of the C$^*$-bialgebra $C(\mathrm{SU}_q(2))$, which is a deformation of the algebra of continuous functions on $\mathrm{SU}(2)$ and is dual to (an analytic version of) $U_h(\mfsl_2)$.
Podle\'{s} showed that there is a one-parameter family of isomorphism classes of such coactions.
From the geometric point of view this is explained by the fact that the covariant Poisson structures on $S^2$ form a Poisson pencil~\cite{MR1087382}.

Tensor categorical counterparts of Hopf algebra coactions are module categories.
Although the precise correspondence, through a Tannaka--Krein type duality, came later \citelist{\cite{MR1976459}\cite{MR3121622}\cite{MR3426224}}, in the context of quantization of Poisson homogeneous spaces there is already a rich accumulation of results obtained from various angles, all related to the \emph{reflection equation}.

This equation was introduced by Cherednik \cite{MR774205} to study quantum integrable systems on the half-line.
While braiding (Yang--Baxter operator) represents scattering of two particles colliding in a $1$-dimensional system, a solution of the reflection equation (reflection operator) represents the interaction of a particle with a boundary.
Adding this operator to a braided tensor category (where the Yang--Baxter operators live) gives rise to a new category with a larger space of morphisms, which admits a monoidal product of the braided tensor category from one side, thus yielding a module category~\cite{MR1644317}, or more precisely, a \emph{braided module category}~\cite{MR3248737}.

Matrix solutions of the reflection equation for the universal $R$-matrix of quasi-triangular Hopf algebras lead to \emph{coideal subalgebras}, as originally pointed out by Noumi \cite{MR1413836} and further clarified by Kolb--Stokman \cite{MR2565052}.
In this direction, the best understood class is that of \emph{quantum symmetric pairs}, that is, the coideals which are deformations of $U(\mfg^\theta)$ for a conjugate $\theta$ of $\sigma$ such that $\mfg^\theta$ is maximally noncompact relative to the Cartan subalgebra defining the deformation $U_h(\mfg)$.
Following Koornwinder's work \cite{MR1215439} on the dual coideals of the Podle\'{s} spheres, Letzter \cite{MR1717368} developed a systematic way of constructing such coideal subalgebras $U_h^{\mbt}(\mfg^\theta) < U_h(\mfg)$ for finite type Lie algebras, which was refined and extended by Kolb to Kac--Moody Lie algebras \cite{MR3269184}.
Next, a \emph{universal $K$-matrix} for $U_h^{\mbt}(\mfg^\theta)$, which gives reflection operators in the representations of $U_h^{\mbt}(\mfg^\theta)$, was defined by Kolb and Balagovi\'c \citelist{\cite{MR2453228}\cite{MR3905136}} expanding on the earlier work of Bao and Wang~\cite{MR3864017} on the (quasi-split) type $\mathrm{AIII}$ and $\mathrm{AIV}$ cases. The construction relied on a coideal analogue of Lusztig's bar involution \citelist{\cite{MR3864017}\cite{MR3414769}}. Kolb~\cite{MR4048733} further showed, developing on the ideas from \citelist{\cite{MR1644317}\cite{MR3248737}}, that these structures give rise  to \emph{ribbon twist-braided module categories}.

On the dual side, a deformation quantization of $U/U^\sigma$ from the reflection equation was developed by Gurevich, Donin, Mudrov, and others \citelist{\cite{arXiv:math/991141}\cite{MR1705663}\cite{MR1964382}\cite{MR1998103}}.
Here one sees a close connection to the theory of \emph{dynamical $r$-matrices} \citelist{\cite{MR1404026}\cite{MR1612160}}.

There is a parallel theory of module categories over the Drinfeld category, that is, the tensor category of finite dimensional $\mfg$-modules with the associator defined by the monodromy of the KZ-equations.
The basic idea is to add an extra pole in these equations, then the reflection operator appears as a suitably normalized monodromy around it.
Conceptually, the usual KZ-equations give flat connections on the configuration space of points in the complement of type $\mathrm A$ hyperplane configurations, and the modified equations are obtained by looking at the complement of type $\mathrm B$ hyperplane configurations.
Following early works of Leibman \cite{MR1289327} and Golubeva--Leksin \cite{MR1940926} on monodromy of such equations, Enriquez \cite{MR2383601} introduced \emph{cyclotomic KZ-equations}.
He also defined \emph{quasi-reflection algebras}, a particular class of quasi-coactions of quasi-bialgebras, which can be considered as type $\mathrm B$ analogues of quasi-triangular quasi-bialgebras.
This formalism turned out to have powerful applications to quantization of Poisson homogeneous spaces, where the associator of a quasi-coaction gives rise to a quantization of a dynamical $r$-matrix~\cite{MR2126485}.

Based on these developments, and guided by the categorical duality between module categories and Hopf algebraic coactions, we proposed a conjecture on equivalence between the following structures \cite{MR3943480}:
 \begin{itemize}
 \item a category of finite dimensional representations of $\mfg^\sigma$, considered as a ribbon twist-braided module category over the Drinfeld category, with the associator and ribbon twist-braid defined by the cyclotomic KZ-equations;
 \item a category of finite dimensional modules over a Letzter--Kolb coideal $U_h^{\mbt}(\mfg^\theta)$,  considered as a ribbon twist-braided module category over the category of $U_h(\mfg)$-modules, with the ribbon twist-braid defined by the Balagovi\'c--Kolb universal $K$-matrix.
\end{itemize}
To be precise, the conjecture was formulated in the analytic setting, that is, $q=e^h$ was assumed to be a real number and the categories carried unitary structures.
In this paper, we give a proof of the corresponding conjecture in the formal setting using the framework of quasi-coactions.

It should be mentioned that Brochier \cite{MR2892463} has already proved an interesting equivalence between two quasi-coactions on $U_h(\mfh)$, where $\mfh<\mfg$ is the Cartan subalgebra and one of the quasi-coactions comes from the cyclotomic KZ-equations associated with a finite order automorphism $\sigma$ such that $\mfg^\sigma = \mfh$. In his setting, the extra deformation parameter space is the formal group generated by the Cartan algebra.
The construction of the equivalence follows the strategy of \cite{MR1047964}, this time relying on the co-Hochschild cohomology studied by Calaque \cite{MR2233127}.

\smallskip
Now, let us sketch what we concretely carry out:
\begin{itemize}
\item Show that the quasi-coactions of Drinfeld's quasi-bialgebra induced by the cyclotomic KZ-equations are generically universal among the quasi-coactions deforming $\Delta$ on $U(\mfg^\sigma)$.
\item Give a complete classification of the corresponding ribbon twist-braids and show that the corresponding $K$-matrices give a complete invariant of the quasi-coactions.
\item In the Hermitian case (see below), when there is a one-parameter family of nonequivalent quasi-coactions, establish a correspondence with Poisson structures on $U / U^\sigma$ by studying coisotropic subgroups which are conjugates (`Cayley transforms') of $U^\sigma$.
\item Make a concrete comparison with the Letzter--Kolb coideals and the Balagovi\'c--Kolb braided module categorical structures.
\end{itemize}

In the first step, the main idea is to reduce the problem to vanishing of obstructions in a suitable version of the co-Hochschild cohomology.
This strategy is quite standard, see \citelist{\cite{MR1047964}\cite{MR2892463}}, but while these papers relied on the braiding/ribbon braids to have a good control of the cohomology, we work with the cohomology classes directly, analogously to Donin--Shnider's approach \cite{MR1390978} to Lie bialgebra quantization, and the identification of the ribbon twist-braids comes only towards the end.
The relevant co-Hochschild cohomology turns out to be isomorphic to $\bigwedge \mfm^\C$ for $\mfm^\C = \mfg \ominus \mfg^\sigma$, and the deformation of a quasi-coaction is controlled by the invariant part of the second cohomology, that is, $(\bigwedge^2 \mfm^\C)^{\mfg^{\sigma}}$.
Up to complexification, this space can be interpreted as the space of $U$-invariant bivectors on $U/U^\sigma$, hence there is a direct connection to equivariant deformation quantization.
This is where one sees the formality principle in action.

At this point we encounter an important dichotomy between the \emph{Hermitian} and the \emph{non-Hermitian} cases.
Although we already discussed it in \cite{MR3943480} based on the parameters $\mbt$ for the coideals $U_h^{\mbt}(\mfg^\theta)$, the following observation is perhaps more illuminating: the dimension of $(\bigwedge^2 \mfm^\C)^{\mfg^{\sigma}}$ is either zero or one, and is equal to that of the center of $\mfg^\sigma$.
In the Hermitian case, and only in this case, this dimension is one and the corresponding homogeneous space $U/U^\sigma$ has an invariant Hermitian structure, induced by an element of the center of $\mfg^\sigma$ (hence the name).

In the non-Hermitian case, the triviality of the center eliminates cohomological obstructions, quickly leading to rigidity of the algebra structure and coaction homomorphisms on $U(\mfg^\sigma)$.
Our results in this case can be summarized as follows.

\begin{Thmintro}[Section \ref{ssec:class-non-Hermitian} and Theorem~\ref{thm:uniq-tw-br-nonherm-case}]\label{thm:mainA}
Let $\mfu^\sigma < \mfu$ be a non-Hermitian irreducible symmetric pair.
Suppose that $\alpha \colon U(\mfg^\sigma)\fpser \to U(\mfg^\sigma) \otimes U(\mfg)\fpser$ and $\Psi \in U(\mfg^\sigma) \otimes U(\mfg)^{\otimes 2} \fpser$ define a quasi-coaction of Drinfeld's quasi-bialgebra $(U(\mfg)\fpser, \Delta, \Phi_{\KZ})$ that deforms $\Delta\colon U(\mfg^\sigma) \to U(\mfg^\sigma) \otimes U(\mfg)$, and let $(\alpha', \Psi')$ be another such pair.
Then $(\alpha, \Psi)$ and $(\alpha', \Psi')$ are obtained from each other by twisting. Moreover, the quasi-coaction $(U(\mfg^\sigma)\fpser, \alpha, \Psi)$ admits a unique ribbon $\sigma$-braid $\mcE$ with prescribed constant term $\mcE^{(0)}\in 1\otimes Z(U)$.
\end{Thmintro}

In the above formulation the ribbon twist-braid is allowed to live in a certain completion of $U(\mfg^\sigma) \otimes U(\mfg)\fps$.
Namely, consider the \emph{multiplier algebra} of the algebra of finitely supported functions on the dual of $U^\sigma$ \cite{MR1378538}, which is the direct product of full matrix algebras
\[
\mcU(G^\sigma) = \prod_\pi \End(V_\pi),
\]
where $\pi$ runs over the irreducible finite dimensional representations of $\mfg^\sigma$ which appear in finite dimensional representations of $\mfg$.
We can further define
\[
\mcU(G^\sigma \times G^n) = \prod_{\pi, \pi_1, \dots, \pi_n} \End(V_\pi) \otimes \End(V_{\pi_1}) \otimes \dots \otimes \End(V_{\pi_n}),
\]
where $\pi_1, \dots \pi_n$ run over the irreducible finite dimensional representations of $\mfg$.
Then we take $\mcE$ as an element of $\mcU(G^\sigma\times G)\fps$.

The situation is more interesting in the Hermitian case.
Even up to equivalence defined by twisting, the quasi-coactions are no longer unique. In this case we show that generic quasi-coactions are equivalent to the ones arising from the cyclotomic KZ-equations with prescribed coefficients \citelist{\cite{MR2126485}\cite{MR3943480}}: the associator $\Psi_{\KZ,s; \mu}$, for parameters $s \in \C\setminus i\Q^\times$ and $\mu \in h\C\fpser$, is given as the normalized monodromy from $w = 0$ to $w = 1$ of the differential equation
\[
H'(w) = \left( \frac{\hbar(t^\mfk_{12} - t^\mfm_{12})}{w + 1} + \frac{\hbar t^\mfu_{12}}{w-1} + \frac{\hbar(2 t^\mfk_{01} + C^\mfk_1) + (s+\mu) Z_1}{w} \right) H(w).
\]
Here we put $\hbar = \frac{h}{\pi i}$, and the coefficients are given as follows:  $t^\mfu$, $t^\mfk$, $t^\mfm $ are the canonical $2$-tensors of $\mfu$, $\mfk = \mfu^\sigma$, and $\mfm = \mfu \ominus \mfk$ respectively, $C^\mfk$ is the Casimir element of $\mfk$ associated to $t^\mfk$, and $Z$ is a normalized element of $\mfz(\mfk)$.

If $s=0$, then $\Psi_{\KZ,s;\mu}$ makes sense in $U(\mfg^\sigma) \otimes U(\mfg)^{\otimes 2} \fpser$, but otherwise we can only say that $\Psi_{\KZ,s;\mu}$ is in $\mcU(G^\sigma\times G^2)\fps$.
It is therefore convenient to start working with the multiplier algebras throughout instead of the universal enveloping algebras.
Fortunately, the concepts of quasi-bialgebras and quasi-coactions have straightforward formulations in this setting, and from the categorical point of view this formalism is actually even more natural when dealing with semisimple module categories.
Then $(\mcU(G^\sigma)\fpser, \Delta, \Psi_{\KZ, s; \mu})$ is a quasi-coaction of $(\mcU(G)\fpser, \Delta, \Phi_{\KZ})$, and our results can be summarized as follows.

\begin{Thmintro}[Theorems \ref{thm:modified-KZ-universality} and \ref{thm:uniq-tw-br-herm-case}]
Let $\mfu^\sigma < \mfu$ be an irreducible Hermitian symmetric pair, and let $\omega$ be an invariant symplectic form on $U / U^\sigma$.
There is a countable subset $A \subset \C$ with the following property: if $\alpha \colon \mcU(G^\sigma)\fpser \to \mcU(G^\sigma \times G)\fpser$ and $\Psi \in \mcU(G^\sigma \times G^2) \fpser$ define a quasi-coaction of $(\mcU(G)\fpser, \Delta, \Phi_{\KZ})$ that deforms $\Delta\colon \mcU(G^\sigma) \to \mcU(G^\sigma\times G)$, and the first order term $\Psi^{(1)}$ of $\Psi$ satisfies $\langle \omega, \Psi^{(1)} \rangle \in \C \setminus A$, then there is a pair $(s, \mu)$, unique up to translation by $(2i\Z,0)$, such that $(\mcU(G^\sigma)\fpser, \alpha, \Psi)$ is equivalent to $(\mcU(G^\sigma)\fpser, \Delta, \Psi_{\KZ, s; \mu})$.
Moreover, $(\mcU(G^\sigma)\fpser, \alpha, \Psi)$ admits a unique ribbon $\sigma$-braid $\mcE$ with prescribed constant term $\mcE^{(0)}\in 1\otimes\exp(-\pi i s Z)Z(U)$.
\end{Thmintro}

We resolve the cohomological obstruction to equivalence by looking at the expansion of $\Psi_{\KZ,s; \mu}$, where we follow Enriquez and Etingof's work \cite{MR2126485} on quantization of dynamical $r$-matrices. Up to a coboundary, $\Psi_{\KZ,s; \mu}$ has the expansion
\begin{equation*}
\Psi_{\KZ, s; \mu} \sim 1 - \hlf{h}\tanh\mathopen{}\left(\hlf{\pi (s + \mu)}\right)  \sum_{\alpha \in \Phi_\nc^+} \frac{(\alpha,\alpha)}{2} 1 \otimes \left( X_\alpha \otimes X_{-\alpha} - X_{-\alpha} \otimes X_{\alpha} \right) + \cdots,
\end{equation*}
where $\Phi_\nc^+$ is the set of positive roots in $\mfm^\C$ with respect to a choice of Cartan subalgebra in~$\mfg^\sigma$, and $X_{\pm \alpha}$ is a normalized root vector for $\pm \alpha$, see Sections \ref{ssec:cyclotomic-KZ} and \ref{sec:class-assoc-herm} for details.
This shows that, under a perturbation of $\mu$, the associator changes in the term one order higher than the perturbation, with a precise control of the cohomology class (formal Poisson structure) of the difference in that term. This leads to the universality of quasi-coactions with the associators $\Psi_{\KZ,s; \mu}$ and
can be interpreted as `poor man's formality' for equivariant deformation quantization.

We next apply these results to the Letzter--Kolb coideals. Since our classification is formulated in the framework of multiplier algebras, we show that the coideals indeed give rise to such structures, essentially by taking a completion. It should be stressed that the formalism of multiplier algebras is important not only for making sense of $\Psi_{\KZ, s; \mu}$. The second and even more important reason is that it allows us to check that the coactions defined by the Letzter--Kolb coideals are twistings of $\Delta$. The point is that since~$\mfg^\sigma$ is not semisimple in the Hermitian case, the standard arguments based on Whitehead's first lemma are not applicable. By working with the multiplier algebras, which are built out of semisimple algebras, we can circumvent the nonvanishing of Lie algebraic cohomological obstructions. We still need to use Letzter's result~\cite{MR1742961} on existence of spherical vectors for this, which means that we have to consider $*$-coideals $U^\mbt_h(\mfg^\theta)$.

Next, in the Hermitian case, we have to verify the condition on the first order term $\Psi^{(1)}$. For this we study Poisson homogeneous structures on $U/U^\sigma$.
More precisely, we have to compare two Poisson structures, corresponding to two ways we obtain the quasi-coactions.
On the one hand, from the cyclotomic KZ-equations we obtain a Poisson pencil \cite{MR1357743}, where one takes the sum of the left action of the standard $r$-matrix $r$ on $U / U^\sigma$ and a scalar multiple of the Kostant--Kirillov--Souriau bracket, which agrees with the bracket defined by the right action of $r$.
On the other hand, from the coideals we obtain the reduction of the Sklyanin bracket to quotients by \emph{coisotropic subgroups}.

Starting from the model $\sigma = \theta$ in the maximally noncompact position, where the subgroup is coisotropic~\cite{MR2102330}, we take a distinguished one-parameter family of subgroups $U^{\theta_\phi}$ that are conjugate to~$U^\theta$ by interpolated Cayley transforms, and show that the associated fixed point subgroups $U^{\theta_\phi}$ remain coisotropic. At the level of Lie algebras, this construction interpolates between the maximally noncompact subalgebra~$\mfg^\theta$ and the maximally compact one $\mfg^\nu$ (which contains $\mfh$).
Moreover, the Lie algebras $\mfg^{\theta_\phi}$ turn out to be the classical limits of the Letzter--Kolb $*$-coideals $U^\mbt_h(\mfg^\theta)$.
By a detailed analysis of the Cayley transforms, we are able to find the relation between the parameters $\phi$ and $\mbt$, as well as to compute the cohomology classes of $\Psi^{(1)}$ for the associators we get. In a bit imprecise form these results are summarized as follows.

\begin{Thmintro}[Theorems \ref{thm:qbialg-compar-KZ-LK-non-Hermitian}, \ref{thm:compar-rib-tw-br-KZ-LK-non-Hermitian}, \ref{thm:qbialg-compar-KZ-LK-Hermitian}, and \ref{compar-rib-tw-br-KZ-Hermitian}]\label{thm:main}
There is a parameter set $\mcT^*$ (consisting of one point $\mbt=0$ in the non-Hermitian case) defining $*$-coideals $U^\mbt_h(\mfg^\theta)$ and satisfying the following properties.
For every $\mbt\in\mcT^*$, the coideal $U^\mbt_h(\mfg^\theta)$ gives rise to a coaction of a multiplier bialgebra which is equivalent to the quasi-coaction $(\mcU(G^\theta_\mbt)\fps,\Delta,\Psi_{\KZ,s;\mu})$ of $(\mcU(G)\fps,\Delta,\Phi_\KZ)$, where $G^\theta_\mbt<G$ is a subgroup conjugate to $G^\theta$, while $s\in\R$ and $\mu\in h\R\fps$ are uniquely determined parameters (equal to $0$ in the non-Hermitian case), with $s$ given by an explicit formula.
Under this equivalence, the Balagovi\'c--Kolb ribbon twist-braids correspond to the ones coming from the cyclotomic KZ-equations.
\end{Thmintro}

This implies a Kohno--Drinfeld type result (Theorems \ref{thm:Kohno-Drinfeld-non-Herm} and \ref{thm:Kohno-Drinfeld-Herm}) for quantum symmetric pairs, stating that representations of type $\mathrm B$ braid groups arising from the coideals and the cyclotomic KZ-equations are equivalent.

A formula for the parameter $\mu$ in Theorem~\ref{thm:main} can in principle be obtained by comparing the eigenvalues of the reflection operators in the two pictures.
In the general case this step might be somewhat involved, but at least for the $\mathrm{AIII}$ case (which corresponds to the symmetric pairs $\mfs(\mfu_p \oplus \mfu_{N-p}) < \mfsu_{N}$) this can be done thanks to the classification of reflection operators by Mudrov \cite{MR1917138}.

So far we have discussed the case of irreducible symmetric spaces of type I, i.e., $U/U^{\sigma}$ with $U$ simple.
However, the \emph{type II} case, corresponding to $U$ itself as a symmetric space, or the quotient of $U \times U$ by the diagonal subgroup, can be handled in essentially the same way as the non-Hermitian type~I cases.
In particular, Theorems \ref{thm:mainA} and \ref{thm:main} can be adapted to this case. This implies that an analogue of Theorem~\ref{thm:main} holds in general for Letzter--Kolb $*$-coideals of $U_h(\mfg)$ with $\mfg$ semisimple.

\medskip
Let us now briefly summarize the contents of the paper.
In Section \ref{sec:prelim} we recall basic definitions and introduce conventions which are used throughout the paper.

In Section \ref{sec:classification} we prove our main conceptual results on classification of quasi-coactions and ribbon twist-braids.
As explained above, the non-Hermitian case is done by a more or less standard cohomological argument, while in the Hermitian case we look into the structure of the associators arising from the cyclotomic KZ-equations.

In Section \ref{sec:int-subgroups} we focus on the Hermitian case and look at conjugates of $\mfu^\sigma < \mfu$ in the maximally compact position by interpolated Cayley transforms.
We show that these conjugates generate coisotropic subgroups and relate them to models arising from the cyclotomic KZ-equations, with an explicit formula for the first order term.

In Section \ref{sec:Letzter-Kolb} we explain how the quantized universal enveloping algebra and the Letzter--Kolb coideals fit into our setting of multiplier quasi-bialgebras and their quasi-coactions.

Finally, in Section \ref{sec:compar} we combine the results of the previous sections and prove our main comparison theorems. We finish the section with a detailed analysis of the $\mathrm{AIII}$ case.

There are three appendices, in which we collect some technical but not fundamentally new results used in the paper.

\medskip
Let us close the introduction with some further problems. First of all, a general formula for $\mu$ in Theorem~\ref{thm:main} would be nice to find, especially if this can be done in a unified way rather than via a case-by-case analysis. Second, the analytic version of the conjecture, as originally proposed in \cite{MR3943480}, remains to be settled, together with a comparison with  the `Vogan picture' introduced there. On the geometric side, one would like to extend the above results in the Hermitian case to all coadjoint orbits of~$U$.

\smallskip
\emph{Acknowledgements}:
K.DC.~thanks A. Brochier for discussions in an early stage of this project.
S.N.~is grateful to P.A.~{\O}stv{\ae}r for a reference.
M.Y.~thanks D.~Jordan and A.~Appel for stimulating discussions.
Last but not least, we would like to thank the anonymous reviewers for their careful reading and valuable suggestions to improve the readability.

\section{Preliminaries}
\label{sec:prelim}

\subsection{Conventions}

We treat $h$ as a formal variable, and put $h^* = h$ when we consider $*$-algebraic structures.
We put
\[
q=e^h\quad\text{and}\quad\hbar = \frac{h}{\pi i},
\]
the latter is mostly reserved for the KZ-equations. We denote the space of formal power series with coefficients in $A$ by
\[
A\fps = \biggl\{a = \sum_{n=0}^\infty  h^n a^{(n)} \biggm| a^{(n)} \in A \biggr\},
\]
and the space of Laurent series by
\[
A\fLauser = \biggl\{a = \sum_{n=k}^\infty h^n a^{(n)} \biggm| a^{(n)} \in A, k \in \Z \biggr\}.
\]
For $a \in A\fLauser$, we denote the smallest $n$ such that $a^{(n)} \neq 0$ by $\ord(a)$.

We denote the $h$-adically completed tensor product of $\C\fpser$-modules by $\hotimes$.
In particular, we have $(A\fpser) \hotimes (B\fpser) = (A \otimes B)\fpser$.

When $A = \C$ and $a \in \C\fpser$ has constant term $a^{(0)}>0$, we take its $n$th root $b = a^\frac{1}{n}$ to be the unique solution of $b^n = a$ such that $b^{(0)}$ is positive. A similar convention is used for $\log$.

\subsection{Simple Lie groups}

Throughout the entire paper $\mfu$ denotes a compact simple Lie algebra and $\mfg$ denotes its complexification. The connected and simply connected Lie groups corresponding to $\mfg$ and~$\mfu$ are denoted by~$G$ and~$U$.

We denote by $(\cdot,\cdot)_\mfg$ the unique invariant symmetric bilinear form on $\mfg$ such that, for any Cartan subalgebra $\mfh < \mfg$, its dual form on $\mfh^*$ has the property that $(\alpha,\alpha)=2$ for every short root $\alpha$. Let $t^\mfu\in\mfu^{\otimes 2}$ be the corresponding invariant tensor:
\begin{equation}
\label{eq:inv-2-tens}
t^\mfu=\sum_i X_i\otimes X^i,
\end{equation}
where $(X_i)_i$ is a basis in $\mfg$ and $(X^i)_i$ is the dual basis.

Recall that $(\cdot,\cdot)_\mfg$ is negative definite on $\mfu$. Therefore, if we define an antilinear involution $*$ on $\mfg$ by letting $X^*=-X$ for $X\in\mfu$, then $(X,Y^*)_\mfg$ becomes an $(\Ad U)$-invariant Hermitian scalar product on~$\mfg$.

We denote the category of finite dimensional algebraic representations of the linear algebraic group~$G$ (equivalently, finite dimensional representations of $\mfg$) by $\Rep G$. It is equivalent to the category of finite dimensional unitary representations of $U$. We write $\pi \in \Rep G$ to say that $\pi$ is a finite dimensional representation of $G$, its underlying space is denoted by $V_\pi$. We also fix a set $\Irr G$ of representatives of the isomorphism classes of irreducible representations.

We will often have to extend the scalars to $\C\fps$. Denote the category we get by $(\Rep G)\fps$. Thus, the objects of $(\Rep G)\fps$ are the $G$-modules over $\C\fps$ that are isomorphic to the modules of the form~$V_\pi\fps$ for $\pi\in\Rep G$.

\subsection{Multiplier algebras}

For $n = 1, 2, \dots$, we put
\[
\mcU(G^n) = \prod_{\substack{\pi_i \in \Irr G,\\ i = 1, \dots, n}} \End(V_{\pi_1}) \otimes \dots \otimes \End(V_{\pi_n})
\]
We view $G$ and $\mfg$ as subsets of $\mcU(G) = \mcU(G^1)$.

Since for every irreducible $\pi\in\Rep G$ there is a unique up to a scalar factor $U$-invariant Hermitian scalar product on $V_\pi$, we have a canonical involution $*$ on~$\mcU(G^n)$. There is also a unique homomorphism
\[
\Delta\colon \mcU(G) \to \mcU(G^2)
\]
characterized by the identities $(\pi_1 \otimes \pi_2)(\Delta(T))S=S\pi(T)$ for all intertwiners $S\colon V_\pi\to V_{\pi_1}\otimes V_{\pi_2}$.
Then $\Delta(g) = g \otimes g$ for $g \in G$. This characterizes the elements of $G$ among the nonzero elements of $\mcU(G)$. Similarly,
the identity $\Delta(X) = X \otimes 1 + 1 \otimes X$ for $X \in \mfg$ characterizes $\mfg$ inside~$\mcU(G)$.

Denote by $\mcO(G)$ the Hopf algebra of regular functions (matrix coefficients of finite dimensional representations) on $G$.
We occasionally write $\mcO(U)$ instead of $\mcO(G)$ when we think of it as a function algebra on~$U$.

There is a nondegenerate pairing between $\mcU(G)$ and $\mcO(G)$ that allows us to identify $\mcU(G)$ with the dual space of $\mcO(G)$.
Concretely, if $\pi$ is irreducible, $T\in\End(V_\pi)$, $v\in V_\pi$, $\ell\in V_\pi^*$, then for the matrix coefficient $a_{v,\ell}\in\mcO(G)$, $a_{v,\ell}(g)=\ell(\pi(g)v)$, we have
\[
\langle a_{v,\ell},T\rangle=\ell(T v),
\]
and $\langle f,T\rangle=0$ for the matrix coefficients $f$ of the irreducible representations $\pi'$ inequivalent to $\pi$.
Similarly, $\mcU(G^n)$ is the linear dual of $\mcO(G)^{\otimes n}$. With respect to this duality the bialgebra structures are related by
\begin{align*}
\langle f_1 \otimes f_2, \Delta(T) \rangle &= \langle f_1 f_2, T \rangle,&
\langle \Delta(f), T_1 \otimes T_2 \rangle &= \langle f, T_1 T_2 \rangle
\end{align*}
for $f_i \in \mcO(G)$ and $T_i \in \mcU(G)$.

\smallskip

We can do the same constructions  for any reductive linear algebraic group $H$ over $\C$. We then also define
\[
\mcU(H \times G^n) = \smashoperator[r]{\prod_{\substack{\pi \in \Irr H, \pi_i \in \Irr G,\\ i = 1, \dots, n}}} \End(V_\pi) \otimes \End(V_{\pi_1}) \otimes \dots \otimes \End(V_{\pi_n})
\]
for $0 \le n < \infty$. In a more invariant form, $\mcU(H \times G^n)$ is the linear dual of $\mcO(H\times G^n)$.

Assume in addition that $H$ is an algebraic subgroup of $G$.
Then the embedding $H\to G$ extends to an embedding of $\mcU(H^{n+1})$ into $\mcU(H\times G^n)$. In particular, the comultiplication $\Delta\colon\mcU(H)\to\mcU(H^2)$ can be viewed as a homomorphism $\mcU(H)\to\mcU(H\times G)$.

Note that in general $H$ is not simply connected. In Lie algebraic terms the category $\Rep H$ consists of the finite dimensional representations of $\mfh$ that are subrepresentations of the finite dimensional representations of $\mfg$ restricted to $\mfh$.

\subsection{Quasi-coactions and ribbon twist-braids}
\label{sec:quasi-coact-rib-tw-br}

The notion of a \emph{quasi-bialgebra} \cite{MR1047964} has a straightforward adaptation to the setting of multiplier algebras, cf.~\cite{MR2832264}*{Section~2}. We will be interested in multiplier quasi-bialgebras of the form $(\mcU(G)\fpser,\Delta_h,\epsilon_h,\Phi)$.
Thus, $\Delta_h$ is a \emph{nondegenerate} homomorphism $\mcU(G)\fpser\to\mcU(G^2)\fpser$, meaning that the images of the idempotents $\Delta_h(\id_{V_\pi})$ ($\pi\in \Irr G$) in $\End(V_{\pi_1}\otimes V_{\pi_2})\fpser$ add up to $1$, $\epsilon_h\colon\mcU(G)\fpser\to\C\fpser$ is a nondegenerate homomorphism, and $\Phi\in\mcU(G^3)\fpser$ is an invertible element (with $\Phi^{(0)}=1$) satisfying the same identities as in~\cite{MR1047964}*{Section~1}.

The assumption of nondegeneracy for the counit $\epsilon_h$ implies that it is determined by its restrictions to the blocks $\End(V_\pi)\fps$ of $\mcU(G)\fps$. Since there are no nonzero ($\C\fps$-linear) homomorphisms $\End(V)\fps\to\C\fps$ for $\dim V>1$ and there is a unique such homomorphism for $\dim V=1$, we conclude that $\epsilon_h$ coincides with the standard counit~$\epsilon$ on~$\mcU(G)\fps$. From now on we will therefore omit $\epsilon_h$ from the notation for a multiplier quasi-bialgebra.

Given a reductive algebraic subgroup $H$ of $G$, a \emph{quasi-coaction} of $(\mcU(G)\fpser,\Delta_h,\Phi)$ on $\mcU(H)\fpser$ is given by a nondegenerate homomorphism $\alpha\colon \mcU(H)\fpser \to \mcU(H \times G)\fps$ and an \emph{associator} $\Psi \in \mcU(H \times G^2)\fpser$ satisfying $\Psi^{(0)}=1$,
\[
(\id\otimes\epsilon)\alpha=\id,
\]
\begin{equation}
\label{eq:mod-assoc}
\Psi (\alpha \otimes \id) \alpha (T) = (\id \otimes \Delta_h) \alpha(T) \Psi \quad (T \in \mcU(H)\fpser),
\end{equation}
the \emph{mixed pentagon equation}
\begin{equation}
\label{eq:mix-pent}
\Phi_{1,2,3} \Psi_{0, 12, 3} \Psi_{0, 1, 2} =  \Psi_{0, 1, 23} \Psi_{01, 2, 3},
\end{equation}
with $\Psi_{01, 2, 3} = (\alpha \otimes \id)(\Psi)$, $\Psi_{0, 12, 3} = (\id_{\mcU(H)} \otimes \Delta_h \otimes \id)(\Psi)$, etc., and the normalization conditions
\[
(\id\otimes\epsilon\otimes\id)(\Psi) = (\id\otimes\id\otimes\epsilon)(\Psi)=1.
\]

A multiplier quasi-bialgebra $(\mcU(G)\fpser,\Delta_h,\Phi)$ defines a tensor category $((\Rep G)\fps,\otimes_h,\Phi)$, where the tensor product $\otimes_h$ on $(\Rep G)\fps$ is defined using $\Delta_h$ and the associativity isomorphism is given by the action of $\Phi$.
A quasi-coaction as above defines then the structure of a \emph{right $((\Rep G)\fps,\otimes_h,\Phi)$-module category} on $(\Rep H)\fps$.
Namely, the functor $\odot_\alpha\colon (\Rep H)\fps\times(\Rep G)\fps\to(\Rep H)\fps$ defining the module category structure is induced by $\alpha$, while the associativity morphisms are defined by the action of~$\Psi$. See~\cite{MR3943480}*{Section~1} for more details, but note that in~\cite{MR3943480} we worked in the analytic setting, meaning that $q=e^h$ was a real number and $\Phi\in\mcU(G^3)$, $\Psi\in\mcU(H\times G^2)$.

Next, let $\mcR \in \mcU(G^2) \fpser$ be an $R$-matrix (with $\mcR^{(0)}=1$) for $(\mcU(G)\fpser,\Delta_h,\Phi)$, that is, $\mcR\Delta_h(\cdot)=\Delta_h^{\mathrm{op}}(\cdot)\mcR$ and $\mcR$ satisfies the \emph{hexagon relations}.
Let $\beta$ be an automorphism of the quasi-triangular multiplier quasi-bialgebra $(\mcU(G)\fps,\Delta_h,\Phi, \mcR)$.
A \emph{ribbon $\beta$-braid} is given by an invertible element $\mcE \in \mcU(H \times G)\fpser$ satisfying
\begin{gather}
\mcE (\id \otimes \beta) \alpha(T) = \alpha(T) \mcE \quad (T \in \mcU(H)\fpser),\label{eq:rib-sig-tw-0}\\
(\alpha \otimes \id)(\mcE) = \Psi^{-1} \mcR_{21} \Psi_{021} \mcE_{02} (\id \otimes\id\otimes \beta)(\Psi_{021}^{-1} \mcR_{12} \Psi),\label{eq:rib-sig-tw-1}\\
(\id \otimes \Delta_h)(\mcE) = \mcR_{21} \Psi_{021} \mcE_{02} (\id \otimes\id\otimes \beta)(\Psi_{021}^{-1} \mcR_{12} \Psi) \mcE_{01} (\id \otimes \beta \otimes \beta)(\Psi^{-1}).\label{eq:rib-sig-tw-2}
\end{gather}
When $\beta$ is the identity map, we just say ``ribbon braid'' instead of ``ribbon $\id$-braid''.
We want to stress that, as opposed to $\Phi$, $\Psi$ and $\mcR$, we do not require $\mcE^{(0)}=1$. A quadruple $(\mcU(H)\fps, \alpha, \Psi, \mcE)$ satisfying~\eqref{eq:rib-sig-tw-0} and~\eqref{eq:rib-sig-tw-1} is a version of a \emph{quasi-reflection algebra} \cite{MR2383601}. In categorical terms, the action of $\mcE$ on $M\odot_\alpha N$ defines the structure of a \emph{ribbon $\beta$-braided module category} on $((\Rep H)\fps,\odot_\alpha, \Psi)$.
See again~\cite{MR3943480}*{Section~1} for more details.

\subsection{Twisting}\label{ssec:twisting}

We can transform a quasi-coaction $(\mcU(H)\fps, \alpha, \Psi)$ of $(\mcU(G)\fpser,\Delta_h,\Phi)$ into a new one as follows.
Suppose that we are given elements $\mcF\in\mcU(G^2)\fps$ and $\mcG\in \mcU(H\times G)\fps$ such that $\mcF^{(0)}=1$, $\mcG^{(0)}=1$ and
\begin{align*}
(\epsilon\otimes\id)(\mcF) &= (\id\otimes\epsilon)(\mcF)=1,&
(\id\otimes\epsilon)(\mcG) &= 1.
\end{align*}
Then the \emph{twisting of the quasi-coaction by $(\mcF,\mcG)$} is the quasi-coaction $(\mcU(H)\fpser,\alpha_\mcG,\Psi_{\mcF,\mcG})$ of the multiplier quasi-bialgebra $(\mcU(G)\fpser,\Delta_{h,\mcF},\Phi_\mcF)$, where
\begin{align*}
\Delta_{h,\mcF}&=\mcF\Delta_h(\cdot)\mcF^{-1}, & \Phi_{\mcF}&=(1\otimes \mcF)(\id\otimes \Delta_h)(\mcF)\Phi(\Delta_h\otimes\id)(\mcF^{-1})(\mcF^{-1}\otimes 1),\\
\alpha_\mcG&=\mcG\alpha(\cdot)\mcG^{-1}, &
\Psi_{\mcF,\mcG}&=(1\otimes\mcF)(\id\otimes \Delta_h)(\mcG)\Psi (\alpha\otimes \id)(\mcG^{-1})(\mcG^{-1}\otimes 1).
\end{align*}

Twisting defines an equivalence relation on the quasi-coactions. In categorical terms it means that we pass from $((\Rep G)\fps,\otimes_h,\Phi)$ to the equivalent tensor category $((\Rep G)\fps,\otimes_{h,\mcF},\Phi_\mcF)$, with the tensor product defined by $\Delta_{h,\mcF}$, and, up to this equivalence, the $((\Rep G)\fps,\otimes_h,\Phi)$-module category $((\Rep H)\fps,\odot_\alpha,\Psi)$ is equivalent to the $((\Rep G)\fps,\otimes_{h,\mcF},\Phi_\mcF)$-module category $$((\Rep H)\fps,\odot_{\alpha_\mcG},\Psi_{\mcF,\mcG}).$$

As the following result shows, twisting often allows one to push all the information on a quasi-coaction into the associators.

\begin{Lem}\label{lem:twisting-to-Delta}
Assume $H$ is a reductive algebraic subgroup of $G$ and $(\mcU(H)\fps, \alpha, \Psi)$ is a quasi-coaction of $(\mcU(G)\fpser,\Delta_h,\Phi)$ such that both $\alpha$ and $\Delta_h$ equal $\Delta$ modulo $h$. Then this quasi-coaction is a twisting of a quasi-coaction  $(\mcU(H)\fps, \Delta, \Psi')$ of $(\mcU(G)\fpser,\Delta,\Phi')$ for some $\Psi'$ and $\Phi'$.
\end{Lem}

\bp
Take irreducible representations $\pi_1$ and $\pi_2$ of $G$. Consider the homomorphisms $f=(\pi_1\otimes\pi_2)\Delta$ and $f_h=(\pi_1\otimes\pi_2)\Delta_h$ from $\mcU(G)\fps$ into $\End(V_{\pi_1}\otimes V_{\pi_2})\fps$. The assumption of nondegeneracy for $\Delta_h$ implies that there exists a finite set $F\subset \Irr G$ such that $f_h$ factors through $\bigoplus_{\pi\in F}\End(V_\pi)\fps$. By taking $F$ large enough we may assume that the same is true for $f$. Since the algebra $\bigoplus_{\pi\in F}\End(V_\pi)$ is semisimple, there are no nontrivial deformations of any given homomorphism $\bigoplus_{\pi\in F}\End(V_\pi)\to \End(V_{\pi_1}\otimes V_{\pi_2})$. Hence there exists $\mcF_{\pi_1,\pi_2}\in \End(V_{\pi_1}\otimes V_{\pi_2})\fps$ such that $\mcF_{\pi_1,\pi_2}^{(0)}=1$ and $f_h=(\Ad \mcF_{\pi_1,\pi_2})f$. Then $\mcF=(\mcF_{\pi_1,\pi_2})_{\pi_1,\pi_2\in\Irr G}\in\mcU(G^2)\fps$ satisfies $\mcF^{(0)}=1$ and $\Delta_h=\mcF\Delta(\cdot)\mcF^{-1}$. Furthermore, since the counit of $(\mcU(G)\fpser,\Delta_h,\Phi)$ is $\epsilon$, we could take $\mcF_{\pi_1,\pi_2}=1$ if either $\pi_1$ or $\pi_2$ were trivial representations. In this case $\mcF$ would additionally satisfy $(\epsilon\otimes\id)(\mcF)= (\id\otimes\epsilon)(\mcF)=1$.

In a similar way we can find $\mcG\in\mcU(H\times G)\fps$ such that $\mcG^{(0)}=1$, $(\iota\otimes\epsilon)(\mcG)=1$ and $\alpha=\mcG\Delta(\cdot)\mcG^{-1}$. Then the twisting by $(\mcF^{-1},\mcG^{-1})$ gives the required quasi-coaction.
\ep

Next, given a quasi-coaction $(\mcU(H)\fps, \alpha, \Psi)$ of $(\mcU(G)\fpser,\Delta_h,\Phi)$, assume in addition we have an automorphism $\beta$ of $(\mcU(G)\fpser,\Delta_h,\Phi)$. If $\mcF$ satisfies $(\beta\otimes\beta)(\mcF)=\mcF$, then $\beta$ remains an automorphism of $(\mcU(G)\fpser,\Delta_{h,\mcF},\Phi_\mcF)$.
Assume also that $\mcR \in \mcU(G^2) \fpser$ is an $R$-matrix for $(\mcU(G)\fpser,\Delta_h,\Phi)$ that is fixed under $\beta$.
Then $\mcR_{\mcF}=\mcF_{21}\mcR\mcF^{-1}$ is an $R$-matrix for $(\mcU(G)\fpser,\Delta_{h,\mcF},\Phi_\mcF)$, again fixed by $\beta$.
Given a ribbon $\beta$-braid $\mcE$ for the original quasi-coaction we get a ribbon $\beta$-braid $\mcE_\mcG$ for the twisted quasi-coaction $(\mcU(H)\fpser,\alpha_\mcG,\Psi_{\mcF,\mcG})$ of $(\mcU(G)\fpser,\Delta_{h,\mcF},\Phi_\mcF,\mcR_\mcF)$ defined by
\begin{equation}\label{eq:braid-twist0}
\mcE_{\mcG}=\mcG\mcE(\id\otimes\beta)(\mcG)^{-1}.
\end{equation}

The condition $(\beta\otimes\beta)(\mcF)=\mcF$ can be relaxed, we will return to this in Section~\ref{ssec:comparison}.

\subsection{Symmetric pairs}

Let $\mfk$ be a proper Lie subalgebra of $\mfu$. We say that $\mfk<\mfu$ is a \emph{symmetric pair}, or more precisely, an \emph{irreducible symmetric pair of type I}, if there is a (necessarily unique) involutive automorphism $\sigma$ of $\mfu$ such that $\mfk=\mfu^\sigma$.
Whenever convenient we extend $\sigma$ to $\mcU(G)$, in particular, to~$\mfg$.
Let $K=U^\sigma$.
The compact group~$K$ is connected by~\cite{MR1834454}*{Theorem~VII.8.2}.
Using the Cartan decomposition of $G$ we can also conclude that $G^\sigma$ is connected.

Given such a symmetric pair, put
\[
\mfm=\{X\in\mfu\mid \sigma(X)=-X\},
\]
which is the orthogonal complement of $\mfk$ in $\mfu$ with respect to the invariant inner product.
We also write $\mfm^\C = \mfm \otimes_\R \C$ for its complexification.

We say that a symmetric pair $\mfk<\mfu$ is \emph{Hermitian}, if $U/K$ is a Hermitian symmetric space. Such symmetric pairs are equivalently characterized by either of the following conditions, see~\cite{MR1661166}*{Proposition~VI.1.3}:
\begin{itemize}
\item the center $\mfz(\mfk)$ is nontrivial (and $1$-dimensional);
\item the space $\mfm$ has a (unique up to a sign) $\mfk$-invariant complex structure.
\end{itemize}

The following closely related characterization will be crucial for us.

\begin{Lem}
\label{lem:inv-wedge}
For any symmetric pair $\mfk<\mfu$, we have $\dim (\bigwedge^2 \mfm)^\mfk = 1$ if $\mfk < \mfu$ is Hermitian, and $\dim (\bigwedge^2 \mfm)^\mfk = 0$ otherwise. We always have $\mfm^\mfk = 0$.
\end{Lem}

\bp
Since $U$ is simple by assumption, $U/K$ is an irreducible symmetric space, so $K$ acts irreducibly on $\mfm = T_{[e]}(U/K)$. As $U/K$ is not one-dimensional, this cannot be the trivial action, and we get $\mfm^\mfk = 0$.

Next, since $\mfm$ has a $\mfk$-invariant inner product, the space $(\bigwedge^2 \mfm)^\mfk$ is isomorphic to the space of $\mfk$-invariant skew-adjoint operators on $\mfm$. Assume we are given such a nonzero operator $A$. Then $A^2=-A^*A$ is self-adjoint, with negative eigenvalues. Hence $A^2$ is diagonalizable, and by irreducibility of the action of $\mfk$ on $\mfm$ we conclude that $A^2$ must be a strictly negative scalar. Therefore by rescaling $A$ we get a $\mfk$-invariant complex structure on $\mfm$. Since there is a unique up to a sign such structure in the Hermitian case and no such structure in the non-Hermitian case, we get the result.
\ep

\begin{Rem}\label{rem:centr-non-herm}
In the non-Hermitian case, the centralizer $Z_U(K)$ of $K$ in $U$ is a finite group that either agrees with the center $Z(U)$ of $U$, or contains it as a subgroup index $2$.
Indeed, we have $\mfz_\mfu(\mfk)=0$ from the above lemma, which implies the finiteness of $Z_U(K)$.
To see that the index of $Z(U)$ in $Z_U(K)$ is at most $2$, observe that any $K$-intertwiner on $\mfm$ has to be a (real) scalar by the vanishing of $(\bigwedge^2 \mfm)^\mfk$.
Then, given $g \in Z_U(K)$, the restriction of the finite-order $K$-intertwiner $\Ad_g$ to $\mfm$ should be $\pm 1$, which implies that either $g \in Z(U)$ or $\Ad_g = \sigma$.
If the inclusion $\mfk < \mfu$ is of \emph{equal rank}, there are elements $g$ satisfying $\Ad_g = \sigma$, hence we obtain $[Z_U(K) : Z(U)] = 2$.
Otherwise there is no such $g$, hence we obtain $Z_U(K) = Z(U)$.
\end{Rem}

An \emph{irreducible symmetric pair of type II} is an inclusion that is isomorphic to the diagonal inclusion of~$\mfu$ into $\mfu \oplus \mfu$ (with a simple compact Lie algebra $\mfu$).
This corresponds to the involution $\sigma(X,Y) = (Y,X)$ on~$\mfu \oplus \mfu$.
For such a pair we can put $\mfm = \{(X, -X) \mid X \in \mfu\}$. Since both $\mfm^{\mfu}$ and $(\medwedge^2 \mfm)^\mfu$ are trivial, such pairs behave in many respects similarly to the non-Hermitian type I pairs. We will therefore mostly focus on the type I case and only make a few remarks on the type II case.

Back to type I symmetric pairs, in the Hermitian case, it is known that an invariant complex structure on $\mfm$ is defined by an element of $\mfz(\mfk)$. The correct normalization is given by the following.

\begin{Lem}
\label{lem:eigenvals-adz}
Assuming that $\mfk<\mfu$ is a Hermitian symmetric pair, let $Z \in \mfz(\mfk)$ be a vector such that $(Z,Z)_\mfg=-1$.
Then on $\mfm$ we have $(\ad Z)^2=-a_\sigma^2\id$, where
\begin{equation*}
a_\sigma = \sqrt{\frac{2 h^\vee c}{\dim \mfm}},
\end{equation*}
$c\in\{1, 2, 3\}$ is the ratio of the square lengths of long and short roots of $\mfg$ and $h^\vee$ is the dual Coxeter number of $\mfg$.
\end{Lem}

\bp
As $(\ad Z)|_{\mfm}$ is $\mfk$-invariant and skew-adjoint, by the proof of the previous lemma we have $(\ad Z)^2=-a^2$ on $\mfm$ for some scalar $a\ge0$. Hence for the Killing form $B_{\mathrm{Kill}}$ on $\mfg$ we have~$B_{\mathrm{Kill}}(Z,Z)=\Tr((\ad Z)^2) = - a^2 \dim \mfm$.
The Killing form and the normalized bilinear form $(\cdot,\cdot)_\mfg$ are related by $B_{\mathrm{Kill}}=2 h^\vee c (\cdot,\cdot)_\mfg$, see~\cite{MR1104219}*{Chapter 6, Exercise~2}. Combining this with $(Z, Z)_\mfg = -1$, we get that $a=a_\sigma$.
\ep

\begin{Cor}
The $\mfk$-invariant complex structures on $\mfm$ are given by $\pm\frac{1}{a_\sigma} (\ad Z)|_\mfm$. For the involutive automorphism~$\sigma$ such that $\mfk=\mfu^\sigma$ we have
\begin{equation*}
\sigma = \exp\mathopen{}\left(\frac{\pi}{a_\sigma} \ad Z\right).
\end{equation*}
\end{Cor}

In particular, we see that $K$ is the stabilizer of $Z$ in $U$ with respect to the adjoint action. As the adjoint and coadjoint representations are equivalent, this leads to yet another known characterization of the Hermitian symmetric pairs: a symmetric pair $\mfk<\mfu$ is Hermitian if and only if the homogeneous $U$-space $U/K$ is isomorphic to a coadjoint orbit of $U$.

\section{Classification of quasi-coactions and ribbon braids}\label{sec:classification}

Throughout this section $\mfk=\mfu^\sigma < \mfu$ denotes a symmetric pair.
Our goal is to classify using the co-Hochschild cohomology a class of quasi-coactions of $(\mcU(G)\fps,\Delta_h,\Phi)$ on $\mcU(G^\sigma)\fps$.

\subsection{Co-Hochschild cohomology for multiplier algebras}
\label{sec:geom-comput}

The \emph{co-Hochschild cochains} will play a central role in this paper. Let $H$ be a reductive algebraic subgroup of $G$. Put $\tB_{G,H}^n = \mcU(H \times G^n)$ for $0 \le n < \infty$, and define a differential $\tB_{G,H}^n\to \tB_{G,H}^{n+1}$ by
\begin{equation}
\label{eq:cohoch-diff}
d_{\cH}(T) = T_{01,2,\dots,n+1} - T_{0,1 2, \dots, n+1} + \dots + (-1)^n T_{0,1,\dots,n(n+1)} + (-1)^{n+1} T_{0,1,\dots,n},
\end{equation}
where $T_{0, \dots, j j + 1, \dots, n+1} = (\id_{\mcU(H \times G^{j-1})} \otimes \Delta \otimes \id_{\mcU(G^{n-j})})(T)$ and $T_{0,1,\dots,n}=T\otimes1$. The group $H$ acts diagonally by conjugation on $\mcU(H \times G^n)$, the differential $d_{\cH}$ is equivariant with respect to this action.
We put
\[
B_{G,H}^n=(\tB_{G,H}^n)^{H}.
\]

\begin{Prop}
\label{prop:comput-cohoch-mult-model}
The cohomology of $\tB_{G,H}$ is isomorphic to the exterior algebra $\bigwedge \mfg/\mfh$ as a graded $H$-module.
\end{Prop}

\bp
The complex $\tB_{G,H}$ is the algebraic linear dual of $\tB_{G,H}' = (\mcO(H) \otimes \mcO(G)^{\otimes n})_{n=0}^\infty$ with the differential $d\colon \mcO(H) \otimes \mcO(G)^{\otimes n} \to \mcO(H) \otimes \mcO(G)^{\otimes n-1}$ given by
\[
d(f_0 \otimes f_1 \otimes \dots \otimes f_n) = \sum_{i=0}^{n-1} (-1)^i f_0 \otimes \dots \otimes f_i f_{i+1} \otimes \dots \otimes f_n + (-1)^n f_n(e) f_0 \otimes \dots \otimes f_{n-1},
\]
where $f_0 f_1$ is the product of $f_0$ and the restriction of $f_1$ to $H$. Thus, the cohomology of $\tB_{G,H}$ is the linear dual of the homology of $\tB_{G,H}'$ as an $H$-module.

The complex $\tB_{G,H}'$ is the standard complex computing the Hochschild homology
\[
\HH_*(\mcO(G), \OHbimod) = \Tor^{\mcO(G) \otimes \mcO(G)}_*(\mcO(G), \OHbimod),
\]
where the bimodule $\OHbimod$ has the underlying space $\mcO(H)$ with the bimodule structure $f.a.f' = f'(e) f a$ for $f, f' \in \mcO(G)$ and $a \in \mcO(H)$. In other words, we are computing
\[
\Tor^{\mcO(G\times G)}_*(\mcO(\Delta), \mcO(H\times\{e\})),
\]
where $\Delta\subset G\times G$ is the diagonal. By~\cite{MR0354655}*{Proposition~VII.2.5}, this is the exterior algebra on $\Tor^{\mcO(G\times G)}_1(\mcO(\Delta), \mcO(H\times\{e\}))$, and the latter is the conormal space of $H \subset G$ at the point~$e$.
Since this conormal space is the dual of $\mfg/\mfh$, we obtain the assertion.
\ep

\begin{Rem}
Proposition~\ref{prop:comput-cohoch-mult-model} and its proof are valid for any linear algebraic group $G$ over $\C$ and any algebraic subgroup $H$, if we define $\mcU(H \times G^n)$ as the dual of $\mcO(H) \otimes \mcO(G)^{\otimes n}$.
\end{Rem}

\begin{Cor} \label{cor:cohom-B-G-H}
For any reductive algebraic subgroup $H < G$, the cohomology of $B_{G,H}$ is isomorphic to $(\bigwedge \mfg/\mfh)^H$.
\end{Cor}

\bp
As the factors $\End(V_\pi) \otimes \End(V_{\pi_1}) \otimes \dots \otimes \End(V_{\pi_n})$ of $\tB^n_{G,H}$ decompose into direct sums of isotypical components, taking the $H$-invariant part commutes with taking cohomology.
\ep

We will mainly need the following particular case.

\begin{Cor}
\label{cor:cohom-B-G-Gtheta}
The cohomology of $B_{G,G^\sigma}$ is isomorphic to $(\bigwedge \mfm^\C)^{\mfk}$.
\end{Cor}

\bp
This follows from the previous corollary, since $G^\sigma$ is connected and $\mfg^\sigma=\mfk^\C$.
\ep

\begin{Rem}
Instead of the multiplier algebras we could use the universal enveloping algebras and define complexes~$\tB_{\mfg,\mfh}$ and $B_{\mfg,\mfh}$, see Appendix~\ref{sec:appendix-dg}. The canonical maps $U(\mfg) \to \mcU(G)$ and $U(\mfh) \to \mcU(H)$ are injective homomorphisms compatible with the coproduct maps $\mcU(G) \to \mcU(G \times G)$ and $\mcU(H) \to \mcU(H \times H)$. Thus we get an inclusion $\tB_{\mfg,\mfh}\to\tB_{G,H}$, and if $H$ is connected, we also get an inclusion $B_{\mfg,\mfh}\to B_{G,H}$.
Corollary~\ref{cor:Buk-cohom} shows that these maps are quasi-isomorphisms.
\end{Rem}

\subsection{Classification of associators and ribbon braids: non-Hermitian case} \label{ssec:class-non-Hermitian}

Assume the symmetric pair $\mfk<\mfu$ is non-Hermitian.

Consider a multiplier quasi-bialgebra $(\mcU(G)\fps,\Delta_h,\Phi)$ such that $\Delta_h=\Delta$ modulo $h$. We claim that up to twisting by $(1,\mcG)$ it has at most one quasi-coaction $(\mcU(G^\sigma)\fps,\alpha,\Psi)$ such that $\alpha=\Delta$ modulo~$h$. Since by Lemma~\ref{lem:twisting-to-Delta} we may assume that both $\Delta_h$ and $\alpha$ equal $\Delta$, the following is an equivalent statement.

\begin{Thm}
\label{thm:Psi-rigid}
Let $\mfk = \mfu^\sigma < \mfu$ be a non-Hermitian symmetric pair, and $\Psi,\Psi'\in\mcU(G^\sigma\times G^2)\fpser$ be two associators defining quasi-coactions of $(\mcU(G)\fps,\Delta,\Phi)$, with the coaction homomorphisms $\alpha=\alpha' = \Delta$.
Then there is an element $\mcH \in 1+h\mcU(G^\sigma \times G)^\mfk\fpser$ such that $(\id\otimes\epsilon)(\mcH)=1$ and $\Psi = \mcH_{0, 12} \Psi' \mcH_{01, 2}^{-1} \mcH_{0,1}^{-1}$.
\end{Thm}

\bp
Suppose that $\Psi^{(k)} = \Psi'^{(k)}$ for $k < n$.
We claim that there is $T \in \mcU(G^\sigma \times G)^\mfk$ such that $(\id\otimes\epsilon)(T)=0$ and $\Psi$ and $\mcH_{0,12} \Psi' \mcH_{01,2}^{-1} \mcH_{0,1}^{-1}$ have the same terms up to (and including) order $n$ for $\mcH = 1 - h^n T$.
The lemma is then proved by inductively applying this claim and taking the product of the elements $1 - h^n T$ we thus get. Note only that by $\mfk$-invariance the elements $T_{0,1}$ and $T_{01,2}$ obtained at different steps commute with each other.

Take the difference of identities~\eqref{eq:mix-pent} for $\Psi$ and $\Psi'$ and consider the terms of order $n$. Since~$\Psi$ and~$\Psi'$ have the same terms up to order $n-1$, we get
\[
(\Psi^{(n)}-\Psi'^{(n)})_{0, 12, 3} + (\Psi^{(n)}-\Psi'^{(n)})_{0, 1, 2}=(\Psi^{(n)}-\Psi'^{(n)})_{0, 1, 23} + (\Psi^{(n)}-\Psi'^{(n)})_{01, 2, 3}.
\]
Since $\Psi$ and $\Psi'$ are $\mfk$-invariant by~\eqref{eq:mod-assoc}, it follows that $\Psi^{(n)} - \Psi'^{(n)}$ is a cocycle in $B_{G, G^\sigma}^2$.
As we are in the non-Hermitian case, by Corollary \ref{cor:cohom-B-G-Gtheta} and Lemma~\ref{lem:inv-wedge}, we have $\Psi^{(n)}-\Psi'^{(n)}=d_\cH(T)$ for some $T\in \mcU(G^\sigma \times G)^\mfk$. As $(\id\otimes\epsilon\otimes\id)(\Psi^{(n)}-\Psi'^{(n)})=0$, we have $(\id\otimes\epsilon)(T)=0$. Thus $T$ satisfies our claim.
\ep

\begin{Rem}\label{rem:que-rigidity-non-Hermitian}
Analogous results are true at the level of the universal enveloping algebras instead of the multiplier algebras. More precisely, given a quasi-bialgebra $(U(\mfg)\fps,\Delta_h,\Phi)$ such that $\Delta_h=\Delta$ modulo~$h$, up to twisting by $(1,\mcG)$ there is at most one quasi-coaction  $(U(\mfg^\sigma)\fps,\alpha,\Psi)$ of this quasi-bialgebra such that $\alpha=\Delta$ modulo $h$. This is proved along the same lines as Lemma~\ref{lem:twisting-to-Delta} and Theorem~\ref{thm:Psi-rigid}, but now relying on Whitehead's lemma for the semisimple Lie algebras $\mfg$ and $\mfg^\sigma$ to show that there are no nontrivial deformations of $\Delta\colon U(\mfg)\to U(\mfg)\otimes U(\mfg)$ and $\Delta\colon U(\mfg^\sigma)\to U(\mfg^\sigma)\otimes U(\mfg)$, and using Corollary~\ref{cor:Buk-cohom} instead of Corollary~\ref{cor:cohom-B-G-Gtheta}.
\end{Rem}

Next let us fix a $\sigma$-invariant $R$-matrix $\mcR \in \mcU(G^2)\fpser$ for $(\Delta, \Phi)$ and look at compatible ribbon $\sigma$-braids.
Note that the left hand side of \eqref{eq:rib-sig-tw-0} becomes $\mcE \Delta(T)$ in the present case.

\begin{Thm}
\label{thm:uniq-br-tw-non-Herm}
Let $\mfk = \mfu^\sigma < \mfu$ be a non-Hermitian symmetric pair, and let $\Phi \in \mcU(G^3)^G\fpser$ and $\mcR \in \mcU(G^2)^G\fpser$ be $\sigma$-invariant elements defining the structure of a quasi-triangular multiplier quasi-bialgebra $(\mcU(G)\fpser, \Delta, \Phi, \mcR)$.
Assume further that we are given a quasi-coaction of the form $(\mcU(G^\sigma)\fpser, \Delta, \Psi)$ by this quasi-bialgebra, and that $\mcE \in \mcU(G^\sigma \times G)\fpser$ is a ribbon $\sigma$-braid for $\mcR$.
Then $\mcE^{(0)}=1\otimes g$ for an element $g$ in the centralizer $Z_U(K)$, and any other ribbon $\sigma$-braid, for the same $\Phi$, $\Psi$ and $\mcR$, and with the same order $0$ term, coincides with $\mcE$.
Furthermore, if $\mcR^{(1)}_{\phantom{1}}+\mcR^{(1)}_{21}\ne0$, then $g\in Z(U)$.
\end{Thm}

The group $Z_U(K)$ is finite as remarked in Remark \ref{rem:centr-non-herm}, so we have at most finitely many ribbon $\sigma$-braids.

\bp
From~\eqref{eq:rib-sig-tw-1} we get that $(\Delta\otimes\id)(\mcE^{(0)})=\mcE_{02}^{(0)}$. This implies that $\mcE^{(0)}=1\otimes g$, with $g=(\epsilon\otimes\id)(\mcE^{(0)})\in\mcU(G)$. From~\eqref{eq:rib-sig-tw-2} we then get $\Delta(g)=g\otimes g$, hence $g\in G$. But then~\eqref{eq:rib-sig-tw-0} shows that $g\in Z_G(G^\sigma)$. Finally, as $\mfz_\mfg(\mfk)=\mfz_\mfu(\mfk)^\C=0$, using the Cartan decomposition of $G$ we see that $Z_G(K)=Z_U(K)$, hence $g\in Z_U(K)$.

Assume now that $\mcE'$ is another ribbon $\sigma$-braid with ${\mcE'}{}^{(0)}=1\otimes g$. We want to show that $\mcE'=\mcE$. It will be convenient to first get rid of~$g$. By multiplying both elements by $1\otimes g^{-1}$ on the right we get new ribbon $\tilde\sigma$-braids $\tilde\mcE$ and $\tilde\mcE'$ in $\mcU(G^\sigma\times G)\fpser$ with the order zero terms $1$, where $\tilde\sigma=(\Ad g)\circ\sigma$.

We argue by induction on $n$ that $\tilde\mcE^{(n)}=\tilde\mcE'^{(n)}$. Suppose that we already know that $\tilde\mcE^{(k)} = \tilde\mcE'^{(k)}$ for $k < n$.
Comparing the terms of degree $n$ in \eqref{eq:rib-sig-tw-1} and \eqref{eq:rib-sig-tw-2}, we obtain
\begin{align*}
X_{01,2} &= X_{0,2},&
X_{0,12} &= X_{0,2} + X_{0,1}
\end{align*}
for $X = \tilde\mcE^{(n)} - \tilde\mcE'^{(n)}$.

The first equality says that $X = 1 \otimes Y$ for $Y = (\epsilon \otimes \id)(X) \in \mcU(G)$.
Then the second equality says $\Delta(Y) = Y_1 + Y_2$, that is, $Y$ is primitive, and we obtain $Y \in \mfg$.

Comparing the terms of degree $n$ in \eqref{eq:rib-sig-tw-0}, we see next that $Y$ has to centralize $\mfk$. Hence $Y=0$ and $\tilde\mcE'^{(n)} = \tilde\mcE^{(n)}$.

\smallskip

It remains to prove the last statement of the theorem. So assume $\mcR^{(1)}_{\phantom{1}}+\mcR^{(1)}_{21}\ne0$.
Let us write $\tB_{G}$, $B_\mfg$, etc., instead of $\tB_{G,\{e\}}$, $B_{\mfg,0}$.

By an analogue of~\cite{MR1047964}*{Proposition~3.1} for the multiplier algebras, by twisting $\Phi$ we may assume that $\Phi=1$ modulo~$h^2$.
Such an analogue is proved in the same way as in~\cite{MR1047964} using that the embedding map $\tB_{\mfg}\to \tB_{G}$ is a quasi-isomorphism by Corollary~\ref{cor:Buk-cohom}.

Namely, consider the normalized skew-symmetrization map $\Alt\colon \tB^3_{G}\to \tB^3_{G}$.
This map kills the coboundaries and transforms the cocycles of the form $X_1\otimes X_2\otimes X_3$, $X_i\in\mfg$, into cohomologous ones by Remark~\ref{rem:co-Hochshc-classes}.
But the classes of such cocycles span the entire space $\rH^3(\tB_G)$ by the same remark and Corollary~\ref{cor:Buk-cohom}.
Therefore if $T\in \tB^3_G$ is a cocycle killed by $\Alt$, then it is a coboundary. The hexagon relations imply that $\Alt(\Phi^{(1)})=0$, so $\Phi^{(1)}=d_\cH(T)$ for some $T\in \tB^2_G$. We may assume that $T$ is $G$- and $\sigma$-invariant, since $\Phi^{(1)}$ has these invariance properties. Then the twisting by $\mcF=1-h T$ proves our claim.

Note that twisting does not change the element $\mcR^{(1)}_{\phantom{1}}+\mcR^{(1)}_{21}$. The hexagon relations imply then that $\mcR^{(1)}\in\mfg\otimes\mfg$ (see the proof of \cite{MR1047964}*{Proposition~3.1}), and since $\mcR$ commutes with the image of $\Delta$, we get $\mcR^{(1)}\in(\mfg\otimes\mfg)^\mfg$.
Hence $\mcR^{(1)}=-\lambda t^\mfu$ for some $\lambda\ne0$, where $t^\mfu$ is the normalized invariant $2$-tensor defined by \eqref{eq:inv-2-tens}.

Next, identity~\eqref{eq:mix-pent} implies that $\Psi^{(1)}$ is a $2$-cocycle in $B_{G, G^\sigma}$, hence by twisting we may assume that $\Psi=1$ modulo $h^2$. We remind also that under twisting the ribbon twist-braids transform via formula~\eqref{eq:braid-twist0}.

Now, by looking at the first order terms in~\eqref{eq:rib-sig-tw-1} for a $\tilde\sigma$-braid $\tilde\mcE$, with $\tilde\mcE^{(0)}=1$, we get
\[
\tilde\mcE^{(1)}_{01,2}=-\lambda t^\mfu_{2,1}+\tilde\mcE^{(1)}_{0,2}-(\id\otimes\id\otimes\tilde\sigma)(\lambda t^\mfu_{1,2}).
\]
The tensor $t^\mfu$ lies in $\mfk\otimes\mfk+\mfm\otimes\mfm$.
Denote the components of $t^\mfu$ in  $\mfk\otimes\mfk$ and $\mfm\otimes\mfm$ by $t^\mfk$ and $t^\mfm$, resp. Then the above identity can be written as
\[
\tilde\mcE^{(1)}_{01,2}=-2\lambda t^\mfk_{1,2}+\tilde\mcE^{(1)}_{0,2}+(\id\otimes\id\otimes\Ad g)(\lambda t^\mfm_{1,2})-\lambda t^\mfm_{1,2}.
\]
Applying $\epsilon$ to the $0$th leg and letting $T=(\epsilon\otimes\id)(\tilde\mcE^{(1)})$, we obtain
\begin{equation}\label{eq:mcE1}
\tilde\mcE^{(1)}=-2\lambda t^\mfk+1\otimes T+(\id\otimes\Ad g)(\lambda t^\mfm)-\lambda t^\mfm.
\end{equation}
But we must have $\tilde\mcE^{(1)}\in\mcU(G^\sigma\times G)$. Since $t^\mfm=\sum_j Y_j\otimes Y^j$ for a basis $(Y_j)_j$ in $\mfm$ and the dual basis~$(Y^j)_j$, this is possible only when $\Ad g$ acts trivially on $\mfm$. Hence $g\in Z(U)$.
\ep

\begin{Rem}
Theorems \ref{thm:Psi-rigid} and \ref{thm:uniq-br-tw-non-Herm} also hold for the type II symmetric pairs with appropriate modifications.
Namely, consider $\tilde G = G \times G$ and its diagonal subgroup $\Delta(G) < \tilde G$, which is the fixed point subgroup of the involution $\sigma(g, h) = (h, g)$ on $\tilde G$.
Then for the quasi-coactions of $(\mcU(\tilde G)\fpser, \Delta, \Phi)$ on the multiplier algebra $\mcU(G)\fpser$, with the coaction map extending $G \ni g \mapsto (g, g, g) \in G \times \tilde G$, and with associators $\Psi \in \mcU(G \times \tilde G^2)\fpser$ and ribbon $\sigma$-braids $\mcE \in \mcU(G \times \tilde G)\fpser$, one can easily prove analogues of these theorems.
First, the proof of Theorem \ref{thm:Psi-rigid} carries over almost without a change.
Indeed, its proof relies on Lemma~\ref{lem:inv-wedge} and Corollary \ref{cor:cohom-B-G-Gtheta}, both of which have analogues for $G \simeq \Delta(G) < \tilde G$.
As for Theorem \ref{thm:uniq-br-tw-non-Herm}, we have $\mfz_{\tilde\mfg}(\mfg) = 0$ for the diagonal inclusion $\mfg < \tilde \mfg$, and $Z_{\tilde G}(G) = Z(U) \times Z(U)$, which is enough to adapt the first half of the proof.
\end{Rem}

\subsection{Associators from cyclotomic KZ equations}\label{ssec:cyclotomic-KZ}

We want to extend the results of the previous subsection to the Hermitian case.
Since $\rH^2(B_{G,G^\sigma})$ is now one-dimensional by Lemma~\ref{lem:inv-wedge}, we should expect a one-parameter family of nonequivalent associators.
In this subsection we define a candidate for such a family arising from the cyclotomic KZ-equations.

Thus, assume $\mfk<\mfu$ is a Hermitian symmetric pair.
We have an element $Z\in\mfz(\mfk)$, unique up to a sign, such that
\[
(Z,Z)_\mfg=-a_\sigma^{-2}.
\]
This normalization is equivalent to $(\ad Z)^2=-1$ on $\mfm$ by Lemma~\ref{lem:eigenvals-adz}.
We fix such $Z$ for the rest of this section.
The operator $\ad Z$ has eigenvalues $\pm i$ on $\mfm^\C$. Denote by $\mfm_{\pm}\subset\mfm^\C$ the corresponding eigenspaces.

We remind that we denote the components of  the normalized invariant $2$-tensor $t^\mfu$ in $\mfk\otimes\mfk$ and $\mfm\otimes\mfm$ by $t^\mfk$ and $t^\mfm$, resp. The tensor $t^\mfm$ lies in $\mfm_+\otimes\mfm_- + \mfm_-\otimes\mfm_+$. We denote the components of $t^\mfm$ in $\mfm_\pm\otimes\mfm_\mp$ by $t^{\mfm_\pm}$.
We thus have
\begin{align}\label{eq:interaction-z-tm}
t^\mfm &= t^{\mfm_+}+t^{\mfm_-},&
(\ad Z \otimes \id)(t^{\mfm_\pm}) &= \pm i t^{\mfm_\pm},&
(\id\otimes\ad Z)(t^{\mfm_\pm}) &= \mp i t^{\mfm_\pm}.
\end{align}

Given $s \in \C$, consider the following elements of $\mcU(G^\sigma\times G^2)\fpser$:
\begin{align}
\label{eq:picard-ingr}
A_{-1} &= \hbar(t^\mfk_{12} - t^\mfm_{12}),&
A_1 &= \hbar t^\mfu_{12},&
A_0 &= \hbar(2 t^\mfk_{01} + C^\mfk_1) + s Z_1,
\end{align}
where $C^\mfk$ is the Casimir element of $\mfk$, the image of $t^\mfk$ under the product map $U(\mfk)\otimes U(\mfk)\to U(\mfk)$.
These lead to \emph{the shifted modified $2$-cyclotomic KZ$_2$-equation} \citelist{\cite{MR2126485}\cite{MR3943480}}
\begin{equation}
\label{eq:cyc-KZ-2}
G'(w) = \left( \frac{A_{-1}}{w + 1} + \frac{A_1}{w-1} + \frac{A_0}{w} \right) G(w).
\end{equation}

\begin{Rem} \label{rem:twisting}
Consider a $\C\fLauser$-valued character $\nu$ on $U(\mfk)$ such that $\nu(Z) = -(2 \hbar a_\sigma^2)^{-1}s$. Then the slicing map $\varsigma_\nu = (\nu\otimes\id)\Delta$ is an algebra homomorphism $U(\mfk)\to U(\mfk)\fLauser$ satisfying
\[
(\varsigma_\nu \otimes \id)(2t^\mfk) = 2 t^\mfk + \hbar^{-1}(1 \otimes s Z)
\]
and commuting with the right coaction $\Delta$ by $U(\mfg)$.
In particular, at least formally speaking, \eqref{eq:cyc-KZ-2} is obtained from the case $s = 0$ by slicing.
But since $\nu$ cannot be extended to $U(\mfk)\fps$, one should be careful with this construction.
\end{Rem}

The normalized monodromy $\Psi_{\KZ,s}\in\mcU(G^\sigma\times G^2)\fpser$ of~\eqref{eq:cyc-KZ-2} from $w = 0$ to $w = 1$ is well-defined as long as the operator $\ad(s Z)$ on $\mcU(G)$ does not have positive integers in its spectrum, cf.~\cite{MR2832264}*{Proposition~3.1}.
Since each matrix block $\End(V_\pi)$ in $\mcU(G)$ is generated by the image of $\mfg$, the eigenvalues of $\ad Z$ are~$i n$ for $n \in \Z$ by Lemma~\ref{lem:eigenvals-adz} and our choice of normalization.
Therefore $\Psi_{\KZ,s}$ is well-defined for all $s\not\in i\Q^\times$. The element $\Psi_{\KZ,s}$ together with the coproduct $\Delta\colon\mcU(G^\sigma)\to\mcU(G^\sigma\times G)$ gives a quasi-coaction of $(\mcU(G)\fpser,\Delta,\Phi_\KZ)$ on~$\mcU(G^\sigma)\fpser$, where $\Phi_\KZ=\Phi(\hbar t^\mfu_{12},\hbar t^\mfu_{23})\in\mcU(G^3)\fpser$ is Drinfeld's KZ-associator for~$G$.

In more detail, $\Psi_{\KZ,s}$ is defined as follows.
Under our restrictions on $s$, a standard argument (see, e.g., \cite{MR2832264}*{Proposition~3.1}) shows that there is a unique $\mcU(G^\sigma\times G^2)\fpser$-valued solution $G_0$ of~\eqref{eq:cyc-KZ-2} on $(0,1)$ such that $G_0(w)w^{-A_0}$ extends to an analytic function in the unit disc with value $1$ at $w=0$. Similarly, there is a unique solution $G_1$ of~\eqref{eq:cyc-KZ-2} such that $G_1(1-w)w^{-A_1}$ extends to an analytic function in the unit disc with value $1$ at $w=0$. Then
\[
\Psi_{\KZ,s}=G_1(w)^{-1}G_0(w)
\]
for any $0<w<1$. We can also write this as
\begin{equation} \label{eq:KZass}
\Psi_{\KZ,s}=\lim_{w\to 1}(1-w)^{-A_1}G_0(w)=\lim_{w\to 1}(1-w)^{-A_1}G_0(w)w^{-A_0}.
\end{equation}

The case $s=0$ is special: in this case, it can be shown, using for example iterated integrals \citelist{\cite{MR1321145}*{Chapter XIX}\cite{MR2383601}}, that $\Psi_{\KZ,0}$ lives in the algebra $U(\mfg^\sigma\otimes \mfg^{\otimes2})\fpser$ rather than in its completion $\mcU(G^\sigma\times G^2)\fpser$.
Note also that this associator is well-defined in the non-Hermitian case as well. We will denote it by $\Psi_\KZ$.

Observe also that if $s\in\R$, then $G_0$ is unitary, hence $\Psi_{\KZ,s}$ is unitary as well. Indeed, in this case $(G_0(w)^*)^{-1}$ has the defining properties of $G_0(w)$, hence coincides with it.

\begin{Prop}[cf.~\cite{MR2126485}*{Proposition 4.7}]
\label{prop:expand-modified-Psi}
For every $s\not\in i\Q^\times$, we have
\[
\Psi_{\KZ,s}= 1 + \frac{h}{\pi i}\left((\log 2)t^\mfu_{12}+\gamma t^\mfm_{12}+ \psi\Bigl(\frac{1}{2}-\frac{is}{2}\Bigr)t^{\mfm_+}_{12}
+ \psi\Bigl(\frac{1}{2}+\frac{is}{2}\Bigr)t^{\mfm_-}_{12}\right)+ O(h^2),
\]
where $\gamma$ is Euler's constant and $\displaystyle \psi=\frac{\Gamma'}{\Gamma}$ is the digamma function.
\end{Prop}

\bp If we restrict to a finite dimensional block of $\mcU(G^\sigma\times G^2)$, then $\ad(s Z_1)$ has a finite number of eigenvalues there, so the corresponding component of $\Psi_{\KZ,s}$ is well-defined for all $s\not\in i N^{-1}\Z^\times$ for some~$N$. As it is analytic in $s$ in this domain, it therefore suffices to consider real~$s$.

Put $H_0(w)=G_0(w)w^{-A_0}$. Then $H_0$ satisfies the differential equation
\[
H'_0(w)=\left( \frac{A_{-1}}{w + 1} + \frac{A_1}{w-1} \right)H_0(w)+\left[\frac{A_0}{w},H_0(w)\right]
\]
and the initial condition $H_0(0)=1$, and by~\eqref{eq:KZass} we have
\[
\Psi_{\KZ,s}=\lim_{w\to 1}(1-w)^{-A_1}H_0(w).
\]

Consider the expansion in $h$. For the order zero terms we immediately get $\Psi^{(0)}_{\KZ,s}=H^{(0)}_0=1$. Next, consider the order one terms. Let us write $H$ for $\pi i H^{(1)}_0$, so that $H_0=1+\hbar H+O(h^2)$. Then
\[
H'(w)=\left( \frac{t^\mfk_{12} - t^\mfm_{12}}{w + 1} + \frac{t^\mfu_{12}}{w-1} \right)+\left[\frac{s Z_1}{w},H(w)\right]
\]
and $H(0)=0$, while
\[
\pi i\Psi_{\KZ,s}^{(1)}=\lim_{w\to1}(H(w)-\mathinner{\log(1-w)} t^\mfu_{12}).
\]

By \eqref{eq:interaction-z-tm}, we have
\begin{multline*}
H(w)=\int^w_0\left(\frac{w}{u}\right)^{\ad (s Z_1)}\left( \frac{t^\mfk_{12} - t^\mfm_{12}}{u + 1} + \frac{t^\mfu_{12}}{u-1} \right) d u \\
= \int^w_0 \left( t^\mfk_{12} \Bigl(\frac1{u+1} + \frac1{u-1} \Bigr) + \Bigl( \left( \frac{w}{u} \right)^{i s} t^{\mfm_+}_{12} + \left( \frac{w}{u} \right)^{-i s} t^{\mfm_-}_{12}  \Bigr) \Bigl( \frac{-1}{u+1} + \frac1{u-1} \Bigr) \right)d u.
\end{multline*}
Note that this integral is well-defined for $0\le w<1$ as $s$ is assumed to be real.
We then get
\[
\pi i\Psi_{\KZ,s}^{(1)}=b t^\mfk_{12}+c(s)t^{\mfm_+}_{12}+c(-s)t^{\mfm_-}_{12},
\]
where
\begin{align*}
b&=\lim_{w\to 1}\left(\int^w_0\left( \frac{1}{u + 1} + \frac{1}{u-1} \right) d u-\log(1-w)\right)=\log 2,\\
c(s)&=\lim_{w\to 1}\left(\int^w_0\left(\frac{w}{u}\right)^{is}\left( \frac{1}{u - 1} - \frac{1}{u+1} \right) d u-\log(1-w)\right).
\end{align*}

To compute $c(s)$, we write $(u^2-1)^{-1}$ as a power series, integrate and get
\[
c(s)=\lim_{w\to 1}\biggl(-\sum^\infty_{n=0}\frac{w^{2n+1}}{n+\frac{1}{2}-\frac{is}{2}}-\log(1-w)\biggr).
\]
Together with the Taylor expansion of $w^{-1}\log(1-w^2)$ and the standard formula
\[
\psi(z)+\gamma=\sum^\infty_{n=0}\left(\frac{1}{n+1}-\frac{1}{n+z}\right)
\]
this gives
\[
c(s)=\psi\Bigl(\frac{1}{2}-\frac{is}{2}\Bigr)+\gamma+\lim_{w\to 1}(w^{-1}\log(1-w^2)-\log(1-w))
=\psi\Bigl(\frac{1}{2}-\frac{is}{2}\Bigr)+\gamma+\log2,
\]
which completes the proof of the proposition.
\ep

Using the formula $\psi(1-z)-\psi(z)=\pi\cot(\pi z)$ it will be convenient to rewrite the result as
\begin{multline}\label{eq:Psi}
\Psi_{\KZ,s}= 1 + \frac{h}{\pi i}\bigg((\log 2)t^\mfu_{12}+\Big(\gamma+\frac{\psi\Bigl(\frac{1}{2}-\frac{is}{2}\Bigr)+\psi\Bigl(\frac{1}{2}+\frac{is}{2}\Bigr)}{2}\Big)t^\mfm_{12}\\
-\frac{\pi i}{2}\tanh\Bigl(\frac{\pi s}{2}\Bigr)\bigl(t^{\mfm_+}_{12}-t^{\mfm_-}_{12}\bigr)\bigg)+ O(h^2).
\end{multline}

Now, take $\mu\in h\C\fpser$. Replacing $s\not\in i\Q^\times$ by $s+\mu$ in \eqref{eq:picard-ingr}, we can construct yet another associator, which we denote by~$\Psi_{\KZ,s;\mu}$.
If $s\in\R$ and $\mu\in h\R\fpser$, then $\Psi_{\KZ,s;\mu}$ is unitary for the same reason as for~$\Psi_{\KZ,s}$.

\begin{Rem}
Similarly to Remark~\ref{rem:twisting}, $\Psi_{\KZ,s;\mu}$ could be obtained from $\Psi_{\KZ,s}$ by slicing by a character~$\nu$ of $U(\mfk)$ satisfying $\nu(Z) = -(2\hbar a_\sigma^2)^{-1}\mu$.
Since such a character does not always extend to $\mcU(G^\sigma)$, to make sense of this we should have allowed in the construction of $\Psi_{\KZ,s}$ arbitrary finite dimensional representations of $\mfg^\sigma$ instead of those in $\Rep G^\sigma$.
Alternatively, with $s\not\in i\R$ fixed, both $\Psi_{\KZ,s;\mu}$ and $\Psi_{\KZ,s+z}$ for small $z$ are specializations of an associator in $\mcU(G^\sigma\times G^2)\llbracket h,\mu\rrbracket$ constructed by treating $\mu$ as a second formal parameter. But this implies that $\Psi_{\KZ,s;\mu}$ is obtained from the Taylor expansion of $\Psi_{\KZ,s+z}$ at $z=0$ by simply taking $\mu$ as the argument:
\begin{equation}\label{eq:taylor}
\Psi_{\KZ,s;\mu}=\sum^\infty_{k=0}\frac{\mu^k}{k!}\frac{d^k\Psi_{\KZ,s}}{d s^k}=\sum_{n,k=0}^\infty\frac{h^n\mu^k}{k!}\frac{d^k\Psi^{(n)}_{\KZ,s}}{d s^k}.
\end{equation}
This also works for $s\in i(\R\setminus\Q^\times)$ if we consider only the components of $\Psi_{\KZ,s;\mu}$ in finite dimensional blocks of $\mcU(G^\sigma\times G^2)$, which are well-defined and analytic in a neighborhood of $s$.
\end{Rem}

\begin{Cor}
\label{cor:Psi-chi-expansion}
For all $s\not\in i\Q^\times$ and $\mu,\nu\in h\C\fpser$, we have
\begin{multline*}
\Psi_{\KZ,s;\mu + \nu} - \Psi_{\KZ,s; \mu} = h^{1+\ord(\nu)}\nu^{(\ord(\nu))}\biggl(\frac{1}{4\pi}\Bigl(\psi'\Bigl(\frac{1}{2}+\frac{is}{2}\Bigr)-\psi'\Bigl(\frac{1}{2}-\frac{is}{2}\Bigr)\Bigr)t^\mfm_{12}\\
-\frac{\pi}{4}\sech^2\Bigl(\frac{\pi s}{2}\Bigr)\bigl(t^{\mfm_+}_{12}-t^{\mfm_-}_{12}\bigr)\biggr)+ O(h^{2 + \ord(\nu)}).
\end{multline*}
\end{Cor}

\bp
By~\eqref{eq:taylor}, we have
\[
\Psi_{\KZ,s;\mu + \nu} - \Psi_{\KZ,s; \mu} = h^{1+\ord(\nu)}\nu^{(\ord(\nu))}\frac{d\Psi^{(1)}_{\KZ,s}}{d s}+ O(h^{2 + \ord(\nu)}).
\]
Hence the result follows from~\eqref{eq:Psi}.
\ep

\subsection{Detecting co-Hochschild classes}

To see that the associators $\Psi_{\KZ,s;\mu}$ are not all equivalent, we need to see that a perturbation of the parameter $\mu$ gives rise to a nontrivial $2$-cocycle in $B_{G,G^\sigma}$.
We can actually see that this is the case from results in Appendix \ref{sec:appendix-dg}, but let us present a concrete cycle to detect this.

Consider the tensor
\begin{equation}\label{eq:fund-class}
\Omega = [t^\mfu_{12}, t^\mfu_{13}] = \sum_{i, j} [X_i, X_j] \otimes X^i \otimes X^j\in \bigl(\medwedge^3 \mfu\bigr)^\mfu
\end{equation}
with $(X_i)_i$ and $(X^i)_i$ as in \eqref{eq:inv-2-tens}.
Every element $X\in\mfg$ defines a function on $G$ such that $g\mapsto (X,(\Ad g)(Z))_\mfg$. This way $\Omega$ defines an element of $\mcO(G^\sigma)\otimes\mcO(G)\otimes\mcO(G)$, which by slightly abusing notation we continue to denote by $\Omega$. Thus, for $(g,h,k)\in G^\sigma\times G\times G$,
\[
\Omega(g,h,k)=\bigl([(\Ad h)(Z),(\Ad k)(Z)],(\Ad g)(Z)\bigr)_\mfg=\bigl([(\Ad h)(Z),(\Ad k)(Z)],Z\bigr)_\mfg,
\]
since $G^\sigma$ stabilizes $Z$. This is a $2$-cycle in the complex $\tilde B'_{G,G^\sigma}$ from the proof of Proposition~\ref{prop:comput-cohoch-mult-model}, as
\[
\Omega(g,g,h)-\Omega(g,h,h)+\Omega(g,h,e)=0
\]
for all $(g,h)\in G^\sigma\times G$. Hence the map $\langle\Omega,\cdot\rangle\colon \mcU(G^\sigma\times G^2)\to\C$ defined by pairing with $\Omega$ passes to~$\rH^2(B_{G,G^\sigma})$. Explicitly, for $T\in \mcU(G^\sigma\times G^2)$ we have
\begin{equation}
\label{eq:pairing}
\langle \Omega,T \rangle=\epsilon(T_0)\bigl([(\ad T_1)(Z),(\ad T_2)(Z)],Z\bigr)_\mfg,
\end{equation}
where $\ad$ denotes the extension of the adjoint representation of $\mfg$ to $\mcU(G)$.

\begin{Prop}
\label{prop:partial-rmat-cohom-nontriv}
The elements $t^\mfk_{12}$, $t^{\mfm_\pm}_{12}$ are $2$-cocycles in $B_{G,G^\sigma}$. Furthermore, $t^\mfk_{12}$ and $t^{\mfm_{\phantom{-}}}_{12}=t^{\mfm_+}_{12}+t^{\mfm_-}_{12}$ are coboundaries, while
\[
\langle\Omega,t^{\mfm_\pm}_{12}\rangle=\pm \frac{i}{2}\dim \mfm.
\]
In particular, $t^{\mfm_+}_{12}$ and $-t^{\mfm_-}_{12}$ represent the same nontrivial class in $\rH^2(B_{G,G^\sigma})$.
\end{Prop}

\bp
It is easy to check that $d_\cH(1\otimes X\otimes Y)=0$ for all $X,Y\in\mfg$. As $t^\mfk_{12}$ and $t^{\mfm_\pm}_{12}$ are $\mfk$-invariant, they are therefore $2$-cocycles in $B_{G,G^\sigma}$.

We have
\[
d_\cH(C^{\mfk}_1) = C^{\mfk}_2 - \Delta(C^{\mfk})_{12} + C^{\mfk}_1 = -2 t^{\mfk}_{12},
\]
so $t^{\mfk}_{12}$ is a coboundary. Similarly, $d_\cH(C^{\mfu}_1) = -2 t^{\mfu}_{12}$, so that $ t^{\mfu}_{12}$ is also a coboundary, and hence $t^\mfm_{12}=t^{\mfu}_{12}-t^\mfk_{12}$ is a coboundary as well.

Next, take a basis $(Y_j)_j$ in $\mfm_+$ and the dual basis $(Y^j)_j$ in $\mfm_-$. Using that $\ad Z$ acts by the scalar $\pm i$ on $\mfm_\pm$, we then compute:
\begin{align*}
\langle \Omega,t^{\mfm_+}_{12}\rangle&=\sum_j \bigl( [(\ad Y_j)(Z), (\ad Y^j)(Z)], Z \bigr)_\mfg=\sum_j ([Y_j, Y^j],Z)_\mfg
=\sum_j (Y^j,[Z,Y_j])_\mfg\\
&=i\sum_j(Y^j,Y_j)_\mfg=i\dim_\C\mfm_+=\frac{i}{2}\dim_\R \mfm.
\end{align*}
The value $\langle\Omega,t^{\mfm_-}_{12}\rangle$ is obtained similarly, but it also follows from the above, as $t^{\mfm_-}_{12}=t^{\mfm_{\phantom{-}}}_{12}-t^{\mfm_+}_{12}$  and~$t^{\mfm_{\phantom{-}}}_{12}$ is a coboundary.
\ep

\begin{Rem}\label{rem:coadj-orb-and-hoch-cochains}
Let us give a different perspective on the above pairing and its nontriviality.

We can view the tensor~\eqref{eq:fund-class} also as a function on $(U/K)^3$ in the same way as above. Let us call this function $\omega$. Then it is again easy to check that $\omega$ is a $2$-cycle in the Hochschild chain complex $(C_n(A,A)=A^{\otimes(n+1)},b)$ for $A=\mcO(U)^\mfk\subset C(U/K)$. Under the Hochschild--Kostant--Rosenberg map this cycle corresponds to the differential $2$-form associated with the Kostant--Kirillov--Souriau bracket on the coadjoint orbit of $(\cdot,Z)_\mfg$, which in turn defines a nonzero class in $\rH^2(U/K;\C)\cong\C$.

We have a left $\mcU(G)$-module structure on $\mcO(U)$ given by right translations: $T.a=a_{(0)}\langle a_{(1)},T\rangle$. Given $T\in B^n_{G,G^\sigma}$, we can then define an $n$-cocycle $D_T$ in the Hochschild cochain complex $(C^n(A,A)=\Hom(A^{\otimes n},A),\delta)$ by
\[
D_T(a_1, \dots, a_n) = \epsilon(T_0)(T_1.a_1)\cdots (T_n.a_n).
\]
The Hochschild cochains act on the chains by contractions: given $D \in C^m(A, A)$, we have
\[
i_D \colon C_n(A,A) \to C_{n-m}(A,A), \quad a_0\otimes \dots \otimes a_n \mapsto a_0 D(a_1, \dots, a_{m})\otimes a_{m+1}\otimes \dots \otimes a_n,
\]
with the convention $i_D = 0$ if $n < m$.

Now, if $T\in B^n_{G,G^\sigma}$ and $c\in C_n(A,A)$ is $U$-invariant (with respect to left translations), then $i_{D_T}c\in A$ is $U$-invariant, hence a scalar. It can be checked that if $b c=0$, then this scalar depends only on the cohomology class of~$T$. Taking $c=\omega$, we recover pairing~\eqref{eq:pairing}: $i_{D_T}\omega=\langle \Omega,T\rangle$.
\end{Rem}

\subsection{Classification of associators: Hermitian case}
\label{sec:class-assoc-herm}

We are now ready to establish, in the Hermitian case, a universality result for the associators $\Psi_{\KZ,s;\mu}$ for generic quasi-coactions $(\mcU(G^\sigma)\fps,\alpha,\Psi)$ of $(\mcU(G)\fps,\Delta,\Phi_\KZ)$ such that $\alpha=\Delta$ modulo $h$. Similarly to Section~\ref{ssec:class-non-Hermitian}, it suffices to consider the case $\alpha=\Delta$.

\begin{Thm}
\label{thm:modified-KZ-universality}
Let $\mfk = \mfu^\sigma <\mfu$ be a Hermitian symmetric pair.
Assume we are given a quasi-coaction $(\mcU(G^\sigma)\fpser,\Delta,\Psi)$ of $(\mcU(G)\fpser,\Delta,\Phi_\KZ)$ such that the number $\langle\Omega,\Psi^{(1)}\rangle$ defined by~\eqref{eq:pairing} is
\[
\text{neither}\quad \pm\frac{i}{2}\dim\mfm, \quad \text{nor}\quad \frac{i(\zeta-1)}{2(\zeta+1)}\dim\mfm
\]
for a root of unity $\zeta\ne\pm1$.
Then there exist $s \not\in i\Q^\times$, $\mu \in h\C\fpser$ and $\mcH \in 1 + h\mcU(G^\sigma\times G)^\mfk\fpser$ such that $(\id\otimes\epsilon)(\mcH)=1$ and $\Psi_{\KZ,s;\mu} =\mcH_{0, 12}\Psi \mcH_{01, 2}^{-1} \mcH_{0,1}^{-1}$.
Furthermore,
\begin{enumerate}[label=\textup{(\roman*)}]
\item\label{it:mod-KZ-un-1} the number $s$ is unique up to adding $2i k$ ($k\in \Z$), and once $s$ is fixed, the element~$\mu$ is uniquely determined;
\item\label{it:mod-KZ-un-2} we can choose $s\in\R$ if and only if $\langle\Omega,\Psi^{(1)}\rangle$ is a purely imaginary number in the interval
\[
(-\frac{i}{2}\dim\mfm,\frac{i}{2}\dim\mfm);
\]
\item\label{it:mod-KZ-un-3} if $s\in\R$ and $\Psi$ is unitary, then $\mu\in h\R\fpser$ and $\mcH$ can be chosen to be unitary.
\end{enumerate}
\end{Thm}

We will later (Remark \ref{rem:s-mu-uniq}) slightly improve this result, by showing that for any $\Psi$ the parameter~$\mu$ is independent of the choice of $s$.

\bp By Proposition~\ref{prop:partial-rmat-cohom-nontriv} and our restrictions on $\Psi$, we can choose $s \not\in i\Q^\times$ such that
\begin{equation}\label{eq:s-detection}
-\frac{1}{2}\tanh\Bigl(\frac{\pi s}{2}\Bigr) \langle \Omega, t^{\mfm_+}_{12}-t^{\mfm_-}_{12}\rangle
=-\frac{i}{2}\tanh\Bigl(\frac{\pi s}{2}\Bigr)\dim\mfm
= \langle \Omega,\Psi^{(1)}\rangle.
\end{equation}
We then start with $\mcH=1$ and $\mu=0$ and modify them by induction on $n$ to have $\Psi_{\KZ,s;\mu} =\mcH_{0, 12}\Psi \mcH_{01, 2}^{-1} \mcH_{0,1}^{-1}$ modulo $h^{n+1}$.

Consider $n=1$. By the proof of Theorem~\ref{thm:Psi-rigid}, $\Psi_{\KZ,s;\mu}^{(1)}-\Psi^{(1)}$ is a $2$-cocycle in $B_{G,G^\sigma}$.
By Lemma~\ref{lem:inv-wedge} and Corollary~\ref{cor:cohom-B-G-Gtheta}, we have $\dim \rH^2(B_{G,G^\sigma})=1$. Hence our choice of~$s$, identity~\eqref{eq:Psi} and Proposition~\ref{prop:partial-rmat-cohom-nontriv} imply that $\Psi_{\KZ,s;\mu}^{(1)}-\Psi^{(1)}$ is a coboundary, so that $\Psi^{(1)}-\Psi^{(1)}_{\KZ,s;\mu}=d_\cH(T)$ for some $T\in\mcU(G^\sigma\times G)^\mfk$.
Letting $\mcH^{(1)}=T$, we then get  $\Psi_{\KZ,s;\mu} = \mcH_{0, 12} \Psi \mcH_{01, 2}^{-1} \mcH_{0,1}^{-1}$ modulo $h^2$.

For the induction step, assume we have $\Psi_{\KZ,s;\mu} =\mcH_{0, 12}\Psi \mcH_{01, 2}^{-1} \mcH_{0,1}^{-1}$ modulo $h^{n+1}$ for some $n\ge1$. Then, again by the proof of Theorem~\ref{thm:Psi-rigid},
\[
\Psi_{\KZ,s;\mu}^{(n+1)}-(\mcH_{0, 12}\Psi \mcH_{01, 2}^{-1} \mcH_{0,1}^{-1})^{(n+1)}
\]
is a $2$-cocycle in $B_{G,G^\sigma}$. On the other hand, by Corollary \ref{cor:Psi-chi-expansion}, for any $a\in\C$, we have
\[
\Psi_{\KZ,s;\mu+h^n a}^{(n+1)}-\Psi_{\KZ,s;\mu}^{(n+1)}=-a\frac{\pi}{4}\sech^2\Bigl(\frac{\pi s}{2}\Bigr)\bigl(t^{\mfm_+}_{12}-t^{\mfm_-}_{12}\bigr)+b t^\mfm_{12}
\]
for some $b\in\C$, and $\Psi_{\KZ,s;\mu+h^n a}^{(k)}=\Psi_{\KZ,s;\mu}^{(k)}$ for $k\le n$. As $t^{\mfm_+}_{12}-t^{\mfm_-}_{12}$ represents a nontrivial cohomology class, the value of $\sech$ is nonzero for our $s$, and $t^{\mfm}_{12}$ is cohomologically trivial, we see that with different choices of $a$ the above difference can represent arbitrary classes in $\rH^2(B_{G,G^\sigma})\cong\C$. In particular, we can find $a\in\C$ such that
\begin{equation}\label{eq:psi-induction}
(\mcH_{0, 12}\Psi \mcH_{01, 2}^{-1} \mcH_{0,1}^{-1})^{(n+1)}-\Psi_{\KZ,s;\mu+h^n a}^{(n+1)}=d_\cH(T)
\end{equation}
for some $T\in\mcU(G^\sigma\times G)^\mfk$. Replacing $\mcH$ by $(1+h^{n+1}T)\mcH$ and $\mu^{(n)}$ by $\mu^{(n)}+a$, we then get $\Psi_{\KZ,s;\mu} =\mcH_{0, 12}\Psi \mcH_{01, 2}^{-1} \mcH_{0,1}^{-1}$ modulo $h^{n+2}$, proving the induction step.

As at the step $n$ of our induction process we only modify $\mu^{(n-1)}$ and $\mcH^{(k)}$ for $k\ge n$, in the limit we get the required $\mu$ and $\mcH$. It remains to prove \ref{it:mod-KZ-un-1}--\ref{it:mod-KZ-un-3}.

\smallskip

\ref{it:mod-KZ-un-3}: Assume $s\in\R$ and that $\Psi$ is unitary. In this case we slightly modify the above inductive procedure to make sure that at every step we have unitarity of $\mcH$ and that $\mu\in h\R\fpser$.

Consider $n=1$. We found $\mcH$ such that  $\Psi_{\KZ,s;\mu} = \mcH_{0, 12} \Psi \mcH_{01, 2}^{-1} \mcH_{0,1}^{-1}$ modulo $h^2$. The unitarity of~$\Psi_{\KZ,s;\mu}$ and~$\Psi$ implies then that the same identity holds for the unitary $\mcH(\mcH^*\mcH)^{-1/2}$ instead of $\mcH$, cf.~the proof of~\cite{MR2832264}*{Proposition~2.3}.

For the induction step, we assume that we have $\Psi_{\KZ,s;\mu} =\mcH_{0, 12}\Psi \mcH_{01, 2}^{-1} \mcH_{0,1}^{-1}$ modulo $h^{n+1}$ for some $n\ge1$, $\mu\in h\R\fpser$ and unitary $\mcH$. Then we take the unique $a\in\C$ such that
\[
(\mcH_{0, 12}\Psi \mcH_{01, 2}^{-1} \mcH_{0,1}^{-1})^{(n+1)}-\Psi_{\KZ,s;\mu}^{(n+1)}+a\frac{\pi}{4}\sech^2 \Bigl(\frac{\pi s}{2}\Bigr)\bigl(t^{\mfm_+}_{12}-t^{\mfm_-}_{12}\bigr)
\]
is a coboundary. By taking adjoints and using that $(t^{\mfm_+})^*=t^{\mfm_-}$, we also get that
\[
\bigl((\mcH_{0, 12}\Psi \mcH_{01, 2}^{-1} \mcH_{0,1}^{-1})^{(n+1)}-\Psi_{\KZ,s;\mu}^{(n+1)}\bigr)^*-\bar a\frac{\pi}{4}\sech^2\Bigl(\frac{\pi s}{2}\Bigr)\bigl(t^{\mfm_+}_{12}-t^{\mfm_-}_{12}\bigr)
\]
is a coboundary. Hence in order to conclude that $a\in\R$ it suffices to show that if we have two unitaries~$\Psi_1$ and~$\Psi_2$ in $1+h\mcU(G^\sigma\times G^2)\fpser$ such that $\Psi_1=\Psi_2$ modulo $h^{n+1}$, then the element $\Psi_1^{(n+1)}-\Psi_2^{(n+1)}$ is skew-adjoint. But this is clear from the identity
\[
(\Psi_1-\Psi_2)^*=-\Psi_1^*(\Psi_1-\Psi_2)\Psi_2^*.
\]

Then we take $T$ satisfying \eqref{eq:psi-induction} and replace $\mu^{(n)}$ by $\mu^{(n)}+a$ and $\mcH$ by
\[
(1+h^{n+1}T)\bigl((1+h^{n+1}T^*)(1+h^{n+1}T)\bigr)^{-1/2}\mcH,
\]
similarly to the step $n=1$.

\smallskip

\ref{it:mod-KZ-un-1}, \ref{it:mod-KZ-un-2}: If $\Psi_{\KZ,s;\mu} =\mcH_{0, 12}\Psi \mcH_{01, 2}^{-1} \mcH_{0,1}^{-1}$, then $\Psi^{(1)}-\Psi^{(1)}_{\KZ,s}=d_\cH(\mcH^{(1)})$. Hence~\eqref{eq:s-detection} is not only sufficient but also necessary for the existence of~$\mcH$ and~$\mu$. Therefore $s$ is determined uniquely up to adding $2i k$ ($k\in \Z$). This also makes \ref{it:mod-KZ-un-2} obvious.

Next, assume $\Psi_{\KZ,s;\mu} =\mcH_{0, 12}\Psi_{\KZ,s;\mu'} \mcH_{01, 2}^{-1} \mcH_{0,1}^{-1}$ for some $s\not\in i\Q^\times$, $\mu,\mu' \in h\C\fpser$ and $\mcH \in 1 + h\mcU(G^\sigma\times G)^\mfk\fpser$. We have to show that $\mu=\mu'$. Assume this is not the case.

Let $n\ge1$ be the smallest order such that $\mu^{(n)}\ne\mu'^{(n)}$. By Corollary \ref{cor:Psi-chi-expansion} we have $\Psi_{\KZ,s;\mu} = \Psi_{\KZ,s;\mu'}$ modulo $h^{n+1}$.
We claim that we can modify $\mcH$ so that we still have $\Psi_{\KZ,s;\mu} =\mcH_{0, 12}\Psi_{\KZ,s;\mu'} \mcH_{01, 2}^{-1} \mcH_{0,1}^{-1}$, but $\mcH=1$ modulo $h^{n+1}$.

We will modify $\mcH$ by induction on $k\le n$ to get $\Psi_{\KZ,s;\mu} =\mcH_{0, 12}\Psi_{\KZ,s;\mu'} \mcH_{01, 2}^{-1} \mcH_{0,1}^{-1}$ and $\mcH=1$ modulo~$h^{k+1}$. Assume we have these two properties for some $k<n$. Then
\[
d_\cH(\mcH^{(k+1)})=\Psi_{\KZ,s;\mu'}^{(k+1)} -\Psi_{\KZ,s;\mu}^{(k+1)}=0.
\]
As $\rH^1(B_{G,G^\sigma})=0$ by Lemma~\ref{lem:inv-wedge} and Corollary~\ref{cor:cohom-B-G-Gtheta}, there exists a central element $S\in\mcU(G^\sigma)$ such that $\mcH^{(k+1)}=S_{01}-S_0$.
Putting $\mcH'=\exp(-h^{k+1}(S_{01}-S_0))$, we see that $\Psi_{\KZ,s;\mu}$ commutes with~$\mcH'_{0,12}$, hence we have $\mcH'_{0, 12}\Psi_{\KZ,s;\mu} \mcH_{01, 2}'^{-1} \mcH_{0,1}'^{-1}=\Psi_{\KZ,s;\mu}$.
It follows that by replacing~$\mcH$ by~$\mcH' \mcH$ we get $\Psi_{\KZ,s;\mu} =\mcH_{0, 12}\Psi_{\KZ,s;\mu'} \mcH_{01, 2}^{-1} \mcH_{0,1}^{-1}$ and $\mcH=1$ modulo~$h^{k+2}$. Thus our claim is proved.

It follows now that $\Psi_{\KZ,s;\mu'}^{(n+1)} -\Psi_{\KZ,s;\mu}^{(n+1)}=d_\cH(\mcH^{(n+1)})$. Since $t^{\mfm_+}_{12}-t^{\mfm_-}_{12}$ is not a coboundary, this contradicts Corollary \ref{cor:Psi-chi-expansion}. Hence $\mu=\mu'$.
\ep

\begin{Rem}
In view of Corollary~\ref{cor:Buk-cohom}, a similar result should in principle be true at the level of the universal enveloping algebras as well.
However, since we only know that $\Psi_{\KZ,s;\mu}\in\mcU(G^\sigma\times G)\fpser$ (for $s\ne0$), in the first place one has to show that $\Psi_{\KZ,s;\mu}$ belongs to $U(\mfg^\sigma) \otimes U(\mfg)^{\otimes 2}\fpser$, at least up to some twist.
\end{Rem}

\subsection{Classification of ribbon braids}
\label{sec:class-ribb-br}

We complement our classification of associators by describing compatible ribbon twist-braids, both in the Hermitian and non-Hermitian cases.

In the following we always take the universal $R$-matrix
\[
\mcR_\KZ = \exp(-h t^\mfu) \in \mcU(G^2)\fpser
\]
for $(\mcU(G)\fpser, \Delta, \Phi_{\KZ})$.

\begin{Thm}
\label{thm:uniq-tw-br-nonherm-case}
If $\mfk = \mfu^\sigma <\mfu$ is a non-Hermitian symmetric pair, then the ribbon $\sigma$-braids for the quasi-coaction $(\mcU(G^\sigma)\fpser,\Delta,\Psi_\KZ)$ of $(\mcU(G)\fpser,\Delta,\Phi_\KZ,\mcR_\KZ)$ are the elements
\[
\mcE_\KZ g_1 = \exp(-h(2t^\mfk_{01}+C^\mfk_1))g_1,
\]
where $g\in Z(U)$.
\end{Thm}

\bp
The fact that $\exp(-h(2t^\mfk_{01}+C^\mfk_1))$ is a ribbon $\sigma$-braid is essentially proved in~\cite{MR2383601}, see~\cite{MR3943480}*{Theorem~3.8}. Hence the elements $\exp(-h(2t^\mfk_{01}+C^\mfk_1))g_1$ ($g\in Z(U)$) are ribbon $\sigma$-braids as well. The claim that these are the only ribbon $\sigma$-braids follows from Theorem~\ref{thm:uniq-br-tw-non-Herm}.
\ep

A similar result holds in the Hermitian case, but the proof is more involved.

\begin{Thm}
\label{thm:uniq-tw-br-herm-case}
If $\mfk = \mfu^\sigma < \mfu$ is a Hermitian symmetric pair, $s\not\in i\Q^\times$ and $\mu\in h\C\fpser$, then the ribbon $\sigma$-braids for the quasi-coaction $(\mcU(G^\sigma)\fpser,\Delta,\Psi_{\KZ,s;\mu})$ of $(\mcU(G)\fpser,\Delta,\Phi_\KZ,\mcR_\KZ)$ are the elements
\[
\mcE_{\KZ,s;\mu} g_1 = \exp(-h(2t^\mfk_{01}+C^\mfk_1)-\pi i(s+\mu)Z_1)g_1,
\]
where $g\in Z(U)$.
\end{Thm}

\bp
The fact that the element $\mcE = \mcE_{\KZ,s;\mu}=\exp(-h(2t^\mfk_{01}+C^\mfk_1)-\pi i(s+\mu)Z_1)$ is a ribbon $\sigma$-braid, and hence that the elements $\mcE(1\otimes g)$ ($g\in Z(U)$) are ribbon $\sigma$-braids as well, follows again from the proof of~\cite{MR3943480}*{Theorem~3.8}.

Let $\mcE'$ be another ribbon $\sigma$-braid. The same argument as in the proof of Theorem~\ref{thm:uniq-br-tw-non-Herm} shows that $\mcE'^{(0)}=1\otimes g$ for an element $g\in Z_G(G^\sigma)$.
We now use the same strategy as in the proof of the last statement of that theorem to get more restrictions on $g$.
Namely, we replace $\mcE'$ by $\mcE'(1\otimes g^{-1})$ to get a ribbon $\tilde\sigma$-braid, where $\tilde\sigma=(\Ad g)\circ\sigma$, and then twist $\mcE'$ and $\Psi_{\KZ,s;\mu}$ further by an element $H$ to get rid of the terms $t^\mfu$ and $t^\mfm$ in $\Psi^{(1)}_{\KZ,s;\mu}$, see~\eqref{eq:Psi} and recall that by Corollary~\ref{cor:Psi-chi-expansion} identity~\eqref{eq:Psi} is still valid for $\Psi_{\KZ,s;\mu}$.
Note for a future use that by the proof of Proposition~\ref{prop:partial-rmat-cohom-nontriv} we can take $\mcH$ of the form
\begin{equation}\label{eq:H-Casimir}
\mcH=1+h(a C^\mfu_1+b C^\mfk_1)
\end{equation}
for appropriate constants $a$ and $b$. Thus, our new associator $\Psi = \mcH_{0,12}\Psi_{\KZ,s;\mu}\mcH^{-1}_{01,2}\mcH^{-1}_{0,1}$ satisfies
\begin{equation}\label{eq:tilde-Psi-lower-term}
\Psi = 1 -\frac{h}{2}\tanh\Bigl(\frac{\pi s}{2}\Bigr)(t^{\mfm_+}_{12}-t^{\mfm_-}_{12})+ O(h^2).
\end{equation}

Looking at the order one terms in~\eqref{eq:rib-sig-tw-1} and applying $\epsilon$ to the $0$th leg, instead of~\eqref{eq:mcE1} we now get
\begin{multline*}
\mcE'^{(1)}=-2t^\mfk+1\otimes T+(\id\otimes\Ad g)(t^\mfm)-t^\mfm\\
+\tanh\Bigl(\frac{\pi s}{2}\Bigr)(t^{\mfm_+}-t^{\mfm_-})+\tanh\Bigl(\frac{\pi s}{2}\Bigr)(\id\otimes\Ad g)(t^{\mfm_+}-t^{\mfm_-}),
\end{multline*}
where $T=(\epsilon\otimes\id)(\mcE'^{(1)})$ and we used that $t^{\mfm_\pm}_{21}=t^{\mfm_\mp}_{12}$. As $\mcE'^{(1)}\in\mcU(G^\sigma\times G)$, this means that
\begin{equation}\label{eq:Ad-g}
\Ad g-\id\mp\tanh\Bigl(\frac{\pi s}{2}\Bigr)\id\mp\tanh\Bigl(\frac{\pi s}{2}\Bigr)\Ad g=0\quad\text{on}\quad \mfm_\pm.
\end{equation}
Hence $\Ad g=e^{\pm\pi s}\id$ on $\mfm_\pm$, which implies that $g=\exp(-\pi is Z)g'$ for some $g'\in Z(G)=Z(U)$.

Without loss of generality we may assume that $g'=e$, and then we want to prove that our original~$\mcE'$ coincides with $\mcE$. It is more convenient to modify $\mcE$ in the same way as $\mcE'$, that is, by replacing it by
\[
\mcH\mcE(1\otimes g^{-1})(\id\otimes\tilde\sigma)(\mcH)^{-1}=\exp(-h(2t^\mfk_{01}+C^\mfk_1)-\pi i\mu Z_1),
\]
where we used that $\mcH$ has the form~\eqref{eq:H-Casimir}.

Thus, our new setup is that we have two ribbon $\tilde\sigma$-braids $\mcE'$ and $\mcE$, with $\tilde\sigma=(\Ad \exp(-\pi is Z))\circ\sigma$, with respect to an associator~$\Psi$ satisfying~\eqref{eq:tilde-Psi-lower-term}, $\mcE'^{(0)}=\mcE^{(0)}=1$,
\begin{equation}\label{eq:mcE(1)}
\mcE^{(1)}=-(2t^\mfk_{01}+C^\mfk_1+\pi i\mu^{(1)} Z_1),
\end{equation}
and the goal is to show that $\mcE'=\mcE$.

We will prove by induction on $n$ that $\mcE'^{(n)}=\mcE^{(n)}$. Consider the case $n = 1$. Put
\[
p=\tanh\Bigl(\frac{\pi s}{2}\Bigr)\quad\text{and}\quad t=t^{\mfm_+}-t^{\mfm_-}.
\]
By the argument in the proof of Theorem \ref{thm:uniq-br-tw-non-Herm}, we have $\mcE'^{(1)} = \mcE^{(1)} + 1 \otimes Y$ for some $Y \in \mfz_\mfg(\mfk)=\C Z$. Put also $T=\mcE'^{(2)}-\mcE^{(2)}$.

Using \eqref{eq:tilde-Psi-lower-term}, formula \eqref{eq:rib-sig-tw-1} for $\mcE$, modulo $h^3$ and terms depending only on $\mcR_\KZ$ and $\Psi$, becomes
\[
(1 + h \mcE^{(1)} + h^2 \mcE^{(2)})_{01,2}
= (1 + h(pt_{1,2} - t^\mfu_{1,2}))(1 + h \mcE^{(1)} + h^2 \mcE^{(2)})_{0,2} (\id\otimes\id \otimes\tilde \sigma)(1 - h(pt_{1,2} + t^\mfu_{1,2})),
\]
where we again used that $t_{21}=-t_{12}$.
We have a similar formula for $\mcE'$.
Taking the difference and comparing the coefficients of $h^2$, we obtain
\[
T_{01, 2} = T_{0,2} + (p t_{1,2} - t^\mfu_{1,2}) Y_2 - Y_2 (\id \otimes\id\otimes \tilde\sigma)(p t_{1,2} + t^\mfu_{1,2}).
\]
Using the identity
\begin{equation}\label{eq:Ad-g2}
(p t - t^\mfm)- (\id\otimes \tilde\sigma)(p t + t^\mfm)=0,
\end{equation}
which is an equivalent form of~\eqref{eq:Ad-g}, we can write this as
\begin{equation}
\label{compar-hbar-2-coeff-1}
T_{01, 2} = T_{0,2} -2t^\mfk_{1,2}Y_2+[p t_{1,2} - t^\mfm_{1,2},Y_2].
\end{equation}

Similarly, formula \eqref{eq:rib-sig-tw-2} for $\mcE$, modulo $h^3$ and terms depending only on $\mcR_\KZ$ and $\Psi$, becomes
\begin{multline*}
(1 + h \mcE^{(1)} + h^2 \mcE^{(2)})_{0,12}
= \Bigl(1 + h \Bigl(\frac{p}{2} t_{1,2} - t^\mfu_{1,2}\Bigr)\Bigr) (1 + h \mcE^{(1)} + h^2 \mcE^{(2)})_{0,2}\\ \times(\id \otimes\id\otimes \tilde\sigma)(1 - h (p t_{1,2} + t^\mfu_{1,2}))
(1 + h \mcE^{(1)} + h^2 \mcE^{(2)})_{0,1}(\id \otimes\tilde\sigma\otimes \tilde\sigma)\Bigl(1+ h\frac{p}{2} t_{1,2}\Bigr),
\end{multline*}
and we have a similar identity for $\mcE'$. Taking the difference and comparing the coefficients of $h^2$, we obtain
\begin{multline*}
T_{0,12} = T_{0,2} + T_{0,1} + \Bigl(\frac{p}{2} t_{1,2} - t^\mfu_{1,2}\Bigr) (Y_1 + Y_2) - Y_2 (\id\otimes\id \otimes \tilde\sigma)(p t_{1,2} + t^\mfu_{1,2}) \\ - (\id \otimes \id\otimes \tilde\sigma)(p t_{1,2} + t^\mfu_{1,2}) Y_1
+ (Y_1 + Y_2) (\id\otimes \tilde\sigma \otimes \tilde\sigma)\Bigl(\frac{p}{2} t_{1,2}\Bigr) + Y_2 \mcE^{(1)}_{0,1} + \mcE^{(1)}_{0,2} Y_1 + Y_1 Y_2.
\end{multline*}
Using that $\tilde\sigma\otimes\tilde\sigma$ is the identity on $\mfm_\pm\otimes\mfm_\mp$ and that $\ad Y_1=-\ad Y_2$ on $\mfm_\pm$, we can write this as
\begin{multline*}
T_{0,12} = T_{0,2} + T_{0,1} + (p t_{1,2} - t^\mfu_{1,2}) (Y_1 + Y_2) - Y_2 (\id\otimes\id \otimes \tilde\sigma)(p t_{1,2} + t^\mfu_{1,2}) \\ - (\id \otimes \id\otimes \tilde\sigma)(p t_{1,2} + t^\mfu_{1,2}) Y_1
+ Y_2 \mcE^{(1)}_{0,1} + \mcE^{(1)}_{0,2} Y_1 + Y_1 Y_2,
\end{multline*}
and then, using again~\eqref{eq:Ad-g2}, we get
\begin{equation}
\label{compar-hbar-2-coeff-2}
T_{0,12} = T_{0,2} + T_{0,1} -2t^\mfk_{1,2}(Y_1+Y_2) + [p t_{1,2}  - t^\mfm_{1,2},Y_2]+ Y_2 \mcE^{(1)}_{0,1} + \mcE^{(1)}_{0,2} Y_1 + Y_1 Y_2.
\end{equation}

Subtracting \eqref{compar-hbar-2-coeff-2} from \eqref{compar-hbar-2-coeff-1}, we obtain
\[
d_{\cH}(T) =2 t^\mfk_{12} Y_1 - \mcE^{(1)}_{01} Y_2 - \mcE^{(1)}_{02} Y_1 - Y_1 Y_2.
\]
By~\eqref{eq:mcE(1)}, this means that
\[
d_\cH(T)=(2 t^\mfk_{12}+2t^\mfk_{02}+C^\mfk_2+\pi i\mu^{(1)} Z_2)Y_1+ (2t^\mfk_{01}+C^\mfk_1+\pi i\mu^{(1)} Z_1)Y_2 - Y_1Y_2.
\]
As $Y\in\C Z$, we can also write this as
\[
d_\cH(T)=(2 t^\mfk_{12}+2t^\mfk_{02}+C^\mfk_2)Y_1+ (2t^\mfk_{01}+C^\mfk_1+2\pi i\mu^{(1)} Z_1)Y_2-Y_1Y_2.
\]

On the other hand, a straightforward computation shows that the right hand side of the above identity is the coboundary of
\[
2t^\mfk_{01}Y_0+C^\mfk_1Y_0+C^\mfk_0Y_1+2\pi i\mu^{(1)}Z_0Y_1-Y_0Y_1.
\]
As $\rH^1(B_{G,G^\sigma})=0$ by Lemma~\ref{lem:inv-wedge} and Corollary~\ref{cor:cohom-B-G-Gtheta}, it follows that there exists~$S$ in the center of $\mcU(G^\sigma)$ such that
\[
T=S_{01}-S_0+2t^\mfk_{0,1}Y_0+C^\mfk_1Y_0+C^\mfk_0Y_1+2\pi i\mu^{(1)}Z_0Y_1-Y_0Y_1.
\]

The only consequence of the above identity that we need is that $T\in\mcU(G^\sigma\times G^\sigma)$. By looking at~\eqref{compar-hbar-2-coeff-1} we see that this implies
\[
[pt-t^\mfm,1\otimes Y]=(p-1)[t^{\mfm_+},1\otimes Y]-(p+1)[t^{\mfm_-},1\otimes Y]\in\mcU(G^\sigma\times G^\sigma).
\]
As $Y$ is a scalar multiple of $Z$, and $\ad Z$ acts by nonzero operators on $\mfm_\pm$, this is possible only if $Y=0$. This shows that $\mcE'^{(1)} = \mcE^{(1)}$.

The induction step is similar. Assuming that $\mcE'^{(k)} = \mcE^{(k)}$ for $k<n$ for some $n\ge2$, we have $\mcE'^{(n)} = \mcE^{(n)} + 1 \otimes Y$ for an element $Y \in \mfz_\mfg(\mfk)=\C Z$. Then, with $T=\mcE'^{(n+1)}-\mcE^{(n+1)}$, comparing the coefficients of $h^{n+1}$ in \eqref{eq:rib-sig-tw-1} we get the same identity~\eqref{compar-hbar-2-coeff-1}. If we do the same for \eqref{eq:rib-sig-tw-2}, the only difference from~\eqref{compar-hbar-2-coeff-2} is that we do not get the term $Y_1Y_2$ at the end. But this has almost no effect on the rest of the argument, we just have to remove the terms $Y_1Y_2$ and $Y_0Y_1$ in the subsequent identities. Thus we get $Y=0$.
\ep

\begin{Rem}\label{rem:s-mu-uniq}
Theorem~\ref{thm:uniq-tw-br-herm-case} implies that in the Hermitian case the ribbon twist-braids contain complete information about the associators.
Namely, assume that $\Psi_{\KZ,s';\mu'} =\mcH_{0, 12}\Psi_{\KZ,s;\mu} \mcH_{01, 2}^{-1} \mcH_{0,1}^{-1}$ for some $\mcH \in 1 + h\mcU(G^\sigma\times G)^\mfk\fpser$. By~\eqref{eq:braid-twist0} and Theorem~\ref{thm:uniq-tw-br-herm-case} it follows that
\[
\mcH \exp(-h(2t^\mfk_{01}+C^\mfk_1)-\pi i(s+\mu)Z_1) (\id\otimes\sigma)(\mcH)^{-1}
=\exp(-h(2t^\mfk_{01}+C^\mfk_1)-\pi i(s'+\mu')Z_1)g_1
\]
for some $g\in Z(U)$. Since $(\epsilon\otimes\id)(\mcH)$ is a $\mfk$-invariant element of $\mcU(G)$ and $\sigma$ is an inner automorphism defined by an element of $K$, we have $\sigma\bigl((\epsilon\otimes\id)(\mcH)\bigr)=(\epsilon\otimes\id)(\mcH)$, and then by applying $\epsilon$ to the $0$th leg we get
\[
\exp(-h C^\mfk-\pi i(s+\mu)Z)=\exp(-h C^\mfk-\pi i(s'+\mu')Z)g.
\]
This implies that $s=s'+2ik$ for some $k\in\Z$, and $\mu=\mu'$.
\end{Rem}

\section{Interpolated subgroups}\label{sec:int-subgroups}

In this section we fix an involutive automorphism $\nu$ of $\mfu$ such that $\mfu^\nu<\mfu$ is an irreducible Hermitian symmetric pair.
We are going to fix a Cartan subalgebra in $\mfu^\nu$ and then apply the results of the previous section to particular conjugates of $\nu$ and true coactions of a quantization of $(\mcU(G),\Delta)$. Along the way we will study a distinguished family of coisotropic subgroups of $U$ that are conjugates of $U^\nu$. The main result is Theorem~\ref{thm:assoc-for-prescribed-classical-limit}, where we in particular relate the induced Poisson structures on the associated homogeneous spaces to a Poisson pencil for the Kostant--Kirillov--Souriau bracket which appears from the cyclotomic KZ-equations.

Throughout this section we use the subscript~$\nu$ for the Lie algebra constructions we had for $\sigma$. Thus, $\mfm_\nu=\{X\in\mfu \mid \nu(X)=-X\}$, $Z_\nu\in\mfz(\mfg^\nu)$.

\subsection{Root vectors and Poisson structure}\label{ssec:Cartan}

Let us quickly review a standard Poisson--Lie group structure on $U$ making $U^\nu$ a Poisson--Lie subgroup (for Hermitian symmetric pairs).

Let $\mft$ be a Cartan subalgebra of $\mfu$ containing $\mfz(\mfu^\nu)$.
Then~$\mft$ is contained in $\mfu^\nu$, and its complexification~$\mfh$ is a Cartan subalgebra of $\mfg$.

Recall that a root $\alpha$ is called \emph{compact} if $\mfg_\alpha\subset\mfg^\nu$, and \emph{noncompact} otherwise.
As $\mfg^\nu$ is the centralizer of $\mfz(\mfu^\nu)$, a root is compact if and only if it vanishes on~$\mfz(\mfu^\nu)$.

As in Section~\ref{ssec:cyclotomic-KZ}, we fix $Z_\nu\in\mfz(\mfu^\nu)$ such that $(Z_\nu,Z_\nu)_\mfg=-a_\nu^{-2}$.
Let us fix an ordered basis in~$i\mft$, with~$-i Z_\nu$ being the first element of the basis, and consider the corresponding lexicographic order on the roots. Then, in this order, any noncompact positive root is bigger than any compact root. Furthermore, the noncompact positive roots are \emph{totally positive} in the sense of~\cite{MR0072427}, meaning that if~$\gamma$ is a noncompact positive root, $\alpha_1, \dots, \alpha_k$ are compact roots, and $m_1,\dots, m_k$ are integers such that $\gamma' = \gamma + \sum_{i=1}^km_i \alpha_i$ is a root, then $\gamma'$ is positive.

We denote by $\Phi$ the set of all roots, and by $\Phi^+$ that of positive roots.
We further denote by $\Phi_\nc^+$ (resp.~$\Phi_\nc^-$) the set of positive (resp.~negative) noncompact roots.
Let $\Pi = \{\alpha_i\}_{i\in I}$ be the set of simple positive roots.
Recall that we denote by $\mfm_{\nu\pm}\subset\mfm_\nu^\C$ the eigenspaces of $\ad Z_\nu$ corresponding to the eigenvalues $\pm i$. It is clear from our choice of the ordering that
\begin{align*}
\mfm_{\nu+}&=\bigoplus_{\alpha \in \Phi_\nc^+}\mfg_\alpha,&
\mfm_{\nu-}&=\bigoplus_{\alpha \in \Phi_\nc^-}\mfg_\alpha.
\end{align*}

Since $2i$ is not an eigenvalue of $\ad Z_\nu$, we have $[\mfm_{\nu+},\mfm_{\nu+}]=0$.
It follows that $\mfp=\mfg^\nu+\mfm_{\nu+}$ is a parabolic subalgebra of $\mfg$.
As $\mfg^\nu$ acts irreducibly on $\mfg/\mfp\cong \mfm_{\nu-}$, this parabolic subalgebra is maximal.
Hence it corresponds to a maximal proper subset of $\Pi$.
This set must consist of compact roots, and hence its complement consists of one noncompact root.
We denote this unique noncompact simple positive root by $\alpha_o$.
It corresponds to the black vertex in a standard \emph{Vogan diagram} of~$\nu$.

For every root $\alpha$, let $H_{\alpha}\in \mfh$ be the element dual to the coroot $\alpha^{\vee} = \frac{2}{(\alpha,\alpha)}\alpha$, so that we have
\[
\alpha(H_{\beta}) = (\alpha,\beta^{\vee})\quad (\alpha,\beta \in \Phi),
\]
where $(\cdot,\cdot)$ is the scalar product dual to the restriction of $(\cdot,\cdot)_\mfg$ to $\mfh$.
For positive roots $\alpha$, we choose root vectors $X_{\alpha}\in \mfg_{\alpha}$ such that the antilinear involution for $\mfu$ satisfies
\[
(X_{\alpha},X_{\alpha}^*)_{\mfg} = \frac{2}{(\alpha,\alpha)}.
\]
Then $[X_\alpha,X^*_{\alpha}]=H_\alpha$, and we put $X_{-\alpha}=X^*_\alpha$.

Let
\begin{align*}
Y_{\alpha} &= \frac{i}{2}(X_{\alpha}+ X_{-\alpha}) \in \mfu,&
Z_{\alpha} &= \frac{1}{2}( X_{\alpha}-X_{-\alpha}) \in \mfu \quad (\alpha \in \Phi^+),
\end{align*}
and consider the antisymmetric tensor
\begin{equation*}\label{EqDefr}
r = \sum_{\alpha>0} (\alpha,\alpha) (Y_{\alpha} \otimes Z_{\alpha} - Z_{\alpha} \otimes Y_{\alpha}) \in \mfu^{\otimes 2}.
\end{equation*}
Note that we can also write
\begin{equation}\label{EqComplr}
r = i \sum_{\alpha>0} \frac{(\alpha,\alpha)}{2} (X_{-\alpha} \otimes X_{\alpha} - X_{\alpha} \otimes X_{-\alpha}) \in \mfg^{\otimes 2}.
\end{equation}
Then $i r \pm t^\mfu$ (with $t^\mfu$ defined by~\eqref{eq:inv-2-tens}) satisfies the classical Yang--Baxter equation, and $\mfu$ becomes a Lie bialgebra with the cobracket
\[
\delta_r(X) = [r, \Delta(X)].
\]
Thus, $U$ becomes a Poisson--Lie group with the \emph{Sklyanin Poisson bracket}
\[
\{f_1,f_2\}_{\Sk} = m (r^{(l,l)} - r^{(r,r)})(f_1 \otimes f_2),
\]
where $m$ is the product map, and for $X \in \mfu$ and $f \in C^{\infty}(U)$ we put
\begin{align*}
 (X^{(l)} f)(g) &= \frac{d}{d t} f(e^{t X} g) |_{t = 0},&
 (X^{(r)} f)(g) &= \frac{d}{d t} f(g e^{t X}) |_{t = 0}.
\end{align*}
Note that if we as usual view $\mfu$ and $U$ as sitting inside $\mcU(G)$, then we can write the Sklyanin bracket as
\[
\{f_1,f_2\}_{\Sk}(g) = \langle f_1 \otimes f_2, [r, g\otimes g] \rangle\quad (g\in U,\ f_1, f_2 \in \mcO(U)).
\]

The Lie algebra $\mfu^\nu$ is the intersection of $\mfu$ with a parabolic subalgebra of $\mfg$, namely, $\mfu^\nu=\mfu\cap\mfp$.
It is well-known that this implies that $U^\nu$ is a Poisson--Lie subgroup of $U$.

\subsection{Satake form}\label{ssec:Satake}

We will have to conjugate $U^\nu$ in order to go beyond the Poisson--Lie subgroups, making it closer to the \emph{coisotropic subgroup} associated with symmetric pairs \cite{MR2102330}.
In order to do so, let us review the \emph{Satake form} of involutions.

For $X\subset I$, we denote by $\Pi_X$ the subset $\{\alpha_j \mid j\in X\}\subset\Pi$. Assume $\theta$ is a nontrivial involution on~$\mfu$ such that its extension to $\mfg$ leaves $\mfh$ globally invariant. We write $\Theta\in \End(\mfh^*)$ for the endomorphism dual to $\theta|_{\mfh}$.

\begin{Def}\label{def:split-pair}
We say that $\theta$ is in \emph{maximally split form}, or in \emph{Satake form}, with respect to $(\mfh,\mfb^+)$, or that $(\mfh,\mfb^+)$ is a \emph{split pair} for $\theta$, if there exists $X\subset I$ satisfying the following conditions:
\begin{enumerate}
\item $\Phi^+ \cap \Theta(\Phi^+) = \Phi^+ \cap \Z \Pi_X$,
\item $\theta = \id$ on $\mfg_{\alpha}$ for all $\alpha \in \Phi^+\cap \Z \Pi_X$.
\end{enumerate}
\end{Def}

The above set $X$ is then uniquely determined, representing the black vertices in the corresponding Satake diagram. Then there exist unimodular $w_{\alpha} \in \C$ such that
\begin{equation*}
\theta(X_{\alpha}) = w_{\alpha} X_{\Theta(\alpha)} \quad (\alpha\in \Phi).
\end{equation*}
Moreover, there exists a unique Dynkin diagram involution $\tau_{\theta}$ such that
\begin{equation}\label{eq:Theta-tau}
\Theta(\alpha) = -w_X\tau_{\theta}(\alpha)\quad (\alpha \in \Phi),
\end{equation}
with $w_X$  the longest element of the Weyl group associated to $X$. This involution leaves the set $X$ globally invariant. See \cite{MR1155464} and \cite{MR3269184}*{Appendix~A} for details.

Put
\[
I_\ns = \{i \in I\setminus X\mid \tau_{\theta}(i) = i \textrm{ and }(\alpha_i,\alpha_j)=0 \textrm{ for all } j \in X\},
\]
which corresponds to the $\tau_{\theta}$-stable white vertices not connected to any black vertices in the Satake diagram.
We then put
\[
I_{\mcS}=\{i\in I_\ns\mid ({\alpha}^\vee_j,\alpha_i) \in 2\Z\textrm{ for all }j\in I_\ns\}.
\]
We also put
\[
I_{\mcC} = \{i \in I \setminus X \mid \tau_{\theta}(i)\neq  i\textrm{ and }(\alpha_i,\Theta(\alpha_i)) = 0\}.
\]
It can be shown that this is the set of white vertices not fixed by $\tau_\theta$ such that, if $\bar \alpha_i$ is the corresponding restricted root, then $2 \bar \alpha_i$ is not a restricted root, see \citelist{\cite{MR0153782}\cite{MR1834454}*{p.~530}}.

We will use the following definition from~\cite{MR3943480}.

\begin{Def}\label{def:S-type-C-type}
We say that a Hermitian symmetric pair $\mfu^{\theta}<\mfu$  is
\begin{itemize}
\item of \emph{S-type}, if $I_{\mcS} \neq \emptyset$,
\item of \emph{C-type}, if there exists $i\in I\setminus(X\cup I_\mcC)$ such that $\tau_\theta(i)\ne i$.
\end{itemize}
\end{Def}

See \cite{MR3943480}*{Appendix C} for a concrete classification of the Hermitian symmetric pairs into these types.
Recall that the restricted root system is always of type $\rC$ or $\rBC$ in the Hermitian case \cite{MR0161943}*{Theorem 2}.
By a case-by-case analysis (see, e.g., \citelist{\cite{MR1064110}*{Reference Chapter, table 9}\cite{MR1834454}*{Chapter X, Table VI}\cite{MR1920389}*{Appendix C}}), one sees that we are in the S-type case exactly when the restricted root system is of type $\rC$, while we are in the C-type case when the restricted root system is of type $\rBC$.

Moreover, by \cite{MR3943480}*{Lemma C.2}, in the C-type case there exists exactly one $\tau_\theta$-orbit of the form $\{o \neq o'=\tau_\theta(o)\}$ in $I\setminus(X\cup I_\mcC)$.
We call the roots $\alpha_o$ and $\alpha_{o'}$ \emph{distinguished}.
By the same lemma, in the S-type case the set $I_\mcS$ consists of one root $\alpha_o$, which we again call distinguished.
Note that we use the same label $o$ for one of the distinguished roots as for a noncompact root in Section~\ref{ssec:Cartan}.
This will be justified in Proposition~\ref{prop:restrictions}.

We now recall the construction of a \emph{cascade} of orthogonal roots in the setting of Section~\ref{ssec:Cartan}.
Let~$\gamma_1$ be the largest in our lexicographic order (hence necessarily noncompact) root of $\mfg^{(0)}=\mfg$, and let $\mfg^{(1)}$ be the centralizer of~$H_{\gamma_1}$.
If $\mfg^{(1)}\not\subset\mfg^\nu$, then let $\gamma_2$ be the largest root such that $X_{\gamma_2}\in\mfg^{(1)}$, and let~$\mfg^{(2)}$ be the centralizer of $\{H_{\gamma_1},H_{\gamma_2}\}$.
Continuing this until we have $\mfg^{(s)} \subset \mfg^\nu$, we obtain a strictly decreasing sequence of Lie subalgebras $(\mfg^{(i)})_{i=1}^s$ and noncompact positive roots $\gamma_1, \dots, \gamma_s$. Furthermore, these roots form a maximal family of mutually orthogonal noncompact roots, and they are mutually {strongly orthogonal}, i.e.,~ $\gamma_i \pm \gamma_{j}$ is never a root, see \cite{MR1920389}*{Lemma~7.143}.

Let $\mfh^+ \subset \mfh$ consist of all $H\in\mfh$ with $\gamma_i(H) = 0$ for all $i$.
Let $\mfh^-$ be the complex linear span of $\{H_{\gamma_i} \mid i = 1, \dots, s \}$.
Then clearly $\mfh = \mfh^+ \oplus \mfh^-$, and hence for the dual spaces we have $\mfh^*=(\mfh^+)^* \oplus (\mfh^-)^*$. Using this decomposition we will often think of the $\gamma_i$ as elements of $(\mfh^-)^*$.

On the other hand, let $\widetilde{\mfh}^-$ be the complex linear span of $\{X_{\gamma_i} - X_{-\gamma_i} \mid i = 1, \dots, s\}$.
Then also $\mfh^+ \oplus \widetilde{\mfh}^-$ is a Cartan subalgebra of $\mfg$. This is a version of Harish-Chandra's construction of maximally split Cartan subalgebras.

To relate the two Cartan subalgebras, we consider the (partial, unitary) \emph{Cayley transform} $\Ad g_1$, where
\[
g_1 = \exp\Bigl(\frac{\pi i}{4} \sum_{i=1}^s (X_{-\gamma_i} + X_{\gamma_i})\Bigr) \in U,
\]
cf.~\cite{MR1920389}*{Section~VI.7}. Then $(\Ad {g_1})(i H_{\gamma_i}) = X_{\gamma_i} -X_{-\gamma_i}$ (see Lemma~\ref{lem:P_0} below), so $(\Ad{g_1})(\mfh^-) = \widetilde{\mfh}^-$, while $\Ad{g_1}$ acts as the identity on $\mfh^+$. We then have the following concrete presentation of maximally split involutions.

\begin{Prop}
\label{prop:concr-form-max-split-invol}
The involution $\theta = (\Ad {g_1})^{-1}\circ \nu\circ(\Ad{g_1})$ is in maximally split form with respect to $(\mfh,\mfb^+)$, with the associated set $X =  \{i\in I \mid H_{\alpha_i} \in \mfh^+\}$.
\end{Prop}

In order to prove this, we will need some detailed information on the restricted roots. We will follow closely the treatment of Harish-Chandra in~\cite{MR0082056}.

For $\lambda, \mu \in \mfh^*$, let us write $\lambda \sim \mu$ when they restrict to the same functional on $\mfh^-$.
For each $i$, let $C_i$ denote the set of compact positive roots $\alpha$ such that $\alpha \sim \hlf1 \gamma_i$.
Similarly, let $P_i$ be the set of noncompact positive roots $\gamma$ such that $\gamma \sim \hlf1 \gamma_i$.

Next, for $i < j$, let $C_{i j}$ denote the set of compact positive roots $\alpha$ such that $\alpha \sim \hlf1(\gamma_i - \gamma_j)$.
Let $P_{i j}$ denote the set of noncompact positive roots $\gamma$ such that $\gamma \sim \hlf1(\gamma_i + \gamma_j)$.

Finally, let $P_0=\{\gamma_1, \dots, \gamma_s\}$, and $C_0$ denote the set of positive roots $\alpha$ such that $\alpha \sim 0$, that is, $H_\alpha \in \mfh^+$ , or equivalently, $\alpha$ is orthogonal to $\gamma_1,\dots,\gamma_s$. The set $C_0$ consists of compact roots, as $\{\gamma_1,\dots,\gamma_s\}$ is a maximal family of mutually orthogonal noncompact positive roots.

\begin{Prop}
\label{prop:structure-of-roots-from-Harish-Chandra}
The set $\Phi^+$ is partitioned by the subsets $P_0$, $C_0$, $(P_i)_{i=1}^s$, $(C_i)_{i=1}^s$, $(P_{i j})_{1 \le i < j \le s}$, and $(C_{i j})_{1 \le i < j \le s}$.
Moreover,
\begin{enumerate}[label=\textup{(\roman*)}]
\item\label{it:structure-of-roots-from-Harish-Chandra-1} if $\alpha\in C_0$, then $\alpha$ is strongly orthogonal to $\gamma_i$ for all $i$;
\item\label{it:structure-of-roots-from-Harish-Chandra-2} for every $1\le i\le s$, the map $\alpha \mapsto \gamma_i - \alpha$ is a bijection from~$C_i$ onto~$P_i$;
\item\label{it:structure-of-roots-from-Harish-Chandra-3} for all $1\le i < j \le s$, the maps $\alpha\mapsto\gamma_i-\alpha$ and $\alpha \mapsto   \gamma_j+\alpha$ are  bijections from $C_{i j}$ onto $P_{i j}$.
\end{enumerate}
\end{Prop}

The proof is practically identical to that of \cite{MR0082056}*{Lemma~16}, we therefore omit the details.

\begin{proof}[Proof of Proposition~\ref{prop:concr-form-max-split-invol}]
First of all observe that by construction $\theta=\id$ on $\mfh^+$ and $\theta=-\id$ on $\mfh^-$.
This already implies that
\[
\Phi^+\cap\Z\Pi_X\subset\{\alpha>0\mid\Theta(\alpha)=\alpha\}\subset \Phi^+\cap\Theta(\Phi^+).
\]

Next, by Proposition \ref{prop:structure-of-roots-from-Harish-Chandra}, every positive root restricts to $0$, $\gamma_i$, $\hlf1 \gamma_i$ or $\hlf1(\gamma_i \pm \gamma_j)$ for some $i$, $j$ with $i < j$. From this we see that the intersection of the restrictions of $\Phi^+$ and $\Theta(\Phi^+)$ is at most $\{0\}$. In particular, if $\alpha\in \Phi^+\cap\Theta(\Phi^+)$, then $\alpha$ restricts to $0$. Decompose such an $\alpha$ into a combination of the simple roots and restrict to $\mfh^-$. Since no nontrivial sum with nonnegative integral coefficients of the vectors $\gamma_i$, $\hlf1 \gamma_i$ and $\hlf1(\gamma_i \pm \gamma_j)$ $(i<j)$ is zero, it follows that $\alpha$ decomposes into a combination of the simple roots that restrict to $0$, that is, $\alpha\in \Phi^+\cap\Z\Pi_X$. This proves property (1) in Definition~\ref{def:split-pair}.

To establish property (2), take $\alpha\in \Phi^+\cap\Z\Pi_X$, that is, $\alpha$ is a positive root restricting to $0$. This root must be compact, since $\{\gamma_1,\dots,\gamma_s\}$ is a maximal family of mutually orthogonal noncompact roots, and it is strongly orthogonal to $\gamma_i$ by Proposition~\ref{prop:structure-of-roots-from-Harish-Chandra} \ref{it:structure-of-roots-from-Harish-Chandra-1}. Therefore $\mfg_\alpha\subset\mfg^\nu$ and $\mfg_\alpha$ centralizers $X_{-\gamma_i} + X_{\gamma_i}$. Hence $\theta$ acts trivially on $\mfg_\alpha$.
\end{proof}

Refining the observation after Definition \ref{def:S-type-C-type}, we have the following.

\begin{Prop}\label{prop:restrictions}
With the above notation, the roots $\gamma_1,\dots,\gamma_s$ are all of the same length.
The restriction map $\alpha\mapsto\alpha|_{\mfh^-}$ defines a bijection between the $\tau_\theta$-orbits in $\Pi\setminus\Pi_X$ and a basis of the restricted root system.
This basis and the distinguished roots are concretely described as follows.

\begin{itemize}
\item{\textup{S-type:}} The restricted root system is of type $\rC_s$, consisting of $\{\pm\hlf1(\gamma_i\pm\gamma_j)\}_{i<j}\cup\{\pm\gamma_i\}_i$, with the basis $\{\hlf1(\gamma_i-\gamma_{i+1})\}^{s-1}_{i=1}\cup\{\gamma_s\}$.
The unique noncompact root $\alpha_o\in\Pi$ is distinguished, and it coincides with~$\gamma_s$.
\item{\textup{C-type:}} The restricted root system is of type $\rBC_s$, consisting of $\{\pm\hlf1(\gamma_i\pm\gamma_j)\}_{i<j}\cup\{\pm\gamma_i\}_i\cup\{\pm\hlf1\gamma_i\}_i$, with the basis $\{\hlf1(\gamma_i-\gamma_{i+1})\}^{s-1}_{i=1}\cup\{\hlf1\gamma_s\}$.
The unique noncompact root $\alpha_o\in\Pi$ is distinguished, and its restriction to $\mfh^-$ is $\hlf1\gamma_s$.
The second distinguished root is the only other simple root~$\alpha_{o'}$ that restricts to  $\hlf1\gamma_s$.
\end{itemize}
\end{Prop}

\bp
Observe that since $\alpha-w_X\alpha\in\Z\Pi_X$ for any root~$\alpha$, we have $\Theta(\alpha)+\tau_\theta(\alpha)\in\Z\Pi_X$ by~\eqref{eq:Theta-tau}. It is well-known and not difficult to see that this implies that the restriction map $\alpha\mapsto\alpha|_{\mfh^-}$ defines a bijection between the $\tau_\theta$-orbits in $\Pi\setminus\Pi_X$ and a basis of the restricted root system.

Next, by a case-by-case analysis (see, e.g., \cite{MR1920389}*{Appendix C} again), we know that the restricted root system is of type $\rC_s$ in the S-type case and of type $\rBC_s$ in the C-type case. By counting the number of roots and Proposition~\ref{prop:structure-of-roots-from-Harish-Chandra} it follows that the restricted root system is $\{\pm\hlf1(\gamma_i\pm\gamma_j)\}_{i<j}\cup\{\pm\gamma_i\}_i$ (S-type) or $\{\pm\hlf1(\gamma_i\pm\gamma_j)\}_{i<j}\cup\{\pm\gamma_i\}_i\cup\{\pm\hlf1\gamma_i\}_i$ (C-type). Now, on the one hand, the restriction of $\Pi\setminus\Pi_X$ gives a basis of the restricted root system. On the other hand, the nonzero restrictions of the positive roots is the set $\{\hlf1(\gamma_i\pm\gamma_j)\}_{i<j}\cup\{\gamma_i\}_i$ (S-type) or $\{\hlf1(\gamma_i\pm\gamma_j)\}_{i<j}\cup\{\gamma_i\}_i\cup\{\hlf1\gamma_i\}_i$ (C-type). From this we may conclude that the basis we get by restriction must be $\{\hlf1(\gamma_i-\gamma_{i+1})\}^{s-1}_{i=1}\cup\{\gamma_s\}$ (S-type) or $\{\hlf1(\gamma_i-\gamma_{i+1})\}^{s-1}_{i=1}\cup\{\hlf1\gamma_s\}$ (C-type). Since this is a basis of root systems of type $\rC_s$ or $\rBC_s$, it follows that the roots $\gamma_1,\dots,\gamma_s$ are of the same length.

It remains to identify the distinguished roots. Consider the unique noncompact root $\alpha_o \in \Pi$.
Its image in the restricted root system must be either $\gamma_s$ (S-type) or $\hlf1 \gamma_s$ (C-type).
Indeed, a noncompact positive root can only restrict to $\gamma_i$, $\hlf1 \gamma_i$ or $\hlf1(\gamma_i+\gamma_j)$ ($i<j$) by Proposition~\ref{prop:structure-of-roots-from-Harish-Chandra}, and the claim follows by taking the intersection with the image of the simple positive roots.

Consider the C-type case. By the remark after the definition of $I_\mcC$, in this case the distinguished roots have the property that if $\beta$ is their common restriction, then $2\beta$ is still a restricted root. Since $\beta$ must be an element of the basis $\{\hlf1(\gamma_i-\gamma_{i+1})\}^{s-1}_{i=1}\cup\{\hlf1\gamma_s\}$ of the restricted root system, we have $\beta=\hlf1\gamma_s$. We already know that $\alpha_o$ restricts to $\hlf1 \gamma_s$, so $\alpha_o$ is one of the distinguished roots, and then the other is~$\tau_\theta(\alpha_o)$.

Consider now the S-type case. Again, we already know that the restriction of $\alpha_o$ is $\gamma_s$. By Proposition~\ref{prop:structure-of-roots-from-Harish-Chandra}, the only positive root restricting to $\gamma_s$ is $\gamma_s$ itself.
Hence $\alpha_o=\gamma_s$ and $\tau_\theta(\gamma_s)=\gamma_s$.
It remains to check that $\alpha_o$ is the unique element of $I_\mcS$.
The equality $\tau_\theta(\gamma_s)=\gamma_s$ means that~$\gamma_s$ is not connected by an arrow to any other vertex in the Satake diagram. It also needs to be separated from the black vertices, as $\gamma_s$ is orthogonal to $\Pi_X$.
Looking at the Satake diagrams, this is already enough to conclude that~$\gamma_s$ is the distinguished vertex, except for the $\mathrm{CI}$ case $\mfu_s \subset \mathfrak{sp}_{s}$.
But in this remaining case we have $\mfh^+=0$, and the restricted roots are the same as the entire roots.
Thus, $\gamma_s$ represents the unique long simple positive root, which is indeed the distinguished root.
\ep

\begin{Cor}\label{cor:center-S-type}
In the S-type case we have
\[
Z_\nu=\frac{i}{2}\sum^s_{j=1}H_{\gamma_j},
\]
as well as $a_\nu=2 / \sqrt s$ for $\mfu_s \subset \mathfrak{sp}_{s}$ and $a_\nu=\sqrt{2/s}$ in all other cases.
\end{Cor}

\bp
Since in the S-type case the compact positive roots restrict to $0$ or $\hlf1(\gamma_i-\gamma_j)$ ($i<j$), such roots vanish on $\sum^s_{j=1}H_{\gamma_j}$. Therefore $\sum^s_{j=1}H_{\gamma_j}\in\mfz(\mfg^\nu)$. As we must have $\gamma_j(Z_\nu)=i$ for any $j$, we get the formula for $Z_\nu$ in the formulation.

A case-by-case verification shows that $\gamma_s=\alpha_o$ is a short root in all cases except for $\mfu_s \subset \mathfrak{sp}_{s}$, while in the last case it is a long root of length $2$. Since the roots $\gamma_1,\dots,\gamma_s$ are all of the same length and $(Z_\nu,Z_\nu)_\mfg=-a_\nu^{-2}$, we then get the formula for $a_\nu$.
\ep

\subsection{A family of coisotropic subgroups}\label{ssec:coisotropic}

Now we are ready to introduce a one-parameter family of involutive automorphisms interpolating between $\nu$ and $\theta = (\Ad {g_1})^{-1}\circ \nu\circ(\Ad{g_1})$, which define coisotropic subgroups of $U$.

For $\phi\in \R$, let
\begin{align*}
g_\phi &= \exp\Bigl(\frac{\pi i \phi}{4} \sum_{i=1}^s (X_{-\gamma_i} + X_{\gamma_i})\Bigr),&
\theta_\phi &= (\Ad {g_\phi})\circ \theta\circ(\Ad{g_\phi})^{-1},
\end{align*}
so that $\theta_0=\theta$ and $\theta_1=\nu$.

\begin{Def}\label{def:K-phi}
We will write $K_\phi$ for $U^{\theta_\phi}=g_{\phi-1}U^\nu g_{1-\phi}$, and denote its Lie algebra by $\mfk_\phi$.
\end{Def}

Our first goal is to understand how the $r$-matrix~\eqref{EqDefr} transforms under~$g_\phi$.

\begin{Prop}
\label{prop:Cayley-rotates-rmat}
For all $\phi\in\R$, we have
\[
(\Ad{g_\phi})^{\otimes 2}(r) - \cos\Bigl(\frac{\pi \phi}{2}\Bigr) r \in \mfu^\nu \otimes \mfu + \mfu \otimes \mfu^\nu.
\]
\end{Prop}

For the proof we need to introduce a more convenient basis for computations.
We will write
\[
e_{\alpha} = -i X_{\alpha}
\]
to adapt to the conventions of~\citelist{\cite{bourbaki-lie-fr-7-8}\cite{MR0214638}}.
Note that then
\[
e_{\alpha}^* = -e_{-\alpha}\quad \text{and}\quad \lbrack e_{\alpha},e_{-\alpha}\rbrack = -H_{\alpha}.
\]

Let us write $[e_\alpha, e_\beta] = N_{\alpha, \beta} e_{\alpha + \beta}$ when $\alpha$, $\beta$, and $\alpha + \beta$ are roots.
We then have $\bar N_{\alpha,\beta} = N_{-\alpha,-\beta}$ and
\begin{align}
\label{eq:Chev-basis-str-coeff-rel}
N_{\alpha,\beta} N_{-\alpha, \alpha + \beta} &= - p(q+1), & \absv{N_{\alpha, \beta}} &= q + 1,
\end{align}
where $p$, resp.~$q$, is the largest integer such that $\beta + p \alpha$, resp.~$\beta - q \alpha$, is a root~\cite{bourbaki-lie-fr-7-8}*{Section VIII.2.4}.
(In fact, if we were more careful in choosing root vectors, we could arrange $N_{\alpha,\beta}$ to be real, with the sign of $N_{\alpha,\beta}$ described in~\cite{MR0214638}.)

Recall from Proposition~\ref{prop:structure-of-roots-from-Harish-Chandra} that $\Phi^+$ is partitioned by the subsets $P_0$, $C_0$, $(P_i)_{i=1}^s$, $(C_i)_{i=1}^s$, $(P_{i j})_{1 \le i < j \le s}$, and $(C_{i j})_{1 \le i < j \le s}$. We will consider these subsets one by one.

We start with $P_0$. For each $i$, put
\[
x_i = e_{\gamma_i} + e_{-\gamma_i}\quad\text{and}\quad y_i = e_{\gamma_i} - e_{-\gamma_i}.
\]

\begin{Lem}\label{lem:P_0}
The map $\Ad g_\phi$ acts as follows:
\begin{align*}
x_i &\mapsto x_i,&
y_i &\mapsto \cos\Bigl(\frac{\pi \phi}{2}\Bigr)y_i-\sin\Bigl(\frac{\pi \phi}{2}\Bigr)H_{\gamma_i},&
H_{\gamma_i} &\mapsto \sin\Bigl(\frac{\pi \phi}{2}\Bigr)y_i+\cos\Bigl(\frac{\pi \phi}{2}\Bigr)H_{\gamma_i}.
\end{align*}
\end{Lem}

\bp
Since $\gamma_i$ is strongly orthogonal to $\gamma_j$ for $j\ne i$, we have
\begin{align*}
\Bigl[\sum_{j=1}^s (e_{-\gamma_j} + e_{\gamma_j}), x_i \Bigr] &= 0,&
\Bigl[\sum_{j=1}^s (e_{-\gamma_j} + e_{\gamma_j}), y_i \Bigr] &= 2H_{\gamma_i},&
\Bigl[\sum_{j=1}^s (e_{-\gamma_j} + e_{\gamma_j}), H_{\gamma_i} \Bigr] &= -2y_i.
\end{align*}
Since
\begin{align}\label{eq:exp}
g_\phi &=\exp\Big(-\frac{\pi \phi}{4}\sum_{j=1}^s (e_{-\gamma_j} + e_{\gamma_j})\Big),&
\exp\begin{pmatrix} 0 & - \lambda \\ \lambda & 0 \end{pmatrix}
&= \begin{pmatrix}  \cos \lambda & -\sin\lambda\\ \sin\lambda& \cos\lambda\end{pmatrix},
\end{align}
we get the result.
\ep

As $\Ad g_\phi$ acts trivially on the orthogonal complement of $\{H_{\gamma_1},\dots,H_{\gamma_s}\}$ in $\mfh$, this lemma already describes the action of $\Ad g_\phi$ on the Cartan subalgebra.

\begin{Lem}\label{lem:C_0}
If $\alpha\in C_0$, then $\Ad g_\phi$ acts trivially on $e_{\pm\alpha}$ and $H_\alpha$.
\end{Lem}

\bp
This follows from Proposition~\ref{prop:structure-of-roots-from-Harish-Chandra} \ref{it:structure-of-roots-from-Harish-Chandra-1}.
\ep

Next, on the root vectors of $P_i$ and $C_i$ we have the following description of $\Ad g_\phi$.

\begin{Lem}\label{lem:PiCi}
Assume $1\le i\le s$, and $\gamma\in P_i$ and $\alpha\in C_i$ are such that $\gamma+\alpha=\gamma_i$. Then $\Ad g_\phi$ acts as follows:
\begin{align*}
e_\gamma &\mapsto \cos\Bigl(\frac{\pi \phi}{4}\Bigr)e_\gamma-\bar N_{\gamma_i,-\gamma}\sin\Bigl(\frac{\pi \phi}{4}\Bigr)e_{-\alpha},&
e_{-\alpha} &\mapsto
N_{\gamma_i,-\gamma}\sin\Bigl(\frac{\pi \phi}{4}\Bigr)e_\gamma+\cos\Bigl(\frac{\pi \phi}{4}\Bigr)e_{-\alpha},
\end{align*}
and we have $|N_{\gamma_i,-\gamma}|=1$.
\end{Lem}

Note that we also get formulas for the action on $e_{-\gamma}$ and $e_\alpha$ by taking adjoints.

\bp
From Proposition~\ref{prop:structure-of-roots-from-Harish-Chandra} we see that $\gamma+\gamma_i$, $\gamma-2\gamma_i$ and $\gamma\pm\gamma_j$ for $j\ne i$ are not roots. Hence
\[
\Bigl[\sum_{j=1}^s (e_{-\gamma_j} + e_{\gamma_j}), e_\gamma\Bigr]=[e_{-\gamma_i},e_\gamma]=N_{-\gamma_i,\gamma}e_{-\alpha}=\bar N_{\gamma_i,-\gamma}e_{-\alpha},
\]
and $N_{-\gamma_i,\gamma}N_{\gamma_i,-\alpha}=-1$ and $|N_{-\gamma_i,\gamma}|=1$ by~\eqref{eq:Chev-basis-str-coeff-rel}. Similarly, $\gamma_i+\alpha$ and $\alpha\pm\gamma_j$ for $j\ne i$ are not roots, hence
\[
\Bigl[\sum_{j=1}^s (e_{-\gamma_j} + e_{\gamma_j}),\bar N_{\gamma_i,-\gamma}e_{-\alpha}\Bigr]=[e_{\gamma_i},\bar N_{\gamma_i,-\gamma}e_{-\alpha}]=-e_\gamma.
\]
The lemma follows by again using~\eqref{eq:exp}.
\ep

Consider now $\gamma\in P_{ij}$, and put $\alpha=\gamma-\gamma_j\in C_{ij}$.
By Proposition~\ref{prop:structure-of-roots-from-Harish-Chandra} \ref{it:structure-of-roots-from-Harish-Chandra-3} we also have roots $\gamma'\in P_{ij}$ and $\alpha'\in C_{ij}$ such that
\[
\gamma_i-\alpha=\gamma'=\gamma_j+\alpha'.
\]
Put $\eps(\gamma)=N_{-\gamma_i, \gamma} N_{-\gamma_j, \gamma'}$, $\eps(\alpha) = N_{-\gamma_i,\alpha} N_{\gamma_j, \alpha'}$, and take the elements
\begin{align*}
x_\gamma &= e_\gamma - \eps(\gamma) e_{-\gamma'},&
y_\gamma &= e_\gamma + \eps(\gamma) e_{-\gamma'},&
x_\alpha &= e_\alpha - \eps(\alpha) e_{-\alpha'},&
y_\alpha &= e_\alpha + \eps(\alpha) e_{-\alpha'}.
\end{align*}

\begin{Lem}\label{lem:PijCij}
If $\gamma\in P_{ij}$ and $\alpha\in C_{ij}$ are such that $\gamma-\alpha=\gamma_j$, then the map $\Ad g_\phi$ acts as follows:
\begin{align*}
x_\gamma &\mapsto x_\gamma,&
x_\alpha &\mapsto x_\alpha,&
y_\gamma &\mapsto \cos\Bigl(\frac{\pi \phi}{2}\Bigr)y_\gamma-\bar N_{\gamma_j,-\gamma}\sin\Bigl(\frac{\pi \phi}{2}\Bigr)y_{\alpha},&
y_\alpha &\mapsto N_{\gamma_j,-\gamma}\sin\Bigl(\frac{\pi \phi}{2}\Bigr)y_\gamma+\cos\Bigl(\frac{\pi \phi}{2}\Bigr)y_{\alpha},
\end{align*}
and we have $|\eps(\gamma)|=|\eps(\alpha)|=|N_{\gamma_j,-\gamma}|=1$.
\end{Lem}

\bp First of all observe that
\begin{align}\label{eq:gamma-alpha}
\gamma_i-\gamma &= \alpha',&
\gamma_i-\gamma' &= \alpha,&
\gamma_j-\gamma &= -\alpha,&
\gamma_j-\gamma' &= -\alpha'.
\end{align}
From Proposition~\ref{prop:structure-of-roots-from-Harish-Chandra} we see that $-\gamma_i-\gamma$ and $2\gamma_i-\gamma$ are not roots, hence $|N_{\gamma_i,-\gamma}|=1$ by the second identity in~\eqref{eq:Chev-basis-str-coeff-rel}. For similar reasons the numbers $N_{\gamma_i,-\gamma'}$, $N_{\gamma_j,-\gamma}$ and $N_{\gamma_j,-\gamma'}$ are of modulus one, and by the first identity in~\eqref{eq:Chev-basis-str-coeff-rel} we have
\begin{equation}\label{eq:N}
N_{\gamma_i,-\gamma}N_{-\gamma_i,\alpha'}=N_{\gamma_i,-\gamma'}N_{-\gamma_i,\alpha}=N_{\gamma_j,-\gamma}N_{-\gamma_j,-\alpha}=
N_{\gamma_j,-\gamma'}N_{-\gamma_j,-\alpha'}=-1.
\end{equation}

We claim that also the following identity holds:
\begin{equation}\label{eq:NN}
N_{-\gamma_j,-\alpha'}N_{-\gamma_i, \gamma}=N_{-\gamma_i, \alpha} N_{-\gamma_j, \gamma}.
\end{equation}
Indeed, the expressions on both sides are precisely the coefficients of $e_{-\gamma'}$ in $[e_{-\gamma_j}, [e_{-\gamma_i}, e_\gamma]]$ and $[e_{-\gamma_i}, [e_{-\gamma_j}, e_\gamma]]$.
By the Jacobi identity,
\[
[e_{-\gamma_j}, [e_{-\gamma_i}, e_\gamma]] - [e_{-\gamma_i}, [e_{-\gamma_j}, e_\gamma]] = [e_\gamma,[e_{-\gamma_i}, e_{-\gamma_j}]].
\]
But we have $[e_{-\gamma_i}, e_{-\gamma_j}] = 0$ by strong orthogonality. Thus our claim is proved.

Now, a simple computation using~\eqref{eq:gamma-alpha}--\eqref{eq:NN} gives
\begin{gather*}
\Bigl[\sum_{k=1}^s (e_{-\gamma_k} + e_{\gamma_k}), x_\gamma\Bigr]=\Bigl[\sum_{k=1}^s (e_{-\gamma_k} + e_{\gamma_k}), x_\alpha\Bigr]=0,\\
\begin{align*}
\Bigl[\sum_{k=1}^s (e_{-\gamma_k} + e_{\gamma_k}), y_\gamma\Bigr] &= 2\bar N_{\gamma_j,-\gamma}y_\alpha,&
\Bigl[\sum_{k=1}^s (e_{-\gamma_k} + e_{\gamma_k}), \bar N_{\gamma_j,-\gamma}y_\alpha\Bigr] &= -2y_\gamma.
\end{align*}
\end{gather*}
The lemma follows again from~\eqref{eq:exp}.
\ep

\bp[Proof of Proposition~\ref{prop:Cayley-rotates-rmat}]
We have
\[
r = -i \sum_{\alpha>0} \frac{(\alpha,\alpha)}{2} (e_{-\alpha} \otimes e_{\alpha} - e_{\alpha} \otimes e_{-\alpha}).
\]
We will use the partition of $\Phi^+$ into the subsets $P_0$, $C_0$, $(P_i)_{i=1}^s$, $(C_i)_{i=1}^s$, $(P_{i j})_{1 \le i < j \le s}$, $(C_{i j})_{1 \le i < j \le s}$, and check how the corresponding components of $r$ transform under $(\Ad{g_\phi})^{\otimes 2}$.

We start with $\gamma=\gamma_i\in P_0$. Up to the factor $-\sqrt{-1}\frac{(\gamma_i,\gamma_i)}{4}$, the corresponding component of $r$ is $x_i\otimes y_i-y_i\otimes x_i$. By Lemma~\ref{lem:P_0}, its image under $(\Ad{g_\phi})^{\otimes 2}$, modulo $\mfg^\nu\otimes\mfg+\mfg\otimes\mfg^\nu$, is
\[
\cos\Bigl(\frac{\pi \phi}{2}\Bigr)(x_i\otimes y_i-y_i\otimes x_i),
\]
as needed.

Next, by Lemma~\ref{lem:C_0}, if $\alpha\in C_0$, then the corresponding components of $(\Ad{g_\phi})^{\otimes 2}(r)$ and $r$ are already in $\mfg^\nu\otimes\mfg^\nu$.

Consider $1\le i\le s$, and take roots $\gamma\in P_i$ and $\alpha\in C_i$ related by $\gamma=\gamma_i-\alpha$. Since $\gamma=-s_{\gamma_i}\alpha$, these roots have the same length. Therefore, up to a factor, the component of $r$ corresponding to~$\gamma$ and~$\alpha$~is
\[
e_{-\gamma}\otimes e_\gamma-e_\gamma\otimes e_{-\gamma}+e_{-\alpha}\otimes e_\alpha-e_\alpha\otimes e_{-\alpha}.
\]
By Lemma~\ref{lem:PiCi}, its image under $(\Ad{g_\phi})^{\otimes 2}$, modulo $\mfg^\nu\otimes\mfg+\mfg\otimes\mfg^\nu$, is
\[
\Bigl(\cos^2\Bigl(\frac{\pi \phi}{4}\Bigr)-\sin^2\Bigl(\frac{\pi \phi}{4}\Bigr)\Bigr)(e_{-\gamma}\otimes e_\gamma-e_\gamma\otimes e_{-\gamma})
=\cos\Bigl(\frac{\pi \phi}{2}\Bigr)(e_{-\gamma}\otimes e_\gamma-e_\gamma\otimes e_{-\gamma}).
\]
This is equal (up to the same factor as before) to the contribution of $\gamma$ and $\alpha$ to $\cos(\frac{\pi \phi}{2})r$ modulo $\mfg^\nu\otimes\mfg+\mfg\otimes\mfg^\nu$.

Consider now $\gamma\in P_{ij}$. Let $\gamma'=\gamma_i+\gamma_j-\gamma$. Using~\eqref{eq:N}, we get from~\eqref{eq:NN} that $N_{-\gamma_i, \gamma} N_{-\gamma_j, \gamma'}=N_{-\gamma_i, \gamma'} N_{-\gamma_j, \gamma}$, that is, $\eps(\gamma)=\eps(\gamma')$. We then have
\[
x_\gamma \otimes y_{\gamma'} - y_{\gamma'} \otimes x_\gamma + x_{\gamma'} \otimes y_\gamma - y_\gamma \otimes x_{\gamma'} = 2\eps(\gamma) (e_\gamma \otimes e_{-\gamma} - e_{-\gamma} \otimes e_\gamma + e_{\gamma'} \otimes e_{-\gamma'} - e_{-\gamma'} \otimes e_{\gamma'}).
\]
The roots $\gamma$ and $\gamma'$ are of the same length, since $\gamma'=-s_{\gamma_i}s_{\gamma_j}\gamma$.
Therefore, up to a factor, the component of $r$ corresponding to $\gamma$ and $\gamma'$ is
\[
x_\gamma \otimes y_{\gamma'} - y_{\gamma'} \otimes x_\gamma + x_{\gamma'} \otimes y_\gamma - y_\gamma \otimes x_{\gamma'}.
\]
Note that it is possible that $\gamma=\gamma'$, but this only changes the overall factor. By Lemma~\ref{lem:PijCij}, the image of the above expression under $(\Ad{g_\phi})^{\otimes 2}$, modulo $\mfg^\nu\otimes\mfg+\mfg\otimes\mfg^\nu$, is
\[
\cos\Bigl(\frac{\pi \phi}{2}\Bigr) (x_\gamma \otimes y_{\gamma'} - y_{\gamma'} \otimes x_\gamma + x_{\gamma'} \otimes y_\gamma - y_\gamma \otimes x_{\gamma'}),
\]
as we need.

Finally, take $\alpha\in C_{ij}$. Let $\alpha'=\gamma_i-\gamma_j-\alpha$. Then, similarly to the previous case, the contribution of~$\alpha$ and~$\alpha'$ to~$r$ is, up to a factor,
\[
x_\alpha \otimes y_{\alpha'} - y_{\alpha'} \otimes x_\alpha + x_{\alpha'} \otimes y_\alpha - y_\alpha \otimes x_{\alpha'}.
\]
By Lemma~\ref{lem:PijCij}, this expression transforms under $(\Ad{g_\phi})^{\otimes 2}$ into an element of $\mfg^\nu\otimes\mfg+\mfg\otimes\mfg^\nu$.
\ep

\begin{Cor}\label{cor:coisotropic}
For every $\phi\in\R$, the subgroup $K_\phi$ of Definition \ref{def:K-phi} is a coisotropic subgroup of $(U,r)$, that is,
\[
\delta_r(\mfk_\phi)\subset\mfk_\phi\otimes\mfu+\mfu\otimes\mfk_\phi.
\]
It is a Poisson--Lie subgroup if and only if $\phi$ is an odd integer.
\end{Cor}

\bp
By definition we have $\mfk_\phi=(\Ad g_{\phi-1})(\mfu^\nu)$. Since $U^\nu$ is a Poisson--Lie subgroup of $(U,r)$, $K_\phi$ is a Poisson--Lie subgroup of $(U,(\Ad g_{\phi-1})^{\otimes 2}(r))$. As
\[
r-\cos\Bigl(\frac{\pi(1-\phi)}{2}\Bigr)(\Ad g_{\phi-1})^{\otimes 2}(r)\in\mfk_\phi\otimes\mfu+\mfu\otimes\mfk_\phi,
\]
this shows that $K_\phi$ is coisotropic in $(U,r)$.

Assume now that $K_\phi$ is a Poisson--Lie subgroup of $(U,r)$ for some $\phi$. Since $K_\phi$ has the same rank as~$U$, it follows that $\mfk_\phi$ must contain the Cartan subalgebra $\mft$ (see, e.g., \cite{MR1995786}*{Proposition~2.1}). Therefore $\theta_\phi=(\Ad g_{\phi-1})\circ\nu\circ(\Ad g_{\phi-1})^{-1}$ acts trivially on $\mfh$. From Lemma~\ref{lem:P_0} we see that this is the case if and only if $\sin(\frac{\pi(1-\phi)}{2})=0$, that is, $\phi$ is an odd integer.

Assume that indeed $\phi=2n+1$ for some $n\in\Z$. In the S-type case, when the sets $P_i$ and $C_i$ are empty for $1\le i\le s$, we see from Lemmas~\ref{lem:C_0} and \ref{lem:PijCij} that $\mfk_\phi^\C=\mfg^\nu$, so $K_\phi=U^\nu$ is a Poisson-Lie subgroup.

Consider the C-type case. If $n$ is even, so that $\cos(\frac{\pi(\phi-1)}{2})=\pm1$, we see from Lemmas~\ref{lem:C_0}--\ref{lem:PijCij} that $\mfk_\phi^\C=\mfg^\nu$, so $K_\phi=U^\nu$ is again a Poisson--Lie subgroup.

Assume now that $n$ is odd.
Then $\sin(\frac{\pi(\phi-1)}{4})=\pm1$, and we see from Lemmas~\ref{lem:C_0}--\ref{lem:PijCij} that $\mfk_\phi^\C$ is spanned by $\mfh$, $X_{\pm\alpha}$ ($\alpha\in C_0$), $X_{\pm\gamma}$ ($\gamma\in P_i$, $1\le i\le s$) and $X_{\pm\alpha}$ ($\alpha\in C_{ij}$).
Moreover, by Proposition~\ref{prop:restrictions}, the nondistinguished simple roots in $\Pi\setminus\Pi_X$ lie in the sets $C_{i,i+1}$, while the distinguished roots satisfy $\alpha_o\in P_s$ and $\alpha_{o'}\in C_s$.
We conclude that we have $X_{\pm\alpha}\in\mfk_\phi^\C$ for the nondistinguished simple roots~$\alpha$, $X_{\pm\alpha_{o}}\in\mfk_\phi^\C$ and $X_{\pm\alpha_{o'}}\not\in\mfk_\phi^\C$. It follows that if $\mfq\subset\mfg$ is the parabolic subalgebra defined by the subset $\Pi\setminus\{\alpha_{o'}\}$ of simple roots, then $\mfq\cap\mfu\subset\mfk_\phi$.
We have a Dynkin diagram involution $\tau_\theta$ mapping $\Pi\setminus\{\alpha_{o'}\}$ onto $\Pi\setminus\{\alpha_o\}$.
Then the corresponding automorphism of $\mfg$ maps $\mfq$ onto $\mfp=\mfg^\nu+\mfm_{\nu+}$. Since $\mfu^\nu=\mfp\cap\mfu$ is a maximal proper Lie subalgebra of $\mfu$, it follows that $\mfq\cap\mfu$ is a maximal proper Lie subalgebra of $\mfu$. Hence $\mfq\cap\mfu=\mfk_\phi$, and therefore $K_\phi$ is a Poisson--Lie subgroup of $(U,r)$.
\ep

\begin{Rem}
We see from the above argument, or directly from Lemmas~\ref{lem:P_0}--\ref{lem:PijCij}, that $(\Ad g_\phi)(\mfu^\nu)=\mfu^\nu$ if and only if $\phi\in 2\Z$ in the S-type case and $\phi\in 4\Z$ in the C-type case.
\end{Rem}

We finish this subsection by exhibiting generators of $\mfk_\phi^\C$.

\begin{Prop}\label{PropFixTheta}
If $\phi\in\R\setminus (1+2\Z)$, then the Lie algebra $\mfk_\phi^\C$ is generated by the following elements: $H \in \mfh^\theta$, $X_{\pm \alpha}$ for $\alpha\in\Pi_X$, $X_{\alpha} + \theta(X_\alpha)$ for the nondistinguished roots $\alpha\in\Pi\setminus\Pi_X$, plus the following elements:
\begin{itemize}
\item{\textup{S-type:}} $X_{\alpha_o}+\theta(X_{\alpha_o})-s_o H_{\alpha_o}$, where $s_o=i\tan(\frac{\pi \phi}{2})$;
\item{\textup{C-type:}} $X_{\alpha_o}+c_o\theta(X_{\alpha_o})$ and $X_{\alpha_{o'}}+c_o^{-1}\theta(X_{\alpha_{o'}})$, where
 $c_o=-\cot(\frac{\pi}{4}(\phi-1))$.
\end{itemize}
\end{Prop}

\begin{proof}
We denote by $\mfg_\phi$ the Lie algebra generated by the elements in the formulation. For $\phi=0$, the generators of $\mfg_0$ are the adjoints of the generators of $\mfg^\theta$ from~\cite{MR3269184}*{Lemma~2.8}. Since $\mfg^\theta$ is $*$-invariant, we therefore get $\mfg_0=\mfg^\theta$, that is, the proposition is true for $\phi=0$. In order to prove it for all $\phi\in\R\setminus\{1+2\Z\}$ it suffices to show that $\mfg_\phi=(\Ad g_\phi)(\mfg_0)$.

We will check how $\Ad g_\phi$ acts on the generators of $\mfg_0$. By definition, $\Ad g_\phi$ is the identity map on $\mfh^\theta=\mfh^+$. By Lemma~\ref{lem:C_0} it is also the identity map on $X_{\pm\alpha}$ for $\alpha\in \Pi_X\subset C_0$.

Next, consider a nondistinguished root $\alpha\in\Pi\setminus\Pi_X$. Then $\alpha\in C_{i,i+1}$ for some $1\le i\le s-1$. As $\theta=(\Ad g_{-1})\circ\nu\circ(\Ad g_1)$, from Lemma~\ref{lem:PijCij} we see that $\theta(x_{\alpha}) = x_{\alpha}$ and $\theta(y_{\alpha}) = -y_{\alpha}$. It follows that $\theta(e_{\alpha}) = -\eps(\alpha)e_{-\alpha'}$, and therefore
\[
X_{\alpha} + \theta(X_{\alpha}) = i e_{\alpha} + i\theta(e_{\alpha}) = ix_{\alpha}.
\]
By Lemma~\ref{lem:PijCij}, $\Ad g_\phi$ acts trivially on this element.

It remains to understand what happens with the generators corresponding to the distinguished roots. Consider the S-type case.
Then the distinguished root is $\alpha_o = \gamma_s$. By Lemma~\ref{lem:P_0} we have $\theta(x_s)=-x_s$ and $\theta(y_s)=y_s$, hence $\theta(e_{\gamma_s})=-e_{-\gamma_s}$. Therefore
\[
X_{\alpha_o}+\theta(X_{\alpha_o})=i y_s.
\]
By Lemma~\ref{lem:P_0} we then get
\[
(\Ad g_\phi)(X_{\alpha_o}+\theta(X_{\alpha_o}))=\cos\Bigl(\frac{\pi \phi}{2}\Bigr)(X_{\alpha_o}+\theta(X_{\alpha_o}))-i\sin\Bigl(\frac{\pi \phi}{2}\Bigr)H_{\gamma_s},
\]
which is exactly the remaining generator of $\mfg_\phi$ multiplied by $\cos(\frac{\pi \phi}{2})$.

Consider now the C-type case. In this case the distinguished roots are $\alpha_o\in P_s$ and $\alpha_{o'}\in C_s$. Generally, if $\gamma\in P_s$ and $\alpha\in C_s$ are such that $\gamma+\alpha=\gamma_s$, then by Lemma~\ref{lem:PiCi} we have $\theta(e_\gamma)=-\bar N_{\gamma_s,-\gamma}e_{-\alpha}$. Applying the same lemma again we get
\[
(\Ad g_\phi)(X_\gamma+\theta(X_\gamma))=\Bigl(\cos\Bigl(\frac{\pi \phi}{4}\Bigr)-\sin\Bigl(\frac{\pi \phi}{4}\Bigr)\Bigr)X_\gamma+\Bigl(\cos\Bigl(\frac{\pi \phi}{4}\Bigr)+\sin\Bigl(\frac{\pi \phi}{4}\Bigr)\Bigr)\theta(X_\gamma).
\]
For $\gamma={\alpha_o}$ the right hand side is, up to the factor $\cos(\frac{\pi \phi}{4})-\sin(\frac{\pi \phi}{4})=-\sqrt{2}\sin(\frac{\pi(\phi-1)}{4})$, the generator of $\mfg_\phi$ corresponding to $\alpha_o$. We similarly get
\[
(\Ad g_\phi)(X_\alpha+\theta(X_\alpha)) = \sqrt{2}\cos\Bigl(\frac{\pi (\phi-1)}{4}\Bigr) X_\alpha-\sqrt{2}\sin\Bigl(\frac{\pi (\phi-1)}{4}\Bigr)\theta(X_\alpha),
\]
so again we see that for $\alpha=\alpha_{o'}$ the right hand side is, up to a factor, the corresponding generator of~$\mfg_\phi$.
Thus the identity $\mfg_\phi=(\Ad g_\phi)(\mfg_0)$ is proved.
\end{proof}

\begin{Def}\label{def:G-phi}
Denote by $G_\phi$ the subgroup $G^{\theta_\phi}=(\Ad g_{\phi-1})(G^\nu)$ of~$G$.
\end{Def}

\subsection{Coactions of quantized multiplier algebras}\label{ssec:coactions}

Let us relate the computation of the previous subsection to the associators from Section \ref{ssec:cyclotomic-KZ}.

Given a reductive algebraic subgroup $H$ (which will be $G_\phi$) of $G$, consider a coaction $(\mcU(H)\fpser,\alpha)$ of a multiplier Hopf algebra $(\mcU(G)\fpser,\Delta_h)$. By Lemma~\ref{lem:twisting-to-Delta}, if $\Delta_h$ and $\alpha$ both equal $\Delta$ modulo $h$, they can be twisted to $\Delta$. We will assume that this can be done by elements satisfying extra properties. Specifically, assume there exist $\mcF\in\mcU(G\times G)\fpser$ and $\mcG\in\mcU(H\times G)\fpser$ such that
\begin{gather}
\label{eq:G-twist}\begin{align}
\mcG^{(0)}&=1,&(\id\otimes\epsilon)(\mcG)&=1,& \alpha &= \mcG\Delta(\cdot)\mcG^{-1}, \end{align}\\
\label{eq:F-twist}\begin{align}
\mcF^{(0)} &= 1,& (\epsilon\otimes\id)(\mcF) &= (\id\otimes\epsilon)(\mcF)=1,& \Delta_{h} &= \mcF\Delta(\cdot)\mcF^{-1},
\end{align}\\
%\notag
\label{eq:Dr-twist-lin-ord-term}
\mcF = 1 + h\frac{i r}{2} + O(h^2), \\
\label{eq:Drinfeld-twist}
(\id\otimes\Delta)(\mcF^{-1})(1\otimes\mcF^{-1})(\mcF\otimes1)(\Delta\otimes\id)(\mcF)=\Phi_\KZ.
\end{gather}
We remind that $\Phi_\KZ=\Phi(\hbar t^\mfu_{12},\hbar t^\mfu_{23})\in U(\mfg)^{\otimes 3}\fpser$ is Drinfeld's KZ-associator for $G$. Then by twisting by $(\mcF^{-1},\mcG^{-1})$ we get a quasi-coaction $(\mcU(H)\fpser,\Delta,\Psi)$ of $(\mcU(G)\fpser,\Delta,\Phi_\KZ)$ and we can try to apply the results of Section~\ref{sec:classification}.

\begin{Thm}
\label{thm:assoc-for-prescribed-classical-limit}
Let $\mfu^\nu < \mfu$ be a Hermitian symmetric pair, and let $G_\phi$ be as in Definition \ref{def:G-phi} for some $\phi\in\R\setminus(1+2\Z)$.
Assume we are given a coaction $\alpha\colon\mcU(G_\phi)\fpser\to\mcU(G_\phi\times G)\fpser$ of a multiplier Hopf algebra $(\mcU(G)\fpser,\Delta_h)$. Assume also that there exist $\mcF\in\mcU(G\times G)\fpser$ and $\mcG\in\mcU(G_\phi\times G)\fpser$ satisfying conditions~\eqref{eq:G-twist}--\eqref{eq:Drinfeld-twist}. Then there exist unique $s_\phi\in\R$ and $\mu\in h\C\fpser$  such that the coaction is obtained by twisting the quasi-coaction $(\mcU(G_\phi)\fpser,\Delta,\Psi_{\KZ,s_\phi;\mu})$ of $(\mcU(G)\fpser,\Delta,\Phi_\KZ)$. The parameter~$s_\phi$ is determined by
\begin{equation}\label{eq:s-phi}
\sin\Bigl(\frac{\pi \phi}{2}\Bigr) = \tanh\Bigl(\frac{\pi s_\phi}{2}\Bigr).
\end{equation}
If, in addition, $\mcF$ and $\mcG$ are chosen to be unitary, then $\mu\in h\R\fpser$.
\end{Thm}

Here $\Psi_{\KZ,s_\phi;\mu}$ is defined for the Hermitian symmetric pair $\mfk_\phi<\mfu$ as in Section~\ref{ssec:cyclotomic-KZ}, using the element $Z_\phi=(\Ad g_{\phi-1})(Z_\nu)$ of~$\mfz(\mfk_\phi)$. We will give examples of coactions satisfying the assumptions of the theorem in Section~\ref{sec:Letzter-Kolb}.

\smallskip

We will need the following lemma for the uniqueness part.

\begin{Lem}[cf.~\cite{MR1080203}*{Proposition~3.2}]\label{lem:unique-twist}
Assume we are given a homomorphism $\Delta_h\colon\mcU(G)\fpser\to\mcU(G\times G)\fpser$ and two elements $\mcF,\mcF'\in \mcU(G\times G)\fpser$ satisfying~\eqref{eq:F-twist} and the identity
\begin{equation}\label{eq:two-twists}
(\id\otimes\Delta)(\mcF^{-1})(1\otimes\mcF^{-1})(\mcF\otimes1)(\Delta\otimes\id)(\mcF)
=(\id\otimes\Delta)(\mcF'^{-1})(1\otimes\mcF'^{-1})(\mcF'\otimes1)(\Delta\otimes\id)(\mcF'),
\end{equation}
defining a $G$-invariant element of $\mcU(G^3)\fpser$.
Then there exists a unique central element $u\in\mcU(G)\fpser$ such that $u=1$ modulo $h$ and
\[
\mcF'=\mcF(u\otimes u)\Delta(u)^{-1}.
\]
If, in addition, $\mcF$ and $\mcF'$ are unitary, then $u$ is also unitary.
\end{Lem}

\bp
To be able to use an inductive construction for $u$, it suffices to show that if $\mcF=\mcF'$ modulo $h^{n+1}$, then there exists a central element $T\in\mcU(G)$ such that
\[
\mcF'^{(n+1)}= \mcF^{(n+1)}+ T\otimes 1+1\otimes T-\Delta(T).
\]
By considering the order $n+1$ terms in~\eqref{eq:two-twists} we get that the element $S=\mcF'^{(n+1)} - \mcF^{(n+1)}$ satisfies
\[
(\id\otimes\Delta)(S)+1\otimes S - S\otimes1 - (\Delta\otimes\id)(S)=0.
\]
This means that $S$ is a $2$-cocycle in the complex $\tB_G=\tB_{G,e}$ from Section~\ref{sec:geom-comput}. Furthermore, $\mcF^{-1}\mcF'$ commutes with the image of $\Delta$, hence $S$ also commutes with the image of $\Delta$. Therefore $S$ is a $2$-cocycle in the complex $B_G=\tB_G^G$. By Proposition~\ref{prop:comput-cohoch-mult-model}, the cohomology of $\tB_G$ is $\bigwedge \mfg$, and then the cohomology of $B_G$ is $(\bigwedge\mfg)^\mfg$. In particular, $\rH^2(B_G)=0$, which implies the existence of $T$.

\smallskip

Assume now that we have two central elements $u$ and $u'$ with the required properties. Consider the central element $v=u' u^{-1}\in\mcU(G)$. Then $v=1$ modulo $h$ and $\Delta(v)=v\otimes v$. Assume $v\ne1$ and take the smallest $n\ge1$ such that $v^{(n)}\ne0$. Then $\Delta(v^{(n)})=v^{(n)}\otimes1+1\otimes v^{(n)}$, hence $v^{(n)}\in\mfg$. But as $v$ is central, we must have $v^{(n)}\in\mfz(\mfg)=0$, which is a contradiction.

\smallskip

Finally, if $\mcF$ and $\mcF'$ are unitary, then $(u^{-1})^*$ has the defining properties of $u$, hence $(u^{-1})^*=u$.
\ep

\bp[Proof of Theorem~\ref{thm:assoc-for-prescribed-classical-limit}]
Taking $s_\phi\in\R$ defined by~\eqref{eq:s-phi}, let us first prove the existence of $\mu$. By twisting the coaction $\alpha$ by $(\mcF^{-1},\mcG^{-1})$ we obtain a quasi-coaction $(\mcU(G_\phi)\fpser,\Delta,\Psi)$ of $(\mcU(G)\fpser,\Delta,\Phi_\KZ)$, where
\[
\Psi=(\id\otimes \Delta)(\mcG^{-1})(1\otimes\mcF^{-1})(\alpha\otimes \id)(\mcG)(\mcG\otimes 1)
=(\id\otimes \Delta)(\mcG^{-1})(1\otimes\mcF^{-1})(\mcG\otimes 1)(\Delta\otimes \id)(\mcG),
\]
and hence
\begin{equation} \label{eq:Psi-r}
\Psi^{(1)}=-\frac{i r_{12}}{2}+d_\cH(\mcG^{(1)}).
\end{equation}

Let us use the subscript $\phi$ for the constructions we had in Section~\ref{sec:classification} applied to the pair $\mfk_\phi<\mfu$.
By Theorem~\ref{thm:modified-KZ-universality}, in order to prove the existence of $\mu$, it suffices to compute $\langle \Omega_\phi,\Psi^{(1)}\rangle$.

Since $\Omega_\phi$ is a cycle in the chain complex $\tB'_{G,G_\phi}$, the term $d_\cH(\mcG^{(1)})$ in~\eqref{eq:Psi-r} does not contribute to the pairing. By Proposition~\ref{prop:Cayley-rotates-rmat}, we have
\[
r-\cos\Bigl(\frac{\pi}{2}(1-\phi)\Bigr)(\Ad g_{\phi-1})^{\otimes 2}(r)\in\mfg_\phi\otimes\mfg+\mfg\otimes\mfg_\phi.
\]
We also have $i r=t^{\mfm_{\nu+}}-t^{\mfm_{\nu-}}$ modulo $\mfg^\nu\otimes\mfg^\nu$, hence $(\Ad g_{\phi-1})^{\otimes 2}(i r)=t^{\mfm_{\phi+}}-t^{\mfm_{\phi-}}$ modulo $\mfg_\phi\otimes\mfg_\phi$. Therefore
\[
i r-\sin\Bigl(\frac{\pi \phi}{2}\Bigr)(t^{\mfm_{\phi+}}-t^{\mfm_{\phi-}})\in\mfg_\phi\otimes\mfg+\mfg\otimes\mfg_\phi.
\]

As $\mfg_\phi$ centralizes $Z_\phi$, any cochain in $1\otimes\mfg_\phi\otimes\mfg+1\otimes\mfg\otimes\mfg_\phi$ pairs trivially with $\Omega_\phi$. Hence
\[
\langle \Omega_\phi,\Psi^{(1)}\rangle=-\frac{1}{2}\sin\Bigl(\frac{\pi \phi}{2}\Bigr) \langle \Omega_\phi,t^{\mfm_{\phi+}}_{12}-t^{\mfm_{\phi-}}_{12}\rangle.
\]
By Theorem~\ref{thm:modified-KZ-universality} and identity~\eqref{eq:s-detection}, it follows that $(\mcU(G_\phi)\fpser,\Delta,\Psi)$ is obtained by twisting the quasi-coaction $(\mcU(G_\phi)\fpser,\Delta,\Psi_{\KZ,s_\phi;\mu})$ for some $\mu\in h\C\fpser$, and if, in addition, $\mcF$ and $\mcG$ are unitary, we can choose $\mu\in h\R\fpser$.

\smallskip

Assume now that the coaction $\alpha\colon\mcU(G_\phi)\fpser\to\mcU(G_\phi\times G)\fpser$ is obtained by twisting the quasi-coaction $(\mcU(G_\phi)\fpser,\Delta,\Psi_{\KZ,s';\mu'})$ of $(\mcU(G)\fpser,\Delta,\Phi_\KZ)$ for some other $s'\in\R$ and $\mu'\in h\C\fpser$. Let $(\mcF',\mcG')$ be a pair defining this twisting. By Lemma~\ref{lem:unique-twist}, we have  $\mcF'=\mcF(u\otimes u)\Delta(u)^{-1}$ for a central element $u\in\mcU(G)\fpser$ such that $u=1$ modulo $h$. But then the pairs $(\mcF',\mcG')$ and $(\mcF,(1\otimes u^{-1})\mcG')$ define the same twistings. In other words, without loss of generality we may assume that $\mcF'=\mcF$. Then the quasi-coaction $(\mcU(G_\phi)\fpser,\Delta,\Psi_{\KZ,s';\mu'})$ is obtained from $(\mcU(G_\phi)\fpser,\Delta,\Psi_{\KZ,s_\phi;\mu})$ by twisting with $(1,\mcG'^{-1}\mcG)$. By Theorem~\ref{thm:modified-KZ-universality} this implies that $s'=s_\phi$ and $\mu'=\mu$. \ep

\begin{Rem}
Using isomorphisms and twistings that are not trivial modulo $h$, we can pass from $G_\phi$ to its conjugate by an element $g\in U$.
Namely, the conjugation by $\Ad g$ in the $0$th leg transforms the quasi-coaction $(\mcU(G_\phi)\fpser,\Delta,\Psi_{\KZ,s_\phi;\mu})$ of $(\mcU(G)\fpser,\Delta,\Phi_\KZ)$ into the isomorphic quasi-coaction
\[
(\mcU(g G_\phi g^{-1})\fpser,(\Ad g)_0\circ\Delta\circ(\Ad g)^{-1},(\Ad g)_0(\Psi_{\KZ,s_\phi;\mu}))
\]
of $(\mcU(G)\fpser,\Delta,\Phi_\KZ)$, and then the twisting by $1\otimes g\in\mcU(g G_\phi g^{-1}\times G)\fpser$ gives the quasi-coaction $(\mcU(g G_\phi g^{-1})\fpser,\Delta,\Psi_{\KZ,s_\phi;\mu})$ of $(\mcU(G)\fpser,\Delta,\Phi_\KZ)$, where $\Psi_{\KZ,s_\phi;\mu}$ now denotes the associator defined by the symmetric pair $(\Ad g)(\mfk_\phi)<\mfu$.
\end{Rem}

Before moving on to the next part, let us explain some geometric structures motivating the above computations.

Starting from the KZ-equations, after fixing a twist $\mcF$ satisfying \eqref{eq:F-twist}--\eqref{eq:Drinfeld-twist}, an associator $\Psi \in \mcU(G^\nu \times G^2)\fpser$ defines an associative product $*^\Psi_h$ on $\mcO(U/U^\nu)\fpser = \mcO(U)^{U^\nu}\fpser$, see~\cite{MR2126485}*{Section 6}.
Moreover, the algebra $(\mcO(U/U^\nu)\fpser, *^\Psi_h)$ becomes a comodule algebra over the \emph{quantized function algebra} $\mcO_h(U)$, the restricted dual Hopf algebra of $(\mcU(G)\fpser, \Delta_h)$.
This structure corresponds to the module category $((\Rep G^\nu)\fpser, \Psi)$ under the Tannaka--Krein type duality for module categories and coactions.

To be more precise, given an element $\Psi \in \mcU(G^\nu \times G^2)^{G^\nu}\fpser$ such that $(\mcU(G^\nu)\fpser, \Delta, \Psi)$ is a quasi-coaction of $(\mcU(G)\fpser, \Delta, \Phi_\KZ)$, we define $f_1 *^\Psi_h f_2$ by
\[
\langle f_1 *^\Psi_h f_2, T \rangle = \langle f_1 \otimes f_2, \mcF \Delta(T) (\epsilon \otimes \id)(\Psi) \rangle \quad (T \in \mcU(G)).
\]
As we have $\mcF^{(1)} - \mcF^{(1)}_{21} = i r$, the corresponding Poisson bracket
\[
\{f_1, f_2\}_{\Psi} = \lim_{h \to 0} \frac{1}{i h} (f_1 *^{\Psi}_h f_2 - f_2 *^{\Psi}_h f_1)
\]
is characterized by
\[
\langle\{f_1, f_2\}_{\Psi}, T \rangle = \bigl\langle f_1 \otimes f_2, r \Delta(T) - i \Delta(T) (\epsilon \otimes \id)(\Psi^{(1)} - \Psi^{(1)}_{0,2,1}) \bigr\rangle.
\]

If we have $\Psi' = \mcH_{0, 12}\Psi \mcH_{01, 2}^{-1} \mcH_{0,1}^{-1}$ with $\mcH \in 1 + h\mcU(G^\nu\times G)^\mfk\fpser$, the invertible transformation $\rho_\mcH$ of $\mcO(U/U^\nu)\fpser$ characterized by
\[
\langle \rho_\mcH(f), T \rangle = \langle f, T (\epsilon \otimes \id)(\mcH) \rangle
\]
satisfies $\rho_\mcH(f_1 *^{\Psi}_h f_2) = \rho_\mcH(f_1) *^{\Psi'}_h \rho_\mcH(f_2)$, i.e., we get isomorphic deformation quantizations from twist equivalent associators.
From this we obtain $\{f_1, f_2\}_{\Psi'} = \{f_1, f_2\}_{\Psi}$ for such $\Psi$ and $\Psi'$.

Combined with the identification of $\rH^2(B_{G,G^\nu})$, we obtain a decomposition
\begin{equation}\label{eq:Pois-bra-for-Psi}
\{f_1, f_2\}_{\Psi} = \{f_1, f_2\}_\alpha + x \{f_1, f_2\}_\beta,
\end{equation}
for some complex number $x$ (which is real for unitary $\Psi$), with
\begin{align*}
\{f_1, f_2\}_\alpha &= m r^{(l,l)}(f_1 \otimes f_2),&
\{f_1, f_2\}_\beta &= i m (t^{\mfm_{\nu+}}_{12}  - t^{\mfm_{\nu-}}_{12})^{(r,r)}(f_1 \otimes f_2).
\end{align*}
Note that $\{f_1, f_2\}_\beta$ is invariant for the left translation action of $U$, while $\{f_1, f_2\}_\alpha$ is equivariant for the Sklyanin bracket on $U$.
The left invariance of $\{f_1, f_2\}_\beta$  implies that the associated Poisson bivectors commute with respect to the Schouten--Nijenhuis bracket.
Note also that we can write $\{f_1, f_2\}_\beta = m r^{(r,r)}(f_1 \otimes f_2)$, and in fact it is the Kostant--Kirillov--Souriau bracket if we identify $U/U^\nu$ with a coadjoint orbit as in Remark \ref{rem:coadj-orb-and-hoch-cochains}, see \cite{MR1357743}.
The Poisson bracket associated with $\Psi_{\KZ,s;\mu}$ is given by $x = \tanh(\hlf{\pi s})$.

Turning to the side of coisotropic subgroups, by Corollary \ref{cor:coisotropic}, $U / K_\phi$ admits a Poisson bracket which is just the restriction of the Sklyanin bracket: $\{f_1, f_2\}_\phi = \{f_1, f_2\}_{\Sk}$ for $f_i \in \mcO(U/K_\phi) \subset \mcO(U)$.
This gives the structure of a Poisson homogeneous space on $U/K_\phi \cong U/U^\nu$ over $(U, r)$.

In fact, any Poisson homogeneous structure of $U/U^\nu$ over $(U, r)$ is of the form \eqref{eq:Pois-bra-for-Psi} (this seems to be folklore, but the idea can be traced back to \cite{MR1087382}*{Appendix}).
We thus obtain a correspondence between the parameters $\phi$ and $s$ through the comparison of the factor $x$.

\begin{Rem}
The bracket \eqref{eq:Pois-bra-for-Psi} defines a Poisson action of $(U, r)$ on $U/U^\nu$ for any $x$, but there is a distinguished range which naturally shows up in our considerations: $-1 < x < 1$.
Starting from the KZ-equations the value of $\tanh(\hlf{\pi s})$ is confined in this range when $s \in \R$ and we have a unitary associator.
On the side of the Cayley transform, this corresponds to the case that $K_\phi$ is coisotropic but \emph{not} a Poisson--Lie subgroup of $(U, r)$.
In this case the Poisson bivector vanishes on a nondiscrete subset of $U/U^\nu$, while in the limit case $x = \pm 1$ it only vanishes at one point, see the next subsection.
When $\absv{x} > 1$, we get a symplectic structure.
\end{Rem}

\subsection{Regularity of ribbon braids}\label{sec:reg-rib-br}

We finish this section with a technical result, which we will need later, showing that in the non-Poisson subgroup case a ribbon braid in the algebra of formal Laurent series must lie in the algebra of formal power series. More precisely, we will prove the following.

\begin{Thm}
\label{Thm:not-PL-subgr-forces-K-mat-to-be-formal-pow-ser}
Let $\mfu^\sigma<\mfu$ be a Hermitian symmetric pair, and $r$ be an $r$-matrix of the form \eqref{EqComplr} for some Cartan subalgebra $\mft < \mfu$ and a choice of positive roots.
Assume we are given a coaction
\[
\alpha\colon\mcU(G^\sigma)\fpser\to\mcU(G^\sigma\times G)\fpser
\]
of a multiplier Hopf algebra $(\mcU(G)\fpser,\Delta_h)$ such that there exist $\mcF\in\mcU(G\times G)\fpser$ and $\mcG\in\mcU(G^\sigma\times G)\fpser$ satisfying conditions~\eqref{eq:G-twist}--\eqref{eq:Drinfeld-twist}. Assume also that there exists a ribbon braid
\[
\mcE \in \prod_{\substack{\rho\in\Irr G^\sigma,\\ \pi\in\Irr G}} \left( \End(V_\rho) \otimes \End(V_\pi) \fLauser \right)
\]
for this coaction with respect to the $R$-matrix $\mcR=\mcF_{21}\exp(-ht^\mfu)\mcF^{-1}$ of $(\mcU(G)\fpser,\Delta_h)$. Then, unless~$U^\sigma$ is a Poisson--Lie subgroup of $(U,r)$, we must have
$\mcE \in \mcU(G^\sigma \times G) \fpser$.
\end{Thm}

Consider the $K$-matrix $\mcK=(\epsilon\otimes\id)(\mcE)$ and put $u=(\epsilon\otimes\id)(\mcG)$. Then from identity~\eqref{eq:rib-sig-tw-1} for our ribbon braid we get
\[
(u\otimes1)\mcE (u^{-1}\otimes1)=\mcR_{21}(1\otimes\mcK)\mcR.
\]
Therefore in order to prove the theorem it suffices to show that $\mcK\in\mcU(G) \fpser$ unless $U^\sigma$ is a Poisson--Lie subgroup of $(U,r)$.

Identities \eqref{eq:rib-sig-tw-0} and \eqref{eq:rib-sig-tw-2} imply
\begin{align}
\mcK (\Ad u)(T)&=(\Ad u)(T)\mcK\quad\text{for all}\quad T\in\mcU(G^\sigma)\fpser,\label{eq:K-0}\\
\Delta_h(\mcK)&=\mcR_{21}(1\otimes\mcK)\mcR(\mcK\otimes1).\label{eq:K-2}
\end{align}
A key step now is to prove the following.

\begin{Prop}
\label{prop:non-invert-K-mat-in-adj-impl-PL-subgr}
With $\sigma$, $\Delta_h$ and $\mcR$ as in Theorem~\ref{Thm:not-PL-subgr-forces-K-mat-to-be-formal-pow-ser}, assume we are given an invertible element
\[
\mcK\in \prod_{\pi\in\Irr G} \left(\End(V_\pi) \fLauser\right)
\]
satisfying conditions \eqref{eq:K-0} and \eqref{eq:K-2} for some $u\in\mcU(G)\fpser$, $u=1$ modulo $h$. Assume also that $\mcK^{\ad} = \ad(\mcK) \in \End(\mfg)\fLauser$ has a negative order term. Then $U^\sigma$ is a Poisson--Lie subgroup of~$(U,r)$.
\end{Prop}

We will prove this by analyzing certain Poisson structures on $U/U^\sigma$.

\smallskip

Take a Cartan subalgebra $\tilde\mft$ and a system $\tilde\Pi$ of simple roots as in Section~\ref{ssec:Cartan}, but for the Hermitian symmetric pair $\mfu^\sigma<\mfu$.
Then there exists $g\in U$ such that $\Ad g$ maps $\mft$ onto $\tilde\mft$, while the dual map maps $\tilde\Pi$ onto $\Pi$.
Now, the Cartan subalgebra $\mft$ is as in Section~\ref{ssec:Cartan} for the automorphism $\nu = (\Ad g)^{-1}\circ\sigma\circ(\Ad g)$ of $\mfu$ and some $Z_\nu\in\mfz(\mfu^\nu)$.

As we remarked in Section \ref{ssec:coactions}, on the compact symmetric space $U / U^\nu$, both the left and the right actions of the $r$-matrix $r$ define (real) Poisson bivector fields, denoted by $r^{(l,l)}$ and $r^{(r,r)}$. (As part of the claim, the bivector field $r^{(r,r)}$ on $U$ preserves the right $U^\nu$-invariant functions.)
Thus, the linear combinations of these commuting bivector fields define Poisson brackets on $U / U^\nu$ as in \eqref{eq:Pois-bra-for-Psi}.

\begin{Lem}
\label{lem:zero-loci-of-Pois-brackets}
The bivector field $r^{(l,l)} - r^{(r,r)}$ on $U / U^\nu$ vanishes only at $[e]$.
Similarly, the bivector field $r^{(l,l)} + r^{(r,r)}$ vanishes only at $[\tilde{w}_0]$, where $\tilde{w}_0\in U$ is any lift of the longest element $w_0$ of the Weyl group.
\end{Lem}

\begin{proof}
As $U^\nu$ is a Poisson--Lie subgroup of $(U,r)$, the first statement follows from the well-known description of the Poisson leaves of the reduction $r^{(l,l)} - r^{(r,r)}$ of the Sklyanin bracket~\cite{MR1037412}.

As for the bivector field $r^{(l,l)} + r^{(r,r)}$, first note that $(\Ad{\tilde{w}_0})^{\otimes 2}(r) = - r$.
This means that the Poisson bivector $r^{(l,l)} + r^{(r,r)}$ on $U / U^\nu$ vanishes at the point $[\tilde{w}_0^{-1}]=[\tilde{w}_0]$.
Thus the $U$-equivariant diffeomorphism
\[
U / U^\nu \to U / (\Ad{\tilde{w}_0})(U^\nu), \quad [g] \mapsto [g \tilde{w}_0^{-1}],
\]
transforms the Poisson bivector $r^{(l,l)} + r^{(r,r)}$ on $U / U^\nu$ into a Poisson bivector $\Pi'$ on $U / (\Ad{\tilde{w}_0})(U^\nu)$ which vanishes at the basepoint.

On the other hand, $(\Ad{\tilde{w}_0})(U^\nu)$ is again a Poisson--Lie subgroup of $U$. Hence $\Pi'$ has to agree with the reduction of the Sklyanin bracket, which vanishes only at $[e]\in U / (\Ad{\tilde{w}_0})(U^\nu)$. As a result $r^{(l,l)} + r^{(r,r)}$ vanishes only at $[\tilde{w}_0]\in  U/ U^\nu$.
\end{proof}

Identity~\eqref{eq:K-2} implies that $\mcK$ satisfies the reflection equation
\begin{equation*}
(\mcK\otimes1)\mcR_{21}(1\otimes\mcK)\mcR=\mcR_{21}(1\otimes\mcK)\mcR(\mcK\otimes1).
\end{equation*}
Using this we are going to introduce another Poisson structure on $U/U^\nu$ following~\citelist{\cite{MR1998103}\cite{MR1964382}}.

Let us write $t$ for $t^\mfu$. Consider a finite dimensional representation $\pi$ of $G$.
Put $t^\pi=(\pi\otimes\pi)(t)$ and consider the set
\[
\mcM^\pi=\{A\in \End(V_\pi) \mid (A\otimes A)t^\pi=t^\pi(A\otimes A)\}.
\]
We have three actions of $g\in U$ on $\End(V_\pi)$ given by multiplication by $\pi(g)$ on the left, on the right, and by conjugation by $\pi(g)$.
For $X\in\mfg$, we will denote by $X^{(l)}$, $X^{(r)}$ and $X^{(\ad)}$ the corresponding vector fields on $\End(V_\pi)$. Thus, $X^{(\ad)}=X^{(l)}-X^{(r)}$.
The \emph{RE bracket} on~$\mcM^\pi$ is defined by
\[
\{f_1, f_2\}_{\RE} =m \bigl(r^{(\ad,\ad)} + i(t^{(l,r)} - t^{(r,l)})\bigr)(f_1 \otimes f_2).
\]
More precisely, since $\mcM^\pi$ is not a smooth manifold in general, this defines a Poisson bracket on the algebra of polynomial functions on $\mcM^\pi$.
But in any case what is going to matter to us is only that $r^{(\ad,\ad)} + i(t^{(l,r)} - t^{(r,l)})$ is a well-defined bivector field on $\End(V_\pi)$.

The following observation is from~\cite{MR1964382}.

\begin{Lem}\label{lem:DoMu}
Put $\mcR^\pi=(\pi\otimes\pi)(\mcR)$, and assume that $A_h\in\End(V_\pi)\fpser$ has constant term $A=A_h^{(0)}\in\mcM^\pi$ and satisfies the reflection equation
\[
(A_h\otimes1)\mcR^\pi_{21}(1\otimes A_h)\mcR^\pi=\mcR^\pi_{21}(1\otimes A_h)\mcR^\pi(A_h\otimes1).
\]
Then the RE bracket vanishes at $A$.
\end{Lem}

\bp
Letting $r^\pi=(\pi\otimes\pi)(r)$, the bivector field $r^{(\ad,\ad)} + i(t^{(l,r)} - t^{(r,l)})$ at the point $A\in\End(V_\pi)$ is
\begin{multline*}
r^\pi(A\otimes A)+(A\otimes A)r^\pi-(1\otimes A)r^\pi(A\otimes 1)-(A\otimes1)r^\pi(1\otimes A)\\
+(1\otimes A)it^\pi(A\otimes 1)-(A\otimes 1)it^\pi(1\otimes A).
\end{multline*}
By looking at the order one terms in the reflection equation and using that $\mcR=1-h(t+i r)+O(h^2)$, $\mcR_{21}=1-h(t-i r)+O(h^2)$, and $(A\otimes A)t^\pi=t^\pi(A\otimes A)$, we see that this bivector is zero.
\ep

We want to apply this to the element $h^k\mcK^{\ad}$ for $k = -\ord(\mcK^{\ad})$, so that $h^k\mcK^{\ad}$ starts with an order zero term.
Denote the orthogonal (with respect to the $U$-invariant Hermitian form) projection $\mfg\to\mfm_{\nu\pm}$ by~$P_\pm$, and put $P_{\pm}^g = (\Ad{g})(P_{\pm})$, so $P_\pm^g$ is the projection onto $\mfm_{\sigma\pm}$.

\begin{Lem}\label{lem:lowest-K}
Under the assumptions of Proposition \ref{prop:non-invert-K-mat-in-adj-impl-PL-subgr}, the lowest order nonzero coefficient of $\mcK^{\ad}$ is, up to a scalar factor, either $P^g_+$ or $P^g_-$.
\end{Lem}

\bp
Let us more generally consider the elements $\mcK^\pi=\pi(\mcK)$ for finite dimensional representations~$\pi$ of~$G$. Denote by $k_\pi$ the order of the lowest nonzero term of $\mcK^\pi$.

Applying the counit to~\eqref{eq:K-2} we get $\epsilon(\mcK)=1$. Consider the contragredient representation $\bar\pi$ of~$G$. The antipode $S_h$ for $(\mcU(G)\fpser,\Delta_h)$ has the form $v S(\cdot)v^{-1}$ for some $v\in\mcU(G)\fpser$, $v=1$ modulo $h$. Applying $m(\id\otimes S_h)$ to~\eqref{eq:K-2} we then get
\[
1= h^{k_\pi+k_{\bar\pi}}\mcK^{\pi,(k_\pi)}S(\mcK^{\bar\pi,(k_{\bar\pi})})+O( h^{k_\pi+k_{\bar\pi}+1})\quad\text{in}\quad\End(V_\pi)\fLauser.
\]
Hence $k_\pi+k_{\bar\pi}\le 0$, and if $k_\pi+k_{\bar\pi}=0$ then $\mcK^{\pi,(k_\pi)}S(\mcK^{\bar\pi,(k_{\bar\pi})})=1$, while if $k_\pi+k_{\bar\pi}<0$ then $\mcK^{\pi,(k_\pi)}S(\mcK^{\bar\pi,(k_{\bar\pi})})=0$.

Consider now the adjoint representation $\ad$. Since it is self-conjugate and by assumption $k_{\ad}<0$, we get
$\mcK^{\ad,(k_{\ad})}S(\mcK^{\ad,(k_{\ad})})=0$. By \eqref{eq:K-0} we know also that $\mcK^{\ad,(k_{\ad})}$ is an intertwiner for $U^\sigma$. As a representation of $\mfg^\sigma$, we have the decomposition
\[
\mfg = \mfz(\mfg^\sigma) \oplus [\mfg^\sigma,\mfg^\sigma] \oplus \mfm_{\sigma+} \oplus \mfm_{\sigma-},
\]
where the derived Lie algebra $[\mfg^\sigma,\mfg^\sigma]$ is either zero, simple or the sum of two simple ideals.
As these components are mutually nonequivalent, $\mcK^{\ad,(k_{\ad})}$ is a linear combination of up to $5$ projections.

The antipode $S$ restricted to the block $\End(\mfg)$ of $\mcU(G)$ is the adjoint map with respect to the invariant form $(\cdot,\cdot)_\mfg$, that is,
\[
(TX,Y)_\mfg=(X,S(T)Y)_\mfg\quad\text{for}\quad T\in\End(\mfg) \quad  (X,Y\in\mfg).
\]
Since the invariant form is nondegenerate on the irreducible components of $\mfg^\sigma=\mfz(\mfg^\sigma) \oplus [\mfg^\sigma,\mfg^\sigma]$, we conclude that the corresponding projections are $S$-invariant. We can also conclude that $S(P^g_+)=P^g_-$. Therefore the identity $\mcK^{\ad,(k_{\ad})}S(\mcK^{\ad,(k_{\ad})})=0$ can be true only if $\mcK^{\ad,(k_{\ad})}$ is a scalar multiple of either~$P^g_+$ or~$P^g_-$.
\ep

\begin{Lem}\label{lem:p-pm-in-Mad}
We have $P_{\pm}\in\mcM^{\ad}$.
\end{Lem}

\bp
Since $t^*=t$, in order to prove that $t^{\ad}$ commutes with $P_+\otimes P_+$ it suffices to check that
\begin{equation}\label{eq:tad}
t^{\ad}(\mfm_{\nu+}\otimes\mfm_{\nu+})\subset \mfm_{\nu+}\otimes\mfm_{\nu+}.
\end{equation}
Recall that $t=t^\mfk+t^{\mfm_{\nu+}}+t^{\mfm_{\nu-}}$, where $\mfk=\mfu^\nu$, and $t^{\mfm_{\nu\pm}}\in \mfm_{\nu\pm}\otimes\mfm_{\nu\mp}$. As $[\mfm_{\nu+},\mfm_{\nu+}]=0$, we see that
\[
t^{\ad}=(\ad\otimes\ad)(t^{\mfk})\quad\text{on}\quad \mfm_{\nu+}\otimes\mfm_{\nu+},
\]
which obviously implies \eqref{eq:tad}. The proof for $P_-$ is similar.
\ep

Consider now the $(\Ad U)$-orbit $O_\pm$ of $P_\pm$. By the previous lemma it is contained in $\mcM^{\ad}$. Since $U^\nu$ stabilizes $P_\pm$ and $\mfu^\nu$ is a maximal proper Lie subalgebra of~$\mfu$, it follows that $U^\nu$ is the connected component of the stabilizer of $P_\pm$. Hence $p_\pm\colon U/U^\nu\to O_\pm$, $[g']\mapsto (\Ad g')(P_\pm)$, is a covering map.

\begin{Lem}\label{lem:Poisson-lift}
The RE bracket on $\mcM^{\ad}$ restricts to $O_\pm$. Being lifted to $U/U^\nu$ via $p_\pm$, this restriction coincides with the bracket defined by $r^{(l,l)} \mp r^{(r,r)}$.
\end{Lem}

\bp
The covering map $p_\pm\colon U/U^\nu\to O_\pm$ intertwines the action by left translations with the adjoint action. From this it is clear that the bivector field $r^{(\ad,\ad)}$ at the points of $O_\pm\subset\mcM^{\ad}$ indeed defines a bivector field on $O_\pm$, and, being lifted to $U/U^\nu$, this gives the bivector field $r^{(l,l)}$.

We next want to compare $r^{(r,r)}$ with $- i(t^{(l,r)} - t^{(r,l)})$. We claim that
\[
(d_{[e]}p_+\otimes d_{[e]}p_+)\bigl(r^{(r,r)}([e])\bigr)=- i(t^{(l,r)} - t^{(r,l)})(P_+).
\]
Since the bivector fields $r^{(l,l)}$ and $r^{(r,r)}$ on $U/U^\nu$ coincide at $[e]$ and the pushforward of $r^{(l,l)}$ is $r^{(\ad,\ad)}$, this is equivalent to
\[
r^{(\ad,\ad)}(P_+)=- i(t^{(l,r)} - t^{(r,l)})(P_+).
\]

Using again that $[\mfm_{\nu+},\mfm_{\nu+}]=0$, we get $(\ad X_\alpha)P_+=0$ and $P_+(\ad X_{-\alpha})=0$ for all $\alpha \in \Phi_\nc^+$.
A simple computation using these properties, together with the fact that $U^\nu$ stabilizes $P_+$, gives
\[
r^{(\ad,\ad)}(P_+)=i\sum_{\alpha \in \Phi_\nc^+}\frac{(\alpha,\alpha)}{2}\bigl(P_+(\ad X_{\alpha})\otimes (\ad X_{-\alpha})P_+ - (\ad X_{-\alpha})P_+\otimes P_+(\ad X_{\alpha})\bigr),
\]
and a similar computation for $- i(t^{(l,r)} - t^{(r,l)})(P_+)$ gives the same answer. Thus our claim is proved.

Since $r^{(r,r)}$ is left $U$-invariant and $- i(t^{(l,r)} - t^{(r,l)})$ is $(\Ad U)$-invariant, we then get
\[
(d_{[g']}p_+\otimes d_{[g']}p_+)\bigl(r^{(r,r)}([g'])\bigr)=- i(t^{(l,r)} - t^{(r,l)})(p_+([g'])) \quad (g'\in U).
\]
This finishes the proof of the lemma for $O_+$. The proof for $O_-$ is similar.
\ep

\bp[Proof of Proposition \ref{prop:non-invert-K-mat-in-adj-impl-PL-subgr}]
By Lemmas~\ref{lem:DoMu}, \ref{lem:lowest-K} and \ref{lem:p-pm-in-Mad}, the RE bracket vanishes either at $P_+^g\in O_+$ or at $P_-^g\in O_-$. By Lemma~\ref{lem:Poisson-lift} this means that either $r^{(l,l)} - r^{(r,r)}$ or $r^{(l,l)} + r^{(r,r)}$ vanishes at $[g]\in U/U^\nu$. But then by Lemma~\ref{lem:zero-loci-of-Pois-brackets} we have either $g\in U^\nu$ or $g\in \tilde w_0U^\nu$, and therefore $U^\sigma=(\Ad g)(U^\nu)$ is either~$U^\nu$ or~$(\Ad \tilde w_0)(U^\nu)$.
\ep

\begin{Rem}
As $r^{(l,l)} - r^{(r,r)}$ has to vanish on $p_+^{-1}(P_+)$, we can also conclude that $p_+$ is injective, that is, $U^\nu$ is exactly the stabilizer of $P_+$. Similarly, or by symmetry, $U^\nu$ is the  stabilizer of $P_-$.
\end{Rem}

\begin{proof}[Proof of Theorem \ref{Thm:not-PL-subgr-forces-K-mat-to-be-formal-pow-ser}]
Assume $G^\sigma$ is not a Poisson--Lie subgroup. To prove the theorem it suffices to show that $\mcK^{\pi}\in\End(V_\pi)\fpser$ for all finite dimensional representations $\pi$ of $G$. By Proposition~\ref{prop:non-invert-K-mat-in-adj-impl-PL-subgr} and identity~\eqref{eq:K-2} this is already the case if $\pi$ belongs to the tensor subcategory generated by $\ad$, that is, $\pi$ factors through the adjoint group $G_{\ad}=\Ad G$. For arbitrary $\pi$, if $\mcK^\pi$ contains a term of negative degree, then, using that $\mcR=1$ modulo $h$, we see from~\eqref{eq:K-2} that $\mcK^{\pi^{\otimes n}}$ contains a term of negative degree for all $n\ge1$.
But when $\pi$ is irreducible, we have $\pi^{\otimes n} \in \Rep G_{\ad}$ for some~$n\ge1$, so this cannot happen.
\end{proof}

\section{Letzter--Kolb coideals} \label{sec:Letzter-Kolb}

Our next goal is to put the Letzter--Kolb coideals in the framework of multiplier algebras. This is again easy in the non-Hermitian case. In the Hermitian case, we will relate the classical limits of the coideals to the coisotropic subgroups of the previous section. Combined with a rigidity result for the fusion rules, this will allow us to define multiplier algebra models of the coideals. We will also cast the Balagovi\'c--Kolb universal $K$-matrix in our setting.

Throughout this section we fix a nontrivial involutive automorphism $\theta$ of $\mfu$, a Cartan subalgebra $\mft$ of~$\mfu$ and a system $\Phi^+$ of positive roots such that $\theta$ is in Satake form with respect to $(\mfh,\mfb^+)$, where $\mfh=\mft^\C$.

\subsection{Quantized universal enveloping algebra and Letzter--Kolb coideal subalgebras}\label{sec:def-QUE-and-coid-alg}

Let $I$ be the label set for the simple roots, so $\Pi=\{\alpha_i\}_{i\in I}$. As in Section~\ref{ssec:Cartan}, for every positive root $\alpha$ we fix an element $X_\alpha\in\mfg_\alpha$ normalized so that $[X_\alpha,X^*_\alpha]=H_\alpha$, and put $X_{-\alpha}=X^*_\alpha$.

\smallskip

The quantized universal enveloping algebra $U_h(\mfg)$ is topologically generated over $\C\fpser$ by a copy of~$U(\mfh)$ and elements $E_i$, $F_i$ ($i\in I$) satisfying the following standard relations:
\begin{gather*}
\begin{align*}
[H,E_i]&= \alpha_i(H)E_i,& [H,F_i]&=-\alpha_i(H)F_i,& [E_i,F_j] &= \delta_{ij}\frac{e^{h d_i H_i}- e^{-h d_i H_i}}{e^{h d_i}- e^{-h d_i}},
\end{align*}\\
\sum_{n=0}^{1-a_{ij}}(-1)^n \qbin{1-a_{ij}}{n}{q_i} E_i^{1-a_{ij}-n} E_j E_i^n = 0 = \sum_{n=0}^{1-a_{ij}}(-1)^n \qbin{1-a_{ij}}{n}{q_i} F_i^{1-a_{ij}-n} F_j F_i^n,
 \quad (i \neq j),
\end{gather*}
where $H\in\mfh$, $H_i=H_{\alpha_i}$, $q_i=q^{d_i}=e^{h d_i}$, $d_i=\hlf1(\alpha_i,\alpha_i)$, $(a_{ij})_{i,j}$ is the Cartan matrix, so $a_{ij} = ({\alpha}^\vee_i,\alpha_j)$, and
\begin{align*}
\qbin{m}{n}{q_i} &= \frac{[m]_{q_i}!}{[n]_{q_i}! [m-n]_{q_i}!},&
[m]_{q_i}! &= \prod_{k=1}^m [k]_{q_i},&
[k]_{q_i} = \frac{q_i^k - q_i^{-k}}{q_i - q_i^{-1}}.
\end{align*}
We will write $U_h(\mfh)$ for the copy of~$U(\mfh)\fpser$ inside $U_h(\mfg)$.

Put $K_i=e^{h d_i H_i}$. More generally, for $\omega\in\mfh^*$, let $h_\omega\in\mfh$ be such that $\alpha(h_\omega)=(\alpha,\omega)$. Define $K_\omega=e^{h h_\omega}$. Then $K_i=K_{\alpha_i}$.

The coproduct $\Delta\colon U_h(\mfg) \to U_h(\mfg) \hotimes U_h(\mfg)$ is defined by
\begin{align*}
\Delta(H) &= H\otimes 1+1\otimes H\quad (H\in\mfh),&
\Delta(E_i) &= E_i\otimes 1+ K_{i}\otimes  E_i,&
\Delta(F_i) &= F_i\otimes K_{i}^{-1} + 1\otimes F_i.
\end{align*}
Finally, the $*$-structure is defined by
\begin{align*}
H_i^* &= H_i,&
E_i^* &= F_i K_i,&
F_i^* &= K_i^{-1}E_i.
\end{align*}

By assumption, the involutive automorphism $\theta$ of $\mfg$ is in Satake form with respect to $(\mfh,\mfb^+)$ (recall Definition~\ref{def:split-pair}).
Then $\theta$ has the following form:
\begin{equation}
\label{eq:factoriz-theta}
\theta = (\Ad {z m_X m_0}) \circ\tau_\theta \circ \tau_0 = (\Ad{z m_X}) \circ\tau_\theta \circ \omega,
\end{equation}
where $m_0$ and $m_X$ are the canonical lifts of $w_0\in W$ and $w_X\in W_X$ to $U$, $\omega$ is the Chevalley automorphism, $\tau_0$ is the diagram automorphism satisfying $\omega = (\Ad m_0)\circ \tau_0$, and $z$ is an element of the maximal torus $\exp(\mft)$. The element $z$ is determined up to a factor in $Z(U)$. It automatically satisfies the following conditions:
\begin{align}\label{eq:z-eigval-rule}
z_i &= 1 \quad (i \in X),&
z_i \bar z_{\tau_\theta(i)} &= (-1)^{2(\alpha_i, \rho_X^\vee)} \quad (i \in I \setminus X),
\end{align}
where $z_i=z(\alpha_i)$ and $\rho_X^\vee$ is half the sum of the positive coroots of the root system generated by $X$. See, for example, \cite{MR3943480}*{Section~2.1} for details.

Consider the following parameter sets:
\begin{gather*}
\mcC  = \{\mbc = (c_i)_{i \in I \setminus X}\mid c_i\in\C\fps^*,\ c_i= c_{\tau_\theta(i)}\textrm{ for }i \in I_{\mcC}\},\\
\mcS = \{\mbs = (s_i)_{i\in I\setminus X} \mid  s_i\in\C\fps,\  s_i = 0 \textrm{ for }i\notin I_{\mcS}\},
\end{gather*}
where $\C\fps^*$ denotes the units of $\C\fps$, that is, the series with nonzero constant terms. We write $\mbt = (\mbc,\mbs)$ for an element of $\mcT = \mcC \times \mcS$.

Fix $\mbt \in \mcT$. For each $i\in I \setminus X$, we define
\begin{equation}\label{eq:generator-B}
B_{i}  = F_i - c_i z_{\tau_\theta(i)}T_{w_X}(E_{\tau_\theta(i)})K_i^{-1} + s_i \kappa_i \frac{K_{i}^{-1}-1}{e^{-d_i h}-1},
\end{equation}
where $\kappa_i = \exp(\pi \sqrt{-1} \psi_i)$ is the square root of $z_i$ with $0 \le \psi_i < 1$, and $T_{w_X}$ is the Lusztig automorphism associated to the longest element $w_X\in W_X$.
It will also be convenient to put
\[
B_i=F_i \quad (i\in X).
\]

Denote by $U_h(\mfg_X)\subset U_{h}(\mfg)$ the closure in the $h$-adic topology of the $\C\fps$-subalgebra generated by the elements $H_i$, $E_i$ and $F_i$ for $i\in X$.

\begin{Def}
\label{def:Letzter-coideal-def}
We define $U_{h}^\mbt(\mfg^\theta) \subset U_{h}(\mfg)$ as the closure in the $h$-adic topology of the $\C\fps$-subalgebra generated by $U_{h}(\mfh^{\theta})$, $U_{h}(\mfg_X)$ and the elements $B_{i}$ for $i\in I \setminus X$.
\end{Def}

\begin{Rem}\label{rem:BK-normalization-of-involution}
In \citelist{\cite{MR3905136}\cite{MR3943480}}, the element $z$ in \eqref{eq:factoriz-theta} is assumed to have the property $z_i = 1$ for all $i \in I \setminus X$ such that $\tau_\theta(i) = i$, which imposes extra conditions on $\theta$.
Although the relevant proofs work in the generality presented here, it is also possible to reduce the situation to this normalized form as follows.
Starting from our convention, choose $z' \in \exp(\mft)$ such that $z'_i = \kappa_i$ for $i$ as above, and $z'_i = 1$ for all other $i \in I$.
Then $\theta' = (\Ad z')^{-1} \circ \theta \circ (\Ad {z'})$ satisfies the normalization condition.
Moreover, $\Ad {z'}$ lifts to a Hopf $*$-algebra automorphism of $U_h(\mfg)$, and we have $(\Ad {z'})(U_h^\mbt(\mfg^{\theta'})) = U_h^\mbt(\mfg^{\theta})$.
\end{Rem}

\begin{Rem}
The above choice of $\kappa_i$ is, of course, a matter of convention.
If we replace $\kappa_i$ by $-\kappa_i$ in~\eqref{eq:generator-B}, the corresponding subalgebra will be conjugate to $U_{h}^\mbt(\mfg^\theta)$ by an inner automorphism of $U_h(\mfg)$ defined by an element of the torus.
\end{Rem}

It follows from~\citelist{\cite{MR1717368}\cite{MR3269184}} that $U_{h}^\mbt(\mfg^\theta)$ is a right coideal of $U_{h}(\mfg)$:
\[
\Delta(U_{h}^\mbt(\mfg^\theta)) \subset U_{h}^\mbt(\mfg^\theta) \hotimes U_{h}(\mfg).
\]
We will only be interested in the coideals $U_{h}^\mbt(\mfg^\theta)$ defined by smaller parameter sets
\[
\mcT^*=\mcC^*\times\mcS^*\subset\mcT^*_\C=\mcC^*_\C\times\mcS^*_\C\subset\mcT.
\]
In the non-Hermitian case both $\mcT^*$ and $\mcT^*_\C$ consist of just one point defined by
\begin{align}\label{eq:no-parameter}
c_i &= e^{-h(\alpha_i^-,\alpha_i^-)},&
s_i &= 0
\end{align}
for all $i\in I\setminus X$, where $\alpha^-_i = \frac{1}{2}(\alpha_i- \Theta(\alpha_i))$.

In the Hermitian case, recall from our discussion in Section~\ref{ssec:Satake} that there are one or two distinguished roots in $I\setminus X$.
Then we define $\mcT^*$ (resp.~$\mcT^*_\C$), by allowing the following exceptions from~\eqref{eq:no-parameter}:
\begin{itemize}
\item S-type: if $\alpha_o$ is the unique distinguished root, then we require $s_o\in i\R\fpser$ (resp.~$s_o\in \C\fpser$ and $s_o^{(0)}\ne\pm1$);
\item C-type: if $\alpha_o$ is a distinguished root, then we require $c_o\in\R\fpser$ and $c_o^{(0)}>0$ (resp.~$c_o\in\C\fpser$ and $c_o^{(0)}\ne \pm i$), and, for both $\mcT^*$ and $\mcT^*_\C$,
\[
c_o c_{\tau_\theta(o)}=e^{-2h(\alpha_o^-,\alpha_o^-)}.
\]
\end{itemize}
By the proof of~\cite{MR3943480}*{Theorem~3.11} (see also Remark \ref{rem:BK-normalization-of-involution}), the coideals $U_{h}^\mbt(\mfg^\theta)$ are $*$-invariant for $\mbt\in\mcT^*$.

We will mainly work with $\mcT^*$ and then explain how our results extend to generic points of $\mcT^*_\C$.
The point of $\mcT^*$ defined by~\eqref{eq:no-parameter} for all $i\in I\setminus X$ is denoted by~$0$, and the corresponding coideal subalgebra is denoted by~$U^\theta_h(\mfg)$.
We will also refer to this as the \emph{standard} or \emph{no-parameter} case. Thus, in the non-Hermitian case, $U^\theta_h(\mfg)$ is the only coideal subalgebra we will be working with.

\smallskip

The classical limit of $U^\mbt_h(\mfg^\theta)$ is given by the following:

\begin{Def}
\label{def:limit-Lie-alg-of-LK-coideal}
For $\mbt\in\mcT$, we define $\mfg^\theta_\mbt$ to be the Lie subalgebra of $\mfg$ generated by $\mfh^{\theta}$, $\mfg_X$ and the elements
\begin{equation*}
X_{-\alpha_i} + c_i^{(0)}\theta(X_{-\alpha_i})+ s_i^{(0)} \kappa_i H_i\quad (i\in I\setminus X).
\end{equation*}
\end{Def}

The image of $U^\mbt_h(\mfg^\theta)$ under the isomorphism $U_h(\mfg)/h U_h(\mfg)\cong U(\mfg)$ (mapping $E_i$ into $X_{\alpha_i}$ and $F_i$ into $X_{-\alpha_i}$) is $U(\mfg^\theta_\mbt)$.
In the standard case, we have $c_i^{(0)}=1$ and $s_i^{(0)}=0$ for all $i\in I\setminus X$, and the corresponding Lie subalgebra $\mfg^\theta_0$ is exactly $\mfg^\theta$ by~\cite{MR3269184}*{Lemma~2.8}, see also Proposition \ref{PropFixTheta}.

\subsection{Untwisting by Drinfeld twist}
\label{ssec:twist-multiplier-model}

Let us quickly review how to relate the quantized universal enveloping algebra to the classical one, see for example \cite{MR1025154} for the details.
First, by $H^2(\mfg, U(\mfg)) = 0$ there is a $\C\fpser$-algebra isomorphism $\pi\colon U_h(\mfg)\to U(\mfg)\fpser$ such that
\begin{align}\label{eq:que-iso}
\pi(H_i)&=H_i,&
\pi(E_i)&=X_{\alpha_i},&
\pi(F_i)&=X_{-\alpha_i} \pmod h.
\end{align}
Moreover, if $\tilde\pi$ is another such isomorphism, then $\tilde\pi=(\Ad u)\pi$ for an element $u\in 1+h U(\mfg)\fpser$ by $H^1(\mfg, U(\mfg)) = 0$.
In a similar way, for any two $\C\fpser$-algebra homomorphisms $\tilde\pi,\pi\colon U_h(\mfg)\to\mcU(G)\fpser$ satisfying~\eqref{eq:que-iso} there is $u\in 1+h\mcU(G)\fpser$ such that $\tilde\pi=(\Ad u)\pi$.
In what follows we fix such a homomorphism $\pi$.

While a particular choice of $\pi$ is not going to matter for our results, in some arguments it is convenient to have extra properties.

\begin{Lem}\label{lem:que-iso}
There is a $*$-preserving $\C\fpser$-algebra isomorphism $\pi\colon U_h(\mfg)\to U(\mfg)\fpser$ such that
\begin{align}\label{eq:que-iso2}
\pi(H_i)&=H_i,&
\pi(K_i^{-1/2}E_i)&=X_{\alpha_i},&
\pi(F_i K_i^{1/2})&=X_{-\alpha_i} \pmod {h^2}.
\end{align}
\end{Lem}

\bp
We have a homomorphism $\rho\colon U(\mfg)\to U_h(\mfg)/h^2U_h(\mfg)$ such that
\begin{align*}
\rho(H_i) &= H_i,&
\rho(X_{\alpha_i}) &= K_i^{-1/2}E_i,&
\rho(X_{-\alpha_i}) &= F_i K_i^{1/2},
\end{align*}
the key point being that since $[n]_{e^h} = n + O(h^2)$, the coefficients of the quantum Serre relations reduce to the classical ones modulo $h^2$. Taking now an arbitrary $\C\fpser$-algebra isomorphism $\pi\colon U_h(\mfg)\to U(\mfg)\fpser$ satisfying~\eqref{eq:que-iso}, we must have that the homomorphism $\pi\circ\rho\colon U(\mfg)\to U(\mfg)\fpser/h^2U(\mfg)\fpser$ is of the form
\[
(\pi\circ\rho)(T)=T+h\delta(T)\quad(T\in U(\mfg))
\]
for a derivation $\delta\colon U(\mfg)\to U(\mfg)$. Replacing $\pi$ by $e^{-h\delta}\circ\pi$ we get an isomorphism $U_h(\mfg)\to U(\mfg)\fpser$ satisfying~\eqref{eq:que-iso2}.

Next, the homomorphism $U_h(\mfg)\to U(\mfg)\fpser$, $T\mapsto \pi(T^*)^*$, also satisfies~\eqref{eq:que-iso2}.
It follows that there exists $u\in 1+h^2U(\mfg)\fpser$ such that $\pi(T^*)^*=u\pi(T)u^{-1}$.
By taking the adjoints and replacing~$T$ by~$T^*$ we also get $\pi(T^*)^*=u^*\pi(T)(u^*)^{-1}$.
Hence $u^*=u v$ for a central element $v\in U(\mfg)\fpser$ such that $v=1$ modulo $h^2$.
This element must be unitary, hence $v^{1/2}\in 1+h^2U(\mfg)\fpser$ is unitary as well.
Replacing~$u$ by~$u v^{1/2}$ we can therefore assume that $u^*=u$. But then replacing $\pi$ by $(\Ad u^{1/2})\pi$ we get a $*$-preserving isomorphism satisfying~\eqref{eq:que-iso2}.
\ep

\begin{Rem}
We may further assume that the exact equality $\pi(H_i) = H_i$ holds in the above lemma, as follows.
In \eqref{eq:que-iso} we can arrange so that $\pi(H_i) = H_i$ holds exactly by \cite{MR1025154}*{Proposition~4.3}.
Then the construction in the proof of Lemma \ref{lem:que-iso} preserves this property. Indeed, first we have $\delta(H_i) = 0$, so that $e^{-h\delta}\circ\pi(H_i) = H_i$ holds.
Then we have $(\Ad u) \pi(H_i) = \pi(H_i^*)^* = H_i$, hence we obtain $(\Ad u^{1/2})\pi(H_i) = H_i$ as well at the last step.
\end{Rem}

With $\pi\colon U_h(\mfg)\to\mcU(G)\fpser$ fixed, there is a unique coproduct $\Delta_h\colon\mcU(G)\fpser\to\mcU(G\times G)\fpser$ such that
\[
(\pi\otimes\pi)\Delta=\Delta_h\pi.
\]
By a \emph{Drinfeld twist} we will mean any element $\mcF\in1+h\mcU(G\times G)\fpser$ such that
\begin{gather}
\notag%\label{eq:F-twist1}
(\epsilon\otimes\id)(\mcF)=(\id\otimes\epsilon)(\mcF)=1, \qquad\qquad
\Delta_{h}=\mcF\Delta(\cdot)\mcF^{-1},\\
\label{eq:Drinfeld-twist1}
(\id\otimes\Delta)(\mcF^{-1})(1\otimes\mcF^{-1})(\mcF\otimes1)(\Delta\otimes\id)(\mcF)=\Phi_\KZ.
\end{gather}

Consider the $r$-matrix $r$ defined by~\eqref{EqComplr} and the corresponding cobracket $\delta_r(X)=[r,\Delta(X)]$ on $\mfg$.

\begin{Lem}\label{lem:Drinfeld-twist}
If $\pi$ is as in Lemma~\ref{lem:que-iso}, then there is a unitary Drinfeld twist $\mcF\in U(\mfg)^{\otimes 2}\fpser$ such that
\begin{equation}
\label{eq:Dr-twist-lin-ord-term2}
\mcF = 1 + h\frac{i r}{2} + O(h^2).
\end{equation}
\end{Lem}

\bp
We start with an arbitrary Drinfeld twist $\mcF\in1+h U(\mfg)^{\otimes 2}\fpser$, which exists by~\cite{MR1080203}.
By our choice of $\pi$ we have
\[
\Delta_{h}(X)=\Delta(X)+h\frac{i\delta_r(X)}{2}+O(h^2)\quad\text{for}\quad X\in\mfg.
\]
It follows that the element $S=\mcF^{(1)}-\hlf1 i r\in U(\mfg)^{\otimes 2}$ commutes with the image of $\Delta$. Since $\Phi_\KZ=1$ modulo $h^2$, identity \eqref{eq:Drinfeld-twist1} implies (similarly to the proof of Lemma~\ref{lem:unique-twist}) that $S$ satisfies the cocycle identity
\[
(\id\otimes\Delta)(S)+1\otimes S - S\otimes1 - (\Delta\otimes\id)(S)=0.
\]
Hence $S=T\otimes1+1\otimes T-\Delta(T)$ for a central element $T\in U(\mfg)$. Replacing $\mcF$ by
\[
\mcF\bigl((1+h T)^{-1}\otimes(1+h T)^{-1}\bigr)\Delta(1+h T)
\]
we get a Drinfeld twist satisfying \eqref{eq:Dr-twist-lin-ord-term2}.
Replacing further $\mcF$ by $\mcF(\mcF^*\mcF)^{-1/2}$ we also get unitarity, see~\cite{MR2832264}*{Proposition~2.3}. Note that this does not destroy~\eqref{eq:Dr-twist-lin-ord-term2}, since $r^*=r$.
\ep

Denote the universal $R$-matrix of $U_h(\mfg)$ (or, the one for $U_q(\mfg)$ in the conventions of~\cite{MR3943480}) by~$\msR$.
Then any Drinfeld twist $\mcF$ satisfies
\begin{equation}\label{eq:R-matrix}
(\pi\otimes\pi)(\msR)=\mcF_{21}\exp(-ht^\mfu)\mcF^{-1}.
\end{equation}
Indeed, this identity holds for a particular Drinfeld twist by~\cite{MR1080203}, but then it must hold for any Drinfeld twist by Lemma~\ref{lem:unique-twist} and the invariance of $\exp(-ht^\mfu)$.
We put $\mcR_h = (\pi\otimes\pi)(\msR)$, which is a universal $R$-matrix for $(\mcU(G)\fps,\Delta_h)$.

\subsection{Parameter case and Cayley transform}
\label{sec:para-case-cayley-trans}

Suppose that $\mfu^\theta<\mfu$ is a Hermitian symmetric pair.
Let us relate the Lie algebras $\mfg^\theta_\mbt$ to the Cayley transform we considered in Section \ref{ssec:Satake}.

Choose $Z_\theta\in\mfz(\mfu^\theta)$ normalized as $(Z_\theta,Z_\theta)_\mfg=-a_\theta^{-2}$.
Let us choose a Cartan subalgebra $\tilde\mft$ of $\mfu$ containing $\mfz(\mfu^\theta)$, and choose positive roots as in Section~\ref{ssec:Cartan}, but now for the pair $(\theta,Z_\theta)$ instead of~$(\nu,Z_\nu)$. We denote the corresponding Borel subalgebra by $\tilde\mfb^+$.

Take $g\in U$ such that $(\Ad g)(\tilde\mft)=\mft$ and $(\Ad g)(\tilde\mfb^+)=\mfb^+$.
Put $\nu=(\Ad g)\circ\theta\circ(\Ad g)^{-1}$ and $Z_{\nu}=(\Ad g)(Z_\theta)$.
Then $\mft$ and our fixed positive roots are defined as in Section~\ref{ssec:Cartan} for our new pair~$(\nu,Z_\nu)$.

Let $g_1$ be the Cayley transform for $\nu$ with respect to $(X_\alpha)_\alpha$.

\begin{Lem}\label{lem:nu-from-theta}
There exist an element $z_\theta\in\exp(\mft)$ such that $(\Ad z_\theta)\circ \theta\circ(\Ad z_\theta)^{-1}$ coincides with the automorphism $\theta' = (\Ad {g_1})^{-1}\circ \nu\circ(\Ad{g_1})$ and $Z_\theta=(\Ad g_1z_\theta)^{-1}(Z_\nu)$.
\end{Lem}

\bp
First note that $\theta'$ is in Satake form with respect to $(\mfh,\mfb^+)$.

We have $\theta=(\Ad g^{-1}g_1 )\circ\theta'\circ (\Ad g^{-1}g_1 )^{-1}$.
It follows that $\theta$ is in Satake form both with respect to $(\mfh,\mfb^+)$ and $((\Ad g^{-1}g_1 )(\mfh),(\Ad g^{-1}g_1 )(\mfb^+))$.
By \cite{MR1155464}*{Corollary 5.32}, we can find $g'\in G^{\theta}$ such that $(\Ad g' g^{-1}g_1 )(\mfh) = \mfh$ and $(\Ad g' g^{-1}g_1 )(\mfb^+) = \mfb^+$. Then $g' g^{-1}g_1\in\exp(\mfh)$. Moreover, we still have
\[
\theta=(\Ad g' g^{-1}g_1 )\circ\theta'\circ (\Ad g' g^{-1}g_1 )^{-1}.
\]

Consider the Cartan decomposition $g' g^{-1}g_1=z_\theta^{-1}a$, so $z_\theta\in\exp(\mft)$ and $a\in\exp(i\mft)$. As $\theta,\theta'$ are $*$-preserving and $\theta\circ(\Ad z_\theta^{-1}a)=(\Ad z_\theta^{-1}a)\circ\theta'$, we also have $\theta'\circ(\Ad a z_\theta) = (\Ad a z_\theta)\circ\theta$.
It follows that $\Ad a^2$ commutes with $\theta'$. This means that $\theta'(a^2)\in Z(G)a^2=Z(U)a^2$, hence $\theta'(a^2)=a^2$, and then $\theta'(a)=a$.
Therefore
\[
\theta=(\Ad z_\theta)^{-1}\circ \theta'\circ(\Ad z_\theta)=(\Ad g_1 z_\theta )^{-1}\circ\nu\circ (\Ad g_1z_\theta ).
\]
We also have
\[
(\Ad g_1z_\theta)^{-1}(Z_\nu)=(\Ad z_\theta^{-1} a g_1^{-1})(Z_\nu)=(\Ad g' g^{-1})(Z_\nu)=(\Ad g')(Z_\theta)=Z_\theta,
\]
where we used that $(\Ad g_1^{-1})(Z_\nu)\in\mfz(\mfu^{\theta'})$ is invariant under $a\in G^{\theta'}$.
\ep

We can now talk about compact/noncompact positive roots for $(\mfh, \mfb^+)$ with respect to $\nu$, as in Section~\ref{sec:int-subgroups}.
Recall that by Proposition~\ref{prop:restrictions}, in the  S-type case the unique noncompact simple root $\alpha_o$ is exactly the distinguished root.
We have the following characterization for the C-type case.

\begin{Lem}\label{lem:noncompact-root}
In the C-type case, the unique noncompact simple root $\alpha_o$ is determined among the distinguished roots $\{\alpha_o,\alpha_{o'}\}$ by the inequality
\[
-i\alpha_o(\tilde Z_\theta)>-i\alpha_{o'}(\tilde Z_\theta),
\]
where $\tilde Z_\theta$ is the component of $Z_\theta\in\mfg=\mfh\oplus\bigoplus_{\alpha\in\Phi}\mfg_\alpha$ lying in $\mfh$.
\end{Lem}

\bp
By the definition of the order structure in Section~\ref{ssec:Cartan}, we have
$\alpha_o(-i Z_\nu)>0=\alpha_{o'}(-i Z_\nu)$. By  Proposition~\ref{prop:restrictions}, $\alpha_o$ and $\alpha_{o'}$ have the same restriction to $\mfh^{-}=\{H\mid \theta(H)=-H\}$.
It follows that if  $Z_\nu=Z_\nu^+ + Z_\nu^-$ is the decomposition of $Z_\nu$ with respect to $\mfh=\mfh^\theta\oplus\mfh^-$, then
\[
-i\alpha_o(Z_\nu^+)>-i\alpha_{o'}(Z_\nu^+).
\]

On the other hand, the inverse $(\Ad g_1)^{-1}$ of the Cayley transform acts trivially on $\mfh^\theta$ and maps $\mfh^-$ onto the linear span of the vectors $X_{\gamma_i}-X_{-\gamma_i}$ (Lemma~\ref{lem:P_0}). As $Z_\theta=(\Ad g_1z_\theta)^{-1}(Z_\nu)$, it follows that $\tilde Z_\theta=Z_\nu^+$, proving the lemma.
\ep

Thus, in both S-type and C-type cases, once $Z_\theta$ is fixed (between the two possibilities), $\nu$ is uniquely and explicitly determined: $\nu$ acts trivially on $\mfh$ and the root vectors $X_{\pm\alpha_i}$ for  $i\in I\setminus\{o\}$, while $\nu(X_{\pm\alpha_o})=-X_{\pm\alpha_o}$.
Since $\alpha_o$ is a noncompact positive root, $\alpha_o(-i Z_\nu)$ is a positive number.
This, together with the normalization $(Z_\nu,Z_\nu)_\mfg=-a_\nu^{-2}$, determines $Z_\nu$.
In the S-type case the pair $(\nu,Z_\nu)$ is therefore independent of the choice of $Z_\theta$.
In the C-type case, changing the sign of $Z_\theta$ swaps the notions of compactness/non\-compactness for the distinguished roots.
Note also that by looking at the basis of the restricted root system obtained by restricting $\Pi\setminus\Pi_X$ to $\mfh^-$, we can recover the roots $\gamma_1,\dots,\gamma_s$ and then the element~$g_1$.

The element $z_\theta$ is not easily determined, but the following lemma will be enough for our purposes.
Recall the factorization of $\theta$ given by \eqref{eq:factoriz-theta}.

\begin{Lem}\label{lem:z-theta-alpha-o-expr}
In the S-type case, $z_\theta(\alpha_o)$ is a square root of $z_o^{-1}$.
\end{Lem}

\bp
Since $o$ is fixed by $\tau_\theta$, we have $\theta(X_{-\alpha_o}) = - z_o X_{\alpha_o}$.
As we already observed in the proof of Proposition~\ref{PropFixTheta}, Lemma~\ref{lem:P_0} implies that $\bigl((\Ad {g_1})^{-1}\circ \nu\circ(\Ad{g_1})\bigr)(X_{\alpha_o})=-X_{-\alpha_o}$.
It follows that $\theta(X_{\alpha_o})=-z_\theta(\alpha_o)^2 X_{-\alpha_o}$, and comparing this with the above formula we get $z_\theta(\alpha_o)^2=z_o^{-1}$.
\ep

Consider now the subgroups $G_\phi$ ($\phi\in\R$) from Definition~\ref{def:G-phi}.
(Note that we have to use $\theta'$ from Lemma \ref{lem:nu-from-theta} as $\theta$ in Section~\ref{ssec:coisotropic}.)
It is convenient now to allow also $\phi\in\C$. Then $G_\phi=(\Ad g_{\phi})(G^{\theta'})$ are still well-defined subgroups of $G$.

\begin{Lem}\label{lem:mbt-vs-phi}
If $\mfu^\theta<\mfu$ is Hermitian and $\mbt\in\mcT^*_\C$, then $\mfg^\theta_\mbt=(\Ad z_\theta)^{-1}(\mfg_\phi)$, where $\phi\in\C$ is any number satisfying the following identity:
\[
z_\theta(\alpha_o) s^{(0)}_o \kappa_o =i\tan\Bigl(\frac{\pi \phi}{2}\Bigr)\quad (\text{\rm S-type})\quad\text{or}\quad c^{(0)}_o=-\cot\Bigl(\frac{\pi}{4}(\phi-1)\Bigr)\quad (\text{\rm C-type}).
\]
In particular, the Lie algebras $\mfg^\theta_\mbt$ are all conjugate to $\mfg^\theta$ in $\mfg$.
\end{Lem}

\bp
This follows from Proposition~\ref{PropFixTheta} (and its obvious extension to complex $\phi$) and the definition of~$\mfg^\theta_\mbt$, combined with Lemma~\ref{lem:nu-from-theta}.
\ep

For $\mbt\in\mcT^*_\C$, let $G^\theta_\mbt = (\Ad z_\theta)^{-1}(G_\phi) \subset G$ be the (connected) algebraic subgroup integrating $\mfg^\theta_\mbt$,  with~$\phi$ as in the lemma above. Then $K_\mbt=(\Ad z_\theta^{-1}g_{\phi-1})(U^\nu)$ is its compact form.
Note that if $\mbt\in\mcT^*$, then we can take $\phi\in\R$, so that $g_\phi\in U$ and hence $\mfk_\mbt=\mfg^\theta_\mbt\cap\mfu$.

\begin{Rem}\label{rem:sign-detection}
In the C-type case we get an element $Z^\mbt_\theta=(\Ad z_\theta^{-1}g_{\phi-1})(Z_\nu)=(\Ad z_\theta^{-1}g_\phi z_\theta^{\phantom{1}})(Z_\theta)\in\mfz(\mfg^\theta_\mbt)$, which by Lemma~\ref{lem:P_0} does not depend on the choice of $\phi$ (such that $c^{(0)}_o=-\cot(\frac{\pi}{4}(\phi-1))$). Using Lemma~\ref{lem:noncompact-root} we can quickly recover $\nu$ from~$Z^\mbt_\theta$. Namely, let $\tilde Z^\mbt_\theta$ be the component of $Z^\mbt_\theta$ in $\mfh$. Then, again by Lemma~\ref{lem:P_0}, $\tilde Z^\mbt_\theta$ and $\tilde Z_\theta$ differ only by an element of~$\mfh^-$. Hence $\alpha_o$ is determined among the distinguished roots $\{\alpha_o,\alpha_{o'}\}$ by the inequality
\[
-i\alpha_o(\tilde Z^\mbt_\theta)+i\alpha_{o'}(\tilde Z^\mbt_\theta)>0.
\]

In the S-type case the element $(\Ad z_\theta^{-1}g_{\phi-1})(Z_\nu)\in\mfz(\mfg^\theta_\mbt)$ does depend on the choice of $\phi$. Here we can take any $Z^\mbt_\theta\in\mfz(\mfg^\theta_\mbt)$ such that $(Z^\mbt_\theta,Z^\mbt_\theta)_\mfg=-a_\theta^{-2}$ and then make the identity $Z^\mbt_\theta=(\Ad z_\theta^{-1}g_{\phi-1})(Z_\nu)$ an extra condition on $\phi$. This works, because by Corollary~\ref{cor:center-S-type} and Lemma~\ref{lem:P_0} we have $(\Ad g_2)(Z_\nu)=-Z_\nu$. We can formulate this in a more intrinsic way with respect to $\mfg^\theta_\mbt$ as follows.
Recall that we have $Z_\nu=\frac{i}{2}\sum_j H_{\gamma_j}$ by Corollary~\ref{cor:center-S-type}.
Then Lemma~\ref{lem:P_0} implies
\[
(\Ad z_\theta^{-1}g_{\phi-1})(Z_\nu) =\frac{1}{2}\cos\Bigl(\frac{\pi\phi}{2}\Bigr)\sum_j(\Ad z_\theta^{-1})(X_{-\gamma_j}-X_{\gamma_j})+\frac{i}{2}\sin\Bigl(\frac{\pi\phi}{2}\Bigr) \sum_j H_{\gamma_j}.
\]
for all $\phi$.
It follows that if $Z^\mbt_\theta=(\Ad z_\theta^{-1}g_{\phi-1})(Z_\nu)$, then
\[
\frac{(Z^\mbt_\theta,X_{\alpha_o})_\mfg}{z_\theta(\alpha_o)\cos(\frac{\pi \phi}{2})} = \frac{(X_{-\alpha_o}, X_{\alpha_o})_\mfg}{2} = \frac{1}{(\alpha,\alpha)} > 0.
\]
\end{Rem}

\subsection{Multiplier algebra model of Letzter--Kolb coideals}
\label{sec:mult-alg-model-of-LK-coids}

Back to the general $\theta$, let us next explain how to cast the Letzter--Kolb coideals in the setting of multiplier algebras.
Let $P$ be the weight lattice.
Denote by $V_\lambda$ an irreducible $\mfg$-module with highest weight $\lambda\in P_+$. We denote by $\pi_\lambda\colon U(\mfg)\to\End(V_\lambda)$ the corresponding homomorphism and use the same symbol for the extension of $\pi_\lambda$ to a homomorphism $U(\mfg)\fps\to\End(V_\lambda)\fps$.
We also put $\pi_{\lambda,h}=\pi_\lambda\pi\colon U_h(\mfg)\to\End(V_\lambda)\fps$.

\begin{Lem}\label{lem:u-conjugate}
For every $\mbt \in \mcT^*$, there exist elements $u_\lambda\in\End(V_\lambda)\fps$, $\lambda\in P_+$, such that
\begin{align}\label{eq:u0}
u_\lambda^{(0)} &= 1,&
\biggl( \bigoplus_{\lambda\in F} (\Ad u_\lambda)\pi_{\lambda,h} \biggr) \bigl( U_{h}^\mbt(\mfg^\theta) \bigr) &= \biggl( \bigoplus_{\lambda\in F}\pi_\lambda \biggr)\bigl(U(\mfg^\theta_\mbt)\fps\bigr)
\end{align}
for any finite subset $F\subset P_+$. If $\pi$ is $*$-preserving, then $u_\lambda$ can in addition be chosen to be unitary.
\end{Lem}

\bp
Let us first fix a finite subset $F\subset P_+$ and show that there exist elements $u_\lambda$, $\lambda\in F$, satisfying~\eqref{eq:u0}.

Denote by $V_F$ the $\mfg$-module $\bigoplus_{\lambda\in F}V_\lambda$ and by $\pi_F$ the representation $\bigoplus_{\lambda\in F}\pi_\lambda$.
Write $\pi_{F,h}$ for~$\pi_F\pi$.
Let $A_h$ be the commutant of $\pi_{F,h}\bigl(U_{h}^\mbt(\mfg^\theta)\bigr)$ in $\End(V_F)\fps$.
It is clear that $A_h$ is a closed $\C\fps$-subalgebra of $\End(V_F)\fps$ and $A_h\cap h\End(V_F)\fps=h A_h$.
It follows that $A_h$ is a free $\C\fps$-module and $A_h/h A_h$ can be considered as a subalgebra of $\End(V_F)$, so that $A_h$ is a deformation of this subalgebra.
We claim that
\[
A_h/h A_h=\End_{\mfg^\theta_\mbt}(V_F).
\]

The inclusion $\subset$ is clear, since the image of $U_{h}^\mbt(\mfg^\theta)$ in $U_h(\mfg)/h U_h(\mfg)\cong U(\mfg)$ is $U(\mfg^\theta_\mbt)$.
For the opposite inclusion, using the Frobenius isomorphism
\[
\End_{\mfg^\theta_\mbt}(V_F)\cong\Hom_{\mfg^\theta_\mbt}(V_0,V_F\otimes \bar V_F),
\]
defined by duality morphisms for $\mfg$-modules, and a decomposition of $V_F\otimes \bar V_F$ into simple $\mfg$-modules $V_\mu$, we see that the problem reduces to the question whether every $\mfg^\theta_\mbt$-invariant vector in $V_\mu$ can be lifted to a $U_{h}^\mbt(\mfg^\theta)$-invariant vector in $V_\mu\fps$.
This is indeed possible by a result of Letzter~\cite{MR1742961}, see Appendix~\ref{sec:spherical} for more details.

\smallskip

Since the algebra $\End_{\mfg^\theta_\mbt}(V_F)$ is semisimple, it has no nontrivial deformations, so there is a $\C\fps$-algebra isomorphism $A_h\cong \End_{\mfg^\theta_\mbt}(V_F)\fps$ that is the identity modulo $h$. Furthermore, there are no nontrivial deformations of the identity homomorphism $\End_{\mfg^\theta_\mbt}(V_F)\to\End(V_F)$, that is, all such deformations are given by conjugating by elements of $1+h\End(V_F)\fps$.
It follows that there is $w\in 1+h\End(V_F)\fps$ such that $w A_h w^{-1}=\End_{\mfg^\theta_\mbt}(V_F)\fps$.

Next, consider the subalgebra $B\subset \End_{\mfg^\theta_\mbt}(V_F)$ spanned by the projections $e_\lambda\colon V_F\to V_\lambda$.
Since $B\subset A_h$, we have $(\Ad w)(B)\subset \End_{\mfg^\theta_\mbt}(V_F)\fps$.
As $B$ is also semisimple, the inclusion map $B\to\End_{\mfg^\theta_\mbt}(V_F)$ cannot be nontrivially deformed, that is, there is $v\in 1+h\End_{\mfg^\theta_\mbt}(V_F)\fps$ such that $\Ad w=\Ad v$ on $B$.
It follows that the element $u=v^{-1}w$ still has the property
\begin{equation}\label{eq:u1}
u A_h u^{-1}=\End_{\mfg^\theta_\mbt}(V_F)\fps,
\end{equation}
but in addition it commutes with the projections $e_\lambda$, $\lambda\in P$.
Hence $u=(u_\lambda)_{\lambda\in F}$ for some $u_\lambda\in 1+h\End_{\mfg^\theta_\mbt}(V_\lambda)\fps$.

By taking the commutants we get from~\eqref{eq:u1} that
\[
u\pi_{F,h}\bigl(U_{h}^\mbt(\mfg^\theta)\bigr)u^{-1}\subset \pi_F\bigl(U(\mfg^\theta_\mbt)\bigr)\fps,
\]
where we used that the $\mfg^\theta_\mbt$-module $V_F$ is completely reducible and hence $\pi_F\bigl(U(\mfg^\theta_\mbt)\bigr)$ is the commutant of $\End_{\mfg^\theta_\mbt}(V_F)$.
The above inclusion becomes an equality modulo $h$.
Since $U_h^\mbt(\mfg^\theta)$ is complete in the $h$-adic topology, we then easily deduce that the inclusion is in fact an equality.
This finishes the proof of the lemma for a fixed finite set $F$, apart from the last statement about unitarity.

\smallskip
Now, consider an increasing sequence of finite subsets $F_n\subset P_+$ with union $P_+$.
For every $n$, choose elements $u^{(n)}=(u^{(n)}_\lambda)_{\lambda\in F_n}$, satisfying~\eqref{eq:u0} for $F=F_n$.
To finish the proof it suffices to show that we can inductively modify $u^{(n+1)}$ in such a way that we get $u^{(n+1)}_\lambda=u^{(n)}_\lambda$ for $\lambda\in F_n$.

For this, consider the element $w=(u^{(n)}_\lambda(u^{(n+1)}_\lambda)^{-1})_{\lambda\in F_n}\in\End(V_{F_n})\fps$.
We have
\begin{align*}
w^{(0)}&=1,&
(\Ad w)\bigl(\pi_{F_n}\bigl(U(\mfg^\theta_\mbt)\bigr)\fps\bigr) &= \pi_{F_n}\bigl(U(\mfg^\theta_\mbt)\bigr)\fps.
\end{align*}
Since $\pi_{F_n}\bigl(U(\mfg^\theta_\mbt)\bigr)$ is semisimple, it follows that there is an element $v\in 1+h\pi_{F_n}\bigl(U(\mfg^\theta_\mbt)\bigr)\fps$ such that $\Ad w=\Ad v$ on $\pi_{F_n}\bigl(U(\mfg^\theta_\mbt)\bigr)$.
Lift $v$ to an element $u\in 1+h U(\mfg^\theta_\mbt)\fps$.
We then modify $u^{(n+1)}$ by replacing $u^{(n+1)}_\lambda$ by $u^{(n)}_\lambda$ for $\lambda\in F_n$ and by $\pi_\lambda(u)u^{(n+1)}_\lambda$ for $\lambda\in F_{n+1}\setminus F_n$.

\smallskip

Finally, assume in addition that $\pi$ is $*$-preserving. In this case it suffices to show that at every stage of the above construction of $u_\lambda$ we can get unitary elements with the required properties. Specifically, we claim that if $u\pi_{F,h}\bigl(U_{h}^\mbt(\mfg^\theta)\bigr)u^{-1}=\pi_F\bigl(U(\mfg^\theta_\mbt)\bigr)\fps$ for a finite set $F$ and an element $u$, $u^{(0)}=1$, then the same identity holds for the unitary $(u u^*)^{-1/2}u$. Indeed, taking the adjoints we get $\pi_{F,h}\bigl(U_{h}^\mbt(\mfg^\theta)\bigr)=u^*\pi_F\bigl(U(\mfg^\theta_\mbt)\bigr)\fps(u^*)^{-1}$. It follows that $\Ad (u u^*)$ defines an automorphism $\beta$ of $\pi_F\bigl(U(\mfg^\theta_\mbt)\bigr)\fps$. As $\beta=\id$ modulo $h$, this automorphism has a unique square root $\beta^{1/2}$ such that $\beta^{1/2}=\id$ modulo $h$. Then $\Ad (u u^*)^{-1/2}=\beta^{-1/2}$ on $\pi_F\bigl(U(\mfg^\theta_\mbt)\bigr)\fps$, and our claim is proved.
\ep

We continue to assume that $t \in \mcT^*$. In the Hermitian case, recall the subgroups $G^\theta_\mbt < G$ from the previous subsection. In the non-Hermitian case, let us put $G^\theta_\mbt = G^\theta$. The collection $(u_\lambda)_{\lambda\in P_+}$ defines an element $u=u^\mbt\in\mcU(G)\fps$ such that
\begin{align}\label{eq:umbt}
u^{(0)} &= 1,&
u\pi(U_{h}^\mbt(\mfg^\theta))u^{-1} &\subset \mcU(G^\theta_\mbt)\fps.
\end{align}
Furthermore, the last inclusion is dense in the sense that the images of both algebras in $\End(V)\fps$ coincide for any finite dimensional $\mfg$-module $V$.

 Consider the homomorphism $\alpha_h\colon\mcU(G^\theta_\mbt)\fps\to\mcU(G\times G)\fps$ defined by
\[
\alpha_h(x)=(u\otimes1)\Delta_h(u^{-1}x u)(u^{-1}\otimes1).
\]
If $x=u\pi(y)u^{-1}$ for some $y\in U_{h}^\mbt(\mfg^\theta)$, we have
\[
\alpha_h(x)=\alpha_h(u\pi(y)u^{-1})=\bigl((\Ad u)\pi\otimes\pi\bigr)\Delta(y).
\]
By the density of $u\pi(U_{h}^\mbt(\mfg^\theta))u^{-1}$ in $\mcU(G^\theta_\mbt)\fps$ we conclude that $\alpha_h(\mcU(G^\theta_\mbt)\fps)\subset \mcU(G^\theta_\mbt\times G)\fps$, and
the strict coassociativity $(\alpha_h\otimes\id)\alpha_h=(\id\otimes\Delta_h)\alpha_h$ holds.
Thus, we get a coaction of $(\mcU(G)\fps,\Delta_h)$ on $\mcU(G^\theta_\mbt)\fps$ making the following diagram commutative:
\[
\begin{tikzcd}
U_{h}^\mbt(\mfg^\theta) \arrow[r,"\Delta"] \arrow[d,"(\Ad u^\mbt)\pi"'] & U_{h}^\mbt(\mfg^\theta)\hotimes U_h(\mfg) \arrow[d,"(\Ad u^\mbt)\pi\otimes\pi"] \\
\mcU(G^\theta_\mbt)\fps \arrow[r,"\alpha_h"'] & \mcU(G^\theta_\mbt\times G)\fps
\end{tikzcd}
\]

\begin{Def}
For $\mbt\in\mcT^*$, we call the coaction $(\mcU(G^\theta_\mbt)\fps,\alpha_h)$ of $(\mcU(G)\fps,\Delta_h)$ the \emph{multiplier algebra model} of the Letzter--Kolb coideal $U^\mbt_h(\mfg^\theta)$.
\end{Def}

It is not difficult to see that up to twisting this model does not depend on the choice of $\pi$ and $u$.

Let us record an immediate consequence of the construction of $\alpha_h$, which we will use later.

\begin{Prop}
\label{prop:conjug-coideal-coact-to-Delta}
For every $\mbt\in\mcT^*$, there is an element $\mcG\in\mcU(G^\theta_\mbt\times G)\fps$ such that
\begin{align*}
\mcG^{(0)}&=1,&
(\id\otimes\epsilon)(\mcG) &= 1,&
\alpha_h=\mcG\Delta(\cdot)\mcG^{-1}.
\end{align*}
If $\alpha_h$ is $*$-preserving, then $\mcG$ can in addition be chosen to be unitary.
\end{Prop}

\bp
Since $\alpha_h = \Delta \bmod h$, Lemma \ref{lem:twisting-to-Delta} implies the existence of such an element $\mcG$.
If $\alpha_h$ is in addition $*$-preserving, then we can replace $\mcG$ by the unitary $\mcG(\mcG^*\mcG)^{-1/2}$.
\ep

\begin{Rem}\label{rem:multiplier-generic}
By the above arguments and Remark~\ref{rem:generic-spherical}, the multiplier algebra model can also be defined for all $\mbt\in\mcT^*_\C$ excluding a countable set of values of $s_o^{(0)}$ (S-type) or $c_o^{(0)}$ (C-type).
\end{Rem}

\begin{Rem}\label{rem:que-model}
By Proposition~\ref{Propform}, for every $\mbt\in\mcT^*_\C$, the algebra $U^\mbt_h(\mfg^\theta)$ is a deformation of $U(\mfg^\theta_\mbt)$.
In the non-Hermitian case, $\mfg^\theta_\mbt = \mfg^\theta$ is semisimple and standard arguments show that if $\pi$ has image $U(\mfg)\fps$, then there exists $u\in 1+h U(\mfg)\fps$ such that $u\pi(U^\theta_h(\mfg))u^{-1}=U(\mfg^\theta)\fps$.
There also exists $\mcG\in 1+h U(\mfg^\theta) \otimes U(\mfg)\fps$ satisfying the conditions in Proposition~\ref{prop:conjug-coideal-coact-to-Delta}, analogously to Remark \ref{rem:que-rigidity-non-Hermitian}. (Moreover, by the remark following Proposition~\ref{prop:coideal-deform} we can go beyond the standard case and consider any $\mbt=(\mbc,\mbs)\in\mcT$ such that $c^{(0)}_i=1$ for all $i\in I\setminus X$.)
In other words, in the non-Hermitian case the multiplier algebra model does not have any particular advantages over the coideal picture.

In the Hermitian case it is still true that $U^\mbt_h(\mfg^\theta)$ is a trivial algebra deformation of $U(\mfg^\theta_\mbt)$, see Proposition~\ref{prop:coideal-deform}. But since in this case the first cohomology of $\mfg^\theta_\mbt$ with coefficients in a finite dimensional module is not always zero, it is not clear whether $u$ and $\mcG$ exist at the level of the universal enveloping algebras.
\end{Rem}

\begin{Rem}\label{rem:mult-model-type-II}
Type II symmetric pairs can be dealt with analogously to the non-Hermitian case.
The relevant involution on $\mfu \oplus \mfu$ in the Satake form is given by $\theta(X, Y) = (\omega(Y),\omega(X))$ for the Chevalley involution $\omega$.
The corresponding Satake diagram is the disjoint union of two copies of the Dynkin diagram of $\mfg$, with the corresponding vertices joined by arrows.
Cohomological considerations as above, both for multiplier algebras and universal enveloping algebras, carry over.
\end{Rem}

\subsection{\texorpdfstring{$K$}{K}-matrix of Balagovi\'{c}--Kolb}\label{ssec:K-matrix}

Next let us recall the construction of universal $K$-matrices for the coideals $U^\mbt_h(\mfg^\theta)$ according to~\citelist{\cite{MR3905136}\cite{MR4048733}\cite{MR3943480}}.
(Strictly speaking, these papers have an extra normalization condition on $\theta$ as in Remark \ref{rem:BK-normalization-of-involution}.
We can either adapt their construction to our setting, or we can first put this extra condition and then use $\Ad {z'}$ as in Remark \ref{rem:BK-normalization-of-involution} to remove it later.)

Denote by $U_q(\mfg)$ the $\C(q^{1/d})$-subalgebra of $U_h(\mfg)\otimes_{\C\fps}\C\fLauser$ generated by $K_\omega$ ($\omega\in P$), $E_i$ and $F_i$, where $q=e^h$ and $d = 4 \det((a_{ij})_{i,j})$.
(We use the same notation in Appendix~\ref{sec:spherical} for the algebra defined over $\K=\C\fLauser$, but since we are not going to use that algebra here, this should not lead to confusion.)
As usual we denote by $x \mapsto \bar x$ the \emph{bar involution}, the $\C$-linear automorphism of $U_q(\mfg)$ characterized by
\begin{align*}
\overline{q^{1/d}} &= q^{-1/d},&
\overline{K_\omega} &= K_{-\omega},&
\overline{E_i} &= E_i,&
\overline{F_i} &= F_i.
\end{align*}
In a similar way as before we define coideals $U^\mbt_q(\mfg^\theta)\subset U_q(\mfg)$ for $\mbt=(\mbc,\mbs)$ such that $c_i,s_i\in\C(q^{1/d})$.

We will first construct, following~\cite{MR3905136}, the $K$-matrix for a particular parameter $\mbt'\in\mcT$ defined by
\begin{align*}
c'_i &= q^{\frac{1}{2}(\alpha_i, \Theta(\alpha_i) - 2 \rho_X)},&
s'_i &= 0,
\end{align*}
where $\rho_X$ is half the sum of the positive roots of the root system generated by $X$.
The parameter $\mbt'=(\mbc',\mbs')$ satisfies the assumptions in~\cite{MR3905136}*{Section 5.4}.

A key ingredient of the construction in~\cite{MR3905136} is a \emph{quasi-$K$-matrix} $\quasiK$. Denote by $U^+\subset U_q(\mfg)$ the $\C(q^{1/d})$-subalgebra generated by the elements~$E_i$, and by $U^+_\mu\subset U^+$ the subspace of vectors of weight $\mu\in Q_+$, where $Q$ is the root lattice. Then
\[
\quasiK = \sum_{\mu \in Q_+} \quasiK_{\mu}\quad (\quasiK_{\mu}\in U^+_\mu),
\]
where the sum is considered in a completion of $U_q(\mfg)$ defined similarly to our multiplier algebra $\mcU(G)$, but over the field $\C(q^{1/d})$. The elements $\quasiK_\mu$ are uniquely determined by $\quasiK_0=1$ and the following recursive relations:
\begin{equation}\label{eq:quasi-K-recurrence}
[F_i,\quasiK_\mu]=\quasiK_{\mu-\alpha_i+\Theta(\alpha_i)}\overline{c_i' X_i}K_i - q^{-(\alpha_i,\Theta(\alpha_i))}K_i^{-1}c_i' X_i\quasiK_{\mu-\alpha_i+\Theta(\alpha_i)} \quad (i\in I),
\end{equation}
with the convention that $\quasiK_{\mu-\alpha_i+\Theta(\alpha_i)}=0$ if $\mu-\alpha_i+\Theta(\alpha_i)\not\in Q_+$.
Here we put
\begin{align*}
X_i &=0 \quad (i\in X),&
X_i &=- z_{\tau_\theta(i)}T_{w_X}(E_{\tau_\theta(i)}) \quad (i\in I\setminus X).
\end{align*}

To use $\quasiK$ in our setting, we need the following integrality property. Let $R\subset\C(q^{1/d})$ be the localization of the ring $\C[q^{1/d}]$ at $q^{1/d}=1$. Denote by $U^{+,\inte}$ the $R$-subalgebra of $U^+$ generated by the elements $E_i$, and put $U^{+,\inte}_{\mu}=U^+_\mu\cap U^{+,\inte}$.

Next, let $I^*\subset I\setminus X$ be a set of representatives of the $\tau_\theta$-orbits in $I\setminus X$. As we already used in Section~\ref{ssec:Satake} (although only in the Hermitian setting), the elements $\alpha_i^-=\hlf1(\alpha_i-\Theta(\alpha_i))$ for $i\in I^*$ form a basis of the restricted root system, and we have $\alpha_{\tau_\theta(i)}^-=\alpha_i^-$ for all $i$.

\begin{Prop}
Take $\mu\in Q_+$, $\mu\ne0$. If $\mu$ has the form
\[
\mu=\sum_{i\in I^*}k_i\alpha_i^-
\]
for some $k_i\in 2\Z_+$, then $\quasiK_\mu\in (q^{1/d}-1)U^{+,\inte}_{\mu}$. Otherwise $\quasiK_\mu=0$.
\end{Prop}

\bp
Let us start with the second statement, that is, $\quasiK_\mu=0$ if either $\Theta(\mu)\ne-\mu$, or $\Theta(\mu)=-\mu$ but in the decomposition $\mu=\sum_{i\in I^*}k_i\alpha_i^-$ some integers $k_i\ge0$ are odd.
This is a refinement of a condition in \cite{MR3905136}*{Section 6.1}, and the proof is basically the same.

To be precise, consider the height of $\mu$ defined by $\htt(\mu)=\sum_{i\in I}m_i$ if $\mu=\sum_{i\in I}m_i\alpha_i$.
We verify the condition by induction on $\htt(\mu)$.
Since $\mu-\alpha_i+\Theta(\alpha_i)=\mu-2\alpha_i^-$ for $i\in I\setminus X$ is either not in $Q_+$ or it satisfies the same assumptions as $\mu$, by the inductive hypothesis we get from~\eqref{eq:quasi-K-recurrence} that $[F_i,\quasiK_\mu]=0$ for all $i\in I$. This means that Lusztig's skew-derivatives $_i r(\quasiK_\mu)$ and $r_i(\quasiK_\mu)$ are zero, which is possible only if $\quasiK_\mu=0$, see~\cite{MR2759715}*{Proposition 3.1.6 and Lemma~1.2.15}.

\smallskip

Turning to the first statement, assume $\mu=\sum_{i\in I^*}k_i\alpha_i^-$ with $k_i\in2\Z_+$. Put $\htt^*(\mu)=\hlf1\sum_{i\in I^*}k_i$. We will prove the statement by induction on $\htt^*(\mu)$.

Consider the case $\htt^*(\mu)=1$. Then $\mu=2\alpha_j^-=\alpha_j-\Theta(\alpha_j)$ for some $j\in I^*$. From~\eqref{eq:quasi-K-recurrence} we then get
\begin{align}\label{eq:quasi-K-recurrence2}
[F_i,\quasiK_\mu] &= 0\quad (i\in I\setminus\{j,\tau_\theta(j)\}),&
[F_i,\quasiK_\mu] &= \overline{c_i' X_i}K_i - q^{-(\alpha_i,\Theta(\alpha_i))}K_i^{-1}c_i' X_i
\quad (i=j,\tau_\theta(j)).
\end{align}
Denote by $U^\inte$ the $R$-subalgebra of $U_q(\mfg)$ generated by the elements $K_i^{\pm1}$, $\frac{K_i-1}{q-1}$, $E_i$ and $F_i$. Then we have an isomorphism $U^\inte/(q^{1/d}-1)U^\inte \to U(\mfg)$ such that
\begin{align*}
K_i^{\pm1} &\mapsto 1,&
\frac{K_i-1}{q-1} &\mapsto d_i H_i,&
E_i &\mapsto X_{\alpha_i},&
F_i &\mapsto X_{-\alpha_i}.
\end{align*}
Since $X_i\in U^\inte$, by~\eqref{eq:quasi-K-recurrence2} we conclude that $[F_i,\quasiK_\mu]\in (q^{1/d}-1)U^\inte$ for all $i\in I$. We claim that this implies that $\quasiK_\mu\in (q^{1/d}-1)U^\inte$, hence $\quasiK_\mu\in (q^{1/d}-1)U^{+,\inte}_\mu$, as $(q^{1/d}-1)U^\inte\cap U^+=(q^{1/d}-1)U^{\inte,+}$ by the triangular decomposition of $U^\inte$.

Indeed, assuming $\quasiK_\mu\ne0$, let $k\in\Z$ be the smallest number such that $(q^{1/d}-1)^k\quasiK_\mu\in U^\inte$. If $k\ge0$, then, on the one hand, the image of $(q^{1/d}-1)^k\quasiK_\mu$ in $U(\mfg)$ is a nonzero element of $U(\mfn^+)_\mu$, and on the other hand this image commutes with $X_{-\alpha_i}$ for all $i$. But this is impossible, hence $k\le -1$.

The inductive step is similar. Using~\eqref{eq:quasi-K-recurrence} and the inductive hypothesis we get $[F_i,\quasiK_\mu]\in (q^{1/d}-1)U^\inte$ for all $i\in I$. Hence $\quasiK_\mu\in (q^{1/d}-1)U^{+,\inte}_{\mu}$.
\ep

Recall that $\pi\colon U_h(\mfg)\to\mcU(G)\fps$ denotes a fixed homomorphism satisfying~\eqref{eq:que-iso}. When it is convenient, we extend it to $U_h(\mfg)\otimes_{\C\fps}\C\fLauser$ and the completion of $U_q(\mfg)$ from~\cite{MR3905136}, but then the target algebra should be $\mcU(G)\fLauser$ and $\prod_{\pi\in\Irr G} \left(\End(V_\pi) \fLauser \right)$, respectively.

\begin{Cor}\label{cor:quasiK}
We have $\pi(\quasiK)\in 1+h\mcU(G)\fps$.
\end{Cor}

Following \cite{MR3905136}, consider a homomorphism $\gamma\colon P\to\C(q^{1/d})^\times$ such that
\begin{align}\label{eq:gamma-cond}
\gamma(\alpha_i) &= c_i' z_{\tau_\theta(i)}\quad (i\in I\setminus X),&
\gamma(\alpha_i) &= 1\quad (i\in X),
\end{align}
and put
\[
\xi(\omega) = \gamma(\omega) q^{-(\omega^+, \omega^+) + \sum_{i \in I} (\alpha_i^-, \alpha_i^-) \omega(\varpi_i^\vee)},
\]
where $\omega^+ = \hlf1 (\omega + \Theta(\omega))$ and $(\varpi_i^\vee)_{i\in I}$ is the dual basis (fundamental coweights) of $(\alpha_i)_{i\in I}$, see~\cite{MR3905136}*{(8.1)}.
It satisfies the relation
\begin{equation*}
\xi(\omega + \alpha_i) = \gamma(\alpha_i) q^{-(\alpha_i, \Theta(\alpha_i)) - (\omega, \alpha_i + \Theta(\alpha_i))} \xi(\omega)
\end{equation*}
for $\omega\in P$ and $i \in I$, which is enough for most of the purposes.
We can view $\xi$ as an element of a completion of $U_q(\mfh)\subset U_q(\mfg)$. Then one takes
\begin{equation*}
\msK' = \mfX \xi T_{w_X}^{-1} T_{w_0}^{-1},
\end{equation*}
where $T_{w_X}$ and $T_{w_0}$ now denote the canonical elements implementing the Lusztig automorphisms.
This gives a universal $K$-matrix for $U^{\mbt'}_q(\mfg^\theta)$ in the conventions of~\cite{MR3905136}.

To pass to our setting, consider the element $\omega_0$ of $\mfh^*$ characterized by
\begin{align*}
(\omega_0,\alpha_i) &= 0\quad (i\in X),&
(\omega_0,\alpha_i) &= \frac{1}{4}(\Theta(\alpha_{\tau_\theta(i)}) - \alpha_{\tau_\theta(i)} -\Theta(\alpha_i) +2\rho_X,\alpha_i)\quad (i\in I\setminus X),
\end{align*}
and use the isomorphism $\Ad K_{\omega_0}$ of $U^{\mbt'}_q(\mfg^\theta)$ onto $U^{\theta}_q(\mfg)$. Namely, define
\begin{equation}
\label{eq:k-mats}
\msK = \tau_\theta\tau_0\bigl((\Ad K_{\omega_0})(\msK')\bigr)=(\Ad K_{\omega_0})\bigl(\tau_\theta\tau_0(\msK')\bigr),
\end{equation}
where $\tau_\theta\tau_0$ is the automorphism of the Hopf algebra $U_q(\mfg)$ induced by the automorphism $\tau_\theta\tau_0$ of the Dynkin diagram.
Finally, using the universal $R$-matrix of $U_q(\mfg)$, we put
\begin{equation}
\label{eq:BK-braid}
\msE=\msR_{21}(1\otimes\msK)(\id\otimes\tau_\theta\tau_0)(\msR).
\end{equation}
This is a ribbon $\tau_\theta\tau_0$-braid for $U_q^\theta(\mfg)$, hence also for $U_h^\theta(\mfg)$, see \cite{MR3943480}*{Section~3.3}. Then
\[
\mcE_h=((\Ad u)\pi\otimes\pi)(\msE)\in\mcU(G^\theta\times G)\fps
\]
is a well-defined ribbon $(\tau_\theta\tau_0)_h$-braid for the multiplier algebra model of $U_h^\theta(\mfg)$, where $u$ is the element~\eqref{eq:umbt} (for $\mbt=0$) and $(\tau_\theta\tau_0)_h$ denotes the unique automorphism of $\mcU(G)\fps$ such that
\begin{equation}\label{eq:twist-auto_h}
\pi\circ \tau_\theta\tau_0=(\tau_\theta\tau_0)_h\circ\pi.
\end{equation}
We call $\msE$ (and also $\mcE_h$) a \emph{Balagovi\'{c}--Kolb ribbon $(\tau_\theta\tau_0)_h$-braid}. Note that this element depends on the choice of~$\gamma$, and the set of these twist-braids forms a torsor over $Z(U)$.

\begin{Rem}
It is not difficult to see that Corollary~\ref{cor:quasiK} and identities~\eqref{eq:z-eigval-rule} imply that $\mcE_h=1\otimes g z m_X m_0$ modulo $h$ for some $g\in Z(U)$. This is consistent with Theorems~\ref{thm:compar-rib-tw-br-KZ-LK-non-Hermitian} and~\ref{compar-rib-tw-br-KZ-Hermitian} below.
\end{Rem}

This finishes our discussion of the ribbon twist-braids in the standard case. Assume now that $\mfu^\theta<\mfu$ is Hermitian and take $\mbt\in\mcT^*_\C$.
Note that $\tau_\theta\tau_0=\id$ now, since $\theta$ is an inner automorphism. The coideal $U^\mbt_h(\mfg^\theta)$ can be obtained from $U^\theta_q(\mfg)$ by twisting and $h$-adic completion similarly to~\cite{MR3943480}*{Theorem C.7}. Namely, define a character $\chi_\mbt\colon U^\theta_q(\mfg)\to\C\fLauser$ as follows:
\begin{itemize}
\item S-type: $\chi_\mbt(K_\omega)=1$ for $\omega\in P^\Theta$, $\chi_\mbt=\epsilon$ on $U_q(\mfg_X)$, $\chi_\mbt(B_i)=0$ for the nondistinguished vertices~$i$,
\[
\chi_\mbt(B_o)=\frac{s_o \kappa_o}{e^{-d_oh}-1};
\]
\item C-type: $\chi_\mbt=\epsilon$ on $U_q(\mfg_X)$, $\chi_\mbt(B_i)=0$ for all $i$, $\chi_\mbt(K_\omega)=\lambda(\omega)$ for $\omega\in P^\Theta$, where $\lambda\colon P\to \C\fLauser^\times$ is any homomorphism such that $\lambda(\alpha_i)=1$ for all $i\in I\setminus\{o\}$ and
\[
\lambda(\alpha_o)=c_o^{-1}e^{-h(\alpha_o^-,\alpha_o^-)}.
\]
\end{itemize}
Then $(\chi_\mbt\otimes\id)\Delta$ maps the generators of $U_q^\theta(\mfg)$ into those of $U^\mbt_h(\mfg^\theta)$, except that in the type S case $B_o$ is mapped into
\[
F_o - c_oz_{\tau_\theta(o)}T_{w_X}(E_{\tau_\theta(o)})K_o^{-1} + s_o \kappa_o \frac{K_{o}^{-1}}{e^{-d_oh}-1},
\]
but this differs only by an additive constant (which may, however, lie in $h^{-1}\C\fps$ rather than in $\C\fps$) from the corresponding generator of $U^\mbt_h(\mfg^\theta)$. By applying this map to the first leg of $\msE$, and using the factorization of $\msE$ given in~\cite{MR4048733}, we get a ribbon braid~$\msE^\mbt$ for $U^\mbt_h(\mfg^\theta)$. Then
\[
\mcE^\mbt_h=((\Ad u^\mbt)\pi\otimes\pi)(\msE^\mbt)
\]
is a ribbon braid for the multiplier algebra model of $U^\mbt_h(\mfg^\theta)$, whenever this model is well-defined.
We call~$\msE^\mbt$ (and~$\mcE^\mbt_h$) again a \emph{Balagovi\'{c}--Kolb ribbon braid}.

One problem, however, is that in the S-type case the construction of $\msE^\mbt$ guarantees only that
\[
\mcE^\mbt_h\in\prod_{\substack{\rho\in\Irr G^\theta_\mbt,\\ \pi\in\Irr G}} \left(\End(V_\rho)\otimes\End(V_\pi) \fLauser \right).
\]

\begin{Prop}
For all $\mbt\in\mcT^*_\C$, we have $(\pi\otimes\pi)(\msE^\mbt)\in\mcU(G\times G)\fps$.
\end{Prop}

\bp
We need only to consider the Hermitian S-type case.
Assume first that $\mbt\in\mcT^*$, that is, $s_o\in i\R\fps$. Then $\mcE^\mbt_h$ is a ribbon braid for the coaction $(\mcU(G^\theta_\mbt)\fps,\alpha_h)$ of $(\mcU(G)\fps,\Delta_h,\mcR_h)$. Hence the assertion follows from Theorem~\ref{Thm:not-PL-subgr-forces-K-mat-to-be-formal-pow-ser}, which is applicable by the results of Section~\ref{sec:mult-alg-model-of-LK-coids} and Corollary~\ref{cor:coisotropic}.
For the general case, observe that by construction the coefficient of $h^k$ of the component of $(\pi\otimes\pi)(\msE^\mbt)$ in $\End(V)\otimes\End(W)\fLauser$ is a rational function in finitely many parameters $s_o^{(n)}$. Since for $k<0$ this function vanishes for purely imaginary $s_o^{(n)}$, it must be zero.
\ep

In particular, if the multiplier algebra model of $U^\mbt_h(\mfg^\theta)$ is well-defined for some $\mbt\in\mcT^*_\C$, then we have $\mcE^\mbt_h\in\mcU(G^\theta_\mbt\times G)\fps$. It would still be interesting to find a more explicit construction of $\msE^\mbt$ similar to that for $\msE$, and provide a more direct proof of the above proposition.

\section{Comparison theorems}
\label{sec:compar}

We will combine the results of the previous sections to compare the Letzter--Kolb coideals with the quasi-coactions defined by the KZ-equations.

\subsection{Twisting of ribbon twist-braids}\label{ssec:comparison}

Let us start by refining the twisting procedure from Section~\ref{ssec:twisting}.
Assume $H$ is a reductive algebraic subgroup of $G$, and $(\mcU(H)\fps, \alpha, \Psi)$ is a quasi-coaction of $(\mcU(G)\fpser,\Delta_h,\Phi)$. Then, given $\mcF\in\mcU(G^2)\fps$ and $\mcG\in \mcU(H\times G)\fps$ such that $\mcF^{(0)}=1$, $\mcG^{(0)}=1$ and
\begin{align*}
(\epsilon\otimes\id)(\mcF) &= (\id\otimes\epsilon)(\mcF)=1,&
(\id\otimes\epsilon)(\mcG) &= 1,
\end{align*}
we get a quasi-coaction $(\mcU(H)\fpser,\alpha,\Psi_{\mcF,\mcG})$ of $(\mcU(G)\fpser,\Delta_{h,\mcF},\Phi_\mcF)$.

Now, assume in addition that $\beta$ is an involutive automorphism of $(\mcU(G)\fpser,\Delta_h,\Phi)$ and $v\in\mcU(G)\fps$ is an element such that $v^{(0)}=1$,
\begin{align}\label{eq:auto-twisting}
v\beta(v) &=1,&
\mcF &= (v\otimes v)(\beta\otimes\beta)(\mcF)\Delta_h(v)^{-1}.
\end{align}

\begin{Prop}
\label{prop:twist-of-twist-braids}
Under the above assumptions, $\beta_v =v\beta(\cdot)v^{-1}$ is an involutive automorphism of $(\mcU(G)\fpser,\Delta_{h,\mcF},\Phi_\mcF)$.
Furthermore, suppose that $\mcR \in \mcU(G^2) \fpser$ is an $R$-matrix for $(\mcU(G)\fpser,\Delta_h,\Phi)$ fixed by $\beta$, and that $\mcE\in \mcU(H\times G)\fps$ is a ribbon $\beta$-braid for $\mcR$.
Then $\mcR_\mcF=\mcF_{21}\mcR\mcF^{-1}$ is an $R$-matrix for $(\mcU(G)\fpser,\Delta_{h,\mcF},\Phi_\mcF)$ fixed by $\beta_v$, and
\begin{equation}\label{eq:braid-twist}
\mcE_{\mcG,v}=\mcG\mcE(\id\otimes\beta)(\mcG)^{-1}(1\otimes v^{-1})=\mcG\mcE(1\otimes v^{-1})(\id\otimes\beta_v)(\mcG)^{-1}
\end{equation}
is a ribbon $\beta_v$-braid for the quasi-coaction $(\mcU(H)\fpser,\alpha_\mcG,\Psi_{\mcF,\mcG})$ of $(\mcU(G)\fpser,\Delta_{h,\mcF},\Phi_\mcF, \mcR_\mcF)$.
\end{Prop}

We call $(\mcU(G)\fpser,\Delta_{h,\mcF},\Phi_\mcF,\beta_v)$ the twisting of $(\mcU(G)\fpser,\Delta_h,\Phi,\beta)$ by $(\mcF,v)$.

\bp
The claims are not difficult to check by a direct computation, but let us explain a more conceptual proof using crossed products (or smashed products), cf.~\cite{MR3943480}*{Remark 1.13}.
Namely, consider the algebra
\[
\mcU(G)\fps\rtimes_\beta\Z/2\Z = \{ a + a' \lambda_\beta \mid a, a' \in \mcU(G)\fps, \lambda_\beta^2 = 1, \lambda_\beta a = \beta(a) \lambda_\beta \}.
\]
We can extend in the usual way the coproduct $\Delta_h$ on $\mcU(G)\fps$ to a coproduct $\tilde\Delta_h$ on $\mcU(G)\fps\rtimes_\beta\Z/2\Z$ by letting $\tilde\Delta_h(\lambda_\beta)=\lambda_\beta\otimes\lambda_\beta$.
Then $(\mcU(G)\fps\rtimes_\beta\Z/2\Z,\tilde\Delta_h,\Phi)$ is a multiplier quasi-bialgebra.

Now, given $(\mcF,v)$ as above, we can twist $(\mcU(G)\fps\rtimes_\beta\Z/2\Z,\tilde\Delta_h,\Phi)$ by $\mcF$ to get a new multiplier quasi-bialgebra $(\mcU(G)\fps\rtimes_\beta\Z/2\Z,(\tilde\Delta_h)_\mcF,\Phi_\mcF)$.
On the other hand, we can first twist $(\mcU(G)\fps,\Delta_h,\Phi)$ by $\mcF$ and then consider the crossed product by $\beta_v$ to get $(\mcU(G)\fps\rtimes_{\beta_v}\Z/2\Z,(\Delta_{h,\mcF})^\sim,\Phi_\mcF)$.
The map
\[
f\colon \mcU(G)\fps\rtimes_\beta\Z/2\Z\to \mcU(G)\fps\rtimes_{\beta_v}\Z/2\Z,\quad a\mapsto a,\quad \lambda_\beta\mapsto v^{-1}\lambda_{\beta_v},
\]
is an isomorphism of these two multiplier quasi-bialgebras.
In particular, $$(\mcU(G)\fps\rtimes_{\beta_v}\Z/2\Z,(\Delta_{h,\mcF})^\sim,\Phi_\mcF)$$ is indeed a multiplier quasi-bialgebra, and hence $\beta_v$ is an automorphism of $(\mcU(G)\fps,\Delta_{h,\mcF},\Phi_\mcF)$.

Let us turn to ribbon twist-braids.
First note that $\mcR$ is still an $R$-matrix for $(\mcU(G)\fps\rtimes_\beta\Z/2\Z,\tilde\Delta_h,\Phi)$ by its $\beta$-invariance.
Moreover, we can view  $(\mcU(H)\fps,\alpha,\Psi)$ as a quasi-coaction of this multiplier quasi-bialgebra.
Then an element $\mcE\in \mcU(H\times G)\fps$ is a ribbon $\beta$-braid for the original quasi-coaction and $\mcR$ if and only if $\mcE(1\otimes\lambda_\beta)$ is a ribbon braid for the new one and $\mcR$ again.

Finally, the map $f$ satisfies
\[
(\id \otimes f)(\mcG \mcE (1 \otimes \lambda_\beta) \mcG^{-1}) = \mcE_{\mcG,v} (1 \otimes \lambda_{\beta_v}),
\]
showing that formula~\eqref{eq:braid-twist}~is a consequence of~\eqref{eq:braid-twist0} for the crossed products and trivial automorphisms.
\ep

\begin{Rem}
Let us also mention a categorical perspective on conditions~\eqref{eq:auto-twisting}, which does not rely on crossed products.
The automorphism $\beta$ defines an autoequivalence $F_\beta$ of $((\Rep G)\fps,\otimes_h,\Phi)$.
The twisting by $\mcF$ produces an equivalent category $((\Rep G)\fps,\otimes_{h,\mcF},\Phi_\mcF)$.
The functor $F_\beta$ gives rise to an autoequivalence of this new category, which, however, is not defined by any automorphism in general.
Conditions~\eqref{eq:auto-twisting} ensure that this autoequivalence is naturally monoidally isomorphic to an autoequivalence defined by an automorphism, namely, to $F_{\beta_v}$.
\end{Rem}

We now return to the setup of Section~\ref{ssec:twist-multiplier-model}.
Let $\pi\colon U_h(\mfg)\to\mcU(G)\fps$ be a homomorphism satisfying~\eqref{eq:que-iso}.
Assume $\beta$ is an involutive automorphism of the Dynkin diagram of $\mfg$.
We denote by the same symbol the corresponding automorphisms of $(U_h(\mfg),\Delta)$ and $(\mcU(G)\fps,\Delta)$; it will always be clear from the context which one we are using.
These are automorphisms of the quasi-triangular (multiplier quasi-)bialgebras $(U_h(\mfg), \Delta, \msR)$ and $(\mcU(G)\fps, \Delta, \Phi_\KZ, \mcR_\KZ)$.
We also note that, similarly to~\eqref{eq:twist-auto_h}, there is a unique automorphism $\beta_h$ of $(\mcU(G)\fps,\Delta_h)$ such that
\[
\pi\circ\beta=\beta_h\circ\pi.
\]

\begin{Lem}\label{lem:automorphism-twisting}
Let $\mcF$ be a Drinfeld twist for $\pi$ (in the sense of Section~\ref{ssec:twist-multiplier-model}). Then there exists a unique element $v\in 1+h\mcU(G)\fps$ such that
\begin{align}\label{eq:auto-twisting2}
\beta_h &= v\beta(\cdot)v^{-1},&
\mcF &= (v\otimes v)(\beta\otimes\beta)(\mcF)\Delta(v)^{-1}.
\end{align}
We also have $v\beta(v)=1$. If, in addition, $\pi$ is $*$-preserving and $\mcF$ is unitary, then $v$ is unitary.
\end{Lem}

In other words, once $\mcF$ is fixed, $\beta_h$ is a twisting of the automorphism $\beta$ of $(\mcU(G)\fps,\Delta)$ in a unique way.

\bp
Since the homomorphisms $\pi\circ\beta$ and $\beta\circ\pi$ are equal modulo $h$, there exists $w\in 1+h\mcU(G)\fps$ such that
$\pi\circ\beta=(\Ad w)\circ\beta\circ\pi$. Then $\beta_h=(\Ad w)\circ\beta$.

We claim that $(w\otimes w)(\beta\otimes\beta)(\mcF)\Delta(w)^{-1}$ is again a Drinfeld twist (for the same $\pi$). Since $\Phi_\KZ$ is invariant under $\beta$, condition~\eqref{eq:Drinfeld-twist1} is satisfied for $(\beta\otimes\beta)(\mcF)$, hence also for $(w\otimes w)(\beta\otimes\beta)(\mcF)\Delta(w)^{-1}$.
It remains to check that
\[
\Delta_h=(w\otimes w)(\beta\otimes\beta)(\mcF)\Delta(w)^{-1}\Delta(\cdot)\Delta(w)(\beta\otimes\beta)(\mcF^{-1})(w^{-1}\otimes w^{-1}),
\]
or equivalently,
\[
\Delta_h\bigl(\beta_h(\cdot)\bigr)=(w\otimes w)(\beta\otimes\beta)(\mcF)\Delta(w)^{-1}\Delta\bigl(\beta_h(\cdot)\bigr)\Delta(w)(\beta\otimes\beta)(\mcF^{-1})(w^{-1}\otimes w^{-1}).
\]
But this is true, as the right hand side of the above identity is easily seen to be equal to $(\beta_h\otimes\beta_h)\Delta_h$.

By Lemma~\ref{lem:unique-twist} it follows that by multiplying $w$ by a central element we get an element $v\in 1+h\mcU(G)\fps$ satisfying~\eqref{eq:auto-twisting2}. Assume $v'$ is another element with the same properties. Then $v^{-1}v'$ is a central element, hence it also equals $v' v^{-1}$ and
\[
\mcF=(v' v^{-1}\otimes v' v^{-1})\mcF\Delta(v' v^{-1})^{-1}.
\]
By the uniqueness part of Lemma~\ref{lem:unique-twist} we conclude that $v' v^{-1}=1$.

Next, since $\beta_h$ and $\beta$ are both involutive, the element $\beta(v)^{-1}$ has the same properties as $v$, hence $\beta(v)^{-1}=v$. Similarly, if $\pi$ is $*$-preserving and $\mcF$ is unitary, then $\beta_h$ is $*$-preserving as well, and the element $(v^*)^{-1}$ has the same properties as $v$, hence $(v^*)^{-1}=v$.
\ep

\subsection{Comparison theorem: non-Hermitian case}
\label{sec:KZ-LK-qcoact-compar-non-Herm}

We are now ready to prove our main results relating the multiplier algebra models of the Letzter--Kolb coideals to cyclotomic KZ-equations.
Let us first consider the non-Hermitian case.

\begin{Thm}\label{thm:qbialg-compar-KZ-LK-non-Hermitian}
Assume $\mfk=\mfu^\theta<\mfu$ is a non-Hermitian symmetric pair, with $\theta$ in Satake form~\eqref{eq:factoriz-theta}.
Then the multiplier algebra model of the Letzter--Kolb coideal $U^\theta_h(\mfg)$, which is a coaction $(\mcU(G^\theta)\fps,\alpha_h)$ of $(\mcU(G)\fps,\Delta_h)$, is obtained by twisting from the quasi-coaction $(\mcU(G^\theta)\fps,\Delta,\Psi_\KZ)$ of the multiplier quasi-bialgebra $(\mcU(G)\fps,\Delta,\Phi_\KZ)$.
Any such twisting extends to a twisting between the automorphism $\tau_\theta\tau_0$ of $(\mcU(G)\fps,\Delta,\Phi_\KZ)$ and the automorphism $(\tau_\theta\tau_0)_h$ of $(\mcU(G)\fps,\Delta_h)$.
\end{Thm}

\bp
Using a Drinfeld twist $\mcF$ and an element $\mcG$ provided by Proposition~\ref{prop:conjug-coideal-coact-to-Delta} we can twist the coaction $(\mcU(G^\theta)\fps,\alpha_h)$ of $(\mcU(G)\fps,\Delta_h)$ to a quasi-coaction  $(\mcU(G^\theta)\fps,\Delta,\Psi)$ of $(\mcU(G)\fps,\Delta,\Phi_\KZ)$ for some $\Psi$. The first statement of the theorem follows then from Theorem~\ref{thm:Psi-rigid}. The second statement, on twisting~$\tau_\theta\tau_0$ to~$(\tau_\theta\tau_0)_h$, follows from Lemma~\ref{lem:automorphism-twisting}.
\ep

\begin{Thm}\label{thm:compar-rib-tw-br-KZ-LK-non-Hermitian}
The twisting provided by Theorem \ref{thm:qbialg-compar-KZ-LK-non-Hermitian} establishes a one-to-one correspondence between the following data:
\begin{itemize}
\item the ribbon $\tau_\theta\tau_0$-braids for the quasi-coaction $(\mcU(G^\theta)\fpser, \Delta, \Psi_\KZ)$ of the quasi-triangular multiplier quasi-bialgebra $(\mcU(G)\fps,\Delta,\Phi_\KZ, \mcR_\KZ)$, given by
\begin{equation}\label{eq:braid-non-Hermitian}
\mcE'_\KZ g_1 = \exp(-h(2t^{\mfk}_{01}+C^\mfk_1))(z m_X m_0 g)_1 \quad (g\in Z(U));
\end{equation}
\item the Balagovi\'{c}--Kolb ribbon $(\tau_\theta\tau_0)_h$-braids~$\msE$ (or their images $\mcE_h$) for the coideal $U^\theta_h(\mfg)$ of the quasi-triangular bialgebra $(U_h(\mfg), \Delta_h, \msR)$, for different choices of $\gamma$ satisfying \eqref{eq:gamma-cond}.
\end{itemize}
Under this correspondence, we have $\mcE_h^{(0)} = 1 \otimes z m_X m_0 g$.
\end{Thm}

\bp
By Theorem~\ref{thm:uniq-tw-br-nonherm-case} we have a complete classification of ribbon $\theta$-braids for the quasi-coaction $(\mcU(G^\theta)\fps,\Delta,\Psi_\KZ)$ of $(\mcU(G)\fps,\Delta,\Phi_\KZ)$.
Since $\theta=(\Ad z m_X m_0)\circ\tau_\theta\circ\tau_0$, the multiplication by $1\otimes z m_X m_0$ on the right gives a one-to-one correspondence between the ribbon $\theta$-braids and the ribbon $\tau_\theta\tau_0$-braids, so the latter ones are given by~\eqref{eq:braid-non-Hermitian}.
As any Drinfeld twist $\mcF$ satisfies~\eqref{eq:R-matrix}, formula~\eqref{eq:braid-twist} provides a correspondence between the ribbon $\tau_\theta\tau_0$-braids and the $(\tau_\theta\tau_0)_h$-braids.
Since both the ribbon $\tau_\theta\tau_0$-braids and the Balagovi\'{c}--Kolb ribbon $(\tau_\theta\tau_0)_h$-braids are torsors over the finite group $Z(U)$, this gives a bijective correspondence stated in the theorem.

Finally, since by definition the elements $\mcF$, $\mcG$, and $v$ used in the twisting have constant terms $1$, we get the claim about $\mcE_h^{(0)}$.
\ep

\begin{Rem}\label{rem:non-Hermitian-UEA-extension}
As we pointed out throughout the paper (see Remarks~\ref{rem:que-rigidity-non-Hermitian} and \ref{rem:que-model}), in the non-Hermitian case there is no real need to consider the multiplier algebra model, so a similar result holds at the level of the universal enveloping algebras, also beyond the standard case.

Let us formulate this more precisely.
Let $\mbt=(\mbc,\mbs)\in\mcT$ be such that $c^{(0)}_i=1$ for all $i\in I\setminus X$ (recall also that, by definition, we have $s_i=0$ for all $i\in I\setminus X$).
If we fix algebra isomorphisms $U_h(\mfg)\cong U(\mfg)\fps$ and $U^\mbt_h(\mfg^\theta) \cong U(\mfg^\theta)\fps$ that are identity modulo $h$ (that is, they are given by~\eqref{eq:que-iso} and Proposition~\ref{prop:coideal-deform}), then the coproduct $\Delta \colon  U_h(\mfg) \to U_h(\mfg)\hotimes U_h(\mfg)$ defines a coproduct $\Delta_h$ on $U(\mfg)\fps$ and a coaction $\alpha_h\colon U(\mfg^\theta)\fps\to U(\mfg^\theta) \otimes U(\mfg)\fps$ of $(U(\mfg)\fps,\Delta_h)$.

The claim then is that this coaction is obtained by twisting from the quasi-coaction $(U(\mfg^\theta)\fps,\Delta,\Psi_\KZ)$ of $(U(\mfg)\fps,\Delta,\Phi_\KZ)$.
Any such twisting extends to a twisting between the automorphisms $\tau_\theta\tau_0$ of $(U(\mfg)\fps,\Delta,\Phi_{\KZ})$ and $(U(\mfg)\fps,\Delta_h)$ and, in the standard case $\mbt=0$, provides a one-to-one correspondence between the ribbon $\tau_\theta\tau_0$-braids as in Theorem~\ref{thm:compar-rib-tw-br-KZ-LK-non-Hermitian}.
\end{Rem}

\begin{Rem}\label{rem:compar-thm-type-II}
The type II symmetric pairs admit analogues of Theorems \ref{thm:qbialg-compar-KZ-LK-non-Hermitian} and \ref{thm:compar-rib-tw-br-KZ-LK-non-Hermitian} and Remark~\ref{rem:non-Hermitian-UEA-extension}, with essentially the same proofs. Indeed, as we have explained along the way, the intermediate results used in the proofs, such as Theorems \ref{thm:Psi-rigid} and \ref{thm:uniq-tw-br-nonherm-case} and Proposition~\ref{prop:conjug-coideal-coact-to-Delta}, all have analogues for the type~II case.
\end{Rem}

\subsection{Comparison theorem: Hermitian case}
\label{sec:KZ-LK-qcoact-compar-Herm}

In the Hermitian case we do need to consider the multiplier algebra model in our approach.

\begin{Thm}\label{thm:qbialg-compar-KZ-LK-Hermitian}
Assume $\mfu^\theta<\mfu$ is a Hermitian symmetric pair, with $\theta$ in Satake form~\eqref{eq:factoriz-theta}.
Take $\mbt\in\mcT^*$ and choose $Z^\mbt_\theta\in\mfz(\mfg^\theta_\mbt)$ such that $(Z^\mbt_\theta,Z^\mbt_\theta)_\mfg=-a_\theta^{-2}$.
Then the coaction $(\mcU(G^\theta_\mbt)\fps,\alpha_h)$ of $(\mcU(G)\fps,\Delta_h)$ is obtained by twisting from the quasi-coaction $(\mcU(G^\theta_\mbt)\fps,\Delta,\Psi_{\KZ,s;\mu})$ of $(\mcU(G)\fps,\Delta,\Phi_\KZ)$
for uniquely defined $s\in\R$ and $\mu\in h\R\fps$, where $\Psi_{\KZ,s;\mu}$ is defined using $Z^\mbt_\theta$.

The parameter $s\in\R$ is determined as follows:
\begin{itemize}
\item{\textup{S-type:}} if $\alpha_o$ is the unique distinguished root and $c = -i s^{(0)}_o$, then
\[
s=\pm\frac{2}{\pi}\log\bigl((1+c^2)^{1/2}+c\bigr),
\]
where $\pm$ is the sign of $\kappa_o (Z^\mbt_\theta,X_{\alpha_o})_\mfg$;
\item{\textup{C-type:}} if $\alpha_o$ is the unique distinguished root such that $-i\alpha_o(\tilde Z^\mbt_\theta)+i\alpha_{\tau_\theta(o)}(\tilde Z^\mbt_\theta)>0$, where $\tilde Z^\mbt_\theta$ is the component of $Z^\mbt_\theta$ in $\mfh$, and $c=c^{(0)}_o$, then
\[
s=\frac{2}{\pi}\log c.
\]
\end{itemize}
\end{Thm}

\bp
Choose $Z_\theta\in\mfz(\mfu^\theta)$ such that $(Z_\theta,Z_\theta)_\mfg=-a_\theta^{-2}$. In the C-type case we require also that if $\alpha_o$ is the distinguished root as in the formulation of the theorem, then $-i\alpha_o(Z_\theta)+i\alpha_{\tau_\theta(o)}(Z_\theta)>0$, which determines $Z_\theta$ uniquely. By Lemma~\ref{lem:nu-from-theta} and the discussion following it we then get a pair $(\nu,Z_\nu)$ as in Section~\ref{sec:int-subgroups}.

By twisting the coaction we may assume that $\pi\colon U_h(\mfg)\to \mcU(G)\fps$ defining the multiplier algebra model is as in Lemma~\ref{lem:que-iso}. Then by Lemma~\ref{lem:Drinfeld-twist} there is a unitary Drinfeld twist such that
\[
\mcF=1+h\frac{i r}{2}+O(h^2).
\]
By Lemma~\ref{lem:u-conjugate} we could also choose $u$ satisfying \eqref{eq:umbt} to be unitary, which means that by twisting~$\alpha_h$ we may assume $\alpha_h$ to be $*$-preserving. Hence, by Proposition~\ref{prop:conjug-coideal-coact-to-Delta}, $\alpha_h$ is a twisting of $\Delta$ by a unitary element~$\mcG$.

By Lemma~\ref{lem:mbt-vs-phi} we have $\mfg^\theta_\mbt=(\Ad z_\theta)^{-1}(\mfg_\phi)$ and $Z^\mbt_\theta=(\Ad z_\theta^{-1}g_{\phi-1})(Z_\nu)$, where $\phi$ is determined as follows (see also Remark \ref{rem:sign-detection}):
\begin{itemize}
\item{\textup{S-type:}} $z_\theta(\alpha_o)s^{(0)}_o\kappa_o=i\tan(\frac{\pi \phi}{2})$ and $\displaystyle\frac{(Z^\mbt_\theta,X_{\alpha_o})_\mfg}{z_\theta(\alpha_o)\cos(\frac{\pi \phi}{2})}>0$;
\item{\textup{C-type:}} $c^{(0)}_o=-\cot(\frac{\pi}{4}(\phi-1))$.
\end{itemize}

The map $\Ad z_\theta$ defines isomorphisms $\mcU(G^\theta_\mbt)\to\mcU(G_\phi)$ and $\mcU(G)\to\mcU(G)$ and  transforms the coaction $(\mcU(G^\theta_\mbt)\fps,\alpha_h)$ of $(\mcU(G)\fps,\Delta_h)$ into a coaction $(\mcU(G_\phi)\fps,\tilde\alpha_h)$ of $(\mcU(G)\fps,\tilde\Delta_h)$. As $[z_\theta\otimes z_\theta,r]=0$, the latter coaction satisfies the assumptions of Theorem~\ref{thm:assoc-for-prescribed-classical-limit}. Hence this coaction is a twisting of the quasi-coaction $(\mcU(G_\phi)\fps,\Delta,\Psi_{\KZ,s;\mu})$ of $(\mcU(G)\fps,\Delta,\Phi_\KZ)$ for uniquely determined $s\in\R$ and $\mu\in h\R\fps$, with $s$ determined by
\[
\sin\Bigl(\frac{\pi \phi}{2}\Bigr) = \tanh\Bigl(\frac{\pi s}{2}\Bigr).
\]

Applying $(\Ad z_\theta)^{-1}$, we conclude that our original coaction is obtained by twisting from the quasi-coaction $(\mcU(G^\theta_\mbt)\fps,\Delta,\Psi_{\KZ,s;\mu})$ of $(\mcU(G)\fps,\Delta,\Phi_\KZ)$, and the pair $(s,\mu)$ is the only one with this property.

It remains to verify the formulas for $s$ in the formulation of the theorem.

In the S-type case, using $s^{(0)}_o=i c$ we can write $z_\theta(\alpha_o)\kappa_o c=\tan(\frac{\pi \phi}{2})$.
We also have $z_\theta(\alpha_o) \kappa_o = \pm 1$ by Lemma \ref{lem:z-theta-alpha-o-expr}.
We thus obtain $\sin(\hlf{\pi \phi})=\pm c(1+c^2)^{-1/2}$, or equivalently
\[
\hlf{\pi s} = \pm \frac{1}{2}\log\mathopen{}\left( \frac{1 + c (1 + c^2)^{-1/2}}{1 - c (1 + c^2)^{-1/2}} \right) = \pm \log\bigl((1+c^2)^{1/2}+c\bigr),
\]
with the $\pm$ being equal to the sign of $c^{-1} \sin(\hlf{\pi \phi}) = z_\theta(\alpha_o) \kappa_o \cos(\hlf{\pi \phi})$.
This is equal to the sign of $\kappa_o (Z^\mbt_\theta,X_{\alpha_o})_\mfg$ because $(z_\theta(\alpha_o)\cos(\hlf{\pi \phi}))^{-1} (Z^\mbt_\theta,X_{\alpha_o})_\mfg > 0$.

In the C-type case, writing $c^{(0)}_o=c > 0$, we have
\begin{equation*}
\sin\Bigl(\hlf{\pi \phi}\Bigr) = \frac{c^2-1}{c^2+1} = \frac{c-c^{-1}}{c+c^{-1}},
\end{equation*}
hence $\hlf{\pi s} = \log c$.
\ep

\begin{Rem}\label{rem:unitarity}
Throughout the paper we made a number of statements about unitarity. We used them in the proof of Theorem~\ref{thm:qbialg-compar-KZ-LK-Hermitian} to make sure that for $\mbt\in\mcT^*$ we get $s\in\R$ and $\mu\in h\R\fps$. It follows that $\Psi_{\KZ,s;\mu}$ is unitary, and once we know this, a standard argument based on polar decomposition shows that if $\pi\colon U_h(\mfg)\to \mcU(G)\fps$ and $\alpha_h$  are $*$-preserving, then the twisting can be done by unitary elements. The same is true for Theorem~\ref{thm:qbialg-compar-KZ-LK-non-Hermitian}.
\end{Rem}

The parameter $\mu$ can in principle be determined by comparing the $K$-matrices using the next theorem.
We will do this in detail in the type $\mathrm{AIII}$ case in Section~\ref{ssec:aiii}.

\begin{Thm}\label{compar-rib-tw-br-KZ-Hermitian}
The twisting provided by Theorem \ref{thm:qbialg-compar-KZ-LK-non-Hermitian} establishes a one-to-one correspondence between the following data:
\begin{itemize}
\item the ribbon braids for the quasi-coaction $(\mcU(G^\theta_\mbt)\fpser, \Delta, \Psi_{\KZ,s;\mu})$ of the quasi-triangular multiplier quasi-bialgebra $(\mcU(G)\fps,\Delta,\Phi_\KZ,\mcR_\KZ)$, given by
\begin{equation}\label{eq:braid-Hermitian}
\mcE''_{\KZ,s;\mu} g_1 = \exp\Big(-h(2t^{\mfk_\mbt}_{01}+C^{\mfk_\mbt}_1)+\pi(1-is-i \mu)(Z^\mbt_\theta)_1\Big)g_1 \quad (g\in Z(U)),
\end{equation}
\item the Balagovi\'{c}--Kolb ribbon braids~$\msE^\mbt$ (or their images $\mcE^\mbt_h$) for the coideal $U^\mbt_h(\mfg^\theta)$ of the quasi-triangular bialgebra $(U_h(\mfg), \Delta, \msR)$, for different choices of $\gamma$ satisfying \eqref{eq:gamma-cond}.
\end{itemize}
Under this correspondence, we have $\mcE^\mbt_h = 1 \otimes \exp(\pi(1-is)Z^\mbt_\theta)g \bmod h$.
\end{Thm}

\bp
This is proved similarly to Theorem~\ref{thm:compar-rib-tw-br-KZ-LK-non-Hermitian}, but now using the classification result from Theorem~\ref{thm:uniq-tw-br-herm-case}, applied to $\sigma=\exp(\pi\ad Z^\mbt_\theta)$, and the fact that the multiplication by $1\otimes \exp(\pi Z^\mbt_\theta)$ on the right gives a one-to-one correspondence between the ribbon $\sigma$-braids and the ribbon braids.
\ep

Analogous results hold for generic $\mbt\in\mcT^*_\C$.
More precisely, we have to exclude a countable set of values of $s^{(0)}_o$ (S-type) and $c^{(0)}_o$ (C-type) for the distinguished roots to be sure that a multiplier algebra model for $U^\mbt_h(\mfg^\theta)$  exists, see Remark~\ref{rem:multiplier-generic}.
We also have to make sure that $\Psi_{\KZ,s;\mu}$ is well-defined, which means that~$s$ should be outside a set $A$ satisfying $i(1+2\Z) \subset A \subset i\Q^\times$.

\begin{Prop}
In the S-type case, for generic $\mbt \in \mcT^*_\C$, the coaction $(\mcU(G^\theta_\mbt)\fps,\alpha_h)$ of $(\mcU(G)\fps,\Delta_h)$ is obtained by twisting from the quasi-coaction $(\mcU(G^\theta_\mbt)\fps,\Delta,\Psi_{\KZ,s;\mu})$ of $(\mcU(G)\fps,\Delta,\Phi_\KZ)$ for $s \in \C$ satisfying $e^{\pi s} = \bigl((1+c^2)^{1/2}+c\bigr)^2$, with $c = -i s_o^{(0)}$, and a uniquely determined $\mu \in h \C\fpser$, where the square root $(1+c^2)^{1/2}$ is chosen such that $(1+c^2)^{1/2} \kappa_o (Z^\mbt_\theta,X_{\alpha_o})_\mfg>0$.
In the C-type case, the same holds for $s$ satisfying $e^{\pi s} = c^2$, with $c = c_o^{(0)}$, where $\alpha_o$ is the unique distinguished root such that $-i\alpha_o(\tilde Z^\mbt_\theta)+i\alpha_{\tau_\theta(o)}(\tilde Z^\mbt_\theta)>0$. Such a twisting establishes a one-to-one correspondence between the ribbon braids~\eqref{eq:braid-Hermitian}
and the Balagovi\'{c}--Kolb ribbon braids.
\end{Prop}

\bp
The proof is essentially identical to that of Theorems~\ref{thm:qbialg-compar-KZ-LK-Hermitian} and~\ref{compar-rib-tw-br-KZ-Hermitian}. Let us only explain where the condition $(1+c^2)^{1/2} \kappa_o (Z^\mbt_\theta,X_{\alpha_o})_\mfg>0$ in the S-type case comes from.

Recall that $(1+ c^2)^{-1} = \cos^2(\frac{\pi \phi}{2})$.
We want to choose the square root $(1+c^2)^{1/2}$ so that $\sin(\frac{\pi \phi}{2}) = c (1+ c^2)^{-1/2}$ holds.
Then we obtain
\[
\frac{e^{\pi s} - 1}{e^{\pi s} + 1} = \tanh\Bigl(\hlf{\pi s}\Bigr) = c (1+ c^2)^{-1/2} = \frac{\bigl((1+c^2)^{1/2}+c\bigr)^2 - 1}{\bigl((1+c^2)^{1/2}+c\bigr)^2 + 1},
\]
which gives the asserted formula for $e^{\pi s}$.
From the proof of Theorem \ref{thm:qbialg-compar-KZ-LK-Hermitian}, we see that the desired choice is given by
\[
(1 + c^2)^{1/2} = \frac{1}{z_\theta(\alpha_o) \kappa_o \cos(\frac{\pi \phi}{2})}.
\]
Then we have
\[
(1+c^2)^{1/2} \kappa_o (Z^\mbt_\theta,X_{\alpha_o})_\mfg = \frac{(Z^\mbt_\theta,X_{\alpha_o})_\mfg}{z_\theta(\alpha_o)\cos(\frac{\pi \phi}{2})} > 0,
\]
hence the condition in the statement of the theorem.
\ep

\subsection{A Kohno--Drinfeld type theorem}
\label{sec:compar-br-gr-rep}

The above results allow us to compare certain representations of type $\mathrm{B}$ braid groups.

Recall that the braid group $\Gamma_n$ of type $\mathrm{B}_n$ is generated by elements $\rho_1$, $\sigma_1,\dots,\sigma_{n-1}$ subject to the following relations:
\begin{align*}
\sigma_i\sigma_j&=\sigma_j\sigma_i\quad(|i-j|>1),& \sigma_i\sigma_j\sigma_i&=\sigma_j\sigma_i\sigma_j\quad (|i-j|=1),\\
\rho_1\sigma_i&=\sigma_i\rho_1\quad (i>1),& \rho_1\sigma_1\rho_1\sigma_1&=\sigma_1\rho_1\sigma_1\rho_1.
\end{align*}
This is the subgroup of the usual (type $\mathrm{A}_{n}$) braid group on $n+1$ strands consisting of the braids with the first strand fixed.

\begin{figure}[ht]
\begin{subfigure}[b]{.3\linewidth}
\centering
\begin{tikzpicture}
\begin{knot}[
%  draft mode=crossings,
]
\stwbr{0}
\strand[only when rendering/.style={dashed}] (0.2,0) -- (0.2,2);
\strand[] (2,0) -- (2,2);
\strand[] (3,0) -- (3,2);
\flipcrossings{1}
\end{knot}
\end{tikzpicture}
\caption{$\rho_1$}
\end{subfigure}
\begin{subfigure}[b]{.3\linewidth}
\centering
\begin{tikzpicture}
\begin{knot}[
%  draft mode=crossings,
]
\strand[only when rendering/.style={dashed}] (0.2,0) -- (0.2,2);
\braidgen{1}{0.5};
\strand[] (1,0) -- (1,0.5);
\strand[] (2,0) -- (2,0.5);
\strand[] (1,1.5) -- (1,2);
\strand[] (2,1.5) -- (2,2);
\strand[] (3,0) -- (3,2);
\flipcrossings{1}
\end{knot}
\end{tikzpicture}
\caption{$\sigma_1$}
\end{subfigure}
\begin{subfigure}[b]{.3\linewidth}
\centering
\begin{tikzpicture}
\begin{knot}[
%  draft mode=crossings,
]
\strand[only when rendering/.style={dashed}] (0.2,0) -- (0.2,2);
\braidgen{2}{0.5};
\strand[] (2,0) -- (2,0.5);
\strand[] (3,0) -- (3,0.5);
\strand[] (2,1.5) -- (2,2);
\strand[] (3,1.5) -- (3,2);
\strand[] (1,0) -- (1,2);
\flipcrossings{1}
\end{knot}
\end{tikzpicture}
\caption{$\sigma_2$}
\end{subfigure}
\caption{Generators of $\Gamma_3$}
\end{figure}
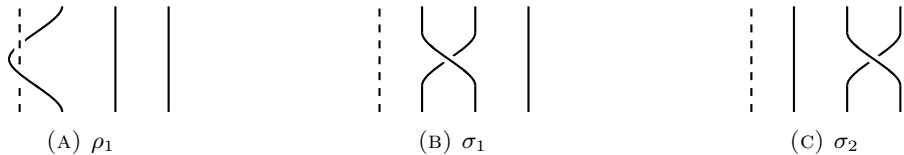

Assume we have a quasi-coaction $(B,\alpha,\Psi)$ of a  quasi-bialgebra $(A,\Delta,\Phi)$, and a ribbon braid $\mcE\in B\otimes A$ with respect to an $R$-matrix $\mcR\in A\otimes A$.
(As before, we will actually use the corresponding variants for multiplier algebras.)
Consider a (left) $B$-module $V$ and an $A$-module $W$. Then we get a representation of $\Gamma_n$ on $V\otimes W^{\otimes n}$, with $\rho_1$ acting by $\mcE$ on the zeroth and first factors and $\sigma_i$ acting by the braiding $\Sigma\mcR$ on the $i$-th and $(i+1)$-st factors, where $\Sigma$ denotes the flip.
More precisely, we have to fix a parenthesization on $V\otimes W^{\otimes n}$ and take into account the associativity morphisms, but different choices lead to equivalent representations by the standard coherence argument.
For example, for $n=3$ we can take
\[
((V\otimes W)\otimes W)\otimes W,
\]
and then the representation is defined by
\begin{align*}
\rho_1 &\mapsto \mcE_{0,1},&
\sigma_1 &\mapsto \Psi^{-1}_{0,1,2}(\Sigma\mcR)_{1,2}\Psi_{0,1,2},&
\sigma_2 &\mapsto \Psi^{-1}_{01,2,3}(\Sigma\mcR)_{2,3}\Psi_{01,2,3}.
\end{align*}

\smallskip

Let us first consider the non-Hermitian case.
Since the involutive automorphisms of quasi-bialgebras are nontrivial, we first need to pass to crossed products, as in the proof of Proposition \ref{prop:twist-of-twist-braids}.

On the side of $q$-deformations, we take the Hopf algebra $U_h(\mfg)\rtimes_{\tau_\theta\tau_0}\Z/2\Z$ and consider the ribbon braid $\msE(1\otimes\lambda_{\tau_\theta\tau_0})$ for a Balagovi\'{c}--Kolb ribbon twist-braid $\msE$ (defined by $\gamma$ satisfying \eqref{eq:gamma-cond}).

Take any $U^\theta_h(\mfg)$-module $V$ and $(U_h(\mfg)\rtimes_{\tau_\theta\tau_0}\Z/2\Z)$-module $W$ that are finitely generated and free as $\C\fps$-modules. Note that a $(U_h(\mfg)\rtimes_{\tau_\theta\tau_0}\Z/2\Z)$-module is the same as a $U_h(\mfg)$-module plus a $\C\fps$-linear isomorphism $u\colon W\to W$ such that $u^2=1$ and $u a=(\tau_\theta\tau_0)(a)u$ for all $a\in U_h(\mfg)$. We then get a representation of $\Gamma_n$ as described above from $\msE$ and $\msR$.

On the side of cyclotomic KZ-equations, we can start with finite dimensional representations of $U(\mfg^\theta)$ and $U(\mfg)\rtimes_{\tau_\theta\tau_0}\Z/2\Z$ on $\tilde V$ and $\tilde W$. The quasi-coaction $(U(\mfg^\theta)\fps,\Delta,\Psi_\KZ)$ of $(U(\mfg)\fps\rtimes_{\tau_\theta\tau_0}\Z/2\Z,\Delta,\Phi_\KZ)$ together with $\mcR_\KZ = e^{-h t}$ and a ribbon $\tau_\theta\tau_0$-braid $\mcE'_\KZ g_1$~\eqref{eq:braid-non-Hermitian} define a representation of $\Gamma_n$ on $(\tilde V\otimes\tilde W^{\otimes n})\fps$.

Then Theorems \ref{thm:qbialg-compar-KZ-LK-non-Hermitian} and \ref{thm:compar-rib-tw-br-KZ-LK-non-Hermitian} and Remark~\ref{rem:non-Hermitian-UEA-extension} imply the following.

\begin{Thm}\label{thm:Kohno-Drinfeld-non-Herm}
Let $\mfu^\theta < \mfu$ be a non-Hermitian symmetric pair, with $\theta$ in Satake form.
Let $V$ be a $U^\theta_h(\mfg)$-module, and $W$ be a $(U_h(\mfg)\rtimes_{\tau_\theta\tau_0}\Z/2\Z)$-module, that are finitely generated and free as $\C\fps$-modules.
Then the representation of $\Gamma_n$ on $V\otimes_{\C\fps}W^{\otimes n}$ defined by $\msE$ and $\Sigma \msR$ is equivalent to the one on $((V/h V)\otimes(W/h W)^{\otimes n})\fps$ defined by $(\mcE'_\KZ g_1, \Psi_{\KZ}, \Sigma \mcR_\KZ, \Phi_{\KZ})$ for the choice of $g$ satisfying $1 \otimes (z m_X m_0 g) = \msE \bmod h$.
\end{Thm}

In the Hermitian case, we can do similar constructions with the following modifications.
First, as $\tau_\theta\tau_0=\id$, we don't have to take crossed products.
Thus, given $\mbt\in\mcT^*$, a Balagovi\'{c}--Kolb ribbon braid $\msE^\mbt$ for $U^\mbt_h(\mfg^\theta)$ and $\Sigma \mcR$ defines a representation of $\Gamma_n$ on $V\otimes_{\C\fps}W^{\otimes n}$.
On the side of KZ-equations, by our construction of $\Psi_{\KZ,s;\mu}$, we can only consider $U(\mfg^\theta_\mbt)$-modules $\tilde V$ that can be integrated to representations of $G^\theta_\mbt$, or equivalently, that are direct summands of finite dimensional $U(\mfg)$-modules.
We use the ribbon braid $\mcE''_{\KZ,s;\mu}$ from \eqref{eq:braid-Hermitian}.

As a consequence of Theorems~\ref{thm:qbialg-compar-KZ-LK-Hermitian} and~\ref{compar-rib-tw-br-KZ-Hermitian}, we get the following result.

\begin{Thm}\label{thm:Kohno-Drinfeld-Herm}
Let $\mfu^\theta < \mfu$ be a Hermitian symmetric pair, with $\theta$ in Satake form, and $\mbt\in\mcT^*$.
Let~$V$ be a $U^\mbt_h(\mfg^\theta)$-module and $W$ be a $U_h(\mfg)$-module that are finitely generated and free as $\C\fps$-modules, and assume also that $V$ is a direct summand of a $U_h(\mfg)$-module with the same property.
Then the representation of $\Gamma_n$ on $V\otimes_{\C\fps}W^{\otimes n}$ defined by $\msE^\mbt$ and $\msR$ is equivalent to the representation on $((V/h V)\otimes(W/h W)^{\otimes n})\fps$ defined by $(\mcE''_{\KZ,s;\mu} g_1,\Psi_{\KZ,s;\mu}, \Sigma \mcR_\KZ, \Phi_\KZ)$, for the subgroup $G^\theta_\mbt<G$ and parameters $(s,\mu)$ from Theorem~\ref{thm:qbialg-compar-KZ-LK-Hermitian}, for the choice of $g \in Z(U)$ satisfying $1 \otimes (\exp(\pi i (1-i s) Z^\mbt_\theta) g) = \msE^\mbt \bmod h$.
\end{Thm}

\begin{Rem}
Since the subgroups $G^\theta_\mbt$ are conjugate to $G^\theta$, we could equally well consider the KZ-equations only for $G^\theta<G$.
We do not do this as the extra choice of conjugator will affect the correspondence $\msE^\mbt = \mcE_\KZ \bmod h$.
\end{Rem}

As a corollary we can also get a version of Theorem~\ref{thm:Kohno-Drinfeld-non-Herm} in the analytic setting. We will prove one such result and then discuss how it can be generalized.

We can define the algebras $U_q(\mfg)$ and $U^\theta_q(\mfg)$ for $q\in\C^\times$ not a nontrivial root of unity.
Furthermore, as has been noted in \citelist{\cite{MR3943480}\cite{MR4070299}}, the constructions of a Balagovi\'{c}--Kolb ribbon twist-braid~$\msE$, the associators $\Psi_{\KZ,s}$ and so on make sense in this setting.
We can therefore consider two types of finite dimensional representations of $\Gamma_n$, defined by $\msE$ and monodromy of KZ-equations. To compare such representations we need a way to associate a representation of~$U^\theta_q(\mfg)$ to a representation of $U(\mfg^\theta)$. To simplify matters let us consider only representations obtained by restriction from representations of~$U_q(\mfg)$ and~$U(\mfg)$. The representation theories of $U_q(\mfg)$ and $U(\mfg)$ are well-understood, so for any finite dimensional $U(\mfg)$-module $V$ we have its quantum analogue $V_q$. This correspondence extends also to representations of the crossed products $U(\mfg)\rtimes_{\tau_\theta\tau_0}\Z/2\Z$ and $U_q(\mfg)\rtimes_{\tau_\theta\tau_0}\Z/2\Z$.

\begin{Cor}\label{cor:analytic-DK}
Take $q>0$, and assume that $\mfu^\theta<\mfu$ is a non-Hermitian symmetric pair.
Consider a finite dimensional $U(\mfg)$-module~$V$ and a finite dimensional $(U(\mfg)\rtimes_{\tau_\theta\tau_0}\Z/2\Z)$-module $W$, and view $V$ as a $U(\mfg^\theta)$-module.
Then the representation of $\Gamma_n$ on $V_q\otimes W_q^{\otimes n}$ defined by a Balagovi\'{c}--Kolb ribbon $\tau_\theta\tau_0$-braid for $U^\theta_q(\mfg)$ and the universal $R$-matrix $\msR$ for $U_q(\mfg)$, corresponding to $h = \log q$, is equivalent to the representation on $V\otimes W^{\otimes n}$ defined by $(\mcE'_\KZ g_1, \Psi_{\KZ}, \Sigma \mcR_\KZ, \Phi_{\KZ})$ given by monodromy of the cyclotomic KZ-equations for the subgroup $G^\theta<G$ and some choice of $g \in Z(U^\theta)$.
\end{Cor}

\bp
We may assume that $V$ and $W$ are equipped with Hermitian scalar products such that they give rise to unitary representations of $U$. In a similar way, the assumption $q>0$ implies that $U_q(\mfg)$ is a $*$-algebra and its representations on $V_q$ and $W_q$ can be turned into $*$-representations.

Theorem~\ref{thm:compar-rib-tw-br-KZ-LK-non-Hermitian} gives us a bijection between the Balagovi\'{c}--Kolb ribbon $\tau_\theta\tau_0$-braids for $U^\theta_h(\mfg)$ and the ribbon $\tau_\theta\tau_0$-braids~\eqref{eq:braid-non-Hermitian}. By specialization this gives us a bijection also in the analytic setting, but a priori it is not given by any formula similar to~\eqref{eq:braid-twist}, as it is not clear when $\mcG$ can be specialized.

Now, using this bijection, it is convenient to extend the representations of $\Gamma_n$ to $\Gamma_n\times Z(U)$, with $Z(U)$ acting on the first factors $W_q$ and $W$ of $V_q\otimes W_q^{\otimes n}$ and $V\otimes W^{\otimes n}$, and prove a formally stronger statement that these representations of $\Gamma_n\times Z(U)$ are equivalent. The representations have the same character, since they are obtained by specialization from the formal case and in that case the representations are equivalent. Therefore it suffices to prove that the representations are completely reducible. For this, in turn, it suffices to show that in both cases the operators of the representations span $*$-algebras.

Observe in general that in the presence of a $*$-involution, if we have a quasi-coaction $(B,\alpha,\Psi)$ of a  quasi-bialgebra $(A,\Delta,\Phi)$ and a ribbon braid $\mcE$ with respect to an $R$-matrix $\mcR$, with unitary $\Psi$ and $\Phi$ and the $R$-matrix satisfying $\mcR^*=\mcR_{21}$, then $\mcE^*$ is also a ribbon braid. Indeed, analogues of identities~\eqref{eq:rib-sig-tw-0} and \eqref{eq:rib-sig-tw-1} for $\mcE^*$ are obtained immediately by taking the adjoints. For \eqref{eq:rib-sig-tw-2}, we in addition have to conjugate by $\mcR_{12}$ and then flip the last two tensor factors.

Since in the formal setting we have a complete classification of ribbon twist-braids, we conclude that every Balagovi\'{c}--Kolb ribbon $\tau_\theta\tau_0$-braid  for $U^\theta_h(\mfg)$ and every ribbon $\tau_\theta\tau_0$-braid~\eqref{eq:braid-non-Hermitian}, being multiplied by $1\otimes\lambda_{\tau_\theta\tau_0}$ on the right, has the property that it coincides with its adjoint up to a factor $1\otimes g$ ($g\in Z(U)$). (For the twist-braids~\eqref{eq:braid-non-Hermitian} this is also not difficult to see by definition, and for the Balagovi\'{c}--Kolb's ones this can be checked by an explicit computation as well~\cite{MR4070299}.) Hence the same is true in the analytic setting, which implies the desired property of the representations.
\ep

\begin{Rem}
Corollary~\ref{cor:analytic-DK} remains true for generic $q\in\C$.
Briefly, this can be proved by viewing both representations as defined over a field of meromorphic functions in $q$.
These representations have the same character by Theorem~\ref{thm:Kohno-Drinfeld-non-Herm}.
They can also be shown to be completely reducible, essentially because everything is determined by restriction to $q>0$, and for every such $q$ the representations are completely reducible.
Hence they are equivalent, and then by specialization we get an equivalence for generic values of $q$.
\end{Rem}

\begin{Rem}
In the Hermitian case, for $q>0$, we can define an analogue of the parameter set $\mcT^*$ for which $U^\mbt_q(\mfg^\theta)$ are $*$-coideals, see~\cite{MR3943480}.
Then the proof of the above corollary still works for such~$\mbt$, but with a caveat.
Assume $\mbt$ is obtained by specialization from a parameter in our set $\mcT^*$.
Then to be able to use Theorem~\ref{thm:Kohno-Drinfeld-Herm}, or even formulate the result in the analytic setting, we need $\mu$ provided by Theorem~\ref{thm:qbialg-compar-KZ-LK-Hermitian} to be specializable.
Assuming we have an explicit formula for $\mu$ as a function of $\mbt$ that can be specialized, this can be further generalized to generic $q \in \C$.
In the type $\mathrm{AIII}$ case analyzed below we see that this is indeed the case, and it is natural to expect that the same it true in all other cases.
\end{Rem}

\subsection{Example: AIII case}\label{ssec:aiii}

In this section we look in detail at the AIII symmetric pairs, that is, the pairs $\mfs(\mfu_p \oplus \mfu_{N-p}) < \mfsu_N$ for $0 < p \le N/2$ and $N\ge2$.

Thus, $\mfu=\mfsu_N$, $\mfg=\mfsl_N(\C)$. The normalized invariant form is $(X,Y)_\mfg=\Tr(X Y)$. As the Cartan subalgebra $\mfh$ we take the diagonal matrices with trace zero. Let $e_{ij}$ be the matrix units of~$M_N(\C)$. Define $L_i\in\mfh^*$ by $L_i(\sum_j a_j e_{jj})=a_i$. As a system of simple roots and generators of $\mfg$ we take
\begin{align*}
\Pi &= \{\alpha_i=L_i-L_{i+1}\}_{1\le i\le N-1},&
H_i &= e_{ii}-e_{i+1,i+1},&
X_{\alpha_i} &= e_{i,i+1},&
X_{-\alpha_i} &= e_{i+1,i}.
\end{align*}
Note that
\begin{align}\label{eq:scalar-product}
(L_i,L_i) &= 1-\frac{1}{N},&
(L_i,L_j) &= -\frac{1}{N}\quad (i\ne j).
\end{align}

Define
\begin{align*}
Z_\nu &= i\diag\Big(\underbrace{1-\frac{p}{N},\dots,1-\frac{p}{N}}_p,\underbrace{-\frac{p}{N},\dots,-\frac{p}{N}}_{N-p}\Big),&
\nu &= \Ad \exp(\pi Z_\nu).
\end{align*}
Then $\mfu^\nu=\mfs(\mfu_p \oplus \mfu_{N-p})$ and the pair $(\nu,Z_\nu)$ is as in Section~\ref{ssec:Cartan}.
We will write $\mfk$ for $\mfu^\nu$. The unique noncompact simple root is $\alpha_p$.

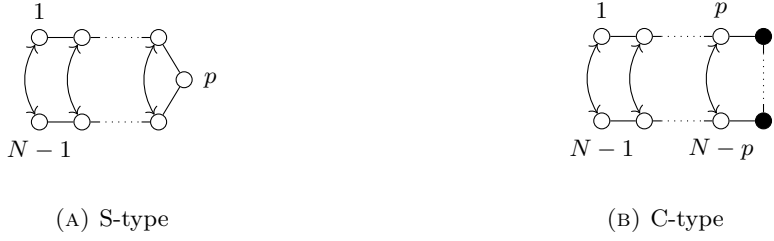
\begin{figure}[ht]
\begin{subfigure}[b]{.45\linewidth}
\centering\begin{tikzpicture}[scale=0.8,baseline,every node/.style={circle,draw,inner sep=2pt,minimum size=6pt}]
\node[label=above:{\small $1$}] (v1) at (0,0) {};
\node (v2) at (1,0) {};
\node (v3) at (2.8,0) {};
\node (v4) at (2.8,-2) {};
\node (v5) at (1,-2) {};
\node[label={[label distance=-8pt]below:{\small $N-1$}}] (v6) at (0,-2) {};
\node[label=right:{\small $p$}] (v7) at (3.4,-1) {};
\draw (v1) -- (v2) -- (1.4,0);
\draw[dotted] (1.4,0) -- (2.4,0);
\draw (2.4,0) -- (v3) -- (v7) -- (v4) -- (2.4,-2);
\draw[dotted] (1.4,-2) -- (2.4,-2);
\draw (1.4,-2) -- (v5) -- (v6);
\draw[<->] (v1) edge[bend right] (v6);
\draw[<->] (v2) edge[bend right] (v5);
\draw[<->] (v3) edge[bend right] (v4);
\end{tikzpicture}
\caption{S-type}
\end{subfigure}
\begin{subfigure}[b]{.45\linewidth}
\centering\begin{tikzpicture}[scale=0.8,baseline,every node/.style={circle,draw,inner sep=2pt,minimum size=6pt}]
\node[label=above:{\small $1$}] (v1) at (0,0) {};
\node (v2) at (1,0) {};
\node[label=above:{\small $p$}] (v3) at (2.8,0) {};
\node[label={[label distance=-8pt]below:{\small $N-p$}}] (v4) at (2.8,-2) {};
\node (v5) at (1,-2) {};
\node[label={[label distance=-8pt]below:{\small $N-1$}}] (v6) at (0,-2) {};
\node[fill] (bv1) at (3.8,0) {};
\node[fill] (bv2) at (3.8,-2) {};
\draw (v1) -- (v2) -- (1.4,0);
\draw[dotted] (1.4,0) -- (2.4,0);
\draw (2.4,0) -- (v3) -- (bv1) -- (3.8,-0.4);
\draw[dotted] (3.8,-0.4) -- (3.8,-1.6);
\draw (2.4,-2) -- (v4) -- (bv2) -- (3.8,-1.6);
\draw[dotted] (1.4,-2) -- (2.4,-2);
\draw (1.4,-2) -- (v5) -- (v6);
\draw[<->] (v1) edge[bend right] (v6);
\draw[<->] (v2) edge[bend right] (v5);
\draw[<->] (v3) edge[bend right] (v4);
\end{tikzpicture}
\caption{C-type}
\label{fig:aiii-c-satake-diag}
\end{subfigure}
\caption{Satake diagrams for AIII symmetric pairs}
\end{figure}

The S-type case corresponds to $N = 2 p$. Then the distinguished simple root is~$\alpha_p$, $X=\emptyset$, and as an involution $\theta$ in Satake form we take
\begin{align*}
\theta &= \Ad m_0, & m_0 &= A_N=
\begin{pmatrix}
 & & & 1 \\
 & & -1 & \\
 & \iddots & & \\
 (-1)^{N-1} & & &
\end{pmatrix},
\end{align*}
so that $z=1$ in~\eqref{eq:factoriz-theta}. It is clear that $(\nu, Z_\nu)$ is associated with $\theta$ as described after Lemma~\ref{lem:noncompact-root}. For every $\mbt\in\mcT^*$ we fix a normalized element $Z^\mbt_\theta\in\mfz(\mfg^\theta_\mbt)$ by requiring $(Z^\mbt_\theta,X_{\alpha_p})_\mfg>0$.
For the standard case $\mbt=0$ we write $Z_\theta = Z^0_\theta \in \mfz(\mfg^\theta)$.
\smallskip

The C-type case corresponds to $0 < p < N/2$. In this case the distinguished simple roots are $\alpha_p$ and $\alpha_{N-p}$, $X=\{\alpha_{p+1},\dots,\alpha_{N-p-1}\}$. As an involution $\theta$ in Satake form we take
\begin{align*}
\theta &= \Ad z m_0 m_X, & z &= e^{\pi i p/N} \diag(\underbrace{(-1)^{p},\dots,(-1)^{p}}_p,-(-1)^{N-p},\dots,-(-1)^{N-p}).
\end{align*}
Note that
\begin{align*}
m_0m_X &=
\begin{pmatrix}
 &  & A_p\\
 & -(-1)^{N-p} I_{N - 2p} & \\
(-1)^{N-p} A_p &  &
\end{pmatrix},&
z m_0m_X &= e^{\pi i p/N}
\begin{pmatrix}
 & & -A_p^t\\
 & I_{N-2p} & \\
 -A_p & &
\end{pmatrix}.
\end{align*}
Again, $(\nu, Z_\nu)$ is the unique pair associated with $\theta$ for which $\alpha_p$ is a noncompact root.
For every $\mbt\in\mcT^*$ we fix a normalized element $Z^\mbt_\theta\in\mfz(\mfg^\theta_\mbt)$ by requiring $-i\alpha_p(\tilde Z^\mbt_\theta)+i\alpha_{N-p}(\tilde Z^\mbt_\theta)>0$.
Again for the standard case we write $Z_\theta = Z^0_\theta \in \mfz(\mfg^\theta)$.

\begin{Thm}\label{thm:AIII}
With the above choices, the parameters $s\in\R$ and $\mu\in h\R\fps$ associated with $\mbt\in\mcT^*$ according to Theorem \ref{thm:qbialg-compar-KZ-LK-Hermitian} are determined as follows (with $q=e^h$):
\begin{itemize}
\item{{\rm S-type:}} $\displaystyle s+\mu=\frac{2}{\pi}\log\Bigl(\Bigl(1-\frac{q(q+1)^2}{4}s_p^2\Bigr)^{1/2}-\frac{q^{1/2}(q+1)}{2}is_p\Bigr)$;
\item{{\rm C-type:}} $\displaystyle s+\mu=\frac{2}{\pi}\log c_p + \frac{h}{\pi}$.
\end{itemize}

In particular, the standard case $\mbt=0$ corresponds to $s+\mu=0$.
\end{Thm}

We will prove the theorem by comparing the eigenvalues of $K$-matrices. In the $\mathrm{AIII}$ case this is facilitated by the knowledge of solutions of the reflection equation~ \cite{MR1917138}.

As in Section~\ref{ssec:K-matrix}, we will work over the field $\C(q^{1/d})$ and then pass to $\C\fLauser$. The fundamental representation of $U_q(\mfg)$ on $V=\C(q^{1/d})^N$ is given by
\begin{align*}
\pi_V(E_i) &= q^{1/2} e_{i,i+1},&
\pi_V(F_i) &= q^{-1/2} e_{i+1,i},&
\pi_V(K_i) &=q e_{ii}+q^{-1}e_{i+1,i+1}.
\end{align*}
Writing $R = q^{-1/N} (\pi_V\otimes \pi_V)(\msR)$ for the universal $R$-matrix $\msR$, we have
\begin{equation*}
\label{eq:R-mat-vec-vec-pres}
R = \sum_{i,j} q^{-\delta_{ij}}e_{ii}\otimes e_{jj} + (q^{-1}-q)\sum_{i<j} e_{ij}\otimes e_{ji}.
\end{equation*}

Consider the $K$-matrix $\msK^\mbt=(\epsilon\otimes\id)(\msE^\mbt)$ of a Balagovi\'{c}--Kolb ribbon braid $\msE^\mbt$ for $\mbt\in\mcT^*$. Then $K^\mbt=\pi_V(\msK^\mbt)$ satisfies the reflection equation
\begin{equation*}
K^\mbt_1 \hat R_{12}K^\mbt_1\hat R_{12}=\hat R_{12}K^\mbt_1\hat R_{12}K^\mbt_1,
\end{equation*}
where $\hat R=\Sigma R$. The invertible solutions of this equation are classified in \cite{MR1917138} as

\begin{equation}\label{EqSolMud}
\begin{pmatrix}
 \lambda+\mu   &   & &  &   &  &   &  &y_1\\
 & \ddots    &&  &  &  &  &     \iddots    &   \\
 & &    \lambda + \mu& &   &  & y_{r} &  &  \\
&  & &\lambda &  &  & &  &  \\
&  &  && \ddots   &  &  &  & \\
 &  &   & &  & \lambda &  & &  \\
 &  &  y_{N-r+1}& &  &  & &  &  \\
 & \iddots &  &  & & & & & \\
y_{N}   &  &&   &  &  &  &
\end{pmatrix},
\end{equation}
where we have the general requirement $y_i y_{N-i+1} = -\lambda \mu \neq 0$. (This takes into account that the conventions in~\cite{MR1917138} are different. Our matrix $R$ corresponds to $R_{21}$ of \cite{MR1917138}, with $q$ replaced by~$q^{-1}$. Note also that the ground field in~\cite{MR1917138} is $\C$, but the arguments there are purely combinatorial and work for any field of characteristic zero.) We are interested only in the nonconstant solutions, since by Theorem~\ref{compar-rib-tw-br-KZ-Hermitian} the matrix $K^\mbt$ is conjugate (by some element $T \in M_N(\C)\fpser$ such that $T^{(0)} = I_N$) to the image of
\begin{equation}\label{eq:K-KZ}
\exp\bigl(-h C^{\mfk_\mbt}+\pi(1-is-i\mu)Z^\mbt_\theta \bigr)g
\end{equation}
in the fundamental representation of~$\mfsu_N$ for some $g\in Z(U)$, which is clearly nonconstant.

Next, let us make a choice of a Balagovi\'{c}--Kolb ribbon braid.
Recall that in the standard case the $K$-matrix $\msK$ and the ribbon braid $\msE$ are given by~\eqref{eq:k-mats} and~\eqref{eq:BK-braid}.
In our present Hermitian case $\tau_\theta\tau_0=\id$, and $\msK$ can be written as
\begin{equation}\label{eq:K-mats2}
\msK=\tilde\quasiK\xi'T^{-1}_{w_X}T^{-1}_{w_0},
\end{equation}
where $\tilde\quasiK=(\Ad K_{\omega_0})(\quasiK)$ and $\xi'=\xi K_{2\omega_0}$, see \cite{MR4070299}*{Section~4}. Furthermore, by \cite{MR4070299}*{Lemma 4.24}, the function $\gamma$ defining $\xi$ can be chosen so that $\xi'=\tau_\theta(z) C_\Theta$, where $C_\Theta$ acts on every vector of weight~$\omega$ by $q^{-(\omega^+,\omega^+)}$, with $\omega^+=\hlf1(\omega+\Theta(\omega))$. Therefore
\begin{equation}\label{eq:xi-prime}
\pi_V(\xi')e_i=\tau_\theta(z)(L_i)q^{-(L_i^+,L_i^+)}e_i=\overline{z(\Theta(L_i))}q^{-(L_i^+,L_i^+)}e_i,
\end{equation}
where we used~\eqref{eq:Theta-tau} and that $z(\alpha_i)=1$ for $i\in X$.
We will assume from now on that $\msK$ is defined using this particular $\xi'$.

The following formula, which will allow us to compute some matrix coefficients of $K=K^0=\pi_V(\msK)$, is probably well-known to experts.

\begin{Lem}
\label{lem:T-w-0-action-vec-rep}
We have
\begin{equation*}
\pi_V(T_{w_0}) = q^{\frac{N-1}{2}}A_N.
\end{equation*}
\end{Lem}

\bp
By definition (see, e.g., \cite{MR1359532}*{Section 8.6}), we have
\begin{align*}
T_{w_0} &= T_{[1]}T_{[2]}\cdots T_{[N-1]},&
T_{[k]} &= T_{k} T_{k-1}\cdots T_1,
\end{align*}
and the operators $\pi_V(T_i)$ are given by
\begin{align*}
\pi_V(T_i) e_j &= e_j\quad (j\neq i,i+1),&
\pi_V(T_i)e_i &= -q^{1/2}e_{i+1},&
\pi_V(T_i)e_{i+1} &= q^{1/2}e_i.
\end{align*}
This gives the result.
\ep

In the nonstandard case we have
\[
\msK^\mbt=(\chi_\mbt\otimes\id)(\msR_{21}(1\otimes\msK)\msR)
\]
for an appropriate character $\chi_\mbt\colon U^\theta_q(\mfg)\to\C\fLauser$, as described in Section~\ref{ssec:K-matrix}.

Write $T = (\id\otimes \pi_V)(\msR_{21}) \in U_q(\mfg) \otimes M_N(\C(q^{1/d}))$. Then $T$ is an upper triangular matrix with coefficients in $U_q(\mfb^-)$, and
\[
K^\mbt = (\chi_\mbt \otimes \id)(T(1\otimes K)T^*),
\]
where we remind that $h^*=h$.

\begin{Lem}\label{lem:K-N1}
We have
\[
K^\mbt_{N1}=\chi_\mbt\Bigl(K_1^{\frac{2-N}{N}}K_2^{\frac{4-N}{N}} \cdots K_{N-1}^{\frac{N-2}{N}}\Bigr)K_{N1}.
\]
\end{Lem}

\bp
By definition we have
\[
K^\mbt_{N1} =\chi_\mbt\Big(\sum_{i,j}T_{Ni}^{\phantom{*}}T_{1j}^*\Big) K_{ij}.
\]
Since $T$ is upper triangular, the contribution of $i \ne N$ is zero. The form of \eqref{EqSolMud}, and the fact that~$K$ is not scalar, further tell us that $K_{N j}=0$ for $j\ne 1$. Therefore
\[
K^\mbt_{N1} = \chi_\mbt(T_{NN}T_{11}^*)K_{N1}.
\]
Since $T_{NN}T_{11}^*=T_{NN}T_{11} = K_1^{\frac{2-N}{N}}K_2^{\frac{4-N}{N}} \cdots K_{N-1}^{\frac{N-2}{N}}$ by the usual factorization of $\msR$ (see, e.g.,~\cite{MR1321145}*{Section XVII.2}), this proves the lemma.
\ep

To get further information on the $K$-matrices we will use that $K^\mbt$ must commute with $\pi_V(U^\mbt_h(\mfg))$. We will treat the S-type and C-type cases separately.

\smallskip

In the S-type case, $\Theta(L_i)=L_{N-i+1}$. The coideal $U^\mbt_h(\mfg^\theta)$ is generated by $U_h(\mfh^\theta)$, the elements
\begin{align*}
B_i &= F_i-q^{-1}E_{N-i}K_i^{-1} \quad (i<p),&
B_p &=F_p-q^{-2}E_p K_p^{-1}+\frac{s_p(K_p^{-1}-1)}{q^{-1}-1},
\end{align*}
and their adjoints.

\begin{Prop}\label{prop:K-matrix-AIII-S}
In the S-type case, for every $\mbt\in\mcT^*$, we have
\begin{equation}\label{eq:K-mats-AIII-S}
K^\mbt =(-1)^{p-1} q^{\frac{1}{2p} -p}
\begin{pmatrix}
 q^{1/2}(q+1)s_p I_p & - A_p^t\\
 A_p & 0
\end{pmatrix}.
\end{equation}
\end{Prop}

\bp
It is easily seen that the nonconstant matrices~\eqref{EqSolMud} commuting with the generators of $U^\mbt_h(\mfg^\theta)$ are precisely of the form
\begin{equation*}
y
\begin{pmatrix}
 q^{1/2}(q+1)s_p I_p & - A_p^t\\
 A_p & 0
\end{pmatrix}
\end{equation*}
for $y\in\C\fLauser^\times$.
(Note that these solutions were also described in \cite{MR2565052}*{Section 5}.)

To determine $y$, we look at the matrix coefficient $K^\mbt_{N1}$. By Lemma~\ref{lem:K-N1} it is independent of $\mbt$, since $\chi_\mbt$ is trivial on the Cartan part in the S-type case. In the standard case we compute $K_{N1}$ using definition~\eqref{eq:K-mats2}.

By~\eqref{eq:xi-prime} and~\eqref{eq:scalar-product} we have $\pi_V(\xi')=q^{\frac{1}{2p}-\hlf1}$, while Lemma \ref{lem:T-w-0-action-vec-rep} describes the action of $T_{w_0}$.
Since~$\widetilde{\mathfrak{X}}$ lies in a completion of $U_q(\mfn^+)$ and has the weight zero component $1$, we conclude that
\[
K^\mbt_{N1}=K_{N1}^{\phantom{\mbt}}  = q^{\frac{1}{2p}-p}.
\]
Hence $y = (-1)^{p-1} q^{\frac{1}{2p} -p}$.
\ep

\bp[Proof of Theorem~\ref{thm:AIII}: S-type case]
Note that the matrix $K^\mbt$ given by~\eqref{eq:K-mats-AIII-S} leaves the two-dimensional spaces spanned by $e_i$ and $e_{2p-i+1}$ invariant. From this we see that it has eigenvalues
\[
x_\pm=(-1)^{p-1} q^{\frac{1}{2p} -p}\Bigl(\frac{q^{1/2}(q+1)}{2}s_p\pm i\Bigl(1-\frac{q(q+1)^2}{4}s_p^2\Bigr)^{1/2}\Bigr),
\]
each of multiplicity $p$.
(Recall that we are assuming $s_p \in i h\R\fpser$, so $\bigl(1-\frac{q(q+1)^2}{4}s_p^2\bigr)^{1/2}$ is well-defined as an element of $\R\fpser$.)

On the other hand, consider a $K$-matrix as in \eqref{eq:K-KZ} but using $(\mfk,Z_\nu)$ instead of $(\mfu^\theta_\mbt,Z^\mbt_\theta)$, and denote its image under the fundamental representation of $\mfsu_N$ by $M$.
(For the moment we leave the choice of $g \in Z(U)$ free.)
Then $K^\mbt$ is conjugate to $M$ by a formal matrix.
Let us compute the eigenvalues of $M$.

The Casimir operator $C^{\mfk}$ equals $C^{\mfsu_p \oplus \mfsu_p} - \Tr(Z_\nu^2)^{-1}Z_\nu^2$. Since the Casimir operator of $\mfsu_p$ acts as the scalar $\frac{p^2 - 1}{p}$ in the fundamental representation of $\mfsu_p$, it follows that $C^\mfk$ acts as $\frac{p^2 - 1}{p}+\frac{1}{2p}=p-\frac{1}{2p}$. Hence the eigenvalues of $M$ are
\[
y_{\pm}=\pm q^{\frac{1}{2p} -p} i e^{\pm\frac{\pi}{2}(s+\mu)}e^{\frac{2\pi i k}{2p}},
\]
each of multiplicity $p$, for some $0\le k\le 2p-1$.

It follows that $y_+$ coincides with $x_+$ or $x_-$. Since $s$ is real, by looking at the order zero terms we can conclude that this is possible only if $e^{\frac{\pi i k}{p}}=\pm1$ and
\[
e^{\frac{\pi}{2}(s+\mu)}=\Big(1-\frac{q(q+1)^2}{4}s_p^2\Big)^{1/2}\pm \frac{q^{1/2}(q+1)}{2}is_p.
\]
Furthermore, since we already know the formula for $s$ by Theorem~\ref{thm:qbialg-compar-KZ-LK-Hermitian}, we see that for $s^{(0)}_p\ne0$ the sign must be $-$ and $g\in Z(SU(N))$ is the scalar matrix $e^{\frac{\pi i k}{p}} = (-1)^{p-1}$.

It remains to handle the case $s^{(0)}_p = 0$, so that $\mfg^\theta_\mbt = \mfg^\theta$.
Then Theorem~\ref{thm:qbialg-compar-KZ-LK-Hermitian} implies that $s = 0$.
We first claim that
\[
Z_\theta = \frac{(-1)^p}{2} A_N.
\]
Since $\theta=\Ad A_N$, we have $A_N\in\mfz(\mfu^\theta)$, hence this formula must be true up to a sign. Then the requirement $(Z_\theta, X_p)_{\mfg} > 0$ forces this choice.

Now, write $M'$ for the image of \eqref{eq:K-KZ} under the fundamental representation of $\mfsu_N$.
From the above formula for $Z_\theta$, we see that $M'$ preserves the span of $e_i$ and $e_{2p-i+1}$ for each $i$, analogously to the situation for $K^\mbt$ observed above.
Restricting to the span of $e_1$ and $e_N$, we find
\[
M'(e_1 + i e_N) = (-1)^p i q^{\frac{1}{2p}-p} e^{(-1)^p \frac{\pi \mu}{2}} e^{\frac{\pi i k}{p}} (e_1 + i e_N),
\]
where $e^{\frac{\pi i k}{p}}$ is the effect of $g \in Z(U)$, and a similar formula for $e_1 - i e_N$ (which we do not use).

Now, we also know that $M'$ and $\pi_V(K^\mbt)$ are conjugate by a formal matrix with constant term $I_N$.
In particular $K^\mbt$ has eigenvectors which are deformations of $e_1 \pm i e_N$, with the same eigenvalues.

From \eqref{eq:K-mats-AIII-S}, the restriction of $K^\mbt$ to the span of $e_1$ and $e_N$ gives
\begin{align*}
K^\mbt \eta &= - q^{\frac{1}{2p} - p} \left(i \Bigl(1 - \frac{t'^2}{4}\Bigr) + \frac{t'}{2} \right) \eta,&\text{with}\qquad
\eta &= \biggl( \biggl( -\frac{t'}{2} - i \Bigl( 1 - \frac{t'^2}{4} \Bigr) \biggr) e_1 + e_N\biggr),
\end{align*}
where $t' = (-1)^p q^{\frac{1}{2}}(q+1) s_p \in i h \R\fpser$.
This eigenvector $\eta$ is a deformation of $-i(e_1 + i e_N)$, hence we obtain the equality of eigenvalues
\[
(-1)^p i q^{\frac{1}{2p}-p} e^{(-1)^p \frac{\pi \mu}{2}} e^{\frac{\pi i k}{p}} = - q^{\frac{1}{2p} - p} \left(i\Bigl( 1 - \frac{t'^2}{4} \Bigr) + \frac{t'}{2} \right),
\]
or equivalently,
\[
e^{(-1)^p \frac{\pi \mu}{2}} e^{\frac{\pi i k}{p}} = (-1)^{p-1} \left( \Bigl( 1 - \frac{t'^2}{4} \Bigr) - i \frac{t'}{2} \right).
\]

When $p$ is even, this implies $e^{\frac{\pi i k}{p}} = -1$ and the formula for $\mu$ follows by taking the logarithm.
When $p$ is odd, we first obtain $e^{\frac{\pi i k}{p}} = 1$, and then the formula for $\mu$ follows by taking the logarithm of inverses (note that $(\sqrt{1 + x^2} + x)(\sqrt{1 + x^2} - x) = 1$).
\ep

Next let us consider the C-type case. Then $\Theta(L_i)=L_i$ for $p+1\le i\le N-p$, and $\Theta(L_i)=L_{N-i+1}$ for all other~$i$.
The coideal $U^\mbt_h(\mfg^\theta)$ is generated by $U_h(\mfh^\theta)$, $U_q(\mfg_X)$, the elements
\begin{align*}
B_i&=F_i-q^{-1}T_{w_X}(E_{N-i})K_i^{-1}\quad (i<p),&
B_p&=F_p-c_p T_{w_X}(E_{N-p})K_p^{-1}
\end{align*}
and their adjoints. Similarly to Lemma~\ref{lem:T-w-0-action-vec-rep} we have
\begin{equation}\label{eq:Twx}
\pi_V(T_{w_X}) =
\begin{pmatrix}
I_p  &  &  \\
 & q^{\frac{N-1}{2}-p} A_{N-2p} &  \\
 & & I_p
\end{pmatrix}.
\end{equation}
It follows that
\begin{align*}
\pi_V(B_i) &= q^{-1/2} e_{i+1,i} -q^{-1/2} e_{N-i,N-i+1}\quad (i<p),&
\pi_V(B_{p}) &= q^{-1/2}e_{p+1,p} - q^{\frac{N}{2}-p}c_p e_{p+1,N-p+1}.
\end{align*}

In the next proposition we assume that the character $\chi_\mbt$ is defined using the unique homomorphism $P\to\R\fps^*$ with values in the power series with positive constant terms such that $\alpha_p\mapsto c_p^{-1}q^{-\hlf1}$ and $\alpha_i\mapsto1$ for $i\ne p$.

\begin{Prop}\label{prop:K-matrix-AIII-C}
In the C-type case, for every $\mbt\in\mcT^*$, we have
\begin{equation}
\label{eq:aiii-ctype-K-formula}
K^\mbt =
\begin{pmatrix}
 (\lambda+\mu) I_p   &   & -q^{\frac{N+1}{2} -p}c_p \lambda A_p^t\\
 & \lambda I_{N-2p}  & \\
 q^{-\frac{N+1}{2} +p}c_p^{-1}\mu A_p & &
\end{pmatrix},
\end{equation}
where $\displaystyle \lambda = e^{-\pi i \frac{p}{N}} q^{\frac{1}{N}-(N -p)-\frac{p}{N}}c_p^{-\frac{2p}{N}}$ and
$\displaystyle\mu = e^{\pi i \frac{N- p}{N}}q^{\frac{1}{N}-p +\frac{N-p}{N} }c_p^{\frac{2(N-p)}{N}}$.
\end{Prop}

\bp
Again, some elementary computations show that a nonconstant solution \eqref{EqSolMud} commutes with the generators of $\pi_V(U^\mbt_h(\mfg^\theta))$ if and only if it has the form~\eqref{eq:aiii-ctype-K-formula}, with no restrictions on $\lambda$ and $\mu$.

Consider first the standard case. Then $c_p=q^{-(\alpha_p^-,\alpha_p^-)}=q^{-1/2}$. By~\eqref{eq:xi-prime} and~\eqref{eq:scalar-product} we have
\[
\pi_V(\xi')e_N=(-1)^p e^{-\pi i\frac{p}{N}}q^{\frac{1}{N}-\hlf1}e_N,
\]
from which we deduce, similarly to the proof of Proposition~\ref{prop:K-matrix-AIII-S}, that
\[
K_{N1}= (-1)^p e^{-\pi i\frac{p}{N}} q^{\frac{1}{N}-\frac{N}{2}}.
\]
It follows that $\mu = -e^{-\pi i \frac{p}{N}} q^{\frac{1}{N} -p}$. In a similar way, using~\eqref{eq:Twx}, we compute
\[
\lambda = K_{p+1,p+1} = e^{-\pi i \frac{p}{N}} q^{\frac{1}{N}-(N -p)}.
\]

For general $\mbt\in\mcT^*$, by Lemma~\ref{lem:K-N1} we have
\[
K^\mbt_{N1}=c_p^{1-\frac{2p}{N}}q^{\hlf1-\frac{p}{N}}K_{N1}= (-1)^p e^{-\pi i\frac{p}{N}} q^{\hlf1-\frac{p}{N}+\frac{1}{N}-\frac{N}{2}}c_p^{1-\frac{2p}{N}},
\]
which gives the asserted formula for $\mu$. Similarly to the proof of Lemma~\ref{lem:K-N1} we also have
\[
\lambda=K^\mbt_{p+1,p+1}=\chi_\mbt\Big( \sum^{N-p}_{i=p+1} T_{p+1,i}^{\phantom{*}}T_{p+1,i}^*\Big)K_{ii}
=\chi_\mbt\Big( \sum^{N-p}_{i=p+1} T_{p+1,i}^{\phantom{*}}T_{p+1,i}^*\Big)K_{p+1,p+1}.
\]
Using again the usual factorization of the $R$-matrix we see that $T_{p+1,i}$ for $p+1<i\le N-p$ is an element of $U_q(\mfh)U_q(\mfn_X^-)$ of weight $L_i-L_{p+1}$, while
\[
T_{p+1,p+1}=K_1^{\frac{1}{N}} \cdots K_p^{\frac{p}{N}}K_{p+1}^{\frac{p+1}{N}-1} \cdots K_{N-1}^{\frac{N-1}{N}-1}.
\]
By definition of $\chi_\mbt$ we conclude that $\lambda$ differs from the standard case by the factor
\[
\chi_\mbt(T_{p+1,p+1}^2)=q_{\phantom{p}}^{-\frac{p}{N}}c_p^{-\frac{2p}{N}}.
\]
This gives the required formula for $\lambda$.
\ep

\bp[Proof of Theorem~\ref{thm:AIII}: C-type case]
Similarly to the S-type case, the matrix $K^\mbt$ given by~\eqref{eq:aiii-ctype-K-formula} has eigenvalues
\begin{align*}
- e^{-\pi i \frac{p}{N}}q^{\frac{1}{N}-p +\frac{N-p}{N} }c_p^{\frac{2(N-p)}{N}}, & &
e^{-\pi i \frac{p}{N}}q^{\frac{1}{N}-(N -p)-\frac{p}{N}}c_p^{-\frac{2p}{N}}
\end{align*}
of multiplicities $p$ and $N-p$, resp., while the image of \eqref{eq:K-KZ} in the fundamental representation of $\mfsu_{N}$ has eigenvalues
\begin{align*}
-e^{\frac{2\pi i k}{N}}e^{-\pi i \frac{p}{N}}q^{\frac{1}{N}-p}e^{\frac{N-p}{N}\pi(s+\mu)}, & &
e^{\frac{2\pi i k}{N}}e^{-\pi i \frac{p}{N}}q^{\frac{1}{N}-(N-p)}e^{-\frac{p}{N}\pi(s+\mu)}
\end{align*}
of multiplicities $p$ and $N-p$, resp., for some $0\le k\le N-1$. By looking at the order zero terms we conclude that $k=0$ and
\[
q^{-\frac{p}{N}}c_p^{-\frac{2p}{N}}=e^{-\frac{p}{N}\pi(s+\mu)},
\]
giving the formula for $s+\mu$ in the statement of the theorem.
\ep

\appendix

\section{Co-Hochschild cohomology}
\label{sec:appendix-dg}

Our goal is to prove an analogue of Corollary~\ref{cor:cohom-B-G-Gtheta} for universal enveloping algebras.
The result essentially follows from computations in~\cite{MR2233127}, but since we formally claim slightly more than what is stated there, we give a self-contained argument.

\smallskip

Let $V$ be a finite dimensional vector space over $\C$ (we could consider any field of characteristic zero), and $W\subset V$ be a subspace.
Viewing $V$ as an abelian Lie algebra, consider its universal enveloping algebra $(U(V),\Delta)$. As an algebra it is the symmetric algebra $\Sym(V)$. But we will mostly need only the coalgebra structure, in which case we write $\Sym^c(V)$. Consider the tensor algebra $T(\Sym^c(V))$ of the vector space $\Sym^c(V)$. We then make $\Sym^c(W)\otimes T(\Sym^c(V))$ into a cochain complex by defining
\[
d\colon \Sym^c(W)\otimes\Sym^c(V)^{\otimes n}\to \Sym^c(W)\otimes\Sym^c(V)^{\otimes (n+1)}
\]
by formula~\eqref{eq:cohoch-diff}, so
\[
d T = T_{01,2,\dots,n+1} - T_{0,1 2, \dots, n+1} + \dots + (-1)^n T_{0,1,\dots,n(n+1)} + (-1)^{n+1} T_{0,1,\dots,n}.
\]
Consider the linear map $p_{V/W}\colon\Sym^c(V)\to V/W$ obtained by composing the projection $\Sym^c(V)\to\Sym^1(V)=V$ with the quotient map $V\to V/W$. It extends to an algebra homomorphism
\[
T(\Sym^c(V))\to\medwedge V/W,
\]
which we continue to denote by $p_{V/W}$. Consider also the counit $\epsilon\colon\Sym^c(W)\to\C$.

\begin{Prop}\label{prop:co-Hoshchild-alg-model1}
The map $\epsilon\otimes p_{V/W}$ defines a quasi-isomorphism of $(\Sym^c(W)\otimes T(\Sym^c(V)),d)$ onto $(\bigwedge V/W,0)$.
\end{Prop}

The result is well-known for $W=0$, see, e.g., \cite{MR1321145}*{Theorem~XVIII.7.1}. We will deduce the proposition from this particular case using the formalism of \emph{twisting morphisms}.

Let $(A, m_A, d_A)$ be a differential graded algebra with product $m_A$ and cohomological differential $d_A \colon A^n \to A^{n+1}$. Let $(C, \Delta_C)$ be a coalgebra (concentrated in degree $0$).
A twisting morphism is a linear map $\alpha\colon C \to A^1$ satisfying
\[
d_A \alpha + \alpha \star \alpha = 0,
\]
where $\alpha \star \alpha = m_A (\alpha \otimes \alpha) \Delta_C$.
Then the degree $1$ map $\id \otimes d_A + d_\alpha$, with
\[
d_\alpha = (\id \otimes m_A) (\id \otimes \alpha \otimes \id) (\Delta_C \otimes \id),
\]
defines the structure of a cochain complex on the graded vector space $C \otimes A$.
We denote this complex by $C \otimes_\alpha A$.

We further assume that
\begin{itemize}
\item $A^n$ and $C$ have auxiliary gradings, called the \emph{weight gradings}; we write $A^{w: m, ch: n}$ for the weight $m$ part of $A^n$ and $C^{w: m}$ for the weight $m$ part of $C$;
\item $m_A$, $d_A$, $\Delta_C$ and $\alpha$ have degree $0$ for the weight grading;
\item $C$ and $A$ are nonnegatively graded with respect to both the weight degree and the cohomological degree.
\end{itemize}

In our application $A=T(\Sym^c(V))$, $C=\Sym^c(W)$ and the weight gradings are defined by declaring the elements of $\Sym^n(V)$ and $\Sym^n(W)$ to be of weight $n$.

\begin{Lem}[cf.~\cite{MR2954392}*{Lemma 2.1.5}]
\label{lem:comparision-lemma}
Let $A$ and $B$ be weight graded differential graded algebras as above and $f\colon A \to B$ be a quasi-isomorphism (i.e., a weight degree preserving homomorphism of dg algebras inducing an isomorphism in cohomology). Then for any weight graded coalgebra $C$ and any twisting morphism $\alpha\colon C\to A^1$ such that $\alpha(C^{w:0}) = 0$, the map $\id \otimes f\colon C \otimes_\alpha A \to C \otimes_\beta B$ is a quasi-isomorphism, where $\beta = f \alpha$.
\end{Lem}

\bp
The assumption $\alpha(C^{w:0}) = 0$ implies that $d_\alpha$ ``carries'' weight from $C$ to $A$. This can be used to build a spectral sequence.

Specifically, define a decreasing filtration $F_s$ of $C \otimes_\alpha A$ by setting
\[
F_s = \bigoplus_{m \ge s} C \otimes A^{w:m}.
\]
We clearly have $(\id\otimes d_A)(F_s)\subset F_s$. Since $(\id \otimes \alpha)\Delta_C(C^{w:n})$ belongs to $\bigoplus_{k\ge 1} C^{w:n-k} \otimes A^{w:k}$ by the assumption $\alpha(C^{w:0}) = 0$, we have $d_\alpha(F_s) \subset F_{s+1}$. Note also that, for each weight $n$, we have $F_0^{w:n} = (C \otimes_\alpha A)^{w:n}$ and $F_s^{w:n} = 0$ for $s > n$.
Hence the associated spectral sequence starting with
\[
E_0^{s, t} = F_s^{ch: s + t} / F_{s+1}^{ch: s + t} = C \otimes A^{w: s, ch: s + t}
\]
is convergent at each weight, and the $E_0$-differential $d_0\colon E_0^{s,t} \to E_0^{s,t+1}$ is just $\id \otimes d_A$.
Therefore
\[
E_1^{s,t} = C \otimes \rH^{s+t}(A^{w:s}).
\]

We can do the same construction for $B$. Then $\id \otimes f$ induces an isomorphism at the $E_1$-page.
Combined with the convergence of the spectral sequences, we obtain the assertion.
\ep

\bp[Proof of Proposition~\ref{prop:co-Hoshchild-alg-model1}]
Denote the complex we get for $W=0$ by $(A,d_A)$, so that $A=T(\Sym^c(V))$. For general $W$, let $C=\Sym^c(W)$ and define a linear map
$\alpha\colon C=\Sym^c(W)\to A^1=\Sym^c(V)$ as the zero map on $\C=\Sym^0(W)$ and the inclusion map $\Sym^n(W)\to\Sym^n(V)$ for $n\ge1$. Then $(\alpha\star\alpha)(1)=0$ and
\[
(\alpha\star\alpha)(x)=\Delta(x)-1\otimes x- x\otimes1\quad (x\in\Sym^n(W), n\ge1).
\]
It follows that $\alpha$ is a twisting morphism. We also have
\[
d_\alpha(x\otimes y)=\Delta(x)\otimes y-x\otimes1\otimes y \quad (x\in\Sym^c(W), y\in T^m(\Sym^c(V))).
\]
Therefore the complex $C\otimes_\alpha A$ is exactly $(\Sym^c(W)\otimes T(\Sym^c(V)),d)$.

Since the proposition is true for $W=0$, by Lemma~\ref{lem:comparision-lemma} we conclude that the map $\id\otimes p_V$ defines a quasi-isomorphism of $(\Sym^c(W)\otimes T(\Sym^c(V)),d)$ onto $(\Sym(W)\otimes \bigwedge V,d)$, where the new differential $d=d_{p_V\alpha}$ is given by $d(1\otimes y)=0$ and
\[
d(x_1 \cdots x_m \otimes y)=\sum_i x_1 \cdots \hat{x}_i \cdots x_m \otimes x_i \wedge y\quad (x_1,\dots,x_m\in W, y\in\medwedge V).
\]

It remains to show that the homomorphism of graded algebras $\Sym(W)\otimes \bigwedge V\to \bigwedge V/W$ defined by $\epsilon$ and the quotient map $V\to V/W$ gives a quasi-isomorphism of $(\Sym(W)\otimes \bigwedge V,d)$ onto $(\bigwedge V/W,0)$.
This is well-known in the case when $V=W$, that is, $\Sym(W) \otimes \bigwedge W$ is quasi-isomorphic to $\C$ concentrated in degree $0$.
The general case follows from this and the standard isomorphism of graded algebras $\bigwedge V \cong \bigwedge W \hotimes \bigwedge V/W$.
\ep

Let $\mfa$ be a finite dimensional complex Lie algebra and $\mfb<\mfa$ a Lie subalgebra.
For $n = 0, 1, \dots$, put
\[
\tB_{\mfa,\mfb}^n = U(\mfb) \otimes U(\mfa)^{\otimes n}.
\]
These spaces form a cochain complex by the same formula as in \eqref{eq:cohoch-diff}. The differential $d_\cH$ is equivariant with respect to the diagonal adjoint action of $\mfb$, so we also obtain a subcomplex $B_{\mfa, \mfb} = (\tB_{\mfa, \mfb})^\mfb$.

It is well-known that the symmetrization map defines an $\mfa$-equivariant coalgebra isomorphism of $\Sym^c(\mfa)$ onto $U(\mfa)$, see~\cite{MR1321145}*{Theorem~V.2.5}. Since the definition of $d_\cH$ uses only the coalgebra structures, it follows that the complexes $(\Sym^c(\mfb)\otimes T(\Sym^c(\mfa)),d)$ and $(\tB_{\mfa,\mfb},d_\cH)$ are isomorphic. Therefore Proposition~\ref{prop:co-Hoshchild-alg-model1} implies the following.

\begin{Prop}\label{prop:co-Hoshchild-alg-model2}
For any finite dimensional complex Lie algebra $\mfa$ and a Lie subalgebra $\mfb$, there is a $\mfb$-equivariant quasi-isomorphism of $(\tB_{\mfa, \mfb},d_\cH)$ onto $(\bigwedge \mfa/\mfb,0)$.
\end{Prop}

\begin{Rem}\label{rem:co-Hochshc-classes}
For any $X_1,\dots,X_n\in\mfa$, the element $1\otimes X_1\otimes\dots\otimes X_n\in\tB_{\mfa,\mfb}^n$ is a cocycle. By the definition of $\epsilon\otimes p_{\mfa/\mfb}$ and the symmetrization map, the image of this cocycle in $\bigwedge^n \mfa/\mfb$ is $\tilde X_1\wedge\dots\wedge\tilde X_n$, where $\tilde X_i$ is the image of $X_i$ in $\mfa/\mfb$.
\end{Rem}

Consider now a reductive algebraic subgroup $H<G$ as in Section~\ref{sec:geom-comput}.

\begin{Cor}
\label{cor:Buk-cohom}
The cohomology $(\tB_{\mfg,\mfh}, d_{\cH})$ is isomorphic to $\bigwedge  \mfg/\mfh$ and the cohomology of $(B_{\mfg,\mfh}, d_{\cH})$ is isomorphic to $(\bigwedge  \mfg/\mfh)^\mfh$. Furthermore, the embedding $\tB_{\mfg,\mfh}\to\tB_{G,H}$ is a quasi-isomorphism, and if~$H$ is connected, then $B_{\mfg,\mfh}\to B_{G,H}$ is a quasi-isomorphism as well.
\end{Cor}

\bp
The first statement follows immediately from Proposition~\ref{prop:co-Hoshchild-alg-model2} and the fact that computing the cohomology commutes with taking the $\mfh$-invariants in the reductive case.

Consider now the embedding $\tB_{\mfg,\mfh} \to\tB_{G,H}$.
This induces a pairing between cochains in $\tB_{\mfg,\mfh}$ and chains $f_0 \otimes \dots \otimes f_n$ in $\tilde B_{G,H}'^n = \mcO(H) \otimes \mcO(G)^{\otimes n}$ from the proof of Proposition~\ref{prop:comput-cohoch-mult-model}.
Unpacking the definitions, for the cocycle $1\otimes X_1\otimes\dots\otimes X_n\in\tB_{\mfg,\mfh}^n$ from Remark \ref{rem:co-Hochshc-classes}, we have
\[
\langle f_0 \otimes \dots \otimes f_n, 1\otimes X_1\otimes\dots\otimes X_n \rangle = f_0(e) \prod_{i=1}^n d_e f_i(X_i).
\]
Restricting to cycles of $\tilde B_{G,H}'$, this reduces to the canonical duality pairing between $\bigwedge \mfg/\mfh \cong \rH(\tB_{\mfg,\mfh})$ and $(\bigwedge \mfg/\mfh)' \cong \rH(\tilde B_{G,H}')$, see the identification of~\cite{MR0354655}*{Proposition~VII.2.5}.

Let us be more concrete.
Choose a basis $X_1,\dots,X_m$ in a complement of $\mfh$ in $\mfg$.
By Remark~\ref{rem:co-Hochshc-classes}, the cohomology classes of
$c_{i_1\dots i_n}=1\otimes X_{i_1}\otimes\dots\otimes X_{i_n}\in\tB^n_{\mfg,\mfh}$, $i_1<\dots<i_n$, form a basis in $\rH^n(\tB_{\mfg,\mfh})$. Choose right $H$-invariant functions $f_1,\dots, f_m\in\mcO(G)$ such that $d_e f_i(X_j)=\delta_{ij}$. Define functions $a_{i_1\dots i_n}$ on $H\times G^n$ by
\[
a_{i_1\dots i_n}(g_0,g_1,\dots,g_n)=\sum_{\sigma\in S_n}\sgn(\sigma)f_{i_1}(g_{\sigma(1)})\dots f_{i_n}(g_{\sigma(n)}),
\]
so that we have $\langle a_{i_1\dots i_n},  c_{j_1\dots j_n}\rangle = \delta_{i_1j_1}\dots\delta_{i_n j_n}$ for all $i_1<\dots<i_n$ and $j_1<\dots<j_n$.
Moreover, using that $f_i(g)=f_i(e)$ for $g\in H$, one can check that the $a_{i_1\dots i_n}$ are cycles in $\tB_{G,H}'$.
We thus obtain a nondegenerate pairing between $\rH^n(\tB_{\mfg,\mfh})$ and $\rH_n(\tilde B_{G,H}')$ which factors through $\rH^n(\tB_{\mfg,\mfh}) \to \rH^n(\tB_{G,H})$. Hence the last map is an isomorphism.

Finally, by considering the $\mfh$-invariants we conclude that if $H$ is connected, then $B_{\mfg,\mfh}\to B_{G,H}$ is a quasi-isomorphism as well.
\ep

\section{Spherical vectors}\label{sec:spherical}

The goal of this appendix is to prove the following result essentially due to Letzter~\cite{MR1742961}.
Since her setting and assumptions are slightly different, we will give a complete argument for the reader's convenience.

\begin{Thm}~\label{thm:Letzter1}
In the notation of Section \ref{sec:mult-alg-model-of-LK-coids}, assume $\mbt\in\mcT^*$ and $\lambda\in P_+$ are such that the highest weight $\mfg$-module $V_\lambda$ has a nonzero $\mfg^\theta_\mbt$-invariant vector.
Then this vector can be lifted to a $U^\mbt_h(\mfg^\theta)$-invariant vector in $V_\lambda\fps$.
\end{Thm}

In the non-Hermitian case we have $\mfg^\theta_\mbt=\mfg^\theta_0=\mfg^\theta$. In the Hermitian case, by Lemma~\ref{lem:mbt-vs-phi}, the Lie subalgebra $\mfg^\theta_\mbt<\mfg$ is conjugate to $\mfg^\theta$. Hence in both cases, by \cite{MR1920389}*{Theorem~8.49} and its proof, necessary (and, as we will see shortly, sufficient) conditions for the existence of a nonzero $\mfg^\theta_\mbt$-invariant vector in $V_\lambda$ are the following:
\begin{equation}\label{eq:Let1}
\text{the weight}\ \lambda\in P_+\ \text{vanishes on}\ \mfh^\Theta,
\end{equation}
\begin{equation}\label{eq:Let2}
(\lambda,\alpha_i^\vee)\in2\Z\ \text{for all}\ i\in I\ \text{such that}\ \Theta(\alpha_i)=-\alpha_i.
\end{equation}

\begin{Lem}\label{lem:Theta}
Given $i\in I$, we have $\Theta(\alpha_i)=-\alpha_j$ for some $j$ if and only if $i\in I\setminus X$ and $\alpha_i$ is orthogonal to $\alpha_k$ for all $k\in X$, in which case we also have $j=\tau_\theta(i)$.
\end{Lem}

\bp
Assume $\Theta(\alpha_i)=-\alpha_j$.
Since $\Theta(\alpha_k)=\alpha_k$ for all $k\in X$, we must have $i\in I\setminus X$, and since $\Theta(\alpha_i)+\alpha_{\tau_\theta(i)}\in\Z X$ and the set $I\setminus X$ is $\tau_\theta$-invariant, it follows also that $j=\tau_\theta(i)$.
As $\Theta(\alpha_i)=-w_X\alpha_{\tau_\theta(i)}$, we have
\[
0\ge(\alpha_j,\alpha_k)=-(\Theta(\alpha_i),\alpha_k)=(w_X\alpha_{\tau_\theta(i)},\alpha_k)=(\alpha_{\tau_\theta(i)},w_X\alpha_k)\ge0,
\]
for all $k\in X$, where the last inequality holds as $w_X\alpha_k\in \Phi^-_X$.
Therefore $\alpha_j=\alpha_{\tau_\theta(i)}$ is orthogonal to $\alpha_k$ for all $k\in X$, hence the same is true for $\alpha_i$ as well.
This proves the lemma in one direction, the other direction is obvious.
\ep

It is well-known that $(U,U^\theta)$, or equivalently $(U,K_\mbt)$, is a Gelfand pair.
Since every $U^\mbt_h(\mfg^\theta)$-invariant vector reduces modulo $h$ to a $K_\mbt$-invariant vector, to prove Theorem~\ref{thm:Letzter1} it therefore suffices to show that there is a nonzero $U^\mbt_h(\mfg^\theta)$-invariant vector whenever conditions~\eqref{eq:Let1} and~\eqref{eq:Let2} are satisfied.

Denote by $\K$ the field $\C[h^{-1},h\rrbracket$.
Instead of working with $U_h(\mfg)$ and $U_h^\mbt(\mfg^\theta)$, we extend the scalars to~$\K$. Recall that we denote $e^h$ by $q$. Consider the $\K$-subalgebra $U_q(\mfg)$ of $U_h(\mfg)\otimes_{\C\fps}\K$ generated by $E_i$, $F_i$ and $K_\omega$ ($\omega\in P$). Consider also the $\K$-subalgebra $U_q^\mbt(\mfg^\theta)$ generated by $K_\omega$ ($\omega\in P^\Theta$), $K_i^{\pm1}$, $E_i$, $F_i$ ($i\in X$) and $B_i$ ($i\in I\setminus X$). Then $V^q_\lambda=V_\lambda\fps\otimes_{\C\fps}\K=V_\lambda\otimes_\C\K$ is the irreducible $U_q(\mfg)$-module with highest weight $\lambda$. If we can show that it contains a nonzero $U^\mbt_q(\mfg^\theta)$-invariant vector $v$, then $h^n v$ becomes a $U^\mbt_h(\mfg^\theta)$-invariant vector in $V_\lambda\fps$ for $n\in\N$ large enough.
To prove that such a vector $v$ exists it is enough, in turn, to show that if $\xi_\lambda\in V^q_\lambda$ is the highest weight vector, then $\xi_\lambda\not\in U_q^\mbt(\mfg^\theta)^+\xi_\lambda$, where $U_q^\mbt(\mfg^\theta)^+$ denotes the augmentation ideal of $U_q^\mbt(\mfg^\theta)$.
Indeed, since $U_q^\mbt(\mfg^\theta)$ is $*$-invariant, the $U_q^\mbt(\mfg^\theta)$-module $V^q_\lambda$ is completely reducible by~\cite{MR1742961}*{Theorem~3.3}.
Then the projection of $\xi_\lambda$ onto a complementary submodule to $U_q^\mbt(\mfg^\theta)^+\xi_\lambda$ is a nonzero invariant vector.
Therefore it suffices to establish the following result.

\begin{Thm}[cf.~\cite{MR1742961}*{Theorem~4.3}]\label{thm:Letzter2}
Assume $\mbt\in\mcT$ and $\lambda\in P_+$ is a weight satisfying conditions~\eqref{eq:Let1} and~\eqref{eq:Let2}.
Then for the highest weight vector $\xi_\lambda\in V^q_\lambda$ we have $\xi_\lambda\notin U_q^\mbt(\mfg^\theta)^+\xi_\lambda$.
\end{Thm}

Note that for this result we no longer need $*$-invariance, so we can take any parameter $\mbt\in\mcT$.
The proof also works for any field extension of $\Q(q^{1/d})$ in place of $\K$ (with parameters $c_i$ and $s_i$ taken from this field), where $d$ is the determinant of the Cartan matrix.

We start by analyzing the rank one case.

\begin{Lem}\label{lem:sl2}
Consider $\mfg=\mfsl_2(\C)$ and the element $B=F-c E K^{-1}+s(K^{-1}-1)\in U_q(\mfsl_2)$, with $c\in\K^\times$ and $s\in\K$.
Then, for every $n\in\Z_+$, the highest weight module $V^q_n$ contains a nonzero vector killed by $B$.
\end{Lem}

\bp
For $n=0$ the lemma is obvious.
For $n=1$, the $U_q(\mfsl_2)$-module $V^q_1$ has the basis $\xi_1,F\xi_1,F^2\xi_1$ over~$\K$, and the actions of $E$ and $K$ on this basis are given by
\begin{align*}
K F^k \xi_1 &= q^{2-2k} F^k \xi_1,&
E F^k \xi_1 &= [2-(k-1)]_q[k]_q F^{k-1}\xi_1.
\end{align*}
One can then easily check that the vector
\[
v=\xi_1+\frac{s(1-q^2)}{c q^2 [2]_q}F\xi_1+\frac{1}{c q^2[2]_q}F^2\xi_1
\]
lies in the kernel of $B$.

For $n\ge 2$, the vector $\xi_1^{\otimes n}\in (V^q_1)^{\otimes n}$ has weight $n$ and generates a $U_q(\mfsl_2)$-submodule isomorphic to~$V^q_n$.
As
$\Delta_q(B)=B\otimes K^{-1}+1\otimes B$, the vector $v^{\otimes n}\in (V^q_1)^{\otimes n}$ is killed by $B$.
Since its weight~$n$ component is nonzero, the projection of this vector onto $V^q_n\subset (V^q_1)^{\otimes n}$ is a nonzero vector killed by~$B$.
\ep

To deal with the general case, let us introduce the following notation.
For a multi-index $J=(j_1,\dots,j_n)$, define $F_J=F_{j_1}\cdots F_{j_n}$ and $B_J=B_{j_1}\cdots B_{j_n}$.
We also let $\operatorname{wt}(J)=\alpha_{j_1}+\dots+\alpha_{j_n}$.
Denote by $U_-$ the unital $\K$-subalgebra of $U_q(\mfg)$ generated by the elements $F_j$, $j\in I$, by $\mcM_{X,+}$ the unital $\K$-subalgebra of $U^\mbt_q(\mfg^\theta)$ generated by the elements $E_j$, $j\in X$, and by $U_\Theta$ the $\K$-algebra generated by the elements $K_\omega$, $\omega\in P^\Theta$.
Fix $\lambda\in P_+$.

\begin{Lem}\label{lem:basis1}
Choose a finite collection $\mcJ$ of multi-indices such that $\xi_\lambda$ and $F_J\xi_\lambda$, $J\in\mcJ$, form a basis of $V^q_\lambda$ over~$\K$.
Then the vectors $\xi_\lambda$ and $B_J\xi_\lambda$, $J\in\mcJ$, also form a basis of $V^q_\lambda$.
\end{Lem}

\bp
By the definition of $B_i$, we have that $B_J\xi_\lambda$ equals $F_J\xi_\lambda$ plus a linear combination of vectors of higher weights.
A simple induction argument using this property gives the result.
\ep

For each $i\in X$, put $m_i=(\lambda,\alpha_i^\vee)\in\Z_+$.

\begin{Lem}\label{lem:basis2}
There are collections $J_i$, $i\in I$, of multi-indices such that the elements $1$, $F_J$ ($J\in\mcJ$, with~$\mcJ$ as in the previous lemma) and $F_J F_i^{m_i+1}$ ($J\in \mcJ_i$, $i\in I$) form a basis of $U_-$ over $\K$.
Moreover the elements $1$, $B_J$ ($J\in\mcJ$) and $B_J B_i^{m_i+1}$ ($J\in \mcJ_i$, $i\in I$) form a basis of the right $U_\Theta\mcM_{X,+}$-module $U_q^\mbt(\mfg^\theta)$.
More generally, for any choice of degree $m_i+1$ polynomials $p_i\in\K[x]$, the elements $1$, $B_J$ ($J\in\mcJ$) and $B_J p_i(B_i)$ ($J\in \mcJ_i$, $i\in I$) form a basis of the right $U_\Theta\mcM_{X,+}$-module $U_q^\mbt(\mfg^\theta)$.
\end{Lem}

\bp
Consider the Verma module $L_\lambda$ with highest weight vector $v_\lambda$ of weight $\lambda$.
The first part of the lemma follows from the well-known facts that the map $U_-\to L_\lambda$, $a\mapsto av_\lambda$, is a linear isomorphism and the kernel of the quotient map $L_\lambda\to V_\lambda$ is $\sum_i U_-F_i^{m_i+1}v_\lambda$.

The second part of the lemma follows then from~\cite{MR3269184}*{Proposition~6.2}.
To be more precise, some of our generators $B_i$ differ from the ones used by Kolb by scalar summands.
Let us denote Kolb's generators by $\tilde B_i$.
Then every element $B_J$ equals $\tilde B_J$ plus a linear combination of the elements $\tilde B_{J'}$ with $\operatorname{wt}(J')<\operatorname{wt}(J)$.
Similarly to the previous lemma, we see that whenever $\{\tilde B_J\}_{J\in\tilde\mcJ}$ is a basis of the right $U_\Theta\mcM_{X,+}$-module $U_q^\mbt(\mfg^\theta)$, then $\{B_J\}_{J\in\tilde\mcJ}$ also forms a basis.
For the same reason we can add to every~$B_J$ any linear combination of the elements $B_{J'}$ with $\operatorname{wt}(J')<\operatorname{wt}(J)$ and still get a basis.
In particular, we can replace every element of the form $B_J B_i^{m_i+1}$ by $B_J p_i(B_i)$.
\ep

\begin{Lem}\label{lem:poly}
If $\lambda\in P_+$ satisfies~\eqref{eq:Let2}, we can find for all $i\in I$, degree $m_i+1$ polynomials $p_i\in\K[x]$ such that $p_i(B_i)\xi_\lambda=0$ and $p_i(0)=0$.
\end{Lem}

\bp
Consider three cases.
If $i\in X$, then $B_i=F_i$ and we can take $p_i(x)=x^{m_i+1}$.

Next, assume $i\in I\setminus X$ but $\Theta(\alpha_i)\ne-\alpha_i$.
Then $i\not\in I_\ns$, hence $s_i=0$ and
\[
B_i= F_i - c_i z_{\tau_\theta(i)}T_{w_X}(E_{\tau_\theta(i)})K_i^{-1}.
\]
As $-\Theta(\alpha_i)=w_X\alpha_{\tau_\theta(i)}\in\Delta^+$ is different from $\alpha_i$, we have $-k\Theta(\alpha_i)-(n-k)\alpha_i\not\le0$ for any $n\ge k\ge1$.
Since any product of $k$ elements $T_{w_X}(E_{\tau_\theta(i)})$ and $n-k$ elements $F_i$ has weight $-k\Theta(\alpha_i)-(n-k)\alpha_i$, it must therefore kill $\xi_\lambda$.
It follows that
$B_i^n\xi_\lambda=F_i^n\xi_\lambda$ for any $n\ge1$.
In particular, we have $B_i^{m_i+1}\xi_\lambda=0$, so in this case we can again take $p_i(x)=x^{m_i+1}$.

Finally, assume $i$ is such that $\Theta(\alpha_i)=-\alpha_i$.
Then, by Lemma~\ref{lem:Theta}, we have $i\in I_\ns$, hence
\[
B_i=F_i - c_i z_i E_i K_i^{-1} + s_i\kappa_i\frac{K_i^{-1}-1}{q^{-d_i}-1}.
\]
The elements $E_i$, $F_i$, $K_i^{\pm1}$ generate a copy of $U_{q^{d_i}}(\mfsl_2)$ in $U_q(\mfg)$.
By acting on $\xi_\lambda$ we get a spin $\frac{m_i}{2}$ $U_{q^{d_i}}(\mfsl_2)$-module with basis $\xi_\lambda,F_i\xi_\lambda,\dots,F_i^{m_i}\xi_\lambda$ over $\K$.
By Lemma~\ref{lem:basis1}, the elements $\xi_\lambda,B_i\xi_\lambda,\dots,B_i^{m_i}\xi_\lambda$ also form a basis.
As $\frac{m_i}{2}\in\Z_+$ by assumption, we can apply Lemma~\ref{lem:sl2} and conclude that there is a nonzero polynomial $f\in\K[x]$ of degree $m\le m_i$ such that $B_if(B_i)\xi_\lambda=0$.
Hence we can take $p_i(x)=x^{m_i+1-m}f(x)$.
\ep

\bp[Proof of Theorem~\ref{thm:Letzter2}]
Take $a\in U_q^\mbt(\mfg^\theta)^+$.
We want to show that $a\xi_\lambda\ne\xi_\lambda$.
By Lemma~\ref{lem:basis2} we can write~$a$~as
\[
a_0+\sum_{J\in\mcJ}B_J a_J+\sum_{i\in I}\sum_{J\in\mcJ_i}B_J p_i(B_i)b_{i;J},
\]
where $a_0$, $a_J$ and $b_{i;J}$ are in $U_\Theta\mcM_{X,+}$ and $p_i$ are the polynomials from Lemma~\ref{lem:poly}.
Note that since $\epsilon(B_i)=0$ and the polynomials $p_i$ have zero constant terms, we must have $\epsilon(a_0)=0$.
By assumption~\eqref{eq:Let1} we have $y\xi_\lambda=\epsilon(y)\xi_\lambda$ for every $y\in U_\Theta\mcM_{X,+}$.
Hence
\[
a\xi_\lambda=\sum_{J\in\mcJ}\epsilon(a_J)B_J\xi_\lambda,
\]
which is different from $\xi_\lambda$ by our choice of $\mcJ$.
\ep

\begin{Rem}\label{rem:generic-spherical}
Theorem~\ref{thm:Letzter1} remains true for $\mbt\in\mcT^*_\C$ if we exclude a finite set (depending on $\lambda$) of values of $s_o^{(0)}$ (S-type) or $c_o^{(0)}$ (C-type). Indeed, let us look for spherical vectors of the form $\xi_\lambda+\sum_{J\in\mcJ}c_J F_J\xi_\lambda$, $c_J\in\K$, where $\mcJ$ is as in Lemma~\ref{lem:basis1}. The sphericity condition gives us a system of linear equations for $c_J$ with coefficients that are rational functions (with complex coefficients) in $q$, $s_o$ or $c_o$. A spherical vector exists if and only if the rank of the matrix $A$ of this system is the same as the rank of the augmented matrix $B$. Take a submatrix of~$A$ of maximal size giving a nonzero minor for some $\mbt\in\mcT^*$. Then the lowest order nonzero term of the minor's expansion in $h$ is a rational function of $s_o^{(0)}$ or $c_o^{(0)}$, so the corresponding minor remains nonzero for all $\mbt\in\mcT^*_\C$ excluding a finite set of values of $s_o^{(0)}$ or $c_o^{(0)}$. On the other hand, if we take any larger minor of $B$ and consider its expansion in $h$, then the coefficients will be rational functions in the parameters $s_o^{(n)}$ or $c_o^{(n)}$. These functions must vanish for all purely imaginary (S-type) or real (C-type) values of the parameters, hence they are identically zero.
\end{Rem}

\section{Coideals as deformations}
\label{app:univ-env-alg-model}

For $\mbt\in\mcT$, recall $\mfg^\theta_\mbt < \mfg$ from Definition \ref{def:limit-Lie-alg-of-LK-coideal}.

\begin{Prop}\label{Propform}
For every $\mbt\in\mcT$, the $\C\fpser$-module $U^\mbt_h(\mfg^\theta)$ is topologically free and the homomorphism $U_h(\mfg)\to U(\mfg)$ induces an isomorphism $U^\mbt_h(\mfg^\theta)/h U^\mbt_h(\mfg^\theta)\cong U(\mfg^\theta_\mbt)$.
\end{Prop}

\bp
Since $U_h(\mfg)$ is topologically free and $U^\mbt_h(\mfg^\theta)\subset U_h(\mfg)$ is closed, to prove both statements it suffices to show that
\[
U^\mbt_h(\mfg^\theta)\cap h U_h(\mfg)=h U^\mbt_h(\mfg^\theta).
\]
For this, in turn, it is enough to check that for all $n\ge2$ we have
\begin{equation}\label{EqCoidhAlg}
U^\mbt_h(\mfg^\theta)\cap h U_h(\mfg)\subset h U^\mbt_h(\mfg^\theta)+h^n U_h(\mfg).
\end{equation}

Let
\[
\mcJ_X=\{J = (i_1,\dots,i_k)\mid i_j\in X\}\subset \cup_{k=0}^{\infty} X^k
\]
be such that the elements $X_{\alpha_{i_1}}\cdots X_{\alpha_{i_k}}$ with $(i_1,\dots,i_k)\in J$ form a basis of $U(\mfn^+_X)\subset U(\mfg_X)$. Consider also a larger set $\mcJ\subset \cup_{k=0}^{\infty} I^k$ giving a basis of $U(\mfn^+)$. For $J=(i_1,\dots,i_k)\in\mcJ$, put
\[
E_J=E_{i_1}\cdots E_{i_k}\in U_h(\mfg)\quad\text{and}\quad B_J=B_{i_1}\cdots B_{i_k}\in U_h(\mfg).
\]
Let $H_1',\dots,H_l'$ be a basis of $\mfh$.
Then by the proof of \cite{MR3269184}*{Proposition 6.1}, the image of the set
\[
\{E_{J}(H_1')^{k_1}\cdots (H_l')^{k_l} B_{J'}\mid J \in \mcJ,k_i\ge0, J'\in \mcJ\}
\]
in $U(\mfg)$ is a basis. This implies that this set is a basis of the free $\C\fps/(h^n)$-module $U_h(\mfg)/h^n U_h(\mfg)$.
On the other hand, the same argument as in the proof of \cite{MR3269184}*{Proposition 6.2} shows that
\[
\{E_{J}(H_1')^{k_1}\cdots (H_l')^{k_l} B_{J'}\mid J \in \mcJ_X,k_i\ge0, J'\in \mcJ\}
\]
generates $U_{h}(\mfg_{\mbt}^{\theta})/(U_{h}(\mfg_{\mbt}^{\theta})\cap h^n U_h(\mfg))$ as a $\C\fpser/(h^n)$-module. These two facts clearly imply~\eqref{EqCoidhAlg}.
\ep

The following lemma slightly generalizes the second Whitehead lemma.

\begin{Lem}\label{LemSecWhite}
Assume $\mfa$ is a finite dimensional Lie algebra over a field of characteristic zero such that the derived Lie subalgebra $[\mfa,\mfa]$ is semisimple and has codimension $1$.
Then $\rH^2(\mfa,V) = 0$ for all finite dimensional $\mfa$-modules $V$.
\end{Lem}

\begin{proof}
This follows from \cite{MR0074780}*{Proposition 1} and the usual second Whitehead lemma, see also \cite{MR2431117}.
\end{proof}

Assume now that $\mbt\in\mcT^*_\C$. In the non-Hermitian case the set $\mcT^*_\C$ consists of one point and we have $\mfg^\theta_0=\mfg^\theta$. In the Hermitian case, by Lemma~\ref{lem:mbt-vs-phi} we have $\mfg^\theta_\mbt\cong\mfg^\theta$. Therefore in both cases the above lemma applies to $\mfg^\theta_\mbt$. As $U^\mbt_h(\mfg^\theta)$ is a deformation of $U(\mfg^\theta_\mbt)$ by Proposition~\ref{Propform}, this leads to the following result.

\begin{Prop} \label{prop:coideal-deform}
For all $\mbt\in\mcT^*_\C$, the isomorphism $U^\mbt_h(\mfg^\theta)/h U^\mbt_h(\mfg^\theta)\cong U(\mfg^\theta_\mbt)$ lifts to an isomorphism $U^\mbt_h(\mfg^\theta)\cong U(\mfg^\theta_\mbt)\fps$ of $\C\fps$-algebras.
\end{Prop}

Note that since $\mfg^\theta_\mbt$ for $\mbt=(\mbc,\mbs)\in\mcT$ depends only $\mbt^{(0)}$, the same result holds for every $\mbt\in\mcT$ such that $\mbt^{(0)}=\mbt'^{(0)}$ for some $\mbt'\in\mcT^*_\C$.

\end{document}